\documentclass[12pt]{amsart}

\usepackage[colorlinks=true,urlcolor=blue,
bookmarks=true,bookmarksopen=true,citecolor=blue
]{hyperref}

\usepackage{graphicx}

\usepackage{epsfig}
\usepackage{amscd}
\usepackage[mathscr]{eucal}
\usepackage{amssymb}
\usepackage{amsxtra}
\usepackage{amsmath}
\usepackage[all]{xy}
\usepackage{enumerate}

\usepackage{color}
\theoremstyle{plain}

\newtheorem{mainthm}{Theorem}
\newtheorem{thm}{Theorem}[subsection]

\newtheorem{cor}[thm]{Corollary}
\newtheorem{corspec}[thm]{Corollary$^*$}
\newtheorem{lem}[thm]{Lemma}
\newtheorem{prop}[thm]{Proposition}
\newtheorem{cnj}{Conjecture}

\theoremstyle{definition}
\newtheorem{dfn}[thm]{Definition}

\theoremstyle{remark}
\newtheorem{rem}[thm]{Remark}
\newtheorem*{remnonum}{Remark}

\newtheorem{ex}[thm]{Example}

\theoremstyle{plain}

\newcommand{\Qed}{\hfill \qedsymbol \medskip}

\newcommand{\hooklongrightarrow}{\lhook\joinrel\longrightarrow}

\newcommand{\R}{\mathbb{R}}
\newcommand{\Z}{\mathbb{Z}}

\newcommand{\N}{\mathbb{N}}
\newcommand{\C}{\mathbb{C}}

\newcommand{\La}{\Lambda}

\newcommand{\Crit}{\textnormal{Crit\/}}

\newcommand{\clift}{\mathbb{T}_{\textnormal{clif}}}
\newcommand{\sym}{\textnormal{symb\/}}

\newcommand{\pbaddress}{biran@math.tau.ac.il}
\newcommand{\ocaddress}{cornea@dms.umontreal.ca}

\begin{document}

\title{Rigidity and uniruling for Lagrangian submanifolds} \date{\today}

\thanks{The first author was partially supported by the ISRAEL SCIENCE
  FOUNDATION (grant No. 1227/06 *); the second author was supported by
  an NSERC Discovery grant and a FQRNT Group Research grant}

\author{Paul Biran and Octav Cornea} \address{Paul Biran, School of
  Mathematical Sciences, Tel-Aviv University, Ramat-Aviv, Tel-Aviv
  69978, Israel} \email{\pbaddress} \address{Octav Cornea, Department
  of Mathematics and Statistics University of Montreal C.P. 6128 Succ.
  Centre-Ville Montreal, QC H3C 3J7, Canada} \email{\ocaddress}

\bibliographystyle{alphanum}

%
%

\maketitle

%
%
\tableofcontents

\section{Introduction}\label{sec:intro}

The purpose of this paper is to explore the topology of monotone
Lagrangian submanifolds $L$ inside a symplectic manifold $M$ by
exploiting the relationships between the quantum homology of $M$ and
various quantum structures associated to the Lagrangian $L$. We show
that the class of monotone Lagrangians satisfies a number of
structural rigidity properties which are particularly strong when the
ambient symplectic manifold contains enough genus-zero
pseudo-holomorphic curves. Indeed, we will see that (very often) if
$M$ is ``highly'' uniruled by curves of area $A$, then $(M,L)$ (or
just $L$) is uniruled by curves of area strictly smaller than $A$ (see
\ref{subsubsec:uniruling} for the definition of the appropriate
notions of uniruling).

\subsection{Setting}
All our symplectic manifolds will be implicitly assumed to be
connected and tame (see~\cite{ALP}).  The main examples of such
manifolds are closed symplectic manifolds, manifolds which are
symplectically convex at infinity as well as products of such. All the
Lagrangians submanifolds will be assumed to be connected and closed
(i.e. compact, without boundary).

We start by emphasizing that our results apply to {\em monotone}
Lagrangians.  These are characterized by the fact that the morphisms:
$$\omega: \pi_{2}(M,L)\to \R\ \ ,\ \ \mu: \pi_{2}(M,L)\to \Z, $$ the
first given by integration and the second by the Maslov index, are
proportional with a positive proportionality constant $\omega= \eta
\mu$ with $\eta>0$. Moreover, we will include here in the definition
of the monotonicity the assumption that the minimal Maslov index,
$$N_{L} = \min \{ \mu(\alpha) \mid \alpha \in \pi_2(M,L),
\, \mu(\alpha) > 0 \}$$ of a homotopy class of strictly positive
Maslov index is at least two, $N_{L}\geq 2$. If $L$ is monotone, then
$M$ is also monotone and $N_{L}$ divides $2C_{M}$ where $C_{M}$ is the
minimal Chern number of $M$
$$C_{M}=\min\{c_{1}(\alpha) | \alpha\in \pi_{2}(M), c_{1}(\alpha)>0 \}~.~$$

\subsubsection{Size of Lagrangians}

Fix a Lagrangian submanifold $L \subset M$.

We say that a symplectic embedding of the closed, standard symplectic
ball of radius $r$, $e:(B^{2n}(r),\omega_{std}) \to (M,\omega)$, is
{\em relative} to $L$ if $$e^{-1}(L)=B^{2n}(r)\cap \R^{n}~.~$$ These
types of embedding were first introduced and used in
\cite{Bar-Cor:Serre} and \cite{Bar-Cor:NATO}.

\

Consider now a vector $v_{p,q}=(r_{1},\ldots, r_{p};\rho_{1},\ldots,
\rho_{q})\in (\R^{+})^{p+q}$.  We will not allow for both $p$ and $q$
to vanish. If just one does, say $p=0$, we will use the notation
$v_{0,q}=(\emptyset;\rho_{1},\ldots,\rho_{q})$.

\begin{dfn}
   {\em The mixed symplectic packing number, $w(M,L:v_{p,q})$, of type
     $v_{p,q}=(r_{1},\ldots, r_{p};\rho_{1},\ldots, \rho_{q})$ of
     $(M,L)$} is defined by:
   $$w(M,L:v_{p,q})=\sup_{\tau> 0} \Bigl(\sum_{i=1}^{p}
   \pi (\tau r_{i})^{2}+\frac{1}{2} \sum_{j=1}^{q}\pi
   (\tau\rho_{j})^{2}\Bigr)$$ where the supremum is taken over all
   $\tau$ such that there are mutually disjoint symplectic embeddings
   $$f_{i}:(B^{2n}(\tau r_{i}),\omega_{0})
   \to (M\backslash L) ,\ 1\leq i\leq p,\
   e_{j}:(B^{2n}(\tau\rho_{j}),\omega_{0})\to M,\ 1\leq j\leq q$$ so
   that the $e_{j}$'s are embeddings relative to $L$.
\end{dfn}

The most widespread example of such vectors $v_{p,q}$ have all their
components equal to $1$. We also notice that $w(M):=w(M,\emptyset:
(1;\emptyset))$ is the well-known Gromov width of $M$: the supremum of
$\pi r^{2}$ over all symplectic embeddings of $B^{2n}(r)$ into $M$. A
similar notion has been introduced in \cite{Bar-Cor:NATO}, see also
\cite{Cor-La:Cluster-2}, to ``measure'' Lagrangians: the width of a
Lagrangian, $w(L)$, is the supremum of $\pi r^{2}$ over all symplectic
embeddings of $B^{2n}(r)$ which are relative to $L$.  With our
conventions, $w(L)=2w(M,L:(\emptyset ;1))$. Moreover, $w(M\backslash
L)$, the Gromov width of the complement of $L$, is given by $w(M,L:
(1;\emptyset))$.

\subsubsection{Uniruling}\label{subsubsec:uniruling}
The main technique used to prove width and packing estimates is based
on establishing uniruling results.

\begin{dfn}
   We say that $(M,L)$ {\em is uniruled of type $(p,q)$ and order $k$}
   (or shorter, $(M,L)$ is {\em $(p,q)$ - uniruled of order $k$}) if
   for any $p$ distinct points $P_{i}\in M\backslash L,\ 1\leq i\leq
   p$, and any $q$ distinct points, $Q_{j}\in L, \ 1\leq j\leq q$,
   there exists a Baire second category (generic) family of almost
   complex structures $\mathcal{J}$ with the property that for each
   $J\in\mathcal{J}$ there exists a non-constant $J$-holomorphic disk
   $u:(D^{2},\partial D^{2})\to (M,L)$ so that $P_{i}\in u(Int(
   D^{2})), \forall i$, $Q_{j}\in u(\partial D^{2}), \forall j$, and
   $\mu(u)\leq k$. In case $L$ is void, we take $q=0$, and instead of
   a disk, $u$ is required to be a non-constant $J$-holomorphic sphere
   so that $P_{i}\in u(S^{2}), \forall i$.
\end{dfn}
If $(M,\emptyset)$ is $(p,0)$-uniruled we will say that $M$ {\em is
  uniruled of type $p$}.  Thus the usual notion of uniruling for a
symplectic manifold - $M$ is uniruled if through each point of $M$
passes a $J$-sphere in some fixed homotopy class in $\pi_{2}(M)$ - is
equivalent in our terminology with $M$ being $1$-uniruled. Similarly,
in case $(M,L)$ is $(0,q)$ - uniruled we will say that $L$ is
$q$-uniruled. Additionally, if $q=1$ we say that $L$ is uniruled.

The relation with packing is given by the following fact:

\begin{lem}\label{lem:area_estimate}
   If the pair $(M,L)$ is $(p,q)$ - uniruled of order $k$ , then for
   any vector $v_{p,q}=(r_{1},\ldots r_{p}; \rho_{1},\ldots \rho_{q})$
   the mixed symplectic packing number $w(M,L : v_{p,q})$ verifies:
   $$w(M,L : v_{p,q})\leq \eta k $$ where $\eta$ is the monotonicity
   constant, $\eta=\omega/\mu$.
\end{lem}

The proof of this is standard and is a small modification of an
argument of Gromov \cite{Gr:phol}. It comes down to the following
simple remark which also explains the $1/2$ factor in the definition
of $w(M,L : v_{p,q})$. If a $J$-curve $u$ with boundary on a
Lagrangian goes through the center of a standard symplectic ball or
radius $r$ embedded in $M$ {\em relative} to $L$ so that $J$ coincides
with the standard almost complex structure inside the ball, then we
have $\pi r^{2}/2\leq \int u^{\ast}\omega$. This is in contrast to the
case when $u$ has no boundary, when the inequality is, as is
well-known, $\pi r^{2}\leq \int u^{\ast}\omega$.

\

The simplest way to detect algebraically that $M$ is $p$-uniruled is
to find some class $\alpha\in \pi_{2}(M)$ and $r\geq 1$ so that, for
distinct points $P_{1},\ldots, P_{p}$, and a generic $J$, the
evaluation at $r$ distinct points on the $J$-spheres of class $\alpha$
which pass through the fixed points $P_{i}, 1\leq i\leq p$, has a
homologically non-trivial image in the product $M^{\times r}$. This
can be translated in terms of Gromov-Witten invariants: if there exist
$\alpha\in \pi_{2}(M)$ and classes $a_{i}\in H_{\ast}(M;\Z_{2}), 1\leq
i\leq r$, so that
\begin{equation}\label{eq:GW_unir}
   GW(pt,\ldots,pt,a_{1},\ldots,a_{r};\alpha)\not=0 \
\end{equation} where the class of the point, $pt\in
H_{0}(M;\Z_{2})$, appears $p$ times, then $M$ is clearly $p$-uniruled
(we recall that the Gromov-Witten invariant $GW(b_{1},\ldots
b_{s};\alpha)$ counts - in this paper with $\Z_{2}$ coefficients - the
number of $J$-spheres in the homotopy class $\alpha\in \pi_{2}(M)$
which each pass through generic cycles representing the homology
classes $b_{i}\in H_{\ast}(M;\Z_{2})$).

\begin{rem}
   In case $p=1$ the condition in (\ref{eq:GW_unir}) gives the notion
   of ``strong uniruled'' which appears in McDuff \cite{McD:Ham_U}
   (with the additional constraint that the degree of the homology
   classes $a_{i}$ are even).
\end{rem}
If we fix $p\geq 2$ and add the requirement that $r=1$, then, by the
splitting property of Gromov-Witten invariants, the uniruling
condition implies $GW(pt,pt,a;\alpha')\not=0$ for some choices of
$a\in H_{\ast}(M;\Z_{2})$ and $\alpha'\in \pi_{2}(M)$. Of course, this
can be re-interpreted in quantum homology as the relation $[pt]\ast
a=[M] e^{\alpha'}+ \cdots$ where $[pt]\in H_{0}(M;\Z_{2})$ represents
the point, $[M]\in H_{2n}(M,\Z_{2})$ is the fundamental class, and the
Novikov ring used is $\Z_{2}[\pi_{2}(M)]$.

A stronger condition will play a key role in the following. Consider
the quantum homology of $M$ with coefficients in
$\Gamma=\Z_{2}[s^{-1},s]$ with $\deg(s)=-2C_{M}$ (where $C_{M}$ is the
minimal Chern number). This is
$QH_{\ast}(M)=H_{\ast}(M;\Z_{2})\otimes\Gamma$.

\begin{dfn} With the notation above we say that $M$ is {\em point
     invertible} if $[pt]$ is invertible in $QH(M)$.  This implies
   that there exists $0 \neq a_{0}\in H_{\ast}(M;\Z_{2})$, $a_{1}\in
   H_{\ast}(M;\Z_{2})\otimes \Z_{2}[s]$, and $k\in \N$ so that, if we
   put $a=a_{0}+a_{1}s$, then in $QH_{\ast}(M)$ we have $$[pt]\ast a =
   [M] s^{k/2C_{M}}~.~$$ The natural number $k$ above is uniquely
   defined and we specify it by saying that $M$ is {\em point
     invertible of order $k$}.
\end{dfn}

Of course, as indicated above, a point-invertible manifold is
$2$-uniruled.  The class of point invertible manifolds includes, for
example, $\C P^{n}$ and the quadric $Q^{2n}\subset \C P^{n+1}$.
Moreover, in view of the product formula for Gromov-Witten invariants,
this class is closed with respect to products.

In general, no such direct algebraic criteria can be found to test the
existence of mixed uniruling of the pair $(M,L)$ or even whether $L$
itself is uniruled because relative Gromov-Witten invariants are not
well-defined in full generality.

\subsection{Main results}
Recall that by the work of Oh \cite{Oh:HF1} if $L \subset M$ is a
monotone Lagrangian - which we will assume from now on, then the Floer
homology $HF(L):=HF(L,L)$ with $\Z_{2}$-coefficients is well-defined
(the construction will be briefly recalled later in the paper). Floer
homology is easily seen to be isomorphic (in general not canonically)
to a quotient of a sub-vector space of $H(L;\Z_{2})\otimes \Lambda$.
Here $H(L;\Z_{2})$ is singular homology and $\Lambda =
\Z_{2}[t^{-1},t]$  where the degree of
$t$ is $|t|=-N_{L}$ (see \S\ref{subsubsec:defin_alg_str} g for the
precise definition). Thus, there are two extremal cases:

\begin{dfn}If $HF(L)=0$ we say that $L$ is {\em narrow}; if there
   exists an isomorphism $HF(L)\cong H(L;\Z_{2})\otimes \Lambda$, then
   we call $L$ {\em wide}. Note that the latter isomorphism is not
   required to be canonical in any sense.
\end{dfn}

Remarkably, all known monotone Lagrangians are either narrow or wide.
We will see that the dichotomy {\em narrow - wide} plays a key role in
structuring the properties of monotone Lagrangians. In particular,
narrow Lagrangians tend to be small in the sense that their width is
bounded and non-narrow ones tend to be {\em barriers} in the sense of
\cite{Bi:Barriers}: the width of their complement tends to be smaller
than that of the ambient manifold. Wide Lagrangians are even more
rigid.

\subsubsection{Geometric rigidity. }
We start with one result concerning narrow Lagrangians which also
shows that the ``narrow - wide'' dichotomy holds in a variety of cases
(related results are due to Buhovsky~\cite{Bu:products}):

\begin{thm}\label{thm:floer=0-or-all}
   Let $L^n \subset M^{2n}$ be a monotone Lagrangian.  Assume that its
   singular homology $H_*(L;\mathbb{Z}_2)$ is generated as a ring
   (with the intersection product) by $H_{\geq n-l}(L;\mathbb{Z}_2)$.
   \begin{enumerate}[i.]
     \item If $N_{L} > l$, then $L$ is either wide or narrow.
      Moreover, if $N_{L}> l+1$, then $L$ is wide.
     \item In case $L$ is narrow, then $L$ is uniruled of order $K$
      with $K=\max\{l+1, n+1-N_{L}\}$ if $N_{L}< l+1$, and $K=l+1$ if
      $N_{L}=l+1$. Moreover, $w(L)\leq 2K\eta$ where $\eta$ is the
      monotonicity constant. In particular, the width of narrow
      monotone Lagrangians $L$ is ``universally'' bounded: $w(L) \leq
      2(n+1) \eta$. In case $L$ is narrow and not a homology sphere
      the bound can be improved to $w(L) \leq 2n \eta$.
      \label{I:dichotomoy-narrow}
   \end{enumerate}
\end{thm}
Note that the finiteness of $w(L)$ from
point~\ref{I:dichotomoy-narrow} is not trivial since $M$ is not
assumed to be compact nor of finite volume or width.  Moreover, when
$L$ is not narrow, $w(L)$ might be infinite. For example,
zero-sections in cotangent bundles (which are wide) have infinite
width.  A class of Lagrangians for which
Theorem~\ref{thm:floer=0-or-all} gives non-trivial information is that
of monotone Lagrangian tori. In this case $H_*(L;\mathbb{Z}_2)$ is
generated by $H_{\geq n-1}(L;\mathbb{Z}_2)$ hence we can take $l=1$.
As $N_L \geq 2 > l$ we see that any monotone Lagrangian torus is
either narrow or wide. In case such a Lagrangian is narrow we have
$w(L) \leq 4 \eta$.

To obtain any meaningful uniruling results for Lagrangians which are
not narrow, the same example of zero sections in cotangent bundles
shows that some additional conditions need to be imposed on the
ambient manifold $M$.

\begin{thm}\label{theo:geom_rig}
   Let $L$ be a monotone Lagrangian in a symplectic manifold $M$ which
   is point-invertible of order $k$.
   \begin{enumerate}[i.]
     \item If $L$ is not narrow, then $(M,L)$ is uniruled of type
      $(1,0)$ of order $<k$. In particular,
      $$w(M\backslash L)\leq (k-N_{L})\eta~.~$$

     \item If $L$ is wide, then $L$ is uniruled of order $<k$ and we
      have:
      \begin{equation}\label{eq:mixed}w(L)+2w(M\backslash L)\leq
         2k\eta
      \end{equation} \label{I:geom_rig-non-narrow}
   \end{enumerate}
\end{thm}
We emphasize that the somewhat surprising  part of the statement is that 
the uniruling involving $L$ is of order strictly lower than $k$ whenever
$M$ is point invertible of order precisely $k$ (in particular, it might happen
 that $M$ itself is uniruled of order precisely $k$).
\begin{rem}
   \begin{enumerate}[a.]
     \item There are a few additional immediate inequalities that are
      worth mentioning: as $M$ is uniruled we have $w(M)\leq k\eta$
      and so $w(L)\leq k\eta$.  Moreover, as $M$ is 2-uniruled, we
      have $w(M,\emptyset; (r_{1},r_{2};\emptyset))\leq k\eta$.
      Obviously, we always have $w(M,\emptyset;
      (r_{1},r_{2};\emptyset))\geq w(M,L; (r_{1};r_{2}))$.
     \item These general inequalities do not imply the inequality
      (\ref{eq:mixed}). Indeed, in contrast to $w(M,L;
      (r_{1};r_{2}))$, the two balls involved in estimating separately
      the width of $L$ and that of its complement are not required to
      be disjoint !
     \item A non-trivial consequence of point~i of the Theorem is that
      if $M$ is point-invertible of order $k$ and $L$ is non-narrow,
      then $N_{L}\leq k/2$.
     \item Assuming the setting of the
      point~\ref{I:geom_rig-non-narrow} of the Theorem we deduce from
      the fact that $L$ is uniruled of order $< k$, that $w(L)\leq
      2(k-N_{L})\eta$. However, this inequality lacks interest because
      $2(k-N_{L})\geq k$ (since $k\geq 2N_{L}$).
   \end{enumerate}
\end{rem}

\subsubsection{Corollaries for Lagrangians in $\C P^{n}$} We endow $\C
P^{n}$ with the standard K\"{a}hler symplectic structure
$\omega_{\textnormal{FS}}$ normalized so that $\int_{\mathbb{C}P^1}
\omega_{\textnormal{FS}} = 1$. With this normalization we have
${\mathbb{C}}P^n \setminus {\mathbb{C}}P^{n-1} \approx
\textnormal{Int\,} B^{2n}(\frac{1}{\sqrt{\pi}})$ hence $w(\C
P^{n})=1$. Note also that for every monotone Lagrangian $L \subset
{\mathbb{C}}P^n$ we have $\eta=1/(2n+2)$ and that $\C P^{n}$ is point
invertible of order $k=2n+2$.

\begin{cor}\label{cor:geom_rig}
   Let $L$ be a monotone Lagrangian in $\C P^{n}$.
   \begin{enumerate}[i.]
     \item At least one of the following inequalities is verified:
      \begin{enumerate}[a.]
        \item $w(L)\leq\frac{n}{n+1}$ \label{I:w-ineq-1}
        \item $w(\C P^{n}\backslash L)\leq \frac{n}{n+1}~.~$ Moreover,
         if $L$ is not narrow then possibility b holds and in fact we
         have $$w({\mathbb{C}}P^n \setminus L) \leq
         \frac{\bigl[\frac{2n}{N_L}\bigr] N_L}{2(n+1)}.$$
         \label{I:w-ineq-2}
      \end{enumerate}
      \label{I:geom_rig-i}
     \item If $L$ is wide, then we have
      $$w(L)+2w(\C P^{n}\backslash L)\leq 2~.~$$ \label{I:wide-pack-cpn}
   \end{enumerate}
\end{cor}
In case $L$ is not narrow, the inequality $w(\C P^{n}\backslash L)\leq
\frac{n}{n+1}$ follows directly from Theorem~\ref{theo:geom_rig}.  If
$L$ is narrow, as $L$ can not be a homology sphere (see
e.g.~\cite{Bi-Ci:closed}) we can take $l=n-1$ in Theorem
\ref{thm:floer=0-or-all} which then implies the inequality at
~\ref{I:geom_rig-i} a above.  Point ii of the Corollary follows from
point ii of Theorem~\ref{theo:geom_rig}.

Corollary~\ref{cor:geom_rig} implies in particular that for any
monotone Lagrangian in $\C P^{n}$ we have
\begin{equation}\label{eq:wi_ineq_gen}
   w(L)+w(\C P^{n}\backslash L)\leq 1+ \frac{n}{n+1} = 2-\frac{1}{n+1}
\end{equation}
or, in other words, {\em any} monotone Lagrangian in $\C P^{n}$ is
either a barrier (in the sense of~\cite{Bi:Barriers}) or its width is
strictly smaller than that of the ambient manifold.  For example, $\R
P^{n} \subset {\mathbb{C}}P^n$ verifies $w(\R P^{n})=1$ and $w(\C
P^{n}\backslash \R P^{n})=1/2$; for the Clifford torus $$\clift^{n} =
\{[z_0: \cdots: z_n] \in {\mathbb{C}}P^n \mid |z_0| = \cdots |z_n|
\}$$ we have $w(\clift^{n})\leq 2/(n+1)$ (an explicit construction due
to Buhovsky~\cite{Bu:packing} shows that we actually have an equality
here) and $w(\C P^{n}\backslash \clift^{n})=n/(n+1)$ so that for $n=2$
both~\ref{I:w-ineq-1} and~\ref{I:w-ineq-2} are sharp. Both $\R P^{n}$
and $\clift^{n}$ show that the inequality at~\ref{I:wide-pack-cpn} is
sharp. We do not know if the inequality~\eqref{eq:wi_ineq_gen} is
sharp.

\

\subsubsection{Spectral rigidity}\label{subsubsec:spec-rig}
To summarize the results above, monotone non-narrow Lagrangians (at
least) in appropriately uniruled symplectic manifolds {\em are
  geometrically rigid}. Of course, by standard Floer intersection
theory, monotone Lagrangians which are not narrow, are also rigid in
the sense that such a Lagrangian can not be disjoined from itself by
Hamiltonian deformation. We now present a different type of rigidity.

Let $\widetilde{Ham}(M)$ be the universal cover of the Hamiltonian
diffeomorphism group of a symplectic manifold $M$. Recall that, by
works of Oh \cite{Oh:spec-inv-1} and Schwarz \cite{Sc:action-spectrum}
we can associate to any $\phi \in \widetilde{Ham}(M)$ and any singular
homology class $\alpha\in H_{\ast}(M;\Z_{2})$ a spectral invariant,
$\sigma(\alpha,\phi)\in\R~.~$ See \S\ref{subsec:proofT3} for the
definition.

Here are two natural notions measuring the variation of an element
$\phi\in\widetilde{Ham}(M)$ on a Lagrangian submanifold $L \subset M$.

\begin{dfn}
   The {\em depth} and, respectively, the {\em height} of $\phi$ on
   $L$ are:
   \begin{align*}
      \mathrm{depth}_{L}(\phi) & =\sup_{[H]=\phi}
      \ \inf_{\gamma\in \Gamma(L)}\int_{S^{1}}H(\gamma(t),t)dt \\
      \mathrm{height}_{L}(\phi)& =\inf_{[H]=\phi} \
      \sup_{\gamma\in\Gamma(L)}\int_{S^{1}}H(\gamma(t),t)dt~,~
   \end{align*}
   where $\Gamma(L)$ stands for the space of smooth loops $\gamma :
   S^{1}\to L$, $H:M\times S^{1}\to \R$ is a normalized Hamiltonian,
   and the equality $[H]=\phi$ means that the path of Hamiltonian
   diffeomorphisms induced by $H$, $\phi^{H}_{t}$, is in the (fixed
   ends) homotopy class $\phi$.
\end{dfn}
\begin{thm}\label{thm:action} Let $L\subset M$ be a monotone
   non-narrow Lagrangian. Then for every $\phi\in\widetilde{Ham}(M)$:
   \begin{enumerate}[i.]
     \item We have $\sigma([M],\phi)\geq \mathrm{depth}_{L}(\phi)$
     \item If $M$ is point invertible of order $k$, then
   $$\ \sigma ([pt],\phi)\geq
   \mathrm{depth}_{L}(\phi)-k\eta ~.~$$
\end{enumerate}
\end{thm}

We will actually prove a more general statement than the one contained
in Theorem \ref{thm:action}, however, even this already has a
non-trivial consequence.

\begin{cor}\label{cor:inters}
   Any two non-narrow monotone Lagrangians in $\C P^{n}$ intersect.
\end{cor}

Here is a quick proof of this Corollary. First, the theory of spectral
invariants shows that for any manifold $M$ so that $QH_{2n}(M) =
\mathbb{Z}_2[M]$ and any $\phi\in\widetilde{Ham}(M)$ we have
$\sigma([pt],\phi^{-1})=-\sigma([M],\phi)$. This is the case for $M=\C
P^{n}$ and thus, as for $\C P^{n}$ we have $k=n+2$, $\eta=1/(2n+2)$,
by Theorem \ref{thm:action} ii we deduce for any $\phi$: $\sigma([\C
P^{n}],\phi)=-\sigma([pt],\phi^{-1})\leq
-\mathrm{depth}_{L}(\phi^{-1})+1=\mathrm{height}_{L}(\phi)+1$ .
Therefore, we have the inequalities:
\begin{equation}\label{eq:in_CP}\mathrm{depth}_{L}(\phi)\leq
   \sigma([\C P^{n}],\phi)
   \leq \mathrm{height}_{L}(\phi)+1~.~
\end{equation}

Assume now that $L_{0}$ and $L_{1}$ are two non-narrow Lagrangians in
$\C P^{n}$ and $L_{0}\cap L_{1}=\emptyset$. In this case, for any two
constants $C_{0}, C_{1}\in \R$ we may find a normalized Hamiltonian
$H$ which is constant equal to $C_{0}$ on $L_{0}$ and is constant and
equal to $C_{1}$ on $L_{1}$. We pick $C_{1}> C_{0} +1$.  Applying the
first inequality in (\ref{eq:in_CP}) to $L_{1}$ and the second to
$L_{0}$ we get:
$$C_{1}\leq \mathrm{depth}_{L_{1}}(\phi)\leq \sigma([\C P^{n}],\phi^{H})\leq
\mathrm{height}_{L_{1}}(\phi)+1\leq C_{0}+1$$ which leads to a
contradiction.

A more general intersection result based on a somewhat different
argument is stated later in the paper, in \S\ref{sec:Lagr_int_str}.

\begin{rem}
   \begin{enumerate}[a.]
     \item We have a stronger result~\cite{Bi-Co:jcirci} which asserts
      that, under slightly different assumptions, the $\Z_{2}$-Floer
      homology of the two Lagrangians involved (when defined) is not
      zero.  However, the proof of this result goes beyond the scope
      of this paper and so it will not be further discussed here (see
      also Remark \ref{rem:int_etc}).
     \item The argument for the proof given above to
      Corollary~\ref{cor:inters} has been first used by Albers
      in~\cite{Alb:extrinisic} in order to detect Lagrangian
      intersections and by
      Entov-Polterovich~\cite{En-Po:rigid-subsets}; Entov-Polterovich
      first noticed that this Corollary follows from an early version
      of our theorem in~\cite{Bi-Co:qrel-long} combined with the
      results in~\cite{En-Po:rigid-subsets}.  Using the terminology
      of~\cite{En-Po:rigid-subsets}, Theorem~\ref{thm:action} implies
      that a monotone non-narrow Lagrangian is heavy.  This is because
      $[M]$ is an idempotent which verifies $\sigma([M],\phi)\geq
      \mathrm{depth}_{L}(\phi)$ for all $\phi$.  Assume now,
      additionally, that $M$ is point invertible of order $k$ and
      moreover that for any $\phi\in\widetilde{Ham}(M)$,
      $\sigma([pt],\phi^{-1})=-\sigma([M],\phi)$.  In this case, we
      deduce $\sigma([M],\phi)= -\sigma([pt],\phi^{-1})\leq
      -\mathrm{depth}_{L}(\phi^{-1})+k\eta=\mathrm{height}_{L}(\phi)+k\eta$
      so that $L$ is even super-heavy.
   \end{enumerate}
\end{rem}

\subsubsection{Existence of narrow Lagrangians}\label{subsubsec:narrow}
Clearly, a displaceable Lagrangian is narrow. For general symplectic
manifolds this is the only criterion for the vanishing of Floer
homology that we are aware of. Unfortunately, except in very
particular cases, this is not very efficient as, for a given
Lagrangian it is very hard to test the existence of disjoining
Hamiltonian diffeomorphisms. Because of this, till now there are very
few examples of monotone, narrow Lagrangians inside closed symplectic
manifolds. One very simple example is a contractible circle embedded
in a surface of genus $\geq 1$.  However, even in $\C P^{n}$ it is
non-trivial to detect such examples. Corollary~\ref{cor:inters} yields
as a byproduct many examples of such narrow monotone Lagrangians: if
one monotone Lagrangian which is not narrow is known, it suffices to
produce another monotone Lagrangian which is disjoint from it.
\begin{ex}\label{ex:narrow}
   There are narrow monotone Lagrangians in $\C P^{n}$, $n \geq 2$.
\end{ex}
Such Lagrangians are obtained using the Lagrangian circle bundle
construction from~\cite{Bi:Nonintersections}. Namely, we take any
monotone Lagrangian $L_0 \subset Q^{2n-2}$ in the quadric hypersurface
(e.g. a Lagrangian sphere) and then push it up to the normal circle
bundle of the complex quadric hypersurface $Q^{2n-2} \subset \C P^{n}$
of appropriate radius such as to get a monotone Lagrangian $L \subset
{\mathbb{C}}P^n$ which is an $S^1$-bundle over $L_0$. As we will see,
this produces a Lagrangian that does not intersect $\R P^{n}$, which
in turn is wide. A detailed construction of narrow Lagrangians in
${\mathbb{C}}P^n$ along these lines is given in~\S\ref{sb:narrow}.

\subsubsection{Methods of proof and homological calculations}
\label{subsubsec:methods} All our results are based on exploiting the
following machinery.  It is well-known that counting
pseudo-holomorphic disks with Lagrangian boundary conditions (and
appropriate incidence conditions) does not lead, in general, to
Gromov-Witten type invariants as these counts strongly depend on the
choices of auxiliary data involved (almost complex structures, cycles
etc).  However, the moduli spaces of pseudo-holomorphic disks are
sufficiently well structured so that these counts appropriately
understood can be used to define a chain complex - which we call the
{\em pearl} complex (this construction was initially proposed by
Oh~\cite{Oh:relative} following an idea of Fukaya and is a particular
case of the more recent cluster complex of Cornea-Lalonde
\cite{Cor-La:Cluster-1} called there {\em linear clusters}). The
resulting homology $QH(L)$ is an invariant which we call the {\em
  quantum homology} of $L$.  The key bridge between the properties of
the ambient manifold and those of the Lagrangian is provided by the
fact that $QH(L)$ has the structure of an augmented two-sided algebra
over the quantum homology of the ambient manifold, $QH(M)$, and, with
adequate coefficients, is endowed with duality. At the same time,
again with appropriate coefficients, $QH(L)$ is isomorphic to the
Floer homology $HF(L,L)$ of the Lagrangian $L$ with itself. Moreover,
many of the additional algebraic structures also have natural
correspondents in Floer theory. However, the models based on actual
pseudo-holomorphic disks rather than on Floer trajectories are much
more efficient from the point of view of applications: they provide a
passage from geometry to algebra which is sufficiently explicit so
that, together with sometimes delicate algebraic arguments, they lead
to the structural theorems listed before. Actually, in this paper we
will not make any essential use of the fact that the Lagrangian
quantum homology can be identified with the Floer homology.

The deeper reason why the models based on pseudo-holomorphic disks are
so efficient has to do with the fact that they carry an intrinsic
``positivity'' which is algebraically useful and is inherited from the
positivity of area (and Maslov index, in our monotone case) of
$J$-holomorphic curves.  These methods also allow us to compute
explicitly the various structures involved in several interesting
cases. In particular, for the Clifford torus in
$\mathbb{T}_{\textnormal{clif}} \subset \C P^{n}$, for Lagrangians, $L
\subset \C P^{n}$ with $2H_{1}(L;\Z)=0$, and for simply-connected
Lagrangians in the quadric $Q$.  The results of these calculations
will be stated in three Theorems in \S\ref{Sbs:examples-comput} once
the algebraic structures involved are introduced.  However, these
calculations imply a number of homological rigidity results as well as
some uniruling consequences which can be stated without further
preparation and so we review these just below.

The first such corollary deals with Lagrangian submanifolds $L \subset
{\mathbb{C}}P^n$ for which every $a \in H_1(L;\mathbb{Z})$ satisfies
$2 a = 0$ (in short: ``$2 H_1(L; \mathbb{Z})=0$'').  It extends some
earlier results obtained by other methods in~\cite{Se:graded} and
in~\cite{Bi:Nonintersections}. Before stating the result let us recall
the familiar example of $\mathbb{R}P^n \subset {\mathbb{C}}P^n$, $n
\geq 2$, which satisfies $2H_1(\mathbb{R}P^n;\mathbb{Z})=0$.
\begin{cor} \label{cor:RP} Let $L \subset {\mathbb{C}}P^n$ be a
   Lagrangian submanifold with $2H_{1}(L;\Z)=0$. Then $L$ is monotone
   with $N_L = n+1$ and the following holds:
   \begin{enumerate}[i.]
     \item There exists a map $\phi:L\to \R P^{n}$ which induces an
      isomorphism of rings on $\mathbb{Z}_2$-homology: $\phi_*:
      H_*(L;\mathbb{Z}_2) \stackrel{\cong}{\longrightarrow}
      H_*(\mathbb{R}P^n;\mathbb{Z}_2)$, the ring structures being
      defined by the intersection product. In particular we have
      $H_i(L;\mathbb{Z}_2)=\mathbb{Z}_2$ for every $0\leq i \leq n$,
      and $H_*(L;\mathbb{Z}_2)$ is generated as a ring by
      $H_{n-1}(L;\mathbb{Z}_2)$. \label{I:cor:RP-H_*}
     \item $L$ is wide. Therefore, as $N_L = n+1$ and in view of
      point~\ref{I:cor:RP-H_*} just stated, we have $HF_i(L,L) \cong
      \mathbb{Z}_2$ for every $i \in \mathbb{Z}$.
     \item Denote by $h = [{\mathbb{C}}P^{n-1}] \in
      H_{2n-2}({\mathbb{C}}P^n; \mathbb{Z}_2)$ the generator. Then $h
      \cap_L [L]$ is the generator of $H_{n-2}(L;\mathbb{Z}_2)$. Here
      $\cap_L$ stands for the intersection product between elements of
      $H_*({\mathbb{C}}P^n;\mathbb{Z}_2)$ and $H_*(L;\mathbb{Z}_2)$.
     \item Denote by $\textnormal{inc}_*:H_i(L;\mathbb{Z}_2) \to
      H_i({\mathbb{C}}P^n;\mathbb{Z}_2)$ the homomorphism induced by
      the inclusion $L \subset {\mathbb{C}}P^n$. Then
      $\textnormal{inc}_*$ is an isomorphism for every $0 \leq
      i=$\,even $\leq n$.
     \item $({\mathbb{C}}P^n, L)$ is $(1,0)$-uniruled of order $n+1$.
     \item $L$ is $2$-uniruled of order $n+1$. Moreover, given two
      distinct points $x,y\in L$, for generic $J$ there is an even but
      non-vanishing number of disks of Maslov index $n+1$ each of
      whose boundary passes through $x$ and $y$.
     \item For $n=2$, $({\mathbb{C}}P^2 ,L)$ is $(1,2)$-uniruled of
      order $6$.
   \end{enumerate}
\end{cor}
Other than $L = \mathbb{R}P^n$ we are not aware of any other
Lagrangian $L \subset {\mathbb{C}}P^n$ satisfying $2 H_1(L;
\mathbb{Z})=0$. In view of Corollary~\ref{cor:RP} it is tempting to
conjecture that the only Lagrangians $L \subset {\mathbb{C}}P^n$ with
$2 H_1(L; \mathbb{Z})=0$ are homeomorphic (or diffeomorphic) to
$\mathbb{R}P^n$, or more daringly symplectically isotopic to the
standard embedding of $\mathbb{R}P^n \hookrightarrow {\mathbb{C}}P^n$.
Note however that in $\mathbb{C}P^3$ there exists a Lagrangian
submanifold $L$, not diffeomorphic to $\mathbb{R}P^3$, with
$H_i(L;\mathbb{Z}_2)=\mathbb{Z}_2$ for every $i$. This Lagrangian is a
quotient of $\mathbb{R}P^3$ by the dihedral group $D_3$. It has
$H_1(L;\mathbb{Z})\cong \mathbb{Z}_4$. This example is due to
Chiang~\cite{Chiang:RP3}.

Our second corollary is concerned with the Clifford torus,
$$\mathbb{T}_{\textnormal{clif}}^n = \{[z_0: \cdots: z_n] \in
{\mathbb{C}}P^n \mid |z_0| = \cdots |z_n| \} \subset
{\mathbb{C}}P^n~.~$$ This torus is monotone and has minimal Maslov
number $N_{\mathbb{T}_{\textnormal{clif}}^n}=2$.  As before, we endow
${\mathbb{C}}P^n$ with the standard symplectic structure
$\omega_{\textnormal{FS}}$ normalized so that $\int_{\mathbb{C}P^1}
\omega_{\textnormal{FS}} = 1$.
\begin{cor}\label{cor:clifford}
   The Clifford torus $\clift^{n} \subset \C P^{n}$ is wide, $(\C
   P^{n},\clift^{n})$ is $(1,0)$-uniruled of order $2n$ and
   $\clift^{n}$ is uniruled of order $2$. For $n=2$,
   $({\mathbb{C}}P^2, \mathbb{T}_{\textnormal{clif}}^2)$ is
   $(1,1)$-uniruled of order $4$. In particular, $w(\C P^{2},
   \clift^{2} : (r,\rho))\leq 2/3$.
\end{cor}

Finally, we also indicate a result concerning Lagrangians in the
smooth complex quadric hypersurface $Q^{2n} \subset
{\mathbb{C}}P^{n+1}$ endowed with the symplectic structure induced
from ${\mathbb{C}}P^n$. The next corollary is concerned with
Lagrangians $L \subset Q^{2n}$ with $H_1(L;\mathbb{Z})=0$. We recall
the familiar example of a Lagrangian sphere in $Q^{2n}$ which can be
realized for example as a real quadric.
\begin{cor} \label{cor:quadric} Let $L \subset Q^{2n}$, $n \geq 2$, be
   a Lagrangian submanifold with $H_1(L;\mathbb{Z})=0$. Then $L$ is
   wide and $(Q,L)$ is $(1,1)$-uniruled of order $2n$. In particular,
   $w(Q,L:(r,\rho))\leq 1$.  If we assume in addition that $n =
   \dim_{\mathbb{C}}Q$ is even, then we also have:
   \begin{enumerate}
     \item[i.] $H_*(L;\mathbb{Z}_2) \cong H_*(S^n; \mathbb{Z}_2)$.
     \item[ii.] $L$ is $3$-uniruled of order $2n$ (an so $w(Q,L :
      (\emptyset;\rho_{1},\rho_{2},\rho_{3}))\leq 1$).
   \end{enumerate}
\end{cor}

\subsection{Structure of the paper}  
The main results of the paper are stated in the introduction and in
\S\ref{sec:lagr_str}.  Namely, in the second section, after some
algebraic preliminaries we review in \S \ref{Sb:main-stat} the
structure of Lagrangian quantum homology. This structure is needed to
state in \S\ref{Sbs:examples-comput} three theorems containing
explicit computations.  Each one of the three corollaries already
described in \S \ref{subsubsec:methods} is a consequence of one of
these theorems. Section \ref{sec:lagr_str} concludes - in
\S\ref{sec:Lagr_int_str} - with the statement of a Lagrangian
intersection result which is a strengthening of Corollary
\ref{cor:inters}.

In~\S\ref{sec:main_quant_proof} and~\S\ref{sec:add_tools} we develop
the tools necessary to prove the results stated in the first two
sections. More precisely, \S\ref{sec:main_quant_proof} contains the
justification of the structure of Lagrangian quantum homology.  While
we indicate the basic steps necessary to establish this structure,
certain technical details are omitted. These details are contained in
our preprint \cite{Bi-Co:qrel-long} and we have decided not to include
them here because they are quite tedious and long and relatively
non-surprising for specialists.  The fourth section contains a number
of auxiliary results which provide additional tools which are
necessary to prove the theorems of the paper.

The actual proofs of the results stated in \S\ref{sec:intro} and
\S\ref{sec:lagr_str} are contained in sections \ref{sec:proof_main}
and \ref{sec:exa_comp}.  Namely, the fifth section contains the proofs
of the three main structural Theorems stated in the introduction as
well as that of the Lagrangian intersection result stated in
\S\ref{sec:Lagr_int_str} and the sixth section contains the proofs of
the three ``computational'' theorems stated in
\S\ref{Sbs:examples-comput} and that of their corresponding three
Corollaries from \S \ref{subsubsec:methods}. The construction of the
example mentioned in \S \ref{subsubsec:narrow} is also included here
as well as a few other related examples.

Finally, in the last section we discuss some open problems derived
from our work.

\subsubsection*{Acknowledgments.}
The first author would like to thank Kenji Fukaya, Hiroshi Ohta, and
Kaoru Ono for valuable discussions on the gluing procedure for
holomorphic disks. He would also like to thank Martin Guest and Manabu
Akaho for interesting discussions and great hospitality at the Tokyo
Metropolitan University during the summer of 2006.

Special thanks from both of us to Leonid Polterovich for interesting comments and his
interest in this project from its early stages as well as for having
pointed out a number of imprecisions in earlier versions of the paper.
We also thank Peter Albers and Misha Entov and Joseph Bernstein for
useful discussions.  We also thank the FIM in Zurich and the CRM in
Montreal for providing a stimulating working atmosphere which
allowed us to pursue our collaboration in the academic year 07-08.

While working on this project our two children, Zohar and Robert, were
born and by the time we finally completed this paper they had already
celebrated their first birthdays. We would like to dedicate this work
to them and to their lovely mothers, Michal and Alina.

\section{Lagrangian quantum structures} \label{sec:lagr_str} In this
section we introduce the algebraic structures and invariants essential
for our applications.  We will then indicate the main ideas in the
proof of the related statements as well as a few technical aspects.
Full details appear in \cite{Bi-Co:qrel-long}.

\subsection{Algebraic preliminaries.} We fix here algebraic notation
and conventions which will be used in the paper.

\subsubsection{Graded modules and chain complexes} \label{sb:graded}

Let $\mathcal{R}$ be a commutative graded ring, i.e. $\mathcal{R}$ is
a commutative ring with unity, $\mathcal{R}$ splits as $\mathcal{R}
= \oplus_{i \in \mathbb{Z}} \mathcal{R}_i$, for every $i, j \in
\mathbb{Z}$ we have $\mathcal{R}_i \cdot \mathcal{R}_j \subset
\mathcal{R}_{i+j}$ and $1 \in \mathcal{R}_0$.  By a graded
$\mathcal{R}$-module we mean an $\mathcal{R}$-module $M$ which is
graded $M = \oplus_{i \in \mathbb{Z}} M_i$ with each component $M_i$
being an $\mathcal{R}_0$-module and moreover for every $i, j \in
\mathbb{Z}$ we have $\mathcal{R}_i \cdot M_j \subset M_{i+j}$.

The chain complexes $(\mathcal{C}, d)$ we will deal with will often be
of the following type. Their underlying space $\mathcal{C} = \oplus_{i
  \in \mathbb{Z}} \mathcal{C}_i$ will be a graded
$\mathcal{R}$-module, and moreover the differential $d$, when viewed as
a map of the total space $d: \mathcal{C} \to \mathcal{C}$, is
$\mathcal{R}$-linear. Since it is not justified to call such complexes
$\mathcal{C}$ ``chain complexes over $\mathcal{R}$'' (as each
$\mathcal{C}_i$ is not an $\mathcal{R}$-module) we have chosen to call
them {\rm $\mathcal{R}$-complexes}. Note that $(\mathcal{C}, d)$ is in
particular also a chain complex of $\mathcal{R}_0$-modules in the
usual sense. Note also that the homology $H(\mathcal{C}, d)$ is
obviously a graded $\mathcal{R}$-module.

Most of our chain complexes $(\mathcal{C},d)$ will be free
$\mathcal{R}$-complexes. By this we mean that (the total space of) the
$\mathcal{R}$-complex $\mathcal{C}$ is a finite rank free module over
$\mathcal{R}$. In other words $\mathcal{C} = G \otimes \mathcal{R}$
where $G$ is a graded finite dimensional $\mathbb{Z}_2$-vector space
and the grading on $\mathcal{C}$ is induced from the grading of $G$
and from the grading of $\mathcal{R}$. The differential $d$ on $\mathcal{C}$ of course
does not need to have the form $d = d_G \otimes 1$. In fact we can
split $d$, in a unique way, as a (finite) sum of operators $d =
\sum_{l \in \mathbb{Z}} \delta_l$ where $\delta_l: G_* \to G_{*-1+l}
\otimes \mathcal{R}_{-l}$.  (Here $G_l$ is identified with $G_*
\otimes 1 \subset G_* \otimes \mathcal{R}_0$ and the operators
$\delta_l$ are extended to $\mathcal{C}$ by linearity over
$\mathcal{R}$). Actually, in most of the complexes below the operators
$\delta_l$ will actually be given as $\delta_l = \sum_{j}
\partial_{l,j} \otimes r_{l,j}$ with $\partial_{l,j}: G_* \to
G_{*-1+l}$ and $r_{l,j} \in \mathcal{R}_{-l}$.

Finally, we say that the differential $d$ of a free
$\mathcal{R}$-complex $(\mathcal{C},d)$ is positive if $\delta_l = 0$
for every $l < 0$. In that case we will call the operator $\delta_0$
the classical component of $d$.

\subsubsection{Coefficient Rings} \label{subsec:coeff} Denote by
$H^D_2(M,L) \subset H_2(M,L;\mathbb{Z})$ the image of the Hurewicz
homomorphisms $\pi_2(M,L) \longrightarrow H_2(M,L)$. Let
$H_2^D(M,L)^{+}$ be the monoid of all the elements $u$ so that
$\omega(u)\geq 0$. Put $\La^{+}=\Z_{2}[H_2^D(M,L)^{+}/\sim]$ with
$\sim$ the equivalence relation $u\sim v$ iff $\mu(u)=\mu(v)$ and
similarly $\La=\Z_{2}[H_2^D(M,L)/\sim]$. We grade these rings so that
the degree of $u$ equals $-\mu(u)$. In practice we will use the
following natural identifications: $\Lambda^+ \cong \mathbb{Z}_2[t]$,
$\Lambda \cong \mathbb{Z}_2[t^{-1}, t]$ induced by $H_2^D(M,L) \ni u
\to t^{\mu(u)/N_L}$. The grading here is chosen so that $\deg t =
-N_L$.  

As mentioned in the introduction, the quantum homology of the ambient
manifold is naturally a module over the ring $\Gamma=\Z_{2}[s^{-1},s]$
where the degree of $s$ is $-2C_{M}$.  There is an obvious embedding
of rings $\Gamma \hookrightarrow \La$ which is defined by $s\to
t^{(2C_{M})/N_{L}}$.  The same embedding also identifies the ring
$\Gamma^{+}=\Z_{2}[s]$ with its image in $\La^{+}$. Using this
embedding we regard $\La$ (respectively $\La^{+}$) as a module over
$\Gamma$ (respectively, over $\Gamma^{+}$) and we define the following
obvious extensions of the quantum homology:
$$QH(M;\Lambda) = H_*(M;\mathbb{Z}_2) \otimes
\Lambda= QH_{\ast}(M)\otimes_{\Gamma} \La, \quad QH(M;\Lambda^{+}) =
H_*(M;\mathbb{Z}_2) \otimes \Lambda^{+}.$$ We endow $QH(M;\Lambda)$
and $QH(M; \Lambda^{+})$ with the quantum intersection product $*$
(see~\cite{McD-Sa:Jhol-2} for the definition).  Similarly, we can
consider the analogous extensions of quantum homology over the ring
$\Lambda'$. Notice that we work here with quantum {\em homology} (not
cohomology), hence the quantum product $*: QH_k(M;\La) \otimes
QH_l(M;\La) \to QH_{k+l-2n}(M;\La)$ has degree $-2n$.  The unit is
$[M] \in QH_{2n}(M;\La)$, thus of degree $2n$.

While we will essentially stick with $\La$, $\La^{+}$ in this paper,
for certain applications it can be useful to also use larger rings
which distinguish explicitly the elements in $H_2^D(M,L)$. This is
done as follows. Let $H_2^S(M,L) \subset H_2(M;\mathbb{Z})$ be the
image of the Hurewicz homomorphism $\pi_2(M) \to H_2(M;\mathbb{Z})$,
and $H_2^S(M)^+ \subset H_2^S(M)$ the semi-group consisting of classes
$A$ with $c_1(A) > 0$. Similarly, denote by $H_2^{D}(M,L)^+ \subset
H_2^D(M,L)$ the semi-group of elements $A$ with $\mu(A) > 0$.  Let
$\widetilde{\Gamma}^+ = \mathbb{Z}_2[H_2^S(M)^+] \cup \{ 1 \}$ be the
unitary ring obtained by adjoining a unit to the non-unitary group
ring $\mathbb{Z}_2[H_2^S(M)^+]$. Similarly we put
$\widetilde{\Lambda}^{+} = \mathbb{Z}_2[H_2^D(M,L)^+] \cup \{1\}$. We
write elements $Q \in \widetilde{\Gamma}^{+}$ and $P \in
\widetilde{\Lambda}^{+}$ as ``polynomials'' in the formal variables
$S$ and $T$:
$$Q(S) = a_0 + \sum_{c_1(A)>0} a_{A} S^{A}, \qquad P(T) = b_0 +
\sum_{\mu(B)>0} b_B T^B \qquad a_0, a_A, b_0, b_B \in \mathbb{Z}_2.$$
We endow these rings with the following grading: $$\deg S^A = -2
c_1(A), \quad \deg T^B = -\mu(B).$$ Note that these rings are smaller
than the rings $\hat{\Gamma}^{\geq 0} = \mathbb{Z}_2[\{A | c_1(A) \geq
0\}]$ and $\hat{\Lambda}^{\geq 0} = \mathbb{Z}_2[\{B | \mu(B) \geq
0\}]$. For example, $\hat{\Lambda}^{\geq 0}$ and $\hat{\Gamma}^{\geq
  0}$ might have many non-trivial elements in degree $0$, whereas in
$\widetilde{\Gamma}^+$ and $\widetilde{\Lambda}^{+}$ the only such
element is $1$.

Let $QH(M;\widetilde{\Gamma}^{+}) = H(M;\mathbb{Z}_2) \otimes
\widetilde{\Gamma}^{+}$ be the quantum homology of $M$ with
coefficients in $\Gamma^{+}$ endowed with the quantum product, which
we still denote by $\ast$ (note that now $\ast$ takes into account the
actual classes of holomorphic spheres not only their Chern numbers).
We have a natural map $H_2^{S}(M)^+ \to H_2^{D}(M,L)^+$ which induces
on $\widetilde{\Lambda}^+$ a structure of a
$\widetilde{\Gamma}^{+}$-module. Put $QH(M;\widetilde{\Lambda}^+) =
QH(M;\widetilde{\Gamma}^{+}) \otimes_{\widetilde{\Gamma}^+}
\widetilde{\Lambda}^{+}$ and endow it with the quantum intersection
product, still denoted $*$.  Note that the quantum product is well
defined with this choice of coefficients, since by monotonicity Chern
numbers of pseudo-holomorphic spheres are non-negative and the only
possible pseudo-holomorphic sphere with Chern number $0$ is constant.
We grade this ring with the obvious grading coming from the two
factors.

The most general rings of coefficients relevant for this paper are
rings $\mathcal{R}$ that are graded commutative
$\widetilde{\Lambda}^+$-algebras. We will usually endow a graded
commutative ring $\mathcal{R}$ with the structure of
$\widetilde{\Lambda}^+$-algebra by specifying a graded ring
homomorphism $q: \widetilde{\Lambda}^+ \to \mathcal{R}$.

Here are a few examples of such rings $\mathcal{R}$ which are useful
in applications.
\begin{enumerate}
  \item Take $\mathcal{R} = \Lambda = \mathbb{Z}_2[t^{-1}, t]$, and
   define $q$ by $q(T^A) = t^{\mu(A)/N_L}$.
   \label{i:r=lambda}
  \item Take $\mathcal{R} = \Lambda^+ = \mathbb{Z}_2[t]$, and define
   $q$ as in~\ref{i:r=lambda}.
  \item Take $\mathcal{R} = \mathbb{Z}_2[H_2^D(M,L)]$ with the obvious
   $\widetilde{\Lambda}^{+}$-algebra structure. We denote this ring by
   $\hat{\Lambda}$.
 \end{enumerate}
Given a graded commutative $\widetilde{\Lambda}^+$-algebra
$\mathcal{R}$ we extend the coefficients of the quantum homology of
the ambient manifold by $QH(M; \mathcal{R}) = QH(M;
\widetilde{\Lambda}^+) \otimes_{\widetilde{\Lambda}^+} \mathcal{R}$.

\subsubsection{A useful filtration} There is a natural decreasing
filtration of $\La^{+}$, $\La$ and $\La'$ by the degrees of $t$, i.e.
\begin{equation}\label{eq:filtration}
   \mathcal{F}^{k}\La = \{P\in \mathbb{Z}_2[t,t^{-1}] \mid
   P(t)=a_{k}t^{k}+a_{k+1}t^{k+1}+\ldots \}~.~
\end{equation}
We will call this filtration the \emph{degree filtration}.  In a
similar way we can define the analogous filtrations on any graded
$\widetilde{\Lambda}^+$-algebra $\mathcal{R}$. This filtration induces
an obvious filtration on any free $\mathcal{R}$-module.

\subsection{Structure of Lagrangian quantum homology}
\label{Sb:main-stat}

Let $f:L\to \R$ be a Morse function on $L$ and let $\rho$ be a
Riemannian metric on $L$ so that the pair $(f,\rho)$ is Morse-Smale.
We grade the elements of $\Crit(f)$ by $|x|=ind_{f}(x)$.  Fix also a
generic almost complex structure $J$ compatible with $\omega$. We
recall that as we work in the monotone case (which, with the
conventions of this paper includes $N_{L}\geq 2$), the Floer homology
$HF_{\ast}(L;\mathcal{R})=HF_{\ast}(L,L;\mathcal{R})$ is well defined
and invariant whenever $\mathcal{R}$ is a commutative
$\Z_{2}[H_{2}^{D}(M,L)]$-algebra (see \S\ref{subsubsec:defin_alg_str}
g for a rapid review of the construction).
\begin{mainthm} \label{thm:alg_main} Let $\mathcal{R}$ be a graded
   commutative $\widetilde{\Lambda}^+$-algebra (e.g. $\mathcal{R} =
   \Lambda$, $\Lambda^{+}$, or $\hat{\Lambda}$). For a generic choice
   of the triple $(f,\rho,J)$ there exists a finite rank, free
   $\mathcal{R}$-chain complex
   $$\mathcal{C}(L;\mathcal{R}; f,\rho,J)=(\Z_2 \langle \Crit(f) \rangle
   \otimes \mathcal{R}, d^{\mathcal{R}})$$ with grading induced by
   Morse indices on the left factor and the grading of $\mathcal{R}$
   on the right. The differential $d^{\mathcal{R}}$ of this complex is
   positive (see~\S\ref{sb:graded}) and its classical component
   coincides with the Morse-homology differential
   $d^{\textnormal{\tiny{Morse}}} \otimes 1$ (see~\S\ref{sb:graded}).
   Moreover, this complex has the following properties:
   \begin{enumerate}[i.]
     \item The homology of this chain complex is a graded
      $\mathcal{R}$-module and is independent of the choices of $(f,
      \rho, J)$, upto canonical comparison isomorphisms. It will be
      denoted by $QH_{\ast}(L;\mathcal{R})$.  There exists a canonical
      (degree preserving) augmentation $\epsilon_{L}:
      QH_{\ast}(L;\mathcal{R})\to \mathcal{R}$ which is an
      $\mathcal{R}$-module map. Moreover, for $\mathcal{R} = \Lambda$
      the augmentation $\epsilon_L$ is non-trivial whenever
      $QH(L;\Lambda) \neq 0$.
     \item The homology $QH(L;\mathcal{R})$ has the structure of a
      two-sided algebra with a unity over the quantum homology of $M$,
      $QH(M;\mathcal{R})$. More specifically, for every $i, j, k \in
      \mathbb{Z}$ there exist $\mathcal{R}$-bilinear maps:
      \begin{align*}
         & QH_i(L;\mathcal{R}) \otimes QH_j(L;\mathcal{R}) \to
         QH_{i+j-n}(L;\mathcal{R}),
         \quad \alpha\otimes \beta \mapsto \alpha \circ \beta, \\
         & QH_k (M;\mathcal{R}) \otimes QH_j(L;\mathcal{R}) \to
         QH_{k+j-2n}(L;\mathcal{R}), \quad a\otimes \alpha \mapsto a
         \circledast \alpha,
      \end{align*}
      where $n = \dim L$. The first map endows $QH(L;\mathcal{R})$
      with the structure of a ring with unity. This ring is in general
      not commutative. The second map endows $QH(L;\mathcal{R})$ with
      the structure of a module over the quantum homology ring
      $QH(M;\mathcal{R})$.  Moreover, when viewing these two
      structures together, the ring $QH(L;\mathcal{R})$ becomes a
      two-sided algebra over the ring $QH(M;\mathcal{R})$. (The
      definition of a two-sided algebra is given below, after the
      statement of the theorem.) The unity of $QH(L;\mathcal{R})$ has
      degree $n=\dim L$ and will be denoted by $[L]$.
      \label{I:qmod-qprod}
     \item There exists a map
      $$i_{L}:QH_{\ast}(L;\mathcal{R}) \to
      QH_{\ast}(M;\mathcal{R})$$ which is a
      $QH_{\ast}(M;\mathcal{R})$-module morphism and which is induced
      by a chain map which is a deformation of the singular inclusion
      (viewed as a map between Morse complexes).  Moreover, this map
      is determined by the relation:
      \begin{equation}\label{eq:inclusion_mod}
         \langle PD(h), i_{L}(x) \rangle =\epsilon_{L}(h \circledast x)
      \end{equation} for $x\in QH(L;\mathcal{R})$,
      $h\in H_{\ast}(M)$, with $PD(-)$
      Poincar\'e duality and $\langle - , - \rangle$ the
      $\mathcal{R}$-linear extension of the
      Kronecker pairing ( i.e. $\langle PD(h),\sum_{r}z_{r}T^{r}\rangle=
      \sum_{r} \langle PD(h) , z_{r}\rangle T^{r}$).
      \label{I:q-inclusion}
     \item The differential $d^{\mathcal{R}}$ respects the degree
      filtration and all the structures above are compatible with the
      resulting spectral sequences. \label{I:spectral-seq}
     \item The differential $d^{\mathcal{R}}$ is in fact defined over
      $\widetilde{\Lambda}^{+}$ in the sense that the relation between
      $\mathcal{C}(L; \mathcal{R}; f, \rho, J)$ and $\mathcal{C}(L;
      \Lambda^{+}; f, \rho, J)$ is that $\mathcal{C}(L; \mathcal{R};
      f, \rho, J) \cong \mathcal{C}(L; \Lambda^{+}; f, \rho, J)
      \otimes_{\widetilde{\Lambda}^+} \mathcal{R}$ and
      $d^{\mathcal{R}} \cong d^{\widetilde{\Lambda}^+} \otimes id$.
      Moreover, any graded $\widetilde{\Lambda}^{+}$-algebra
      homomorphism $\mathcal{R} \to \mathcal{R}'$ (e.g. the inclusion
      $\Lambda^+ \to \Lambda$) induces in homology a canonical
      morphism $QH(L;\mathcal{R}) \to QH(L; \mathcal{R}')~.~$
      \label{I:extension}
     \item If $\mathcal{R}$ is a commutative
      $\mathbb{Z}_2[H_2^D(M,L)]$-algebra (e.g. $\mathcal{R} =
      \Lambda$), then there exists an isomorphism $$QH_*(L;
      \mathcal{R}) \to HF_*(L;\mathcal{R})$$ which is canonical up to
      a shift in grading.
      \label{I:comparison}
   \end{enumerate}
\end{mainthm}
The existence of the morphism $QH(L;\mathcal{R}) \to QH(L;
\mathcal{R}')$ at point~\ref{I:extension} of the Theorem is not a
purely algebraic statement about extension of coefficients. Rather, it
means that the canonical extension of coefficients morphisms
$H_*(\mathcal{C}(L; \mathcal{R}; f, \rho, J)) \to H_*(\mathcal{C}(L;
\mathcal{R}'; f, \rho, J))$ do not depend on $(f, \rho, J)$ in the
sense that they are compatible with the canonical comparison
isomorphisms relating the homologies associated to any two triples
$(f_0, \rho_0, J_0)$ and $(f_1, \rho_1, J_1)$.  In view of
point~\ref{I:extension} we will denote from now on the differential
$d^{\mathcal{R}}$ by $d$ whenever the ring $\mathcal{R}$ is fixed and
there is no risk of confusion.

By a {\em two-sided algebra} $A$ over a ring $R$ we mean that $A$ is a
module over $R$, that $A$ is also a (possibly non-commutative) ring,
and the two structures satisfy the following compatibility conditions:
$$\forall\, r \in R \textnormal{ and } a,b\in A
\textnormal{ we have } r(ab)=(ra)b=a(rb).$$ In other words, the first
identity means that $A$, when considered as a left module over $R$, is
an algebra over $R$, and the second one means that $A$ continues to be
an algebra over $R$ when viewed as a right module over $R$, where the
left and right module operations are the same one.

\medskip Before going on any further we would like to point out that,
the existence of a module structure asserted by
Theorem~\ref{thm:alg_main} has already some non-trivial consequences.
For instance, the fact that $QH_*(L;\Lambda)$ is a module over
$QH_*(M;\Lambda)$ implies that if $a \in QH_k(M;\Lambda)$ is an
invertible element of degree $k$, then the map $a \circledast
(-)$ gives rise to {\em isomorphisms} $QH_i(L;\Lambda) \to
QH_{i+k-2n}(L;\Lambda)$ for every $i \in \mathbb{Z}$, or in other
words, $QH_*(L;\Lambda)$ is $(k-2n)$-periodic. In view of
point~\ref{I:comparison} of the theorem the same periodicity holds for
the Floer homology $HF_*(L)$ too. Note that there is yet another
obvious periodicity for $QH_*(L)$ that always holds (regardless of the
module structure). Namely multiplying by $t \in \Lambda$ always gives
isomorphisms $QH_*(L; \Lambda) \cong QH_{*-N_L}(L;\Lambda)$. This
follows immediately from the fact that $QH(L; \Lambda)$ is a graded
$\Lambda$-module and that $t \in \Lambda_{-N_L}$ is invertible. The
above two periodicities, when applied together, provide a powerful
tool in the computations of our invariants.

\medskip In most of the applications below we will take the ring of
coefficients $\mathcal{R}$ to be either $\Lambda$ or $\Lambda^{+}$.
Therefore we will sometimes drop the ring of coefficients from the
notation and use the following abbreviations:
\begin{align*}
   \mathcal{C}(L;f,\rho,J) & = \mathcal{C}(L;\Lambda;f,\rho,J), \quad
   QH(L) = QH(L; \Lambda) ~,~\\
   \mathcal{C}^{+}(L;f,\rho,J) & =
   \mathcal{C}(L;\Lambda^{+};f,\rho,J), \quad Q^{+}H(L) =
   QH(L;\Lambda^{+}).
\end{align*}
We will call the complex $\mathcal{C}(L;f,\rho,J)$ (respectively
$\mathcal{C}^{+}(L;f,\rho,J)$) the {\em (positive) pearl complex}
associated to $f,\rho,J$ and we will call the resulting homology the
(positive) {\em quantum homology} of $L$. In the perspective
of~\cite{Cor-La:Cluster-1, Cor-La:Cluster-2} the complex
$\mathcal{C}(L;f,\rho,J)$ corresponds to the {\em linear cluster
  complex}.

\begin{rem} \label{r:alg-thm}
   \begin{enumerate}[a.]
     \item The complex $\mathcal{C}(L;f,\rho,J)$ was first suggested
      by Oh~\cite{Oh:relative} (see also
      Fukaya~\cite{Fu:Morse-homotopy}) and, from a more recent
      perspective, it is a particular case of the cluster complex as
      described in Cornea-Lalonde~\cite{Cor-La:Cluster-1}. The module
      structure over $Q^{+}H(M)$ discussed at point~\ref{I:qmod-qprod}
      is probably known by experts - at least in the Floer homology
      setting - but has not been explicitly described yet in the
      literature. The product at~\ref{I:qmod-qprod} is a variant of
      the Donaldson product defined via holomorphic triangles - it
      might not be widely known in this form.  The map $i_{L}$ at
      point~\ref{I:q-inclusion} is the analogue of a map first studied
      by Albers in~\cite{Alb:extrinisic} in the absence of bubbling.
      The spectral sequence appearing at~\ref{I:spectral-seq} is a
      variant of the spectral sequence introduced by Oh
      \cite{Oh:spectral}. The compatibility of this spectral sequence
      with the product at point~\ref{I:qmod-qprod} has been first
      mentioned and used by Buhovsky~\cite{Bu:products} and
      independently by Fukaya-Oh-Ohta-Ono~\cite{FO3}. The comparison
      map at~\ref{I:comparison} is an extension of the
      Piunikin-Salamon-Schwarz construction~\cite{PSS}, it extends
      also the partial map constructed by Albers in~\cite{Alb:PSS} and
      a more general such map was described independently
      in~\cite{Cor-La:Cluster-1} in the ``cluster" context.  We also
      remark that this comparison map (with coefficients in $\La$)
      identifies all the algebraic structures described above with the
      corresponding ones defined in terms of the Floer complex.
     \item The isomorphism $QH(L) \cong HF(L)$ at
      point~\ref{I:comparison} of Theorem~\ref{thm:alg_main} is an
      important structural property of the Lagrangian quantum
      homology. However, we would like to point out that this property
      of $QH(L)$ is in fact not used in any of the applications
      presented in this paper. There is only one minor exception to
      this rule. Namely, our definition of wide and narrow Lagrangians
      $L$ goes via $HF(L)$. However we could have defined these
      notions directly using $QH(L)$, and actually in the rest of the
      paper this will be the more relevant definition. The reason we
      have chosen to define wide and narrow using Floer homology is
      two-fold.  Firstly, Floer homology is already well known in
      symplectic topology, and we wanted to base the notions of wide
      and narrow on a familiar concept. Secondly, it is easier to
      produce examples of narrow Lagrangians this way, simply by using
      the fact that if a Lagrangian $L$ is Hamiltonianly displaceable
      then $HF(L)=0$.
   
      We insist on separating between $HF$ and $QH$ because
        we do not view our Lagrangian quantum homology as a
      Lagrangian intersections invariant. Moreover, the results in
      this paper suggest that  Lagrangian quantum homology has
      applications beyond  Lagrangian intersections and
      thus we believe that this homology should be developed and studied
      in its own right.
   \end{enumerate}
\end{rem}

\subsection{Some computations} \label{Sbs:examples-comput} Here we
present a few explicit computations of the various quantum structures
mentioned in Theorem~\ref{thm:alg_main} performed on three examples:
Lagrangians $L \subset {\mathbb{C}}P^n$ with $2 H_1(L;\mathbb{Z})=0$
(e.g. $L = \mathbb{R}P^n$), the Clifford torus
$\mathbb{T}^2_{\textnormal{clif}} \subset {\mathbb{C}}P^2$ and
Lagrangians $L$ in the quadric with $H_1(L;\mathbb{Z})=0$ (e.g.
spheres).  The proofs of the three results listed here are given in
\S\ref{sec:exa_comp}. More results in this direction can be found
in~\cite{Bi-Co:qrel-long}.

We work here over the ring $\Lambda$.  We start with Lagrangians $L
\subset {\mathbb{C}}P^n$ that satisfy $2 H_1(L;\mathbb{Z})=0$. Recall
from Corollary~\ref{cor:RP} that $QH_i(L) \cong HF_i(L) \cong
\mathbb{Z}_2$ for every $i \in \mathbb{Z}$. Denote by $\alpha_i \in
QH_i(L)$ the generator. Denote by $h = [{\mathbb{C}}P^{n-1}] \in
H_{2n-2}({\mathbb{C}}P^n;\mathbb{Z}_2)$ the class of a hyperplane.
Recall also that in the quantum homology $QH({\mathbb{C}}P^n)$ we
have:
\begin{equation} \label{Eq:qh-cpn} h^{*j} =
   \begin{cases}
      h^{\cap j}, & 0 \leq j \leq n \\
      [{\mathbb{C}}P^n] s, & j=n+1
   \end{cases}
\end{equation}
As we will see (and is stated in Corollary \ref{cor:RP}) $N_L = n+1$,
thus the embedding $\Gamma \hookrightarrow \Lambda$ is given by $s \to
t^2$. It follows that in $QH({\mathbb{C}}P^n; \Lambda)$ the last
relation of~\eqref{Eq:qh-cpn} becomes $h^{* (n+1)} = [{\mathbb{C}}P^n]
t^{2}$.  Finally note that both $h$ and $[pt]$ are invertible elements
in $QH({\mathbb{C}}P^n)$.
\begin{thm} \label{T:qstruct-2H_1=0} Let $L \subset {\mathbb{C}}P^n$
   be a Lagrangian with $2 H_1(L;\mathbb{Z}_2)=0$. Then:
   \begin{enumerate}
     \item[i.] For every $i,j \in \mathbb{Z}$, $\alpha_i \circ
      \alpha_j = \alpha_{i+j-n}$.
     \item[ii.] For every $i \in \mathbb{Z}$, $h \circledast \alpha_i
      = \alpha_{i-2}$.
   \end{enumerate}
   Furthermore, denote by $h_j \in H_j({\mathbb{C}}P^n;\mathbb{Z}_2)$
   the generator (so that $h_{2n-2} = h$, $h_{2k} = h^{\cap (n-k)}, \,
   \forall \, 0\leq k \leq n$, $h_{\textnormal{odd}}=0$ etc.) then:
   \begin{enumerate}
     \item[iii.]For $n=$\,even we have:
      \begin{align*}
         & i_L(\alpha_{2k})=h_{2k}, \quad  \forall\, 0 \leq 2k \leq n, \\
         & i_L(\alpha_{2k+1}) = h_{2k+n+2} t,\quad \forall\, 1 \leq
         2k+1 \leq n-1.
      \end{align*}
     \item[iv.] For $n=$\,odd we have:
      \begin{align*}
         & i_L(\alpha_{2k}) = h_{2k} +
         h_{2k+n+1} t, \quad \forall\, 0 \leq 2k \leq n, \\
         & i_L(\alpha_{2k+1})=0, \quad \forall\, k \in \mathbb{Z}.
      \end{align*}
   \end{enumerate}
\end{thm}

The next result describes our computations for, mainly, the
$2$-dimensional Clifford torus $\mathbb{T}^2_{\textnormal{clif}}
\subset {\mathbb{C}}P^2$.
\begin{thm} \label{T:clif2-qstruct} The Clifford torus $\clift^{n}$ is
   wide for every $n\geq 1$.  Let $w \in
   H_2(\mathbb{T}^2_{\textnormal{clif}};\mathbb{Z}_2) \hookrightarrow
   QH_2(\mathbb{T}_{\textnormal{clif}};\mathbb{Z}_2)$ be the
   fundamental class.  There are generators $a,b\in
   H_{1}(\mathbb{T}^2_{\textnormal{clif}}, \mathbb{Z}_2) \cong
   QH_1(\mathbb{T}^2_{\textnormal{clif}})$, and $m \in
   QH_0(\mathbb{T}^2_{\textnormal{clif}})$ which together with $w$
   generate $QH(\mathbb{T}^2_{\textnormal{clif}})$ as a
   $\Lambda$-module and verify the following relations:
   \begin{enumerate}[i.]
     \item $a \circ b=m + w t$, $b \circ a=m$, $a \circ a=b \circ b= w
      t$, $m \circ m = mt + wt^2$. \label{I:qprod-T}
     \item $h \circledast a=at$, $h \circledast b=bt$, $h \circledast
      w=wt$, $h \circledast m = mt$. Here $h =[\mathbb{C}P^1] \in
      H_2({\mathbb{C}}P^2;\mathbb{Z}_2)$ is the class of a projective
      line.
     \item $i_{L}(m)=[pt]+ht+[\C P^{2}]t^{2}$, \,
      $i_{L}(a)=i_{L}(b)=i_{L}(w)=0$
   \end{enumerate}
\end{thm}
We remark that, as the formulas in~\ref{I:qprod-T} indicate, the
quantum product on $QH(L)$ is in general {\em non-commutative} (even
if we work over $\mathbb{Z}_2$).

\begin{rem}
   \begin{enumerate}[a.]
     \item The fact that the Clifford torus is wide and point i of
      Theorem~\ref{T:clif2-qstruct} have been obtained before by Cho~
      in \cite{Cho:Clifford} and \cite{Cho:products} by a different
      approach. From the perspective of \cite{Cho:products} the
      Clifford torus is a special case of a torus which appears as a
      fibre of the moment map defined on a toric variety. See
      also~\cite{Cho-Oh:Floer-toric} for related results in this
      direction.
     \item Given that $\clift^{2}$ is wide we have
      $QH_*(\mathbb{T}^2_{\textnormal{clif}}) \cong
      H_*(\mathbb{T}^2_{\textnormal{clif}};\mathbb{Z}_2) \otimes
      \Lambda$.  Note however that such an isomorphisms cannot be made
      canonical in all degrees (see also \S\ref{subsec:Cliff_calc}).
      Nevertheless there is a canonical embedding
      $H_2(\mathbb{T}^2_{\textnormal{clif}}) \hookrightarrow
      QH_2(\mathbb{T}^{2}_{\textnormal{clif}})$ and the isomorphism
      $QH_1(\mathbb{T}^2_{\textnormal{clif}}) \cong
      H_1(\mathbb{T}^{2}_{\textnormal{clif}};\mathbb{Z}_2)$ is
      canonical. (See \cite{Bi-Co:qrel-long}, \cite{Bi-Co:Yasha-fest}
      for more details on this).
   \end{enumerate}
\end{rem}

We now turn to the third example: Lagrangians in the quadric.  Let $L
\subset Q^{2n}$ be a Lagrangian submanifold of the quadric (where
$dim_{\mathbb{R}} Q = 2n$) that satisfies $H_1(L;\mathbb{Z})=0$.  Such
Lagrangians are monotone and the minimal Maslov number is $N_L = 2n$.
Recall that by Corollary~\ref{cor:quadric} $L$ is wide hence $QH_*(L)
\cong (H(L;\mathbb{Z}_2) \otimes \Lambda)_{\ast}$. As $\deg t = -2n$
we have $QH_0(L) \cong H_0(L;\mathbb{Z}_2)$ and $QH_n(L) \cong
H_n(L;\mathbb{Z}_2)$. Denote by $\alpha_0 \in QH_0(L)$ and $\alpha_n
\in QH_n(L)$ the respective generators. Finally, denote by $[pt] \in
H_0(Q;\mathbb{Z}_2)$ the class of a point.
\begin{thm} \label{T:quadric-qstruct} Let $L \subset Q$ be as above.
   Then:
   \begin{enumerate}
     \item[i.] $[pt]\circledast \alpha_0 = -\alpha_0 t$, $[pt]
      \circledast \alpha_n = -\alpha_n t$.
     \item[ii.] $i_L(\alpha_0) = [pt] - [Q]t$, where $[Q] \in
      H_{2n}(Q;\mathbb{Z}_2)$ is the fundamental class.
     \item[iii.] If $n$ is even then $\alpha_0 \circ \alpha_0 =
      \alpha_n t$.
   \end{enumerate}
\end{thm}

\begin{rem}
   The significance of the signs in the formulae above comes from the
   fact that we expect our machinery to hold with coefficients in $\Z$
   and, if so, these are the signs that we obtain when taking into
   account orientations. As we shall see these signs play a
   significant role in some applications - see Corollary*
   \ref{c:jcirci-quad}.
\end{rem}

\subsection{A criterion for Lagrangian
  intersections.}\label{sec:Lagr_int_str}

We describe here a criterion for Lagrangian intersections which is
somewhat more general than Corollary~\ref{cor:inters} and which is
stated in terms of the machinery described in
Theorem~\ref{thm:alg_main}.

Let $L_0, L_1 \subset M$ be two monotone Lagrangian submanifolds.  Let
$\Lambda_0 = \mathbb{Z}_2[t_0^{-1}, t_0]$, $\Lambda_{1} =
\mathbb{Z}_2[t_1^{-1}, t_1]$ be the associated rings, graded by $\deg
t_0 = -N_{L_0}$ and $\deg t_1 = N_{L_1}$. Recall
from~\S\ref{subsec:coeff} that we also have the ring $\Gamma =
\mathbb{Z}_2[s^{-1},s]$, $\deg s = -2C_M$, and that $\Lambda_0$,
$\Lambda_1$ are $\Gamma$-modules. Consider now the ring
$\Lambda_{0,1}=\Lambda_0 \otimes_{\Gamma} \Lambda_1$ with the grading
induced form both factors (it is easy to see that this grading is well
defined). Equivalently, $$\Lambda_{0,1} \cong \mathbb{Z}_2[t_0^{-1},
t_1^{-1}, t_0, t_1] / \{ t_0^{2C_M/N_{L_0}} = t_1^{2C_M/N_{L_1}} \}.$$
Note that $\Lambda_{0,1}$ is  a $\Lambda_0$-algebra, a
$\Lambda_1$-algebra as well as $\Gamma$-algebra. Thus we have well
defined quantum homologies $QH(L_0;\Lambda_{0,1})$,
$QH(L_0;\Lambda_{0,1})$ as well as $QH(M;\Lambda_{0,1})$.

With the above notation we have two canonical maps. The first one is
the quantum inclusion $i_{L_0}: QH_*(L_0;\Lambda_{0,1}) \to
QH_*(M;\Lambda_{0,1})$, mentioned at point~\ref{I:q-inclusion} of
Theorem~\ref{thm:alg_main}. The second map is $j_{L_1}:
QH_*(M;\Lambda_{0,1}) \to QH_{*-n}(L_1;\Lambda_{0,1})$, defined by
$j_{L_1}(a) = a \circledast [L_1]$. Consider the composition:
$$j_{L_1} \circ i_{L_0} :
QH_*(L_0;\Lambda_{0,1}) \longrightarrow QH_{*-n}(L_1;\Lambda_{0,1}).$$

\begin{thm} \label{cor:inters_i_j} If $j_{L_1} \circ i_{L_0} \neq 0$,
   then $L_0\cap L_1 \neq \emptyset$.
\end{thm}

\begin{rem} \label{rem:int_etc} 
   \begin{enumerate}[a.]
     \item It is possible to show~\cite{Bi-Co:jcirci} (see
      also~\cite{Bi-Co:Yasha-fest}) that the condition $j_{L_0}\circ
      i_{L_1} \neq 0$ implies the non-vanishing of the Floer homology
      $HF(L_0,L_1)$ (when defined).
     \item The map $j_{L_1}$ has appeared before in a different
      setting in the work of Albers~\cite{Alb:extrinisic}.
   \end{enumerate}
\end{rem}

Here is a consequence of this theorem which provides a different proof
of Corollary~\ref{cor:inters}. To state it we fix some more notation.
As discussed before, for any Lagrangian submanifold the inclusion of
the associated coefficient rings $\La^{+}\to \La$ induces a map of
pearl complexes (when defined) $p:\mathcal{C}(L;\La^{+}; f,\rho,J)\to
\mathcal{C}(L;\La;f,\rho,J)$ which is canonical in homology. Denote by
$IQ^{+}(L)$ the image of $p_{\ast}:QH(L;\La^{+})\to QH(L;\La)$, the
map induced in homology by $p$, and notice that $IQ^{+}(L)$ is a
$\La^{+}$-module so that it makes sense to say whether a class $z\in
IQ^{+}(L)$ is divisible by $t$ in $IQ^{+}(L)$: this means that there
is some $z'\in IQ^{+}(L)$ so that $z=tz'$.

\begin{cor}\label{cor:inter_Mas_Chern}
   Let $L \subset M$ be a non-narrow monotone Lagrangian submanifold.
   Let $[pt]\in QH(M;\La)$ be the class of the point. If the product
   $[pt]\circledast [L]$ is not divisible by $t^{2C_{M}/N_L}$ in
   $IQ^{+}(L)$ then $L$ must intersect any non-narrow monotone
   Lagrangian in $M$.
\end{cor}
Any non-narrow monotone Lagrangian $L \subset \C P^{n}$ satisfies the
condition in the statement and so Corollary \ref{cor:inter_Mas_Chern}
implies Corollary \ref{cor:inters}. Indeed, put $z = [pt] \circledast
[L] \in IQ^{+}_{-n}(L)$. Assume that $z = t^{2C_{{\mathbb{C}}P^n}/N_L}
z'$ for some $z' \in IQ^+(L)$. We have $2C_{{\mathbb{C}}P^n} = 2n+2$
and $|t^{2C_{{\mathbb{C}}P^n}/N_L}| = -(2n+2)$. Therefore, $|z'| = -n
+ 2n+2 = n+2$. But for degree reasons $IQ^+_l(L)=0$ for every $l > n$
and so $z' = 0$. In particular $z = 0$. On the other hand as $[pt] \in
QH(M; \Lambda)$ is invertible and $[L] \neq 0$ we must have $z \neq
0$. A contradiction.

\

The proof of Corollary \ref{cor:inter_Mas_Chern} is given in
\S\ref{sb:jcirci} after the proof of Theorem \ref{cor:inters_i_j}.

\begin{rem}
   \begin{enumerate}[a.]
     \item By Theorem~\ref{thm:alg_main}, $L$ is non-narrow if and
      only if $[L] \neq 0 \in QH(L)$. The reason is that $[L]$ is the
      unity of $QH(L)$ when viewed as a ring. Moreover, whenever $M$
      is point invertible and $L$ is not narrow the product
      $[pt]\circledast [L]$ does not vanish. Of course, the
      non-divisibility condition in the statement of
      Corollary~\ref{cor:inter_Mas_Chern} is an additional strong
      restriction.
     \item The criterion in Corollary \ref{cor:inter_Mas_Chern} does
      not apply to Lagrangians $L$ in the quadric which satisfy
      $H_{1}(L;\Z)=0$ so it does not lead to intersection results in
      this case. However, later in the paper (in
      Corollary~\ref{c:jcirci-quad}) we will see that
      Theorem~\ref{cor:inters_i_j} can also be applied to this setting
      but by working with integer coefficients, thus under the
      assumption that our machinery continues to work when taking into
      account orientations.
   \end{enumerate}
\end{rem}

\subsection{Simplification of notation} \label{Sbs:simp-not} As
mentioned before, whenever we use the rings $\Lambda$ and
$\Lambda^{+}$ we will drop them from the notation in the following
way:
\begin{equation} \label{Eq:simp-not}
   \begin{aligned}
      \mathcal{C}(L;f,\rho,J) & = \mathcal{C}(L;\Lambda;f,\rho,J),
      \quad
      QH(L) = QH(L; \Lambda),\\
      \mathcal{C}^{+}(L;f,\rho,J) & =
      \mathcal{C}(L;\Lambda^{+};f,\rho,J), \quad Q^{+}H(L) =
      QH(L;\Lambda^{+}).
   \end{aligned}
\end{equation}

Another simplification is the following. Theorem~\ref{thm:alg_main}
involves three different algebraic operations: the quantum
intersection product $*$, the Lagrangian quantum product $\circ$, and
the external module operation $\circledast$:
\begin{equation} \label{Eq:simp-not-oper}
   \begin{aligned}
      & *:QH_k(M;\mathcal{R}) \otimes QH_l(M;\mathcal{R}) \to
      QH_{k+l-2n}(M;\mathcal{R}), \\
      & \circ: QH_i(L;\mathcal{R}) \otimes QH_j(L;\mathcal{R}) \to
      QH_{i+j-n}(L;\mathcal{R}), \\
      & \circledast: QH_k(M;\mathcal{R}) \otimes QH_j(L;\mathcal{R})
      \to QH_{i+j-2n}(L;\mathcal{R}).
   \end{aligned}
\end{equation}
As all these operations commute in the sense that $QH(L;\mathcal{R})$
is an algebra over $QH(M;\mathcal{R})$ we will sometimes denote all
these operations by $*$.


\section{Sketch of proof for Theorem~\ref{thm:alg_main}}
\label{sec:main_quant_proof}We will explain the ideas behind the proof
but, as mentioned in the introduction, we will not prove here this
theorem in full.  However, all the technical details which are omitted
here can be found in \cite{Bi-Co:qrel-long}. The reason for proceeding
in this way is that, on one hand, many of the actual technical
verifications are not novel for specialists but quite long so
including them here does not seem judicious. On the other hand, it is
not possible to apply efficiently this theorem in the absence of a
good understanding of the underlying moduli spaces and thus it is
important to give a sufficiently detailed description of the
construction of our machinery.  We will also shortly review the main
ideas behind the proof of transversality as well as the basic argument
needed to prove the identities contained in the statement of the
theorem.

\subsection{The moduli spaces}\label{subsec:moduli_def}
It is useful to view our further constructions as a ``quantum''
version of standard constructions in Morse theory. In particular, in
Morse theory, the Morse differential is modeled by a tree with one
entry and one exit but no interior vertex. The same is true for a
Morse morphism which relates two Morse complexes.  The intersection
product is modeled on trees with two entries and one exit. For the
associativity of this product, are required trees with three entries
and one exit.  The quantum version of this construction consists in
allowing each edge in these simple trees to be subdivided by a finite
number of quantum contributions represented by pseudo-holomorphic
disks or spheres.  Such contributions can also appear at the vertices
of the trees.  Obviously, a more precise definition is required and we
proceed to give one below.

\

\underline{A. Combinatorial preliminaries.}  The trees needed here are
of a reasonably simple type because we only use some rather elementary
algebraic structures.  The vertices of these trees will be of two
types, corresponding to $J$-holomorphic disks (with boundary on $L$)
or $J$-holomorphic spheres, and the edges will correspond to flow
lines of Morse functions some defined on $L$ and some on $M$.  The
entries and the exit will correspond to critical points of these Morse
functions. Here is a more precise description, unavoidably quite
tedious.  Conditions i-iii below simply model the data: each edge in
the tree needs to carry a label (which geometrically corresponds to a
particular Morse function). Each interior vertex will correspond to
some $J$-holomorphic sphere or disk so that it needs to carry a label
given by some homotopy class etc. A stability restriction is needed
and is added as condition iv.  In the compactifications of such moduli
spaces appear configurations where one (or more) edges are represented
by flow lines of zero length.  The corresponding geometric objects
also appear by disk (or sphere) bubbling off.  For our construction it
is crucial that each configuration of this type appears exactly twice:
once by bubbling off and once by the degeneration of a flow line.  The
purpose of condition v is to insure precisely this property.  The
point vi describes how the flow lines arriving at a vertex represented
by a $J$-holomorphic curve are anchored to that curve.

\

Here are the precise details of the construction we consider connected
trees $\mathcal{T}$ with oriented edges embedded in
$\R\times[0,1]\subset \R^{2}$ with entries lying on the line
$\R\times\{1\}$ and a single exit which is situated on the line
$\R\times\{0\}$ and so that the edges strictly decrease the
$y$-coordinate. Clearly, at each internal vertex there is precisely
one ``exiting'' (or departing) edge and at least one ``entering'' (or
arriving) edge.  There will be at most three entries and one exit.  We
call such a tree, $\mathcal{T}$, {\em $(M,L)$-labeled} if the
following additional structure is given:
\begin{itemize}
  \item[i.] The entries and the exit have valence one (and they are
   the only vertices with this property).  The vertices of the tree -
   except for the entries and the exit - are labeled by elements of
   $\lambda\in H_{2}^{D}(M,L)$ or by elements $\mu\in H_{2}^{S}(M)$
   with $\omega (\lambda)\geq 0$, $\omega(\mu)\geq 0$.  The first kind
   of vertex will be called of disk type and the second will be called
   spherical. The set of vertices of $\mathcal{T}$ (including entries
   and the exit) is denoted by $v(\mathcal{T})$, the set of the
   spherical vertices is denoted by $v_{S}(\mathcal{T})$ and the set
   of disk type vertices is denoted by $v_{D}(\mathcal{T})$. The set
   of interior vertices will be denoted by
   $v_{int}(\mathcal{T})=v_{D}(\mathcal{T})\cup v_{S}(\mathcal{T})$.
   The class of an interior vertex $v$ will be denoted by $[v]\in
   H_{2}^{D}(M,L)$ or $\in H_{2}^{S}(M)$.
\end{itemize}
Let $\mathcal{F}_{L}$ be a finite set of Morse functions defined on
$L$ and let $\mathcal{F}_{M}$ be a finite set of Morse functions
defined on $M$. Put $\mathcal{F}=\mathcal{F}_{L}\cup \mathcal{F}_{M}$.
An $(M,L)$-labeled tree $\mathcal{T}$ is called $\mathcal{F}$-{\em
  colored} if it satisfies the following three properties:
\begin{itemize}
  \item[ii.] The set of edges of $\mathcal{T}$ is denoted by
   $e(\mathcal{T})$ and is partitioned into two classes, the edges of
   type $L$, $e_{L}(\mathcal{T})$, and the edges of type $M$,
   $e_{M}(\mathcal{T})$.  Each edge $e$ of type $L$ is colored by a
   Morse function $f_{e}\in \mathcal{F}_{L}$ and each edge $e$ of type
   $M$ is colored by a Morse function $f_{e}\in \mathcal{F}_{M}$.  For
   $v\in v(\mathcal{T})$ we let $n_{L}(v)$ be the number of edges of
   type $L$ which are incident to $v$ and we let $n_{M}(v)$ be the
   number of those edges of type $M$.  For an edge $e$ we let
   $e_{-}\in v(\mathcal{T})$ be the (initial) vertex where $e$ starts
   and we let $e_{+}$ be the end (or final) vertex of $e$.  If a
   vertex $v\in v_{S}(\mathcal{T})$, then $n_{L}(v)=0$.  If $v\in
   v_{D}(\mathcal{T})$, then $n_{L}(v)\geq 1$. If $e\in
   e_{L}(\mathcal{T})$ and $e_{-}$ (respectively $e_{+}$) is not an
   entry (respectively, not the exit), then $e_{\pm}\in
   v_{D}(\mathcal{T})$.

  \item[iii.] Each entry as well as the exit is labeled by a critical
   point of the Morse function corresponding to the incident edge.  In
   other words, for all edges $e$, if $e_{-}$ is an entry, then this
   implies that $e_{-}$ is labeled by a critical point of the function
   $f_{e}$ and similarly for the exit.  Any two distinct entries
   correspond to critical points of different Morse functions.

  \item[iv.]  At each vertex, distinct arriving edges are labeled by
   different Morse functions.  If a vertex $v\in v_{D}(\mathcal{T})$
   has the property $\omega([v])=0$ and $n_{L}(v)\leq 2$, then
   $n_{M}(v)\geq 1$.  If a vertex $v\in v_{S}(\mathcal{T})$ has the
   property $\omega([v])=0$, then $n_{M}(v)\geq 3$.
\end{itemize}

The coloring of our trees will be usually described by means of an
{\em exit rule}.  Namely, fix as before a collection $\mathcal{F}$ of
Morse functions (some on $L$, some on $M$). Notice that, for a planar
tree $\mathcal{T}$, at each vertex $v$, the planarity of the tree
induces an order among the arriving edges (by the values of the
$x$-coordinates of the intersections of these edges with a horizontal
line close to the vertex but above it).

\begin{itemize}
  \item[v.]  An exit rule $\Theta$ associates to each ordered vector,
   $(f_{1},\ldots, f_{s})$ with $f_{i}\in\mathcal{F}$, and symbol $S$
   which can be either $L$ or $M$, a new function $\Theta
   (f_{1},\ldots f_{s}; S)\in \mathcal{F}$.  An $\mathcal{F}$-colored
   tree $\mathcal{T}$ is called $\Theta$-{\em admissible} if, for each
   vertex of $\mathcal{T}$ whose exit edge is of type $S$ and whose
   arriving edges are colored, in order, by $(f_{1},\ldots, f_{s})$,
   the departing edge is colored by $\Theta( f_{1},\ldots,f _{s};
   S)\in \mathcal{F}_{S}$.
\end{itemize}

Given an exit rule $\Theta$ notice that, for any $(M,L)$-labeled tree
$\mathcal{T}$, if a coloring of the entry edges is given, then there
exists a unique $\mathcal{F}$-coloring of $\mathcal{T}$ that is
$\Theta$-admissible. Note also that, in order to color $\mathcal{T}$
in this way, we do not always need to know the value of $\Theta$ on
all possible configurations (since some of them might not appear in
any relevant trees).

\

We recall that the moduli spaces that we intend to construct consist
of $J$-holomorphic disks and spheres joined by Morse trajectories.  To
proceed from trees to these moduli spaces we need an additional
structure which describes how the flow lines are ``anchored'' to the
$J$-curves.  The structure in question is as follows:

\begin{itemize}
  \item[vi.]  A {\em marked point selector} for an
   $\mathcal{F}$-colored tree $\mathcal{T}$ is given by an assignment
   $Q$ which associates to each vertex $v\in v_{S}(\mathcal{T})$ a
   collection $Q_{v}$ of distinct points in $S^{2}$ which is in 1-1
   correspondence with the incident edges and, similarly, $Q$
   associates to a vertex $v\in v_{D}(\mathcal{T})$ a collection
   $Q_{v}\subset D$ so that if an edge $e$ is of type $M$ its
   corresponding marked point is in $Int(D)$ and if the edge $e$ is of
   type $L$ the corresponding marked point is in $\partial D$.
   Moreover, for $v\in v_{D}(\mathcal{T})$ the order among the marked
   points in $\partial D$ matches the order of the incident edges of
   type $L$ {\em clockwise} around the circle. If $e$ is an arriving
   edge (at some internal vertex) the respective marked point is
   denoted by $q_{+}(e)$ and if the edge is the exiting one, then the
   marked point is denoted by $q_{-}(e)$.
\end{itemize}
We denote $\mathcal{F}$-colored trees together with a marked point
selector $Q$ by $(\mathcal{T},Q)$ and we refer to the pair
$(\mathcal{T},Q)$ as an {\em $\mathcal{F}$-colored tree with marked
  points}.  The marked point selectors that will be used here satisfy
an additional property: they only depend on the type of the edge $e$,
the valence of the vertex $v$, on whether the edge $e$ is an exit edge
or an entry one and, in this last case, on the planar order of the
edge among the arriving edges at the vertex $v$.  In other words, we
can view such a marked point selector as an abstract rule which
associates a certain marked point to each edge incident to a vertex of
{\em any} $\mathcal{F}$-colored tree. In view of this, if $Q$ and $Q'$
are marked point selectors we can write $Q=Q'$ if the two
corresponding rules agree.

For a tree $\mathcal{T}$ we indicate its entries and the exit by a
symbol like $(x, y, z: w)$ where the first components - in this case,
they are three - are the labels of the entries {\em written in the
  planar order} and the last component indicates the label of the
exit.  We call this data the {\em symbol} of the tree $\mathcal{T}$.
The {\em class} of the tree $\mathcal{T}$, $[\mathcal{T}]\in
H_{2}^{D}(M,L)$ is defined to be the sum of the classes of the
interior vertices. We denote the symbol of the $\mathcal{F}$-colored
tree $\mathcal{T}$ by $\sym(\mathcal{T})$.

\

\underline{B. Construction of the moduli spaces.}  Fix an
$\mathcal{F}$-colored tree with marked points $(\mathcal{T}, Q)$.  Fix
also a pair $\rho=(\rho_{M}, \rho_{L})$ where $\rho_{L}$ is a
Riemannian metric on $L$ and $\rho_{M}$ is a Riemannian metric on $M$.
For every $f\in\mathcal{F}$ let $\gamma^{f}_{t}$ be the associated
negative gradient flows (with respect to the metric $\rho_{L}$ for the
functions defined on $L$ and with respect to the metric $\rho_{M}$ for
the functions defined on $M$). Denote by $(x_{1},\ldots, x_{l}: y)$
the symbol of $\mathcal{T}$.

\

For an $\omega$-compatible almost complex structure $J$ and a class
$\lambda\in H_{2}^{D}(M,L)$ (or in $H_{2}^{S}(M)$) let
$\mathcal{M}(\lambda, J)$ be the moduli space of parametrized
$J$-disks (respectively $J$-spheres) in the class $\lambda$.

\

The {\em pearl} moduli space modeled on $(\mathcal{T},Q)$ will be
denoted by $\mathcal{P}_{\mathcal{T},Q}( J, \rho)$ (or, if the data
involved is clear from the context, just $\mathcal{P}_{\mathcal{T}}$)
and it is defined as follows.  If $\mathcal{T}$ has no interior vertex
or, equivalently, it consists of precisely of one edge $e$ connecting
the entry (which is labeled by a critical point $x=x_{1}$ of $f_{e}$)
to the exit labeled by $y\in \Crit(f_{e})$, then
$\mathcal{P}_{\mathcal{T}}$ is the unparametrized moduli space of flow
lines of $\gamma^{f_{e}}$ connecting $x$ to $y$.

In case $\mathcal{T}$ contains an internal vertex, consider the
product
  $$\Pi(\mathcal{T})= \prod_{v\in v_{int}(\mathcal{T})} \mathcal{M}([v], J)$$  and let  $S_{\mathcal{T},Q}$  consist of all
  $\{u_{v}\}_{v \in v_{int}(\mathcal{T})}\in \Pi(\mathcal{T})$ subject
  to the constraints:

  \begin{itemize}
    \item[a.] For each internal edge $e\in e(\mathcal{T})$ there is
     $t\geq 0$ (called the {\em length} of $e$) such that
     $$\gamma_{t}^{f_{e}}(u_{e_{-}}(q_{-}(e)))=u_{e_{+}}(q_{+}(e))~.~$$
    \item[b.] For an entry edge, $e$, let $x_{i}$ be the critical
     point labeling the vertex $e_{-}$. We have
$$\lim_{t\to-\infty}\gamma_{t}^{f_{e}}(u_{e_{+}}(q_{+}(e)))=x_{i}~.~$$
\item[c.] For the exit edge $e$ we have
 $$\lim_{t\to \infty}\gamma_{t}^{f_{e}}(u_{e_{-}}(q_{-}(e)))=y~.~$$
\end{itemize}

Finally, define $\mathcal{P}_{\mathcal{T}, Q}=S_{\mathcal{T},Q}/\sim$
where $\sim$ is given by the action of the obvious reparametrization
groups which act on the $\mathcal{M}([v],J)$'s and preserve the marked
points.

\begin{figure}[htbp]
   \psfig{file=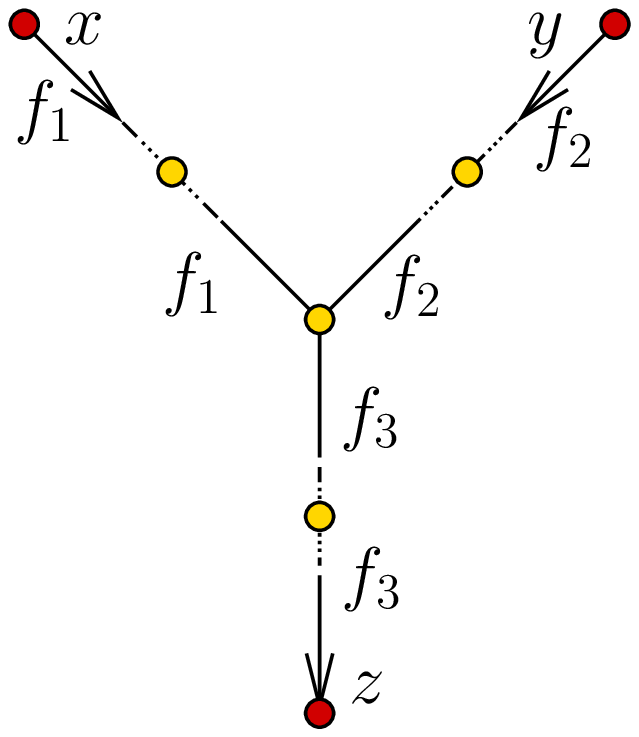, width=0.35 \linewidth} \quad
   \psfig{file=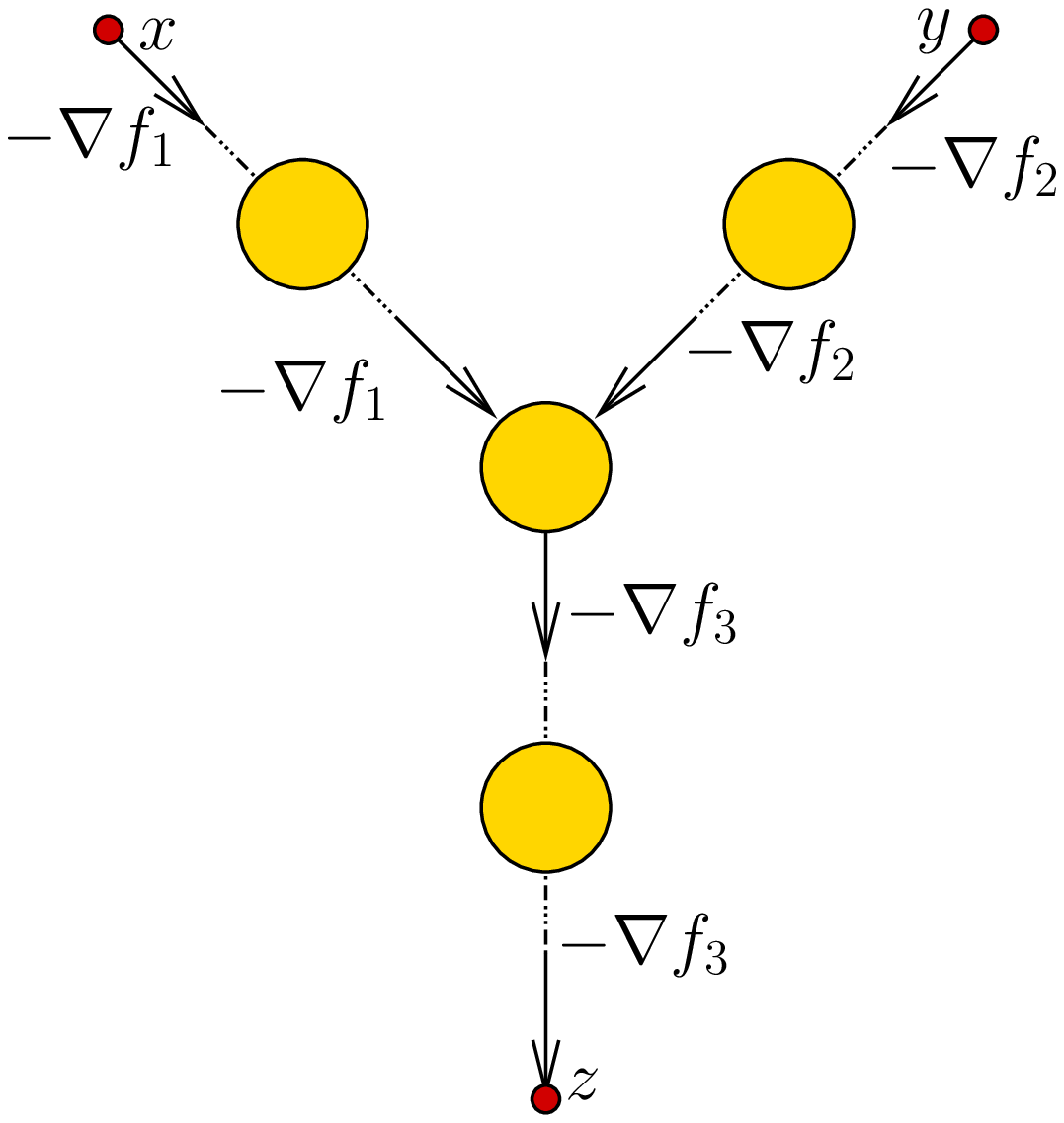, width=0.58 \linewidth}
   \caption{A tree of symbol $(x,y:z)$ on the left, and a pearly
     trajectory corresponding to it on the right.}
   \label{f:pearls-prod}
\end{figure}

The moduli space $\mathcal{P}_{\mathcal{T},Q}$ has a virtual dimension
which only depends on the structure encoded in the definition of the
colored trees with marked points.  This virtual dimension will be
denoted by $\delta(\mathcal{T})$.  When transversality is achieved, it
coincides with the actual manifold dimension of
$\mathcal{P}_{\mathcal{T},Q}$.  As we will see in the next section,
under this transversality assumption, the space
$\mathcal{P}_{\mathcal{T},Q}$ is a manifold, in general non-compact,
with a boundary consisting of configurations where some edge of
$\mathcal{T}$ has $0$-length.

Assume that the symbol of $\mathcal{T}$ is $(x_{1},\ldots, x_{k} : y)$
and that there are $s$ entries among the $x_{i}$'s which are critical
points of functions in $\mathcal{F}_{M}$.  Then the formula giving
this virtual dimension is:
\begin{equation}\label{eq:virtual_dim}
   \delta(\mathcal{T})=\sum _{i}|x_{i}| - |y| +\mu[\mathcal{T}] +\epsilon(k) -(s+k-1)n
\end{equation}
where $\epsilon(k)=-1$ if $k=1$, $y\in L$, and $\epsilon(k)=0$
otherwise.

\

\underline{C. Equivalence of trees.}  In the sequel two
$\mathcal{F}$-colored trees will be viewed as equivalent if the
underlying topological trees are isomorphic by a tree isomorphism
which preserves the order of the entering edges at each vertex and
which also preserves the labels and the coloring.

\begin{rem} Most of our moduli spaces are constructed according to the
   recipe above. In particular, they are all modeled on
   $(M,L)$-labeled trees. However, sometimes we need to work with
   variants of the last part of the construction. For example, we
   might use instead of Morse functions, Morse cobordisms; instead of
   a single almost complex structure we might require a family of such
   structures. Moreover, sometimes, some of the curves used in the
   construction satisfy a perturbed Cauchy-Riemann equation or the
   domains of some of the ``vertices'' in our trees will not be
   spheres or disks but rather, cylinders or strips etc.  In all these
   cases we will describe explicitly the (generally minor)
   modifications that are needed in the construction above.
\end{rem}

\subsection{Definition of the algebraic structures}\label{subsubsec:defin_alg_str}
The formalism given above allows us to define all the particular
moduli spaces needed for our various operations and we will describe
all these constructions below.  In all these cases, we indicate the
relevant moduli spaces by following the scheme above.  In each case we
will describe the various structures involved, namely, the class of
Morse functions $\mathcal{F}$, the exit rule $\Theta$ (we will give
its values only over that part of its domain which is relevant), the
marked point selector $Q$ as well as the symbol $\sym({\mathcal{T}})$
of the relevant trees.  We will also indicate in each case the formula
for the virtual dimension of the respective moduli spaces.

The definitions of our operations and their properties depend on the
transversality results which will be reviewed in the next section.
Moreover, the various relations that need to be proved require to
understand the compactification of these moduli spaces, a description
of their boundary and a gluing formula. This part will be discussed in
the last subsection.

\

Let $\mathcal{R}$ be a graded commutative
$\tilde{\Lambda}^{+}$-algebra as in \S\ref{subsec:coeff}. As before,
we fix a pair $\rho=(\rho_{L},\rho_{M})$ of Riemannian metric on $L$
and on $M$ as well as an almost complex structure $J$ compatible with
$\omega$.

\

{\em a.  The pearl complex and its differential}. Here and in the
points b and c below all the internal vertices are of disk type and all
internal edges are or type $L$ so that we omit from the notation of
$\Theta$ the symbol $S$ as $S=L$ in these three cases.

We consider a single Morse function $f:L\to \R$ and put
$\mathcal{F}=\{f\}$.  The pearl complex is
$$\mathcal{C}(L;\mathcal{R}; f,\rho_{L},J)=(\Z_{2}\langle\Crit(f)\rangle\otimes \mathcal{R}, d)~.~$$

The differential $d$ is defined for generic choices of our data. To
describe it, we consider $\mathcal{F}$-colored trees with marked
points, $(\mathcal{T},Q)$, with symbol $(x:y)$ with $x,y\in\Crit(f)$
and so that the marked point selector associates to each $e\in
e(\mathcal{T})$, $q_{-}(e)=+1\in \partial D$ and
$q_{+}(e)=-1\in\partial D$. It is easy to see that the virtual
dimension of the associated moduli spaces is given by
$\delta(\mathcal{T})=|x|-|y|+\mu[\mathcal{T}]-1$.

\begin{figure}[htbp]
   \psfig{file=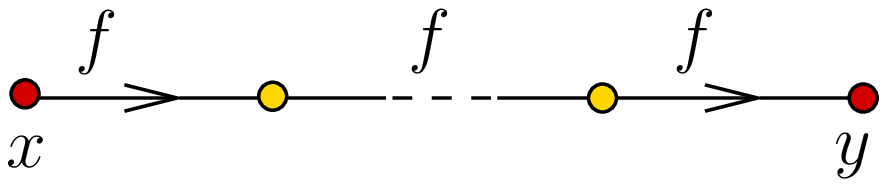, width=0.38 \linewidth} \\
   \psfig{file=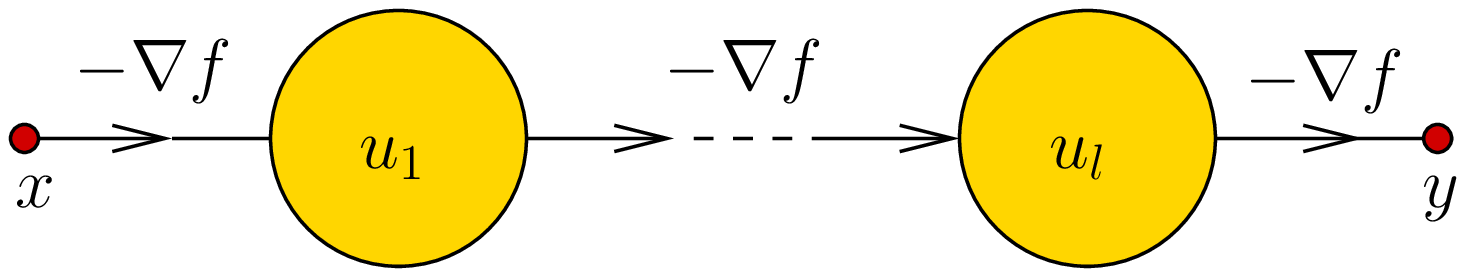, width=0.58 \linewidth}
   \caption{A tree of symbol $(x:y)$ at the top, and a pearly
     trajectory corresponding to it at the bottom.}
   \label{f:pearls-diff}
\end{figure}

We now put:
\begin{equation}\label{eq:pearl_moduli}\ dx=
   \sum_{y, (\mathcal{T},Q)}\#_{2}(\mathcal{P}_{\mathcal{T},Q})\ y\ T^{[\mathcal{T}]}
\end{equation}
where, $y, (\mathcal{T},Q)$ go over all the trees $(\mathcal{T},Q)$ as
above and we only count elements in $\mathcal{P}_{\mathcal{T},Q}$ when
the associated virtual dimension is $0$ (we will use the same
convention in the other examples below).  The relation $d^{2}=0$ is
obtained by using the same type of moduli spaces but with virtual
dimension equal to $1$. Notice that if $f$ has a single maximum, $P$,
then, for degree reasons, $P$ is a cycle in the pearl complex
$\mathcal{C}(L;\mathcal{R};f,\rho_{L},J)$ (the point here is that the
differential is defined over $\widetilde{\La}^{+}$).

We will omit $L$, $J$, $\rho$, $\mathcal{R}$ from the notation if they
are clear from the context.

{\em b. The quantum product}. In this case $\mathcal{F}=\{f_{1},
f_{2}, f_{3}\}$ with the three functions $f_{i}$ all defined on $L$.
The product is defined by:
\begin{equation}\label{eq:moduli_prod}
   \circ: \mathcal{C}(f_{1})\otimes_{\mathcal{R}}\mathcal{C}(f_{2})\to \mathcal{C}(f_{3}) \ , \ x\circ y =\sum_{(\mathcal{T},Q),z}(\#_{2}\mathcal{P}_{\mathcal{T},Q})\ z \ T^{[\mathcal{T}]}
\end{equation}
where the sum is taken over all the $\mathcal{F}$-colored trees with
marked points $(\mathcal{T}, Q)$ of symbol $(x,y:z)$ with
$x\in\Crit(f_{1})$, $y\in \Crit(f_{2})$ and $z\in\Crit(f_{3})$ which
are $\Theta$-admissible with $Q$ and $\Theta$ as follows.  First, the
marking selector verifies: if $e_{+}$ is of valence at most $2$ then
$q_{+}(e)=-1\in \partial D$; if $e_{-}$ is of valence at most $3$,
$q_{-}(e)=+1\in \partial D$; if $e_{+}$ is of valence $3$, and $e$ is
the $j$-th entering edge (in the planar order) at the vertex $e_{+}$
(clearly, $j\in\{1,2\}$), then $q_{+}(e)= e^{-\frac{2\pi j}{3}i}\in
\partial D$. In other words, at a vertex of valence $3$, the marked
(or incidence) points are the roots of order three of the unity.
Finally, the exit rule is $\Theta (f_{i})=f_{i} \ \forall
i\in\{1,2,3\}$, $\Theta(f_{1},f_{2})=f_{3}$.  The virtual dimension in
this case is $\delta(\mathcal{T})=|x|+|y|-|z|- n+\mu[\mathcal{T}]$.
Schematically, the trees used here and the associated configurations
are depicted in Figure \ref{f:pearls-prod}.

Similar moduli spaces but of virtual dimension $1$ are used to show
that the linear map defined by (\ref{eq:moduli_prod}) defines a chain
morphism and thus descends to homology.

A useful remark here is that we can also use instead of the three
functions $f_{1}$, $f_{2}$, $f_{3}$ only two function $f_{1}$ and
$f_{2}$ with the same exit rule as above except that for the vertex of
valence $3$ we require $\Theta(f_{1},f_{2})=f_{2}$. It is easy to see
that this definition provides a product
\begin{equation}\label{eq:simpl_q_prod}
   \circ :\mathcal{C}(f_{1})\otimes_{\mathcal{R}}\mathcal{C}(f_{2})\to \mathcal{C}(f_{2})
\end{equation}
which coincides in homology with the product given before (see also
the invariance properties described at point e). This is particularly
useful in verifying the associativity of the product as described at
point f below as it allows one to work in that verification with only
three Morse functions. Another reason why this description of the
product is useful is that, assuming that $f_{1}$ has a single maximum
$P$, we see that if a moduli space $\mathcal{P}_{\mathcal{T},Q}$ used
to define (\ref{eq:simpl_q_prod}) is of symbol $(P,y:z)$ and of
dimension $0$, then $y=z$ and $\mathcal{P}_{\mathcal{T},Q}$ consists
of the unique Morse trajectory of $f_{1}$ joining $P$ to $y$. Thus
$P\circ y=y$ hence $P$ is a unity at the {\em chain level} for the
product defined in (\ref{eq:simpl_q_prod}).

{\em c. The module structure.} We now have
$\mathcal{F}=\{f_{1},f_{2}\}$ with one Morse function $f_{1}:M\to \R$
and one Morse function $ f_{2}:L\to \R$.  We let
$CM(f_{1};\mathcal{R})=\Z_{2}\langle
\Crit(f_{1})\rangle\otimes\mathcal{R}$ be the Morse complex of $f_{1}$
tensored with the ring $\mathcal{R}$ (endowed with the Morse
differential $d=d_{Morse}\otimes 1$).  The module action is defined
by:
\begin{equation}\label{eq:moduli_module}
   \circledast: C(f_{1})\otimes_{\mathcal{R}} \mathcal{C}(f_{2})\to \mathcal{C}(f_{2}) \ , \ a\circledast x =\sum_{(\mathcal{T},Q),y}(\#_{2}\mathcal{P}_{\mathcal{T}})\ y \ T^{[\mathcal{T}]}
\end{equation}
where the sum is taken over all the $\mathcal{F}$-colored trees
$(\mathcal{T},Q)$ of symbol $(a,x:y)$ with $a\in \Crit(f_{1})$ and
$x,y\in \Crit(f_{2})$ which are $\Theta$-admissible for $Q$ and
$\Theta$ defined as follows: for all edges $e$ of type $L$,
$q_{+}(e)=-1\in\partial D$, $q_{-}(e)=+1\in \partial D$; if $e$ is an
edge of type $M$ (there can in fact be at most one such edge), then
$q_{+}(e)=0\in D$; $\Theta (f_{2})=f_{2}$,
$\Theta(f_{1},f_{2})=f_{2}$.  The virtual dimension in this case is
$\delta =|a|+|x|-|y|-2 n+\mu([\mathcal{T}])$.

\begin{figure}[htbp]
   \psfig{file=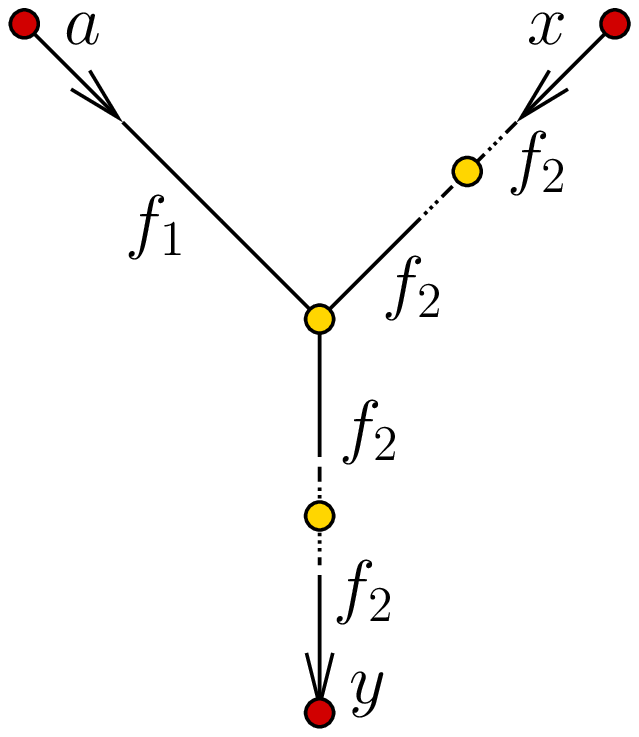, width=0.25 \linewidth} \quad
   \psfig{file=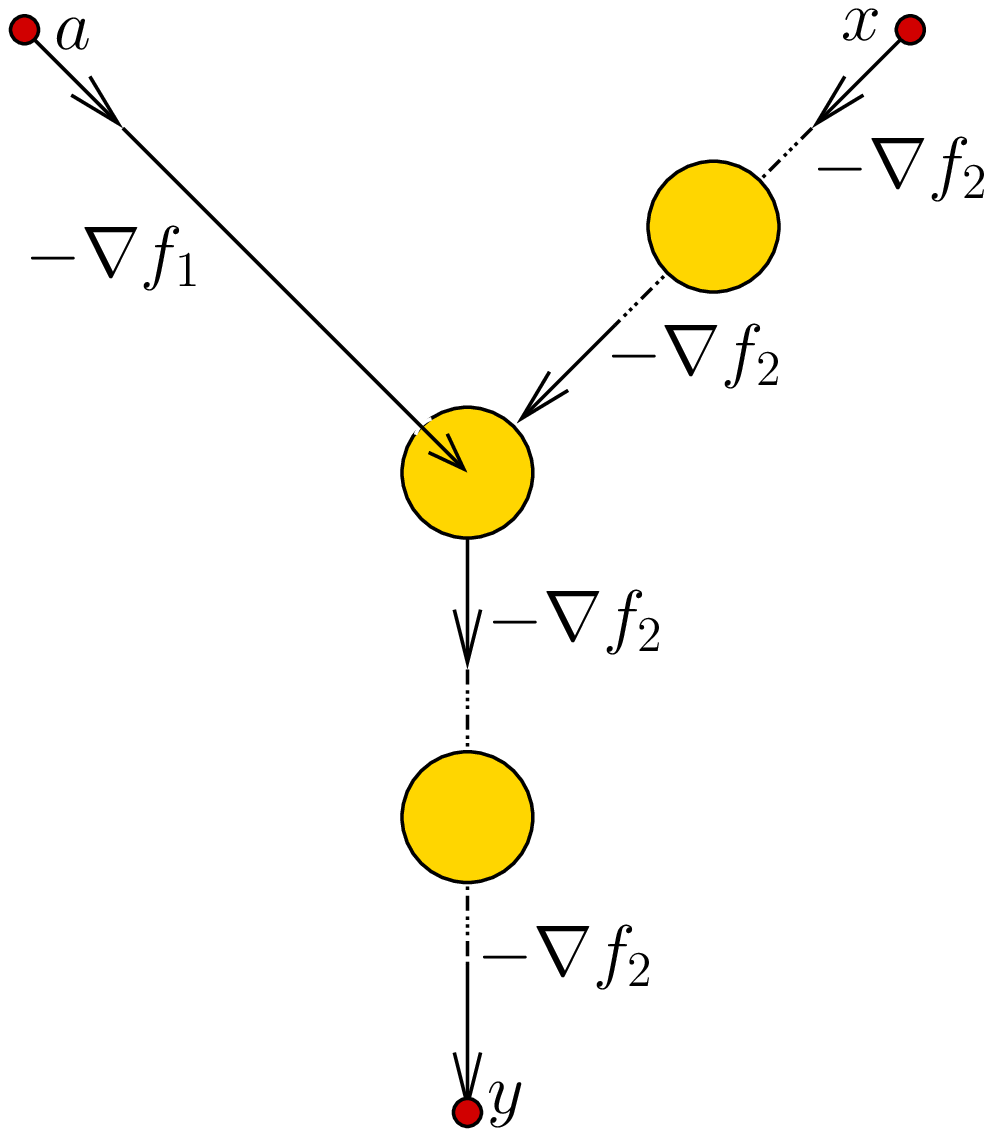, width=0.48 \linewidth}
   \caption{A tree of symbol $(a,x:y)$ on the left, and a pearly
     trajectory corresponding to it on the right.}
   \label{f:pearls-module}
\end{figure}

The same type of moduli spaces but of virtual dimension $1$ serve to
prove that this operation passes to homology. However, at this step a
modification is needed and has to do with the proof of transversality:
we need that in these moduli spaces if a vertex $v$ is of valence
three, then the corresponding curve $u_{v}$ is not pseudo-holomorphic
but rather it carries a small Hamiltonian perturbation of type:
\begin{equation} \label{Eq:delbar-pert}
   \begin{cases}
      u:(D, \partial D) \to (M,L) \\
      \partial_s u + J(u)\partial_t u = -X_F(s,t,u)-J(u)X_G(s,t,u)
   \end{cases}
\end{equation}
with $F,G: D \times M \to \mathbb{R}$ well chosen Hamiltonians and
$X_{F}$ and $X_{G}$ the respective Hamiltonian vector fields (see
\cite{McD-Sa:Jhol-2} and \cite{Bi-Co:qrel-long} for details).  The
reason why these perturbations are needed will be explained in the
next section and we refer to \cite{Bi-Co:qrel-long} for the full
construction.

{\em d. The inclusion $i_{L}$.} In this case we use one Morse function
$f_{1}:L\to \R$ and another Morse function $f_{2}:M\to \R$ and
$\mathcal{F}=\{f_{1},f_{2}\}$.  The relevant $\mathcal{F}$-colored
trees with marked points have symbol $(x:a)$ with $x\in \Crit(f_{1})$,
$a\in\Crit(f_{2})$. The marking is chosen as follows: for all the
edges $e$ of type $L$, $q_{-}(e)=+1$, $q_{+}(e)=-1$; for the edge $e$
of type $M$, $q_{-}(e)=0\in D$ (it is easy to see that the stability
condition iv in \S\ref{subsec:moduli_def} together with the form of
the symbol imply that there can only be a unique edge of type $M$. The
exit rule is $\Theta(f_{1};L)=f_{1}$, $\Theta(f_{1};M)=f_{2}$ (notice
that, this is the first place where the symbol $S$ in the definition
of the exit rule at point v in \S\ref{subsec:moduli_def} is of use;
moreover, because the symbol is $(x:a)$, the only disk type vertex
with the exit edge of type $M$ is the one just before the end of the
tree).  The virtual dimension is in this case
$\delta=|x|-|y|+\mu([\mathcal{T}])$ and the quantum inclusion is
defined by
 $$i_{L}:\mathcal{C}(f_{1})\to CM(f_{2};\mathcal{R})\ ; \  i_{L}(x)=\sum (\#_{2}\mathcal{P}_{\mathcal{T},Q})\ a\ T^{[\mathcal{T}]}~.~$$

 \

 {\em e. Invariance.} Assume given two sets of data $(f, \rho_{L}, J)$
 and $(f',\rho'_{L}, J')$ so that the pearl complexes $\mathcal{C}(L;
 \mathcal{R}; f,\rho_{L},J)$ and $\mathcal{C}(L;\mathcal{R};
 f,\rho'_{L},J')$ are defined. We now need to construct a chain
 morphism:
$$\phi^{F,\tilde{\rho}_{L},\tilde{J}}: 
\mathcal{C}(L; f,\rho_{L},J)\to\mathcal{C}(L;f',\rho'_{L},J')$$ which
induces a canonical isomorphism in homology (we omit the ring
$\mathcal{R}$ from the notation).  This morphism is associated to:
$\tilde{J}=\{J_{t}\}$, a smooth one parametric family of almost
complex structures with $J_{0}=J, J_{1}=J'$, $F:L\times [0,1]\to \R$,
a Morse homotopy (see \cite{Bi-Co:qrel-long} as well as
\cite{Co-Ra:Morse-Novikov}) between $f$ and $f'$ , $\tilde{\rho}_{L}$
a metric on $L\times [0,1]$ with
$\tilde{\rho}|_{L\times\{0\}}=\rho_{L}$ and $\tilde{\rho}|_{L\times
  \{1\}}=\rho'_{L}$.  In other words, we use here a slight
modification of our standard construction by taking
$\mathcal{F}=\{F\}$ and using trees as at point a, but with $F$
replacing $f$, $\tilde{\rho}_{L}$ replacing $\rho_{L}$ and $\tilde{J}$
instead of $J$. The symbol is $(x:y)$ with $x\in
\Crit_{k+1}(F)|_{L\times\{0\}}=\Crit_{k}(f)$ and
$y\in\Crit_{s}(F)|_{L\times\{1\}}=\Crit_{s}(f')$.  In particular, both
the marked point selector $Q$ and the exit rule are the same as at
point a.  The points a,b, c, in \S\ref{subsec:moduli_def} B. are also
modified as follows.

The set $S_{\mathcal{T},Q}$ is now a subset of the product
$$\Pi(\mathcal{T}, \tilde{J})=
\prod_{v\in v_{int}(\mathcal{T}) ,\ t\in [0,1]} \mathcal{M}_{t}([v],
\tilde{J})$$ where
$$\mathcal{M}_{t}([v],\tilde{J})=\{u:(D,\partial D)\to (M\times\{t\},
L\times\{t\}) \ \ | \ \bar{\partial}_{J_{t}}(u)=0\}~.~$$

The flow $\gamma_{t}^{f}$ is replaced by the negative gradient flow,
$\gamma_{t}^{F}$, of $F$ with respect to $\tilde{\rho_{L}}$ (which is
a flow on $L\times [0,1]$ ) and points a,b, c now apply without
further modifications.  In short, the curves which appear at the start
(and respectively the end) of the edge $e$ are $J_{t}$-holomorphic
where $t$ is determined by the second coordinate of the starting point
(respectively, end) of the flow line of $-\nabla(F)$ which corresponds
to $e$. Notice that in our construction all intervening curves are
genuinely $J_{t_{0}}$-holomorphic for some $t_{0}\in [0,1]$ in
contrast to the continuation method familiar in Floer theory.

The virtual dimension is $\delta=|x|-|y|+\mu[\mathcal{T}]$. The
morphism is defined by:
$$\phi^{F,\tilde{\rho}_{L},\tilde{J}}( x) = 
\sum(\#_{2}\mathcal{P}_{\mathcal{T},Q})\ y\ T^{[\mathcal{T}]}~.~$$

An additional parameter is required to show that the morphism induced
in homology is canonical - by constructing a chain homotopy between
any two morphisms as above which is associated to a Morse homotopy of
Morse homotopies.  Perfectly similar constructions provide chain
homotopies which proves the invariance of the quantum product and of
the module structure.

{\em f. The associativity type relations.} The purpose here is to
define the moduli spaces needed to prove the associativity of the
quantum product as well as the other relations at point ii of Theorem
\ref{thm:alg_main}.

For the associativity of the quantum product we will use three
functions $f_{i}:L\to \R$, $i\in \{1,2,3\}$ and the moduli spaces to
be considered are modeled on trees $\mathcal{T}$ of symbol
$(x_{1},x_{2},x_{3}:w)$ with $x_{i}\in\Crit(f_{i})$ and $w\in
\Crit(f_{3})$; the exit rule is $\Theta( f_{k_{1}},\ldots ,
f_{k_{i}})=f_{\max\{k_{1},\ldots, k_{i}\}}$.  We will now define a
particular {\em family} of marked point selectors
$\bar{Q}=\{Q_{\theta}\}$ consisting of one marked point selector
$Q_{\theta}$ for each $\theta\in (0, 2\pi/3)$.  This $Q_{\theta}$ is
as in the definition of the quantum product for all vertices of
valence $2$ and $3$ and in case one vertex $v$ is of valence $4$ then
the first two edges arriving at $v$ (in the planar order) and the exit
edge are attached at the roots of the unity of order $3$ - in the same
way as for the vertices of valence $3$. The third arriving edge $e$
verifies $q_{\theta,+}(e)=e^{i\theta}$.  The moduli spaces used to
prove the associativity of the quantum product are
$$\mathcal{P}_{\mathcal{T},\bar{Q}}=\cup_{\theta\in
  (0,2\pi/3)}\mathcal{P}_{\mathcal{T},Q_{\theta}}\times \{\theta\}
~.~$$

The resulting virtual dimension of this moduli space is $\delta =
|x_{1}|+|x_{2}|+|x_{3}| - |w|+\mu[\mathcal{T}]+1$ (the $+1$ comes from
the additional parameter $\theta$).

Both $0$ and $1$-dimensional such moduli spaces are needed to verify
associativity: the $0$ dimensional moduli spaces are used to define a
chain homotopy $\eta :\mathcal{C}(f_{1})\otimes_{\mathcal{R}}
\mathcal{C}(f_{2})\otimes_{\mathcal{R}}
\mathcal{C}(f_{3})\to\mathcal{C}(f_{3})$ and the $1$ dimensional
moduli spaces are used to prove the relation $((-\circ -) \circ -)
+(-\circ (-\circ -))=(d\eta +\eta d)(-\otimes -\otimes -)$. More
details appear in \cite{Bi-Co:qrel-long}.

\medskip To prove the relation $(a*b)\circledast x=a*(b\circledast x)$
with $a,b\in QH(M;\mathcal{R})$ and $x\in QH(L;\mathcal{R})$ we use
two functions $f_{1}, f_{2}: M\to \R$ and $f_{3}:L\to \R$.  The moduli
spaces in question are modeled on trees $\mathcal{T}$ of symbol
$(a,b,x:y)$ with $a\in\Crit(f_{1})$, $b\in\Crit(f_{2})$,
$x,y\in\Crit(f_{3})$.  The exit rule is $\Theta (f_{k_{1}},\ldots,
f_{k_{s}})=f_{\max\{k_{1},\ldots, k_{s}\}}$.  Again we will need to
define a special family of marked point selectors, denoted in this
case by $\tilde{Q}=\{Q_{\tau}\}$ for $\tau\in (-1,0)$. The marked
point selector $Q_{\tau}$ is as at point c for all vertices of valence
$2$ or $3$. If a vertex is of valence 4 then the marked points are the
same as at point c for the edges of type $L$.  At this vertex there
are also two entering edges of type $M$ and the respective marked
points are as follows: for the edge $e_{1}$ colored with $f_{1}$, we
put $q_{+}(e_{1})=0\in D^{2}$; for the edge $e_{2}$, colored with
$f_{2}$, we put $q_{+}(e_{1})=\tau\in (-1,0)\subset Int(D^{2})$.
Finally the moduli spaces needed here are:
$$\mathcal{P}_{\mathcal{T},\tilde{Q}}=\cup_{\tau\in
  (-1,0)}\mathcal{P}_{\mathcal{T},Q_{\tau}}\times \{\tau\} ~.~$$ We
will again need moduli spaces of this sort and of dimensions $0$ and
$1$. As at point c, to achieve transversality, some of the disks
appearing in these moduli spaces will need to be perturbed by using
perturbations as described by equation (\ref{Eq:delbar-pert}). More
precisely, in the moduli spaces of dimension $0$, if a vertex is of
valence $4$, then its corresponding curve is a perturbed $J$-disk. In
the moduli spaces of dimension $1$, the disks of valence $3$ as well
as the disk of valence $4$ (if present) need to be perturbed. Again,
for more details see \cite{Bi-Co:qrel-long}.

{\em g. Comparison with Floer homology.}  The version of Floer
homology that we need is defined in the presence of a generic
Hamiltonian $H:M\times [0,1]\to \R$.  Consider the path space
$\mathcal{P}_{0}(L)=\{\gamma\in C^{\infty}( [0,1], M)\ | \
\gamma(0)\in L\ , \ \gamma(1)\in L \ , \ [\gamma]=1\in\pi_{2}(M,L)\}$
and inside it the set of (contractible) orbits, or chords,
$\mathcal{O}_{H}\subset \mathcal{P}_{0}(L)$ of the Hamiltonian flow
$X_{H}$. Assuming $H$ to be generic we have that $\mathcal{O}_{H}$ is
a finite set.  Fix a generic almost complex structure $J$.

There is a natural epimorphism $p:\pi_{1}(\mathcal{P}_{0}(L))\to
H_{2}^{D}(M,L)$ and we take $\tilde{\mathcal{P}}_{0}(L)$ be the
regular, abelian cover associated to $\ker (p)$ so that
$H_{2}^{D}(M,L)$ acts as the group of deck transformations for this
covering.  Consider all the lifts
$\tilde{x}\in\tilde{\mathcal{P}}_{0}(L)$ of the orbits $x\in
\mathcal{O}_{H}$ and let $\tilde{\mathcal{O}}_{H}$ be the set of these
lifts.  Fix a base point $\eta_{0}$ in $\tilde{\mathcal{P}}_{0}(L)$
and define the degree of each element $\tilde{x}$ by
$|\tilde{x}|=\mu(\tilde{x},\eta_{0})$ with $\mu$ being here the
Viterbo-Maslov index.  Let $\mathcal{R}$ be a commutative
$\Z_{2}[H_{2}^{D}(M,L)]$-algebra (e.g.  $\mathcal{R}=\Lambda$, or $\La
'$ or $\Z_{2}[H_{2}^{D}(M,L)]$ itself but not $\La^{+}$ or
$\widetilde{\La}^{+}$).

The Floer complex is the $\mathcal{R}$-module:
$$CF_{\ast}(H,J) = \Z_{2}\langle\tilde{\mathcal{O}}_{H}\rangle
\otimes_{\Z_{2}[H_{2}^{D}(M,L)]} \mathcal{R}~.~$$ The differential is
given by $d\tilde{x}=\sum \#\mathcal{M}(\tilde{x},\tilde{y})
\tilde{y}$ where $\mathcal{M}(\tilde{x},\tilde{y})$ is the moduli
space of solutions $u:\R\times [0,1]\to M$ of Floer's equation
$\partial u/\partial s+J\
\partial u/\partial t+\nabla H(u,t)=0$ which verify $u(\R\times
\{0\})\subset L,\ u(\R\times\{1\})\subset L$ and they lift in
$\tilde{\mathcal{P}}_{0}(L)$ to paths relating $\tilde{x}$ and
$\tilde{y}$. Moreover, the sum is subject to the condition
$\mu(\tilde{x},\tilde{y}) -1=0$.

The comparison map from the pearl complex
$$\phi_{f,H}:\mathcal{C}(L; f,\rho_{L}, J)\to CF(L;H,J)$$
is defined by the PSS method (see \cite{PSS} and, in the Lagrangian
case, \cite{Bar-Cor:NATO}, \cite{Cor-La:Cluster-1},\cite{Alb:PSS}) as
well as the map in the opposite direction
$$\psi_{H,f}: CF(L;H,J) \to \mathcal{C}(L; f,\rho_{L}, J)~.~$$

In our language, the map $\phi_{f,H}$ is defined by counting elements
in moduli spaces modeled on trees of symbol $(x:\gamma)$ with $x\in
\Crit(f)$, $\gamma\in \tilde{\mathcal{O}}_{H}$ - thus notice a first
modification of the ``pearl'' construction, the exit of the tree is
labeled in this case by an orbit.  There will be just one Morse
function $f:L\to \R$ and the exit rule as well as the marked point
selector are as at point a (in \S\ref{subsubsec:defin_alg_str}).
However, the last vertex in the tree, the exit, will no longer
correspond to a critical point but rather to a solution $u:\R\times
[0,1]\to M$ of the equation
\begin{equation}\label{eq:perturbed}
   \partial u/\partial s+ J\partial u/\partial t+\beta(s)\nabla H(u,t)=0
\end{equation}
so that $\beta:\R\to [0,1]$ is an appropriate increasing smooth
function supported in the interval $[-1,+\infty)$ and which is
constant equal to $1$ on $[1,+\infty)$. This solution $u$ has also to
verify $u(\R\times\{0\})\subset L$, $u(\R\times\{1\})\subset L$,
$\lim_{s\to\infty}u(s,-)=\gamma(-)$ and $\lim_{s\to
  -\infty}u(s,-)=P\in L$ so that condition c in
\S\ref{subsec:moduli_def} B which describes the geometric relation
associated to the exit edge $e$, is replaced by: `` $\exists\ t>0$ so
that $\gamma^{f_{e}}_{t}(u_{e_{-}}(q_{-}(e)))=P$''.  The map
$\psi_{H,f}$ is given by using similar moduli spaces but with the
first vertex being a perturbed one (the perturbation will use the
function $\beta'=1-\beta$) and starting from an element of
$\tilde{\mathcal{O}}_{H}$. Proving that these maps are chain morphisms
and that their compositions induce inverse maps in homology depends,
in the first instance, on using one-dimensional moduli spaces as above
and, in the second, on yet some other moduli spaces which will produce
the needed chain homotopies. For $\phi_{f,H}\circ\psi_{H,f}$ these
moduli spaces are again modeled on trees with a single entry and exit,
as in the differential of the pearl complex, but both the exit and
entry vertices are of the perturbed type as in (\ref{eq:perturbed})
(with a perturbation $\beta'$ for the entry and $\beta$ for the exit).
In the case of $\psi_{H,f}\circ\phi_{f,H}$ one of the internal
vertices satisfies a perturbed equation but a function $\beta''$ with
support in an interval of type $[-r,r]$ is used instead of $\beta$
(see again \cite{Alb:PSS},\cite{Bi-Co:qrel-long} for details).

{\em h. The augmentation.} Fix a pearl complex $\mathcal{C}(L;
\mathcal{R}; f,\rho, J )$ where $\mathcal{R}$ is a
$\tilde{\Lambda}^{+}$ algebra (as in \S\ref{subsec:coeff}). Define
$$\epsilon_{L} :\mathcal{C}(L;\mathcal{R}; f,\rho , J ) \to \mathcal{R}$$
by $\epsilon_{L} (x) = 0$ for all critical points $x \in\Crit_{>0}(f)$
and $\epsilon_{L}(x) = 1$ for those critical points $x \in\Crit(f)$
with $|x|=1$. Notice that a (local) minimum $x_{0}$ can not appear in
the differential $dy = \sum a_{z,A} zT^{A}$ of any critical point y
except for $A = 0$ and $|y | = 1$. Indeed, a moduli space
$\mathcal{P}_{\mathcal{T}}$ modeled on a tree $\mathcal{T}$ of symbol
$(y: x_{0})$ as at the point a in this section is of dimension $|y
|-1+\mu[\mathcal{T} ]$ and thus can only be of dimension $0$ if
$[\mathcal{T} ] = 0$. Since for each critical point of index $1$ there
are precisely two flow lines emanating from it, we deduce that
$\epsilon_{L}\circ d=0$ and so $\epsilon_{L}$ is a chain map. The same
type of argument, now applied to the comparison map constructed in the
invariance argument at point e shows that, in homology, $\epsilon_{L}$
commutes with the canonical isomorphisms.

\subsection{Transversality} As mentioned before we will not give here
the full proof of transversality (we refer to \cite{Bi-Co:qrel-long}
for that). However, we will review the main ideas.

Given an $\mathcal{F}$-colored tree with marked points
$(\mathcal{T},Q)$ as defined in \S\ref{subsec:moduli_def} we discuss
the proof of the fact that, for generic $J$, the associated moduli
space $\mathcal{P}_{\mathcal{T},Q}$ is a manifold of dimension equal
to the virtual dimension $\delta(\mathcal{T})$. The finite family
$\mathcal{F}$ of Morse functions defined on $L$ or on $M$ is fixed
throughout the section and it contains at most three functions defined
on $L$ and two defined on $M$. The only moduli spaces to be treated
are those appearing in \S\ref{subsubsec:defin_alg_str}.

In the argument, slightly more general such moduli spaces will also be
needed. As before, the numbers of entries will always be at most $3$
and there will be a single exit. However, we will not impose any
particular restriction on the exit rule (in particular, all possible
exit rules will be allowed in the inductive argument below). Secondly,
we will need to prove the regularity of moduli spaces of type
$$\mathcal{P}_{\mathcal{T},Q}=
\bigcup_{s\in U}\mathcal{P}_{\mathcal{T},Q_{s}}\times \{s\}$$ where
$Q=\{Q_{s}\}_{s\in U}$ is a {\em family} of marked point selectors
$Q_{s}$ so that at most two of the marked points provided by $Q_{s}$
(and which are associated to vertices of valence at least $3$) are
allowed to take the values in the set $U$.  Here $U=U_{1}\times U_{2}$
where both $U_{i}\subset D$ are connected submanifolds without
boundary of dimension at most $2$.  These types of moduli spaces have
already appeared in the discussion of associativity at the point f in
\S \ref{subsubsec:defin_alg_str} and some additional ones will appear
in the transversality argument.  More precisely, our allowed choices
for these sets $U_{i}$ are as follows.  If $\dim U_{i}=0$, then
$U_{i}$ coincides with one of the marked points appearing in the
description of the marked point selectors in
\S\ref{subsubsec:defin_alg_str} (in other words, $U_{i}$ is one of the
points $+1,-1,e^{2\pi i/3} , e^{4\pi i/3}, 0\in D$); if $\dim
U_{i}=1$, then $U_{i}$ is one of the following two choices
$(-1,0)\subset D$ or $\{e^{it}\}_{0\leq t\leq 2\pi/3}\subset
\partial D$ (both have been already used at point f in
\S\ref{subsubsec:defin_alg_str}); finally, if $\dim U_{i}=2$, then
$U_{i}= Int(D)$.  We will still refer to these moduli spaces by
$\mathcal{P}_{\mathcal{T},Q}$ and refer to them as
$\mathcal{F}$-colored moduli spaces with marked points and, by a
slight abuse of notation, $Q$ will still be referred to as a marked
point selector.  The virtual dimension of these moduli spaces is given
by a formula similar to (\ref{eq:virtual_dim}) to which is added
another term depending on the dimension of the sets $U_{i}$ as above
and on the valence of the vertices to which these marked points are
associated. In view of this, we denote this virtual dimension by
$\delta(\mathcal{T},Q)$.

\

Let $\mathcal{P}^{\ast}_{\mathcal{T},Q}$ be the moduli spaces
associated to $\mathcal{F}$-colored trees with marked points
$(\mathcal{T},Q)$ which satisfy the additional condition that all the
$J$-holomorphic curves $u_{v}$ corresponding to the internal vertexes
$v\in v(\mathcal{T})$ have the property that they are {\em simple} and
that they are {\em absolutely distinct}. We recall that a curve
$u:\Sigma\to M$ is simple if it is injective at almost all points
$z\in Int(\Sigma)$ in the sense that $du_{z}\not=0$ and
$u^{-1}(u(z))=\{z\}$. The curves $(u_{v})$ are absolutely distinct if
no single curve $u_{v}$ has its image included in the union of the
images of the others, $Im(u_{v})\not\subset \cup_{v'\in
  v(\mathcal{T})\backslash\{v\}} Im(u_{v'})$.  By a straightforward
adaptation of now standard techniques, as in \cite{McD-Sa:Jhol-2}
Chapter 3 in particular Proposition 3.4.2, we obtain that
$\mathcal{P}^{\ast}_{\mathcal{T},Q}$ is a manifold of dimension
$\delta(\mathcal{T},Q)$, in general non-compact, with a boundary
consisting of configurations so that some edges in $\mathcal{T}$ are
represented by gradient flow lines of $0$-length (recall that we allow
the length of edges to be $\geq 0$).  Notice that, in case some
perturbed $J$-holomorphic curves appear also in the elements of
$\mathcal{P}_{\mathcal{T},Q}$ as at c in
\S\ref{subsubsec:defin_alg_str}, there is no need to impose any
similar condition to them: a choice of generic perturbations insures
the needed transversality.  To simplify the argument, we focus in the
proof below on the case where just a single almost complex structure
appears in the definition of our moduli spaces.  However, if as for
the invariance argument, point e in \S\ref{subsubsec:defin_alg_str},
we need to deal with a family $\tilde{J}=\{J_{t}\}_{t\in [0,1]}$ of
almost complex structures, then the ``absolutely distinct'' condition
only needs to be verified for the disks that are $J_{t}$-holomorphic
for each $t$ at a time and by taking this remark into account the
argument below adapts easily to this setting.

The key point is to show that
$\mathcal{P}^{\ast}_{\mathcal{T},Q}=\mathcal{P}_{\mathcal{T},Q}$ as
long as $\delta(\mathcal{T},Q)\leq 1$. In turn, the proof of this is
by induction.  To be more explicit, fix the symbol $
\sym(\mathcal{T})=(x_{1},x_{2},\ldots , x_{l}: y)$ of the tree
$\mathcal{T}$.  Fix some $k\in\N$. The combinatorial data used to
define $\mathcal{F}$-colored trees with marked points
$(\mathcal{T},Q)$ so that $\mu[\mathcal{T}]\leq k$ is finite. Thus, up
to isomorphism, there are only finitely many such trees.  Suppose, by
induction, that for all $\mathcal{F}$-colored trees with marked points
$(\mathcal{T'},Q')$ of symbol of length at most $4$ and with
$\mu[\mathcal{T}']<\mu([\mathcal{T}])$ and
$\delta(\mathcal{T}',Q')\leq 1$, we have
\begin{equation}\label{eq:regular_moduli}
   \mathcal{P}^{\ast}_{\mathcal{T}',Q'}=\mathcal{P}_{\mathcal{T}',Q'}~.~
\end{equation}
To prove identity (\ref{eq:regular_moduli}) for $\mathcal{T}$ it
suffices to show that the following {\em simplification} step is true:
\begin{equation}\label{eq:implication_regularity}
   \begin{matrix}
      \mathcal{P}_{\mathcal{T},Q}\not=
      \mathcal{P}^{\ast}_{\mathcal{T},Q} \ \Rightarrow\ & \exists
      (\mathcal{T}',Q')\ \mathrm{\ \ such\ that\ \ \
      }\sym(\mathcal{T}')=\sym(\mathcal{T}),
      \mu([\mathcal{T}'])<\mu([\mathcal{T}]), \\ &
      \delta(\mathcal{T}',Q')< 0 \ ,
      \mathcal{P}_{\mathcal{T}',Q'}\not=\emptyset ~.~
   \end{matrix}
\end{equation}
Indeed, if $\delta(\mathcal{T}',Q')<0$ , the identity
(\ref{eq:regular_moduli}) together with the regularity of the moduli
spaces consisting of simple, absolute distinct curves implies that
$\mathcal{P}_{\mathcal{T}',Q'}=\emptyset$ and the conclusion follows
by contradiction.

The key to prove (\ref{eq:implication_regularity}) is a structural
result concerning $J$-holomorphic disks which is the disk counterpart
of the {\em multiply-covered $\leftrightarrow$ almost everywhere
  injective} dichotomy valid in the case of $J$-holomorphic spheres.
One such result is due to Lazzarini \cite{Laz:discs}\cite{Laz:decomp}
(an alternative one is due to Kwon-Oh \cite{Kw-Oh:discs}). Here are
more details on this point.

Let $u:(D,\partial D) \to (M,L)$ be a non-constant $J$-holomorphic
disk. Put $\mathcal{C}(u)=u^{-1}(\{ du=0 \})$. Define a relation
$\mathcal{R}_u$ on pairs of points $z_1, z_2 \in \textnormal{Int\,}D
\setminus \mathcal{C}(u)$ in the following way:
\begin{equation*}
   z_1 \mathcal{R}_u z_2 \Longleftrightarrow
   \begin{cases}
      & \forall \textnormal{ neighborhoods } V_1, V_2 \textnormal{ of
      } z_1,z_2, \\
      & \exists \textnormal{ neighborhoods } U_1, U_2
      \textnormal{ such that:} \\
      & \textnormal{(i) } z_1 \in U_1 \subset V_1 ,
      z_2 \in U_2 \subset V_2. \\
      & \textnormal{(ii) } u(U_1)=u(U_2).
   \end{cases}
\end{equation*}
Denote by $\overline{\mathcal{R}}_u$ the closure of $\mathcal{R}_u$ in
$D \times D$. Note that $\overline{\mathcal{R}}_u$ is reflexive and
symmetric but it may fail to be transitive (see~\cite{Laz:decomp} for
more details on this). Define the {\em non-injectivity graph} of $u$
to be:
$$\mathcal{G}(u) = \{ z \in D \, | \, \exists z' \in \partial D
\textnormal{ such that } z \overline{\mathcal{R}}_u z' \}.$$ It is
proved in~\cite{Laz:decomp,Laz:discs} that $\mathcal{G}(u)$ is indeed
a graph and its complement $D \setminus \mathcal{G}(u)$ has finitely
many connected components. We use the following theorem due to
Lazzarini (See Proposition~4.1 in~\cite{Laz:decomp} as well
as~\cite{Laz:discs}).

\begin{thm}[Decomposition of disks, \cite{Laz:decomp,Laz:discs}]
   \label{T:decomposition}
   Let $u:(D,\partial D) \to (M,L)$ be a non-constant $J$-holomorphic
   disk. Then for every connected component $\mathfrak{D} \subset D
   \setminus \mathcal{G}(u)$ there exists a surjective map
   $\pi_{\overline{\mathfrak{D}}}: \overline{\mathfrak{D}} \to D$,
   holomorphic on $\mathfrak{D}$ and continuous on
   $\overline{\mathfrak{D}}$, and a simple $J$-holomorphic disk
   $v_{\mathfrak{D}}:(D,\partial D) \to (M,L)$ such that
   $u|_{\overline{\mathfrak{D}}} = v_{\mathfrak{D}}\circ
   \pi_{\overline{\mathfrak{D}}}$. The map
   $\pi_{\overline{\mathfrak{D}}}:\overline{\mathfrak{D}} \to D$ has a
   well defined degree $m_{\mathfrak{D}} \in \mathbb{N}$ and we have
   in $H_2^{D}(M,L;\mathbb{Z})$:
   $$[u]=\sum_{\mathfrak{D}} m_{\mathfrak{D}} [v_{\mathfrak{D}}],$$
   where the sum is taken over all connected components $\mathfrak{D}
   \subset D \setminus \mathcal{G}(u)$.
\end{thm}

Two Lemmas, \ref{L:non-simple} and \ref{L:uv}, to be stated a bit
later, are easy consequences of the theorem above and, as we will see,
they reduce our problem to a sequence of combinatorial verifications.

\

Returning to the proof of (\ref{eq:implication_regularity}) we proceed
in two steps. First we discuss the argument insuring that all
$J$-curves involved are simple. The second step will show that they
can also be assumed to be absolutely distinct.  We focus here on the
case $\dim(L)\geq 3$ and will comment on the case $\dim(L)\leq 2$ at
the end.

Thus, suppose that $u\in \mathcal{P}_{\mathcal{T},Q}$ is so that
$u=(u_{v})_{v\in v_{int}(\mathcal{T})}$ and for some internal vertex
$v\in v(\mathcal{T})$ the corresponding $J$-holomorphic curve $u_{v}$
is not simple.

In the trees used in this paper a sphere-type vertex does not carry
more than three incidence points. Therefore, in case $u_{v}$ is a
$J$-sphere it can clearly be replaced by a simple one $u'_{v}$ and the
marked point selector is not modified. This means that we may take in
this case $\mathcal{T}'$ to be topologically the same tree as
$\mathcal{T}$ except that the label of the vertex $v$ is now
$[u'_{v}]$ instead of $[u_{v}]$.  Thus we may now suppose that $u_{v}$
is a $J$-disk.  To deal with this case we will make use of the
following consequence of Theorem \ref{T:decomposition}.  We refer to
\cite{Bi-Co:qrel-long} for the proof.

\begin{lem} \label{L:non-simple} Suppose $n=\dim L \geq 3$. Then there
   exists a second category subset $\mathcal{J}_{\textnormal{reg}}
   \subset \mathcal{J}(M,\omega)$ such that for every $J \in
   \mathcal{J}_{\textnormal{reg}}$ the following holds. For every
   non-constant, non-simple $J$-holomorphic disk $u:(D, \partial D)
   \to (M,L)$ there exists a $J$-holomorphic disk $u':(D,
   \partial D) \to (M,L)$ with the following properties:
   \begin{enumerate}
     \item $u'(D) = u(D)$ and $u'(\partial D) = u(\partial D)$.
     \item $u'$ is simple.
     \item $\omega([u']) < \omega([u])$. In particular, if $L$ is
      monotone we also have $\mu([u']) < \mu([u])$.
   \end{enumerate}
\end{lem}

We apply Lemma \ref{L:non-simple} to replace the $J$-disk $u_{v}$ by
the simple disk $u'_{v}$ provided by the Lemma. Thus, to
prove~\eqref{eq:implication_regularity}, the relevant tree
$\mathcal{T}'$ that we are looking for is identified with
$\mathcal{T}$ except that the vertex $v$ will now be labeled by
$[u'_{v}]$.  A slightly delicate point needs to be made concerning the
marked point selector $Q'$ corresponding to $\mathcal{T}'$. The way
this is constructed is the following: as $u'_{v}(D)=u_{v}(D)$, and
$u'_{v}(\partial D)=u_{v}(\partial D)$, the points $u_{v}(q_{\pm}(e))$
(where $e$ is an incident edge at $v$) can be lifted to the domain of
$u'_{v}$ and used as marked points there. Of course, this works only
if all these points, $u_{v}(q_{\pm}(e))$, are distinct. If this is not
the case some additional vertices need to be included in the tree so
that they correspond to constant disks or spheres which are related to
the vertex $v$ by edges colored by functions in $\mathcal{F}$ and of
$0$-length.

We still need to verify that $\delta (\mathcal{T}',Q')< 0$.  Given
that $N_{L}\geq 2$ and so $\mu(u'_{v})<\mu(u_{v})-1$ this inequality
is automatic if $Q'=Q$ because in this case
$\delta(\mathcal{T}',Q')\leq \delta(\mathcal{T},Q)-N_{L}$.  This is
the case if $v$ carries two or three marked points all on $\partial
D$. The same is true also if $v$ carries two marked points, one on the
boundary and one in the interior of $D$.  Suppose now that $v$ carries
two boundary marked points, $-1$ and $+1$, and the interior marked
point $0$ (as at point c in \S\ref{subsubsec:defin_alg_str}).  In this
case the marked point selector for $\mathcal{T}'$ can not be assumed
to be the same as that for $\mathcal{T}$: the internal marked point
for $u_{v'}$ can not be assumed anymore to be as assigned by $Q$ but
can be anywhere inside $D$ - in other words in this case
$Q'=\{Q_{s}\}_{s\in Int(D)}$.  In this situation we have
$\mu[\mathcal{T}']\leq\mu([\mathcal{T}])-N_{L}$ and it is easy to see
that $\delta(\mathcal{T}',Q')\leq \delta(\mathcal{T},Q)-N_{L}+1$.
Thus, if $\delta(\mathcal{T},Q)=0$ we still have
$\delta(\mathcal{T}',Q')<0$ so that (\ref{eq:regular_moduli}) remains
true for the moduli spaces needed to define the module structure
without the need to use any perturbations.  However, to prove the fact
that the operation defined there is a chain morphism we need to use
moduli spaces as before but which verify $\delta(\mathcal{T},Q)=1$.
This is precisely why we use perturbed $J$-holomorphic disks in this
case: as mentioned before, the proof of the transversality of the
relevant evaluation maps requires only the non-perturbed
$J$-holomorphic curves to be simple and absolutely distinct. The same
issue appears for both the $0$ and $1$-dimensional moduli spaces used
to prove the associativity of the module action as in the second part
of point f in \S\ref{subsubsec:defin_alg_str} and this shows that the
perturbations indicated there are necessary.  Full details for these
arguments are found in \cite{Bi-Co:qrel-long}.

We now pass to the second step: showing that the $J$-curves
$\{u_{v}\}_v$ are absolutely distinct. The main tool is the next
result which can be deduced too from Theorem \ref{T:decomposition}.
\begin{lem} \label{L:uv} Suppose $n=\dim L \geq 3$. Then there exists
   a second category subset $\mathcal{J}_{\textnormal{reg}} \subset
   \mathcal{J}(M,\omega)$ such that for every $J \in
   \mathcal{J}_{\textnormal{reg}}$ the following holds. If
   $u,w:(D,\partial D) \to (M,L)$ are simple $J$-holomorphic disks
   such that $u(D) \cap v(D)$ is an infinite set, then:
   \begin{itemize}
     \item[i.] either $u(D) \subset w(D)$ and $u(\partial D) \subset
      w(\partial D)$; or
     \item[ii.] $w(D) \subset u(D)$ and $w(\partial D) \subset
      u(\partial D)$.
   \end{itemize}
\end{lem}

This implies that if the $J$-curves in $\{u_{v}\}_{v \in
  v_{int}(\mathcal{T})}$ are not absolutely distinct, then there exist
two vertices $v_0$ and $v_1$ both corresponding to $J$-holomorphic
(unperturbed curves) so that $u_{v_1}(D)\subset u_{v_0}(D)$ and
$u_{v_1}(\partial D)\subset u_{v_0}(\partial D)$.  The aim now is to
show that we can ``simplify'' both $\{u_v\}_v$ and the tree by
eliminating $v_1$ (as well as possibly other vertices and edges) and
thus produce a new tree $(\mathcal{T}',Q')$ of lower Maslov number and
with $\delta(\mathcal{T}',Q')<0$ as well as a new element $\{u'_v\}_{v
  \in v_{int}(\mathcal{T}')} \in \mathcal{P}_{(\mathcal{T}', Q')}$
thus arriving at a contradiction.

There are three different cases to consider:
\begin{enumerate}[i.]
  \item $v_0$ and $v_1$ are independent, in the sense that they are on
   different branches of the tree.
  \item $v_0$ is above $v_1$ in the tree, in the sense that by
   following the tree starting from $v_0$ we reach $v_1$.
  \item $v_1$ is above $v_0$ in the tree.
\end{enumerate}
In the first two cases we obtain the new tree $\mathcal{T}'$ by simply
taking the branch in the tree above $v_0$ but containing $v_0$ and
pasting it in $\mathcal{T}$ with $v_0$ in the place of $v_1$.  Thus in
the tree $\mathcal{T}'$ the vertex $v_1$ has disappeared and has been
replaced with $v_{0}$. To avoid confusion we denote this vertex in
$\mathcal{T}'$ by $\hat{v}_{1}$. The corresponding pearly element
$\{u'_v\}_{v \in v_{int}(\mathcal{T}')}$ will satisfy $u'_v = u_v$ for
every $v \neq \hat{v}_1$ and $u'_{\hat{v}_1} = u_{v_0}$.  A similar
construction can be performed in the third case. Here is a more
precise description of this operation in each of the cases i-iii.

In case i we first remove from $\mathcal{T}$ the branch $B_{v_0}$ of
the tree lying above $v_0$ (and including $v_0$). Then we also remove
from $\mathcal{T}$ the path going from $v_0$ to the branch point below
$v_0$ which is closest to $v_0$. Denote the remaining tree by
$\mathcal{T}_0$. We define $\mathcal{T}'$ by gluing $B_{v_0}$ to
$\mathcal{T}_0$ identifying $v_0$ with $v_1$. This new vertex will be
now denoted by $\hat{v}_1$. We label $\hat{v}_1$ by the homology class
of $v_0$ and we define $Q'$ at $\hat{v}_1$ using the marked points of
both $v_0$ and $v_1$ except of the exit marked point of $v$ which
becomes irrelevant now and is hence dropped. See
figure~\ref{f:tree-oper-1} for an example.
\begin{figure}[htbp]
   \psfig{file=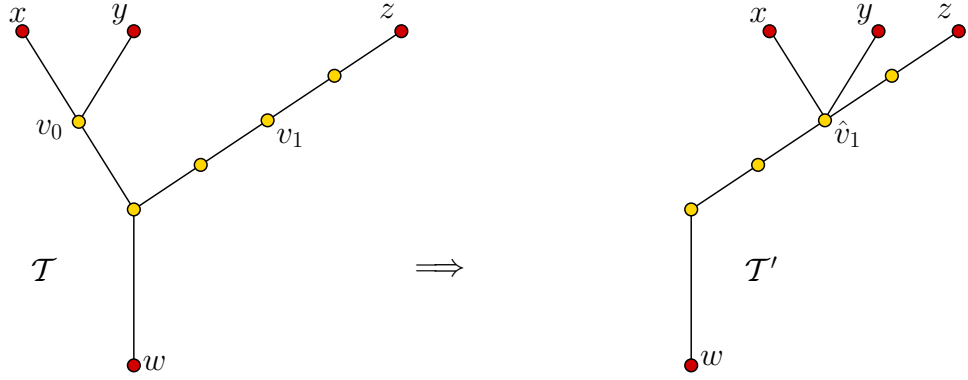, width=0.8 \linewidth}
   \caption{Passing from $\mathcal{T}$ to $\mathcal{T}'$ - case i.}
   \label{f:tree-oper-1}
\end{figure}

In case ii, if there is a branch point $\tilde{v}_{0,1}$ between $v_0$
and $v_1$ we define $\mathcal{T}'$ as follows. We delete from
$\mathcal{T}$ the branch $B_{v_0}$ as in case i above. We also delete
from $\mathcal{T}$ the path between $v_0$ and $\tilde{v}_{0,1}$ and
denote the remaining tree by $\mathcal{T}_0$. We define $\mathcal{T}'$
as in case i by gluing $B_{v_0}$ to $\mathcal{T}_0$ identifying $v_0$
with $v_1$, calling the this new vertex $\hat{v}_1$. As in case i
above, we label $\hat{v}_1$ by the homology class of $v_0$ and define
$Q'$ using the marked points of both $v_0$ and $v_1$, excluding the
exit marked point of $v_0$. See figure~\ref{f:tree-oper-2} for an
example.
\begin{figure}[htbp]
   \psfig{file=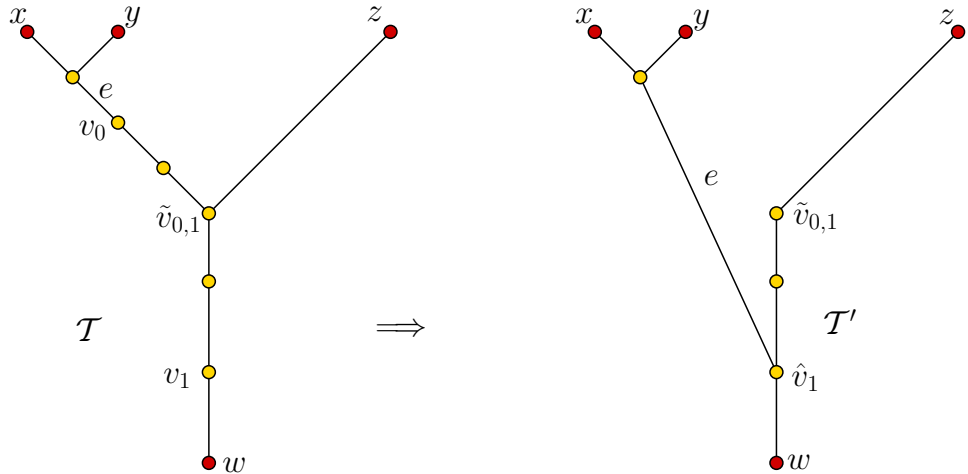, width=0.8 \linewidth}
   \caption{Passing from $\mathcal{T}$ to $\mathcal{T}'$ - case ii.}
   \label{f:tree-oper-2}
\end{figure}
To conclude case ii we need to describe $\mathcal{T}'$ in case there
is no branch point between $v_0$ and $v_1$. In that case, we just
define $\mathcal{T}'$ by removing the path between $v_0$ and $v_1$ and
identifying $v_0$ with $v_1$. This new vertex $\hat{v}_1$ is labeled
by the label of $v_0$ and the marked points now are inherited from
$v_0$ and $v_1$ except of the exiting marked point of $v_0$ and the
corresponding entering marked point of $v_1$ which are now dropped.

Suppose we are now in case iii, i.e. $v_0$ is lower than $v_1$ in the
tree. This case is dealt with similarly to ii. In this case, the tree
$\mathcal{T}'$ is obtained as before but with the roles of the
vertices $v_0$ and $v_1$ reversed: the branch above $v_1$ and
containing $v_1$ is grafted to the tree in the place of $v_0$ and the
branch leaving from $v_1$ and reaching the first branch point
separating $v_1$ and $v_0$ (or the portion in the tree between $v_1$
and $v_0$ if no such branch point exists) is omitted. The new vertex
(corresponding to $v_0$ and $v_1$ is now called $\hat{v}_0$. Again the
$J$-curves associated to the vertices of $\mathcal{T}'$ are the same
as the corresponding curves associated to the vertices of
$\mathcal{T}$ except that $u_{\hat{v}_0}=u_{v_0}$.

There is yet another point at which care should be taken (in all cases
i-iii). It may happen that some of the relevant marked points of $v_0$
and of $v_1$ coincide (again, we disregard those marked points that
are dropped as above), and in this case the description given above
for $(\mathcal{T}', Q')$ is incomplete. If such a coincidence of
marked points occurs we need to insert some additional vertices,
corresponding to constant $J$-curves, carrying distinguished marked
points as well as connecting edges. This modification is
straightforward and we will not go into more detail about it.

It now easily follows that the resulting tree $\mathcal{T}'$ has a
strictly lower Maslov index than $\mathcal{T}$. The dimension
verification is also immediate except if $v_1$ carries some internal
marked points. If there is a single such marked point and
$\delta(\mathcal{T},Q)=0$, then we take $Q'=\{Q_{s}\}_{s\in
  \textnormal{Int\,} D}$ (because the internal marked point may now
take any value inside $D$) and we still have, as in the reduction to
simple disks, $\delta(\mathcal{T}',Q')<0$. If $v_1$ carries two
internal incidence points or if it carries one but
$\delta(\mathcal{T},Q)=1$, then, by the particular choice of the
moduli spaces in \S\ref{subsubsec:defin_alg_str}, $v_1$ corresponds to
a perturbed $J$-holomorphic disk in contradiction to our starting
assumption.

\

The case $n\leq 2$ is easily reduced to a number of combinatorial
problems. The assumptions $n\leq 2$, $N_{L}\geq 2$ and
$\delta(\mathcal{T},Q)\leq 1$ imply that the total number of
$J$-curves is relatively small (for example, there are at most two for
the verifications involving the pearl complex) so that combinatorial
arguments apply in many of these cases. In fact, it is not hard to use
directly Theorem~\ref{T:decomposition} to deal with trees
$\mathcal{T}$ in which the total Maslov index of the vertices
represented by $J$-disks is at most $6$ (even if there might be
additional vertices corresponding to perturbed disks). This covers all
the verifications involved with the pearl complex and its invariance,
the product and its associativity and invariance, the definition of
the module structure and its invariance.  This also works for the
proof of the relation $(a\ast b)\circledast x=a\circledast
(b\circledast x), \ a,b\in QH(M;\mathcal{R}),\ x\in QH(L;\mathcal{R})$
for $N_{L}\geq 3$. Finally, the remaining case can also be dealt with
combinatorially.

\subsection{Compactness and the final step} The transversality
arguments in the previous section show that our moduli spaces are
manifolds.  We will start here by describing the structure of the
compactification of these moduli spaces. For this, besides the
transversality results described before, we only need the Gromov
compactness theorem (for disks see Frauenfelder \cite{Fr:msc}).  We
first remark that given an $\mathcal{F}$-colored tree with marked
points $(\mathcal{T},Q)$ and an associated moduli space
$\mathcal{P}_{\mathcal{T},Q}$ - constructed as described in
\S\ref{subsec:moduli_def}- there is a natural Gromov type topology on
$\mathcal{P}_{T}$ as well as a natural compactification
$\bar{\mathcal{P}}_{T}$.

In short, the elements of $\bar{P}_{\mathcal{T}}\backslash
P_{\mathcal{T}}$ are modeled on the tree $\mathcal{T}$ and the only
modification with respect to our definition in
\S\ref{subsec:moduli_def} concerns the points a, b, c, at the end of
that section.  Specifically, the product $\Pi(\mathcal{T})$ is
replaced by its compactification
$$\bar{\Pi}(\mathcal{T}) =
\prod_{v\in v_{int}(\mathcal{T})}\ \bar{\mathcal{M}}([v], J)\ , $$
where $\bar{\mathcal{M}}([v],J)$ is the Gromov compactification of
$\mathcal{M}([v],J)$ so that, for each internal vertex $v$, the
associated geometric object $u_{v}\in \bar{\mathcal{M}}([v],J)$.  The
points a,b, and c, are then replaced by the following variants:
\begin{itemize}
  \item[a'.] For each internal edge $e\in e(\mathcal{T})$, the points
   $u_{e_{-}}(q_{-}(e))$ and $u_{e_{+}}(q_{+}(e))$ are related by a
   possibly broken flow line (possibly of $0$-length) of
   $\gamma^{f_{e}}$.
  \item[b'.] For an entry edge, $e$, let $x_{i}$ be the critical point
   labeling the vertex $e_{-}$.  The point $x_{i}$ is related to the
   point $u_{e_{+}}(q_{+}(e))$ by a possibly broken flow line
   (possibly of $0$-length) of $\gamma^{f_{e}}$
  \item[c'.] For the exit edge $e$ so that the vertex $e_{+}$ is
   labeled by the critical point $y$ of $f_{e}$, the point
   $u_{e_{-}}(q_{-}(e))$ is related to $y$ by a possibly broken flow
   line (possibly of $0$-length) of $\gamma^{f_{e}}$.
\end{itemize}
A remark is needed concerning the marked point selectors. The various
marked points which correspond to the same vertex in the
configurations described above are again required to be distinct and
are given in the same way as that described in
\S\ref{subsubsec:defin_alg_str}. In particular, each time two (or
more) such incidence points ``merge'' a ghost curve needs to be
introduced.

\

From now on we will only focus on $\mathcal{F}$-colored trees that
are of virtual dimension $\delta(\mathcal{T})\leq 1$ and, in view of
our transversality results, we may assume that
(\ref{eq:regular_moduli}) is satisfied so that
$\mathcal{P}^{\ast}_{\mathcal{T}}=\mathcal{P}_{\mathcal{T}}$ (the role
of the marked point selector is less crucial in this part and we will
omit it from the notation).  Under this hypothesis, the first key
remark is that each element in
$\bar{u}\in\bar{\mathcal{P}}_{\mathcal{T}}\backslash
\mathcal{P}_{\mathcal{T}}$ contains exactly one configuration among
the three types below:
\begin{itemize}
  \item[i.] a flow line broken exactly once.
  \item[ii.] a vertex $v_{\bar{u}}\in v(\mathcal{T})$ corresponds to a
   cusp curve with precisely two components (which can be ghosts).
  \item[iii.]  a flow line of length $0$.
\end{itemize}
The reason for this is that if more than a single such configuration
occurs we can extract from $\bar{u}$ an object
$u'\in\mathcal{P}_{\mathcal{T}'}$ with $\delta(\mathcal{T}')< 0$ which
is impossible because such a moduli space of negative virtual
dimension is regular and thus void.

The second important remark is that the condition $N_{L}\geq 2$
insures that no ``lateral'' bubbling is possible. More explicitly,
this means that if the element $\bar{u}$ verifies condition ii, then
the incidence points associated to the vertex $v_{u}$ are distributed
among the two components of the cusp curve so that not all of them are
in just one component. This happens because, otherwise, the component
which does not carry any of these incidence points can be omitted thus
giving rise to an object $u'$ which belongs to a moduli space of
virtual dimension lower by at least $N_{L}$ than
$\delta(\mathcal{P}_{\mathcal{T}})$ which again is not possible.

\

The last step is to use the description of the compactification given
above to verify the various relations required to establish the
theorem (as described at the points a, b, c, d, e, f, g, and i in \S
\ref{subsubsec:defin_alg_str}). The technical ingredient for this
verification is gluing.  Gluing procedures have already appeared for
example in \cite{FO3} and for full details we refer again to
\cite{Bi-Co:qrel-long}. This gluing procedure insures that, when
$\delta(\mathcal{T})=1$, each element $\bar{u}$ which is modeled on
the tree $\mathcal{T}$ and which verifies exactly one of the
properties i, ii, iii above actually belongs to
$\bar{\mathcal{P}}_{\mathcal{T}}\backslash \mathcal{P}_{\mathcal{T}}$
and appears as a boundary element of
$\bar{\mathcal{P}}_{\mathcal{T}}$.

Finally, the verification of the relations mentioned involves in an
essential way the fact that our algebraic operations are defined by
using $\Theta$-admissible trees. The role of the exit rule $\Theta$
(as described at point v in \S\ref{subsec:moduli_def}) is as follows:
for a tree $\mathcal{T}$ with $\delta(\mathcal{T})=1$, if $\bar{u}\in
\partial \bar{\mathcal{P}}_{\mathcal{T}}$ verifies ii above, then, due
to the fact that ``lateral'' bubbling is not possible, $\bar{u}$ is
also an element of $\partial\bar{\mathcal{P}}_{\mathcal{T}'}$ where
$\mathcal{T}'$ is the tree obtained from $\mathcal{T}$ by replacing
the vertex $v_{\bar{u}}$ by two vertices (corresponding to the two
components of the cusp curve associated to $v_{\bar{u}}$) related by
an edge of length $0$ whose type is {\em uniquely} determined by the
exit rule.  Moreover, by gluing, each $\bar{u}\in
\partial \bar{\mathcal{P}}_{\mathcal{T}}$ verifying iii is an element
in the boundary of a moduli space modeled on a tree obtained from
$\mathcal{T}$ by replacing the two vertices related by the edge of
$0$-length by a single vertex.  Denote by
$\partial_{1}(\bar{\mathcal{P}}_{\mathcal{T}})$ the parts of the
boundary of $\bar{\mathcal{P}}_{\mathcal{T}}$ formed by the points
verifying i.  and fix the symbol $(x_{1},x_{2},\ldots x_{l} : y)$ of
$\mathcal{T}$. When summing over all trees (of virtual dimension $1$)
and of fixed symbol we see that the configurations of types ii and iii
cancel (as we will see below, due to the presence of perturbations in
some of our moduli spaces, an additional argument is sometimes needed
at this point) and so we deduce:
$$\sum_ {\mathrm{ symb\/} (\mathcal{T'}) =
  (x_{1},\ldots x_{l}: y)}\#\
\partial_{1}(\bar{\mathcal{P}}_{\mathcal{T}'})=0~.~$$

The relations that need to be justified are then obtained by
identifying each element $\bar{u}\in
\partial_{1}(\mathcal{P}_{\mathcal{T}})$ of type i with precisely one
element of the product
$\mathcal{P}_{\mathcal{T}_{1}}\times\mathcal{P}_{\mathcal{T}_{2}}$
where $\mathcal{T}_{1}$ and $\mathcal{T}_{2}$ are the two trees
obtained as follows: first, introduce in $\mathcal{T}$ an additional
vertex $\tilde{v}$ on the edge which corresponds to the broken flow
line and then let $\mathcal{T}_{1}$, $\mathcal{T}_{2}$ be the two
(sub)-trees which have in common only the vertex $\tilde{v}$ and whose
union gives $\mathcal{T}$.

Clearly, in what concerns the comparison with Floer homology - point g
in \S\ref{subsubsec:defin_alg_str}~- the argument above needs to be
modified slightly. The required modification is however obvious and we
will not discuss it further.  However, a more substantial addition to
the argument is needed in the case of the perturbations of type
(\ref{Eq:delbar-pert}) which were introduced in the moduli spaces
needed to verify that the module action is a chain map and to check
some of the related associativity - as at points c and f in
\S\ref{subsubsec:defin_alg_str}.  This happens in precisely two cases.
The first - concerning the fact that the module operation is a chain
map - has to do with the identification of an element
$\bar{u}\in\partial_{1} (\mathcal{P}_{\mathcal{T}})$ with an element
of the product $\mathcal{P}_{\mathcal{T}_{1}}\times
\mathcal{P}_{\mathcal{T}_{2}}$.  The problem here is that, by the
definition of the relevant $1$-dimensional spaces at the end of c in
\S\ref{subsubsec:defin_alg_str} we see that such a $\bar{u}$ can be
viewed as as product of two configurations modeled on two trees
$\mathcal{T}_{1}$ and $\mathcal{T}_{2}$ but one of these
configurations contains a vertex of valence three which corresponds to
a perturbed curve. At the same time both
$\mathcal{P}_{\mathcal{T}_{1}}$ and $\mathcal{P}_{\mathcal{T}_{2}}$
are moduli spaces of virtual dimension $0$ and so, following the
definition of the module action and the pearl differential, they do
not contain perturbed curves.

The second case concerns the verification of the associativity type
relations involving the module action and it arises if the initial
curve leading by bubbling to an element $\bar{u}\in
\partial\bar{\mathcal{P}}_{\mathcal{T}}$ is in fact a perturbed curve
(verifying (\ref{Eq:delbar-pert})), and carrying $4$ marked points,
two of which are interior points and each of the two components of the
resulting cusp curve carries one interior point. The problem in this
case is that just one of the resulting cusp curves verifies the
perturbed equation and the other one is a usual $J$-holomorphic curve
(the definition of the marked point selector in this case implies that
the lower component in the tree is the perturbed one) and this
configuration $\bar{u}$ does not actually appear as an object of type
iii.  The reason is that in the relevant moduli spaces all the
vertices of valence three correspond to perturbed curves and thus, the
configurations of type iii in this case contain a cusp curve with both
components being perturbed.

The solution to these two issues turns out to be simple: a further
analysis of the moduli spaces involved in both cases shows that if the
relevant perturbations are small enough - which can be obviously
assumed - then the two types of configurations which are compared in
each case are in bijection. This is proved by a cobordism argument
which is possible because both the perturbed and the un-perturbed
configurations are regular - see again \cite{Bi-Co:qrel-long} for more
details.

\section{Additional tools}\label{sec:add_tools}

In this section we introduce a number of additional tools which will
be useful for the proof of the main theorems and in related
computations.

\subsection{Minimal pearls}\label{subsec:minimal} 
As before, we assume here that $L\subset (M,\omega)$ is monotone.
Suppose that for some almost complex structure $J$ and Morse function
$f:L\to \R$ the pearl complex $\mathcal{C}^{+}(L;f,\rho_{L}, J)$ is
defined. It is clear that if $f$ is a perfect Morse function, in the
sense that the differential of its Morse complex is trivial, then the
pearl complex is most efficient for computations. Clearly, not all
manifolds admit perfect Morse functions. However, we will see that,
algebraically, we can always reduce the pearl complex to such a
minimal form (a similar construction in the cluster set-up has been
sketched in \cite{Cor-La:Cluster-1}).

\

It is crucial to work here over a ``positive'' coefficient ring. We
will use in this section $\La^{+}=\Z_{2}[t]$.  In the algebraic
considerations below the fact hat $\deg(t)\leq -2$ plays an important
role.

Let $G$ be a finite dimensional graded $\Z_2$-vector space and let
$\mathcal{D}=(G\otimes \La^{+}, d)$ be a chain complex with a
differential $d$ which is $\La^{+}$-linear - in other words
$\mathcal{D}$ is a $\La^{+}$-chain complex.  For an element $x\in G$
let $d(x)=d_{0}(x)+d_{1}(x)t$ with $d_{0}(x)\in G$. In other words
$d_{0}$ is obtained from $d(x)$ by treating $t$ as a polynomial
variable and putting $t=0$. Clearly $d_{0}:G\to G$, $d_{0}^{2}=0$.
Similarly, for a chain morphism $\xi$ we denote by $\xi_{0}$ the
$d_{0}$-chain morphism obtained by making $t=0$.  Let $\mathcal{H}$ be
the homology of the complex $(G,d_{0})$. We refer to this homology as
$d_{0}$-homology in contrast to $d$-homology which is denoted by
$H_{\ast}(\mathcal{D})$.

\begin{prop} \label{prop:min_model} With the notation above there
   exists a chain complex
   $$\mathcal{D}_{min}=
   (\mathcal{H}\otimes \La^{+},\delta), \ {\rm with}\ \delta_{0}=0$$
   and chain maps $\phi:\mathcal{D}\to \mathcal{D}_{min}$,
   $\psi:\mathcal{D}_{min}\to \mathcal{D}$ so that:
   $\phi\circ\psi=id$, $\phi_{0}$ and $\psi_{0}$ induce isomorphisms
   in $d_{0}$-homology and $\phi$ and $\psi$ induce isomorphisms in
   $d$-homology.  Moreover, the properties above characterize
   $\mathcal{D}_{min}$ up to (a generally non-canonical)
   isomorphism.
\end{prop} 
Concerning the the uniqueness part of the statement  see also~\S\ref{sbsb:further-min}.

\

Here is an important consequence of Proposition~\ref{prop:min_model}:
\begin{cor}\label{cor:min_pearls}
   There exists a complex
   $\mathcal{C}^{+}_{min}(L)=(H_{\ast}(L;\Z_{2})\otimes
   \La^{+},\delta)$, with $\delta_{0}=0$ and so that, for any
   $(L,f,\rho,J)$ such that $\mathcal{C}^{+}(L;f,\rho,J)$ is defined,
   there are chain morphisms $\phi:\mathcal{C}^{+}(L;f,\rho,J)\to
   \mathcal{C}^{+}_{min}(L)$ and $\psi : \mathcal{C}^{+}_{min}(L)\to
   \mathcal{C}^{+}(L;f,\rho,J)$ which both induce isomorphisms in
   quantum homology as well as in Morse homology and verify
   $\phi\circ\psi=id$. The complex $\mathcal{C}^{+}_{min}(L)$ with
   these properties is unique up to (a generally non-canonical)
   isomorphism.
\end{cor}

We call the complex provided by this corollary the {\em minimal pearl
  complex} and the maps $\phi$, $\psi$ the {\em structure maps
  associated} to $\mathcal{C}^{+}(L;f,\rho,J)$ (or shorter, to $f$).
This terminology originates in rational homotopy where a somewhat
similar notion is central. There is a slight abuse in this notation
as, while any two complexes as provided by the corollary are
isomorphic this isomorphism is not canonical. Obviously, in case a
perfect Morse function exists on $L$ any pearl complex associated to
such a function is already minimal.  As mentioned before, in the
arguments below it is essential that the differential and morphisms
are defined over $\Lambda^{+}$ (but the same constructions also work
over $\widetilde{\Lambda}^{+}$; see \S\ref{subsec:coeff} for the
various Novikov rings available).  In case we need to work over
$\Lambda$ we define
$\mathcal{C}_{min}(L)=\mathcal{C}^{+}_{min}(L)\otimes_{\Lambda^{+}}\Lambda$.

\begin{rem}\label{rem:product_min}
   a. An important consequence of the existence of the chain morphisms
   $\phi$ and $\psi$ is that all the algebraic structures described
   before (product, module structure etc.) can be transported and
   computed on the minimal complex. For example, the product is the
   composition:
   \begin{equation}\label{eq:prod_min}
      \mathcal{C}^{+}_{min}(L)\otimes
      \mathcal{C}^{+}_{min}(L)\stackrel{\psi_{1}\otimes\psi_{2}}
      {\longrightarrow}\mathcal{C}^{+}(L;f_{1},\rho,J)\otimes
      \mathcal{C}^{+}(L;f_{2},\rho,J)\stackrel{\ast}{\to}
      \mathcal{C}^{+}(L;f_{3},\rho,J)\stackrel{\phi_{3}}{\to}
      \mathcal{C}^{+}_{min}(L)
   \end{equation}
   where $\psi_{i}, \phi_{i}$ are the structural maps given by
   Corollary \ref{cor:min_pearls} and which correspond to the
   complexes associated to $f_{i}$. There is a cycle in
   $\mathcal{C}_{min}^{+}(L)$ equal to $\phi(P)$ where $P$ is the
   maximum of any Morse function $f$ so that
   $\mathcal{C}^{+}(L;f,\rho,J)$ is defined and so that $f$ has a
   single maximum; $\psi,\phi$ are the associated structural maps. By
   degree reasons in dimension $n=\dim(L)$ $\phi=\phi_{0}$ and
   $\psi=\psi_{0}$ and so this cycle is independent on the choice of
   $f$ and of that of the associated structural maps and it coincides
   with $[L] \otimes 1$ where $[L]$ is the fundamental class of $L$.
   By a slight abuse of notation we will continue to denote by $[L]$
   both the cycle $\phi(P)$ as well as its quantum homology class.  In
   homology, the product defined by (\ref{eq:prod_min}) has as unity
   the fundamental class $[L]$.  Moreover, with the simplified
   description of the quantum product given in (\ref{eq:simpl_q_prod})
   - where $f_{2}=f_{3}$ we obtain a product so that $[L]$ is the unity
   at the chain level.  It also follows from the fact that
   $\phi_{0},\psi_{0}$ induce isomorphisms in Morse homology that the
   ``minimal'' product described above is a deformation of the
   intersection product.

   b. A consequence of point a is that $HF(L)\cong QH(L)=0$ iff there
   is some $x \in \mathcal{C}^+_{min}(L) = H(L; \Z_{2})\otimes \La^+$
   so that $\delta x=[L]t^{k}$. Indeed, suppose that $QH(L)=0$. Then,
   as for degree reasons $[L]$ is a cycle in $\mathcal{C}_{min}(L)$,
   we obtain that it has to be also a boundary. This means that there
   exists $a \in \mathcal{C}_{min}(L)$ so that $\delta a = [L]$.
   Multiplying $a$ by a large enough positive power $k$ of $t$ gives
   an element $x = a t^k$ which now lies in $\mathcal{C}^{+}_{min}(L)$
   and such that $\delta x=[L]t^{k}$. Conversely, if $\delta x = [L]
   t^k$ then $[L]$ is a boundary in $\mathcal{C}_{min}(L)$. On the
   other hand the cycle $[L] \in \mathcal{C}_{min}(L)$ represents the
   unity for the product on $H_*(\mathcal{C}_{min}(L)) \cong QH_*(L)$
   just mentioned at point a above. Thus the unity is $0$, hence
   $QH(L)=0$.

   c. It is also useful to note that there is an isomorphism
   $QH(L)\cong H(L;\Z_{2})\otimes \La$ iff the differential
   $\delta$ in $\mathcal{C}_{min}(L)$ is identically zero.
\end{rem}
We now proceed to the proof of the Proposition and of its Corollary.

\

{\em Proof of Proposition \ref{prop:min_model}.}  We start with a
useful algebraic property.
\begin{lem}\label{lem:isomorphisms_min} Let $\mathcal{D'}=(G'\otimes
   \La^{+},d')$ and $\mathcal{D}''=(G''\otimes\La^{+},d'')$ be two
   $\La^{+}$-chain complexes.  If a chain morphism $\xi:
   \mathcal{D}'\to\mathcal{D}''$ which is $\La^{+}$-linear is so that
   $\xi_{0}$ induces an isomorphism in $d_{0}$-homology, then $\xi$
   induces an isomorphism in $d$-homology.
\end{lem}
\begin{proof} 
   Recall that the filtration $\mathcal{F}^{k}\La^{+}=t^{k}\La^{+}$
   induces a filtration, called the degree filtration, on any free
   $\La^{+}$- module. The resulting spectral sequence induced on any
   $\La^{+}$-chain complex is called the degree spectral sequence.
   Clearly, $\xi$ respects the degree filtration and thus it induces a
   morphism relating the degree spectral sequences of $\mathcal{D}'$
   and $\mathcal{D}''$.  We notice that $E^{1}(\xi)$ is identified
   with the morphism induced by $\xi_{0}$ in $d_{0}$-homology.
   Therefore, this is an isomorphism.  As we work over $\Z_{2}$ this
   implies that $H_{\ast}(\xi)$ is an isomorphism.  \end{proof}

\begin{rem}\label{rem:unique} 
   a. Under the assumptions in Lemma \ref{lem:isomorphisms_min}, the
   same spectral sequence argument also shows that the chain morphism
   $$\xi\otimes id_{\La}:\mathcal{D}'\otimes_{\La^{+}}\La \to 
   \mathcal{D}''\otimes_{\La^{+}}\La$$ induces an isomorphism in
   homology.

   b. Let $G'$, $G''$ be finite dimensional, graded $\Z_2$-vector
   spaces.  We claim that a $\La^{+}$-linear morphism $$\xi:G'\otimes
   \La^{+}\to G''\otimes \La^{+}$$ is an isomorphism if and only if
   $\xi_{0}$ is an isomorphism. Indeed, any such $\xi$ can be viewed
   as a morphism of chain complexes by assuming that the differentials
   in the domain and target are trivial. We deduce from Lemma
   \ref{lem:isomorphisms_min} that, if $\xi_{0}$ is an isomorphism,
   then $\xi$ is an isomorphism.  Conversely, if $\xi$ is an
   isomorphism, then $t\xi:t(G'\otimes \La^{+})\to t (G''\otimes
   \La^{+})$ is an isomorphism. As $\xi_{0}$ is identified with the
   quotient morphism
   $$\frac{G'\otimes \La^{+}}{t(G'\otimes \La^{+})}\to
   \frac{G''\otimes \La^{+}}{t (G''\otimes \La^{+})}$$ induced by
   $\xi$, it follows that $\xi_{0}$ is an isomorphism.
\end{rem}

We now return to the proof of Proposition~\ref{prop:min_model}.  Start
by choosing a basis for the complex $(G,d_{0})$ as follows: $G=\Z_2
\langle x_{i}: i\in I \rangle\oplus\ \Z_2 \langle y_{j}:j\in
J\rangle\oplus\ \Z_2\langle y'_{j}: j\in J\rangle$ so that
$d_{0}x_{i}=0$, $d_{0}(y_{j})=0$, $d_{0}y'_{j}=y_{j}$, $\forall j\in
J$.  For further use, we denote $B_{X}=\{x_{i}: i\in I\}$,
$B_{Y}=\{y_{j} : j\in J\}$, $B_{Y'}=\{y'_{j} : j\in J\}$.

Clearly, $\mathcal{H}\cong \Z_2 \langle x_{i}\rangle $ and we will
identify further these two vector spaces and denote
$\mathcal{D}_{min}=\Z_2 \langle\tilde{x}_{i}\rangle\otimes \La^{+}$
where $\tilde{x}_{i}, i\in I$ are of the same degree as the $x_{i}$'s
(the differential on $\mathcal{D}_{min}$ remains to be defined). We
will construct $\phi$ and $\psi$ and $\delta$ so that
$\phi_{0}(x_{i})=\tilde{x}_{i}$, $\phi_{0}(y_{j})=\phi_{0}(y'_{j})=0$
and $\psi_{0}(\tilde{x}_{i})=x_{i}$. If $\psi_{0}$ and $\phi_{0}$
verify these properties, then, they induce an isomorphism in
$d_{0}$-homology and, by Remark \ref{lem:isomorphisms_min}, $\psi$ and
$\phi$ induce isomorphisms in $d$-homology.

The construction is by induction. We fix the following notation:
$\mathcal{D}^{k}= \Z_2\langle x_{i},\ y'_{j},\ y_{j}\ : \ \
|x_{i}|\geq k, |y'_{j}|\geq k\ \rangle\otimes \La^{+}$.  Similarly, we
put $\mathcal{D}^{k}_{min}=\Z_2\langle\tilde{x}_{i} \ : \ |x_{i}|\geq
k \rangle\otimes \La^{+}$.  Notice that there are some generators in
$\mathcal{D}^{k}$ which are of degree $k-1$, namely the $y_{j}$'s of
that degree. With this notation we also see that $\mathcal{D}^{k}$ is
a sub-chain complex of $\mathcal{D}$ (because $dy'_{i}=d_{1}(y'_{i})t$
and so $|d_{1}(y'_{i})|\geq |y'_{i}|+1$, the same type of relation
holds for $x_{i}$ and for $y_{i}$ we have
$dy_{i}=y'_{i}+d_{1}(y_{i})t$) .  Assume that $n$ is the maximal
degree of the generators in $G$.  For the generators of
$\mathcal{D}^{n}$ we let $\phi$ be equal to $\phi_{0}$, we put
$\delta=0$ on $\mathcal{D}_{min}^{n}$ and we also let $\psi=\psi_{0}$
on $\mathcal{D}_{min}^{n}$. To see that $\phi:\mathcal{D}^{n}\to
\mathcal{D}_{min}^{n}$ is a chain morphism with these definitions it
suffices to remark that if $y\in B_{Y}$, $|y|=n-1$ , then
$y=d_{0}y'=dy'$ and so $dy=0$.

We now assume $\phi,\delta,\psi$ defined on $\mathcal{D}^{n-s+1}$,
$\mathcal{D}_{min}^{n-s+1}$ so that $\phi,\psi$ are chain morphisms,
they induce isomorphisms in homology and $\phi\circ \psi=id$.  We
intend to extend these maps to $\mathcal{D}^{n-s}$,
$\mathcal{D}_{min}^{n-s}$.  We first define $\phi$ on the generators
$x\in B_{X}$, $y'\in B_{Y'}$ which are of degree $n-s$:
$\phi(x)=\tilde{x}$, $\phi(y')=0$. We let
$\delta(\tilde{x})=\phi^{n-s+1}(dx)$ ( when needed, we use the
superscript $(-)^{n-s+1}$ to indicate the maps previously constructed
by induction). Here it is important to note that, as $d_{0}x=0$, we
have that $dx\in\mathcal{D}^{n-s+1}$. We consider now the generators
$y\in B_{Y}\cap\mathcal{D}^{n-s}$ which are of degree $n-s-1$ and we
put $\phi(y)=\phi^{n-s+1}(y-dy')$. This makes sense because $y-dy'\in
\mathcal{D}^{n-s+1}$.  We write $dy'=y+y''$ and we first see $\phi
(dy')=0=\delta(\phi(y'))$ so that, to make sure that $\phi^{n-s}$ is a
chain morphism with these definitions, it remains to check that
$\delta\phi(y)=\phi(d y)$ for all $y\in B_{Y}$ of degree $n-s-1$.  But
$\delta \phi(y)=-\delta \phi^{n-s+1}(y'')$ and as $\phi^{n-s+1}$ is a
chain morphism, we have $\delta \phi^{n-s+1}(y'')=\phi^{n-s+1} d(y'')$
which implies our identity because $d y''+dy=d^{2}y'=0$.  Clearly,
$\phi^{n-s}_{0}$ induces an isomorphism in $d_{0}$-homology and hence
in $d$-homology too.

To conclude our induction step it remains to construct the map $\psi$
on the generators $\tilde{x}$ of degree $n-s$.  We now consider the
difference $dx-\psi^{n-s+1}(\delta \tilde{x})$ and we want to show
that there exists $\tau\in \mathcal{D}^{n-s+1}$ so that
$d\tau=dx-\psi^{n-s+1}(\delta \tilde{x})$ and $\tau\in \ker
(\phi^{n-s+1})$. Assuming the existence of this $\tau$ we will put
$\psi(\tilde{x})=x-\tau$ and we see that $\psi$ is a chain map and
$\phi\circ\psi=id$.  To see that such a $\tau$ exists remark that
$w=dx-\psi^{n-s+1}(\delta \tilde{x})\in \mathcal{D}^{n-s+1}$ and
$dw=-d(\psi^{n-s+1}(\delta \tilde{x}))=-\psi^{n-s+1}(\delta\circ
\delta \tilde{x})=0$ (because $\psi^{n-s+1}$ is a chain map).
Moreover, $\phi(w)=\phi^{n-s+1}(dx)-\delta\tilde{x}=0$ because
$\phi^{n-s+1}\circ\psi^{n-s+1}=id$.  Therefore $w$ is a cycle
belonging to $\ker (\phi^{n-s+1})$. But $\phi^{n-s+1}$ is a chain
morphism which induces an isomorphism in homology and which is
surjective. Therefore $H_{\ast}(\ker(\phi^{n-s+1}))=0$. Thus there
exists $\tau\in\ker(\phi^{n-s+1})$ so that $d\tau=w$ and this
concludes the induction step.

This construction concludes the first part of the statement and to
finish the proof of the proposition we only need to prove the
uniqueness result. For this, suppose $\phi': \mathcal{D}\to
\mathcal{D}'$ and $\psi':\mathcal{D}'\to \mathcal{D}$ are chain
morphisms so that $\phi'\circ\psi'=id$ with $\mathcal{D}'=(H\otimes
\La^{+},\delta')$, $\delta'_{0}=0$ and $H$ some graded,
$\Z_{2}$-vector space and $\phi'$, $\psi'$, $\phi'_{0}$, $\psi'_{0}$
induce isomorphisms in the respective homologies.  We want to show
that there exists a chain map $c:\mathcal{D}_{min}\to \mathcal{D}'$ so
that $c$ is an isomorphism.  To this end we define
$c(u)=\phi'\circ\psi(u)$, $\forall u\in \mathcal{D}_{min}$.  Now
$H_{\ast}(\phi_{0})$ and $H_{\ast}(\phi'_{0})$, $H_{\ast}(\psi_{0})$,
$H_{\ast}(\psi'_{0})$ are all isomorphisms (in $d_{0}$-homology). So
$H(c_{0})$ is an isomorphism but as $\delta_{0}=0=\delta'_{0}$ we
deduce that $c_{0}$ is an isomorphism.  By Remark \ref{rem:unique} b,
the map $c$ is an isomorphism.  \qed

\

{\em Proof of Corollary \ref{cor:min_pearls}.}  Fix a triple
$f^{0},\rho^{0},J^{0}$ and assume that
$\mathcal{C}^{+}(L;f^{0},\rho^{0}, J^{0})$ is defined. Apply
Proposition \ref{prop:min_model} to this complex.  Denote by
$(\mathcal{C}^{+}_{min},\phi,\psi)$ the resulting minimal complex and
the chain morphisms as in the statement of \ref{prop:min_model}. The
only part of the statement which remains to be proved is that given a
different set of data $(f',\rho',J')$ so that
$\mathcal{C}^{+}(L;f',\rho',J')$ is defined, there are appropriate
morphisms $\phi',\psi'$ as in the statement.  There are comparison
morphisms: $h:\mathcal{C}^{+}(L;f'\rho',J')\to
\mathcal{C}^{+}(L;f^{0}\rho^{0},J^{0})$ as well as
$h':\mathcal{C}^{+}(L;f^{0}\rho^{0},J^{0})\to
\mathcal{C}^{+}(L;f'\rho',J')$ so that, by construction, both $h$ and
$h'$ are inverse in homology and both induce isomorphisms in Morse
homology (and these two isomorphisms are also inverse). Define
$\phi':\mathcal{C}^{+}(L;f',\rho',J')\to \mathcal{C}^{+}_{min}$,
$\psi'':\mathcal{C}^{+}_{min}\to \mathcal{C}^{+}(L;f',\rho',J')$ by
$\phi'=\phi\circ h$ and $\psi''=h'\circ \psi$. It is clear that
$\phi'$, $\psi''$, $\phi'_{0}$ and $\psi''_{0}$ induce isomorphisms in
homology. Moreover, as $h_{0}$ and $h'_{0}$ are inverse in homology
and $\delta_{0}=0$ in $\mathcal{C}^{+}_{min}$ it follows that
$\phi'_{0}\circ \psi''_{0}=id$.  This means by Lemma
\ref{lem:isomorphisms_min} that $v=\phi'\circ\psi''$ is a chain
isomorphism so that $v_{0}$ is the identity. We now put
$\psi'=\psi''\circ v^{-1}$ and this verifies all the needed
properties.  The uniqueness of $\mathcal{C}^{+}_{min}(L)$ now follows
from the uniqueness part in Proposition \ref{prop:min_model}.  \qed

\subsubsection{Further remarks on minimal models}
\label{sbsb:further-min} While the minimal complex $\mathcal{D}_{min}$
associated to $\Lambda^+$-complex $\mathcal{D}$ is unique (upto
isomorphism),  this is not the case  for the structural
maps $\phi$ and $\psi$. For these maps we expect uniqueness
in a weaker sense such as uniqueness upto chain homotopy, however we
will not pursue this direction here. On the other hand, we will use
minimal models in~\S\ref{sec:proof_main} quite frequently. In fact,
in~\S\ref{sec:proof_main} we will have to use the specific choice of the
morphisms $\phi$, $\psi$ (as well as $\phi_0$, $\psi_0$) that is
constructed in the proof of Proposition~\ref{prop:min_model}. It seems
plausible that this can be avoided by axiomatizing more the theory of
minimal models, but we will not do this here since we view the minimal
model as a purely computational tool.

\subsection{Geometric criterion for the vanishing of $QH(L)$}
\label{Sb:criteria-QH}

Let $L\subset (M,\omega)$ be a monotone Lagrangian submanifold.
Remark \ref{rem:product_min} b provides a criterion for the vanishing
of $QH(L)$.  We provide here a more geometric such criterion which is
useful when $N_{L}=2$ which we will assume in this section.

Denote by $\partial:H_2(M,L;\mathbb{Z}) \to H_1(L;\mathbb{Z})$ the
boundary homomorphism and by
$\partial_{\mathbb{Z}_2}:H_2(M,L;\mathbb{Z}) \to H_1(L;\mathbb{Z}_2)$
the composition of $\partial$ with the reduction mod $2$,
$H_1(L;\mathbb{Z}) \to H_1(L;\mathbb{Z}_2)$. Given $A \in H_2^D(M,L)$
and $J \in \mathcal{J}(M,\omega)$ consider the evaluation map
$$ev_{A,J}: (\mathcal{M}(A,J) \times \partial D)/G \longrightarrow L,
\quad ev_{A,J}(u,p) = u(p),$$ where $G = \textnormal{Aut}(D) \cong
PSL(2,\mathbb{R})$ is the group of biholomorphisms of the disk.

For every $J \in \mathcal{J}(M,\omega)$ let $\mathcal{E}_2(J)$ be the
set of all classes $A \in H_2^D(M,L)$ with $\mu(A)=2$ for which there
exist $J$-holomorphic disks with boundary on $L$ in the class $A$:
$$\mathcal{E}_2(J) = \{ A \in H_2^D(M,L) \mid \mu(A)=2, \quad
\mathcal{M}(A,J) \neq \emptyset \}.$$ Define:
$$\mathcal{E}_2 = \bigcap_{J \in \mathcal{J}(M,\omega)}
\mathcal{E}_2(J).$$ Standard arguments show that:
\begin{enumerate}
  \item $\mathcal{E}_2(J)$ is a finite set for every $J$.
  \item There exists a second category subset
   $\mathcal{J}_{\textnormal{reg}} \subset \mathcal{J}(M,\omega)$ such
   that for every $J \in \mathcal{J}_{\textnormal{reg}}$,
   $\mathcal{E}_2(J) = \mathcal{E}_2$. In other words, for generic
   $J$, $\mathcal{E}_2(J)$ is independent of $J$.
  \item For every $J \in \mathcal{J}$ and every $A \in
   \mathcal{E}_2(J)$ the space $\mathcal{M}(A,J)$ is compact and all
   disks $u \in \mathcal{M}(A,J)$ are simple.
  \item For $J \in \mathcal{J}_{\textnormal{reg}}$ and $A \in
   \mathcal{E}_2$, the space $(\mathcal{M}(A,J) \times \partial D)/G$
   is a compact smooth manifold without boundary. Its dimension is $n
   = \dim L$. In particular, for generic $x \in L$, the number of
   $J$-holomorphic disks $u \in \mathcal{M}(A,J)$ with $u(\partial D)
   \ni x$ is finite.
  \item For every $A \in \mathcal{E}_2$ and $J_0, J_1 \in
   \mathcal{J}_{\textnormal{reg}}$ the manifolds $(\mathcal{M}(A,J_0)
   \times \partial D)/G$ and $(\mathcal{M}(A,J_1) \times \partial
   D)/G$ are cobordant via a compact cobordism. Moreover, the
   evaluation maps $ev_{A,J_0}$, $ev_{A,J_1}$ extend to this
   cobordism, hence $\deg_{\mathbb{Z}_2} ev_{A,J_0} =
   \deg_{\mathbb{Z}_2} ev_{A,J_1}$.  In other words
   $\deg_{\mathbb{Z}_2} ev_{A,J}$ depends only on $A \in
   \mathcal{E}_2$.
  \item In fact, the set $\mathcal{J}_{\textnormal{reg}}$ above can be
   taken to be the set of all $J \in \mathcal{J}(M,\omega)$ which are
   regular for all classes $A \in H_2^D(M,L)$ in the sense that the
   linearization of the $\overline{\partial}_J$ operator is surjective
   at every $u \in \mathcal{M}(A,J)$.
\end{enumerate}

Let $J \in \mathcal{J}_{\textnormal{reg}}$ and let $x \in L$ be a
generic point. Define a one dimensional $\mathbb{Z}_2$-cycle
$\delta_x(J)$ to be the sum of the boundaries of all $J$-holomorphic
disks with $\mu=2$ whose boundaries pass through $x$. Of course, if a
disk meets $x$ along its boundary several times we take its boundary
in the sum with appropriate multiplicity. Thus the precise definition
is:
\begin{equation} \label{Eq:delta-x} \delta_x(J) = \sum_{A \in
     \mathcal{E}_2} \; \sum_{(u,p) \in ev_{A,J}^{-1}(x)} u(\partial
   D).
\end{equation}
By the preceding discussion the homology class $D_1 = [\delta_x(J)]
\in H_1(L;\mathbb{Z}_2)$ is independent of $J$ and $x$. In fact
\begin{equation} \label{Eq:D1} D_1 = \sum_{A \in \mathcal{E}_2}
   (\deg_{\mathbb{Z}_2} ev_{A,J})
   \partial_{\mathbb{Z}_2}A.
\end{equation}

\begin{prop} \label{P:criterion-QH=0-1} Let $L \subset (M,\omega)$ be
   a monotone Lagrangian submanifold with $N_L = 2$.  If $D_1 \neq 0$
   then $QH_*(L)=0$. \label{I:criterion-QH-1}
\end{prop}

\begin{proof}
   Choose a generic $J \in \mathcal{J}(M,\omega)$. Let $f:L \to
   \mathbb{R}$ be a generic Morse function with precisely one local
   maximum at a point $x \in L$ and fix a generic Riemannian metric on
   $L$.  Denote by $(CM_*(f), \partial_0)$, $(\mathcal{C}_*(f,J),d)$
   the Morse and pearl complexes associated to $f$, $J$ and the chosen
   Riemannian metric. As discussed in \S\ref{subsubsec:defin_alg_str}
   b, $x$ is a cycle in the pearl complex of $f$ and its quantum
   homology class is the unity.

   For degree reasons the restriction of $d$ to $CM_{n-1}(f) \subset
   \mathcal{C}_{n-1}(f,J)$ is given by $d=\partial_0 + \partial_1 t$,
   where $\partial_1:CM_{n-1}(f) \to CM_n(f)=\mathbb{Z}_2 x$ counts
   pearly trajectories with holomorphic disks of Maslov index $2$.
   Since $x$ is a maximum of $f$, no $-\nabla f$ trajectories can
   enter $x$ (i.e. $W_x^s(f) = \{x\}$). Therefore for every $y \in
   \textnormal{Crit}_{n-1}(f)$ we have
   \begin{equation} \label{Eq:del-1}
      \partial_1 y = \#_{\mathbb{Z}_2} \bigl( W_y^u(f) \cap
      \delta_x(J)\bigr)x.
   \end{equation}

   Assume now that $D_1 \neq 0$. By Poincar\'{e} duality there exists
   an $(n-1)$-dimensional cycle $C$ in $L$ such that
   $$\#_{\mathbb{Z}_2} C \cap \delta_x(J) \neq 0.$$
   Let $z \in CM_{n-1}(f)$ be a $\partial_0$-cycle representing $[C]
   \in H_{n-1}(L;\mathbb{Z}_2)$. Then $$d(z) = \partial_1(z)t =
   \#_{\mathbb{Z}_2} \bigl( W_z^u \cap \delta_x(J)\bigr)x t =
   \#_{\mathbb{Z}_2} \bigl(C \cap \delta_x(J)\bigr)x t = axt$$ for
   some non-zero scalar $a$. (Of course, $a \neq 0$ is the same as
   $a=1$ here, since we work over $\mathbb{Z}_2$. However we wrote
   $ax$ to emphasize that the argument works over every field.)  It
   follows that $[x] = 0 \in QH_n(L)$. But, as $[x]$ is the unity of
   $QH_*(L)$, we deduce $QH_*(L)=0$.
\end{proof}

\subsection{Action of the symplectomorphism group} \label{Sb:act-symp}

We now describe a property of our machinery which is very useful in
computations when symmetry is present. In this section $\mathcal{R}$
is any of the rings described in \S\ref{subsec:coeff}.

\begin{prop}\label{cor:symm} Let $\phi:L\to L$ be a diffeomorphism which
   is the restriction to $L$ of an ambient symplectic diffeomorphism
   $\bar{\phi}$ of $M$. Let $f,\rho,J$ be so that the pearl complex
   $\mathcal{C}(L;\mathcal{R};f,\rho,J)$ is defined. There exists a
   chain map:
   $$\tilde{\phi}:\mathcal{C}(L;\mathcal{R};f,\rho,J)\to
   \mathcal{C}(L;\mathcal{R};f,\rho,J)$$ which respects the degree
   filtration, induces an isomorphism in homology, and so that the
   morphism $E^{1}(\tilde{\phi})$ induced by $\tilde{\phi}$ at the
   $E^{1}$ level of the degree spectral sequence coincides with
   $H_{\ast}(\phi)\otimes {id}_{\Lambda^{+}}$ (where $H_{\ast}(\phi)$
   is the isomorphism induced by $\phi$ on singular homology).  The
   map $\bar{\phi}\to \tilde{\phi}$ induces a representation:
   $$\bar{h}:Symp (M,L)\to Aut (QH_{\ast}(L;\mathcal{R}))$$
   where $Aut(QH_{\ast}(L;\mathcal{R}))$ are the ring automorphisms of
   $QH_{\ast}(L;\mathcal{R})$ preserving the augmentation and
   $Symp(M,L)$ are the symplectomorphisms of $M$ which keep $L$
   invariant. The restriction of $\bar{h}$ to $Symp_{0}(M)\cap
   Symp(M,L)$ takes values in the automorphisms of $QH(L;\mathcal{R})$
   as an algebra over $QH(M;\mathcal{R})$ (here $Symp_{0}(M)$ is the
   component of the identity in $Symp(M)$).
\end{prop}

\begin{proof} 
   To ease notation, we omit the ring $\mathcal{R}$ in the writing of
   the pearl complexes below.

   Assume that $\phi:L\to L$ is a diffeomorphism which is the
   restriction to $L$ of the symplectomorphism $\bar{\phi}$ and
   $f,\rho,J$ are such that the chain complex
   $\mathcal{C}(L;f,\rho,J)$ is defined. Let
   $f^{\phi}=f\circ\phi^{-1}$.  There exists a basis preserving
   isomorphism
   $$h^{\phi}:\mathcal{C}(L;f,\rho,J)\to
   \mathcal{C}(L;f^{\phi},\rho^{\ast},J^{\ast})$$ induced by $x \to
   \phi(x)$ for all $x\in \Crit(f)$ where $\rho^{\ast},J^{\ast}$ are
   obtained by the push-forward of $\rho,J$ by means of $\phi$ and the
   symplectomorphism $\bar{\phi}$. The isomorphism $h^{\phi}$ acts in
   fact as an identification of the two complexes.

   Next, there is also the standard comparison chain morphism,
   canonical up to chain homotopy
   $$c: \mathcal{C}(L;f^{\phi},\rho^{\ast}, J^{\ast})\to
   \mathcal{C}(L;f,\rho,J)~.~$$

   We now consider the composition $\tilde{\phi}=c\circ h^{\phi}$.  It
   is clear that this map induces an isomorphism in homology and that
   it preserves the the ring structure and the augmentation (as each
   of its factors does so). We now inspect the Morse theoretic
   analogue of these morphisms - in the sense that we consider instead
   of the complexes $\mathcal{C}(L, f,- )$ the respective Morse
   complexes $C(f,-)$. It is easy to see that the Morse theoretic
   version of $\tilde{\phi}$ induces in Morse homology precisely
   $H_{\ast}(\phi)$. But this means that at the $E^{1}$ stage of the
   degree filtration the morphism induced by $\tilde{\phi}$ has the
   form $H_{\ast}(\phi)\otimes id_{\mathcal{R}}$.

   We now denote $k=\bar{h}(\bar{\phi})$ and we need to verify that
   for any two elements $\bar{\phi}, \bar{\psi}\in Symp(M)$ we have
   $\bar{h}(\bar{\phi}\circ\bar{\psi})=\bar{h}(\phi)\circ\bar{h}(\psi)$.
   It is easy to see that this is implied by the commutativity of the
   following diagram:
   $$\xy\xymatrix@+10pt{ \mathcal{C}(L;f')
     \ar[r]^{h^{\phi}}\ar[d]_{c} &
     \mathcal{C}(L;(f')^{\phi})\ar[d]^{c'}\\
     \mathcal{C}(L;f)\ar[r]_{h^{\phi}}& \mathcal{C}(L,f^{\phi}) }
   \endxy $$ for any two Morse function $f$ and $f'$ so that the
   respective complexes are defined. To verify this commutativity,
   first we use some homotopy $H$, joining $f$ to $f'$, to provide the
   comparison morphism $c$ and we then use the homotopy $H\circ
   \phi^{-1}$ to define $c'$.

   Finally, recall that the module structure of $QH(L)$ over $QH(M)$
   is defined by using an additional Morse function $F:M\to \R$. If we
   put $F^{\bar{\phi}}=F\circ \bar{\phi}^{-1}$ we see easily that the
   external operations defined by using $f,F,\rho,J$ and
   $f^{\phi},F^{\bar{\phi}},\rho^{\ast},J^{\ast}$ are identified one
   to the other via the application $h^{\phi}$ (extended in the
   obvious way to the critical points of $F$).  There is a usual
   comparison map $\bar{c}$ relating the Morse complex of
   $F^{\bar{\phi}}$ to that of $F$. Together with $c$ the map
   $\bar{c}$ identifies - in homology - the external product
   associated to $f^{\phi},F^{\bar{\phi}},\rho^{\ast},J^{\ast}$ and
   the external product associated to $f,F,\rho,J$.  At the level of
   the quantum homology of $M$ the composition $\bar{c}\circ h^{\phi}$
   induces $H_{\ast}(\bar{\phi})\otimes id_{\mathcal{R}}$.  Therefore,
   if $\bar{\phi}\in Symp_{0}(M)$, it follows that this last map is
   the identity and proves the claim.
\end{proof}
\begin{rem}
   It results from the proof above that for $\bar{h}(\phi)$ to be an
   algebra automorphism it is sufficient that $\bar{\phi}$ induce the
   identity at the level of the singular homology of $M$, e.g. when
   $\phi$ is homotopic to the identity.
\end{rem}

\subsection{Duality} \label{Sbs:duality} We start by fixing some
algebraic notation and conventions. Let $\mathcal{R}$ be a commutative
$\widetilde{\Lambda}^{+}$ algebra.  Suppose that
$(\mathcal{C},\partial)$ is a free $\mathcal{R}$-chain complex.  Thus
$\mathcal{C}=G\otimes \mathcal{R}$ with $G$ some graded
$\mathbb{Z}_2$-vector space. We let
$$\mathcal{C}^{\odot}=\hom_{\mathcal{R}}(\mathcal{C},\mathcal{R})$$
graded so that the degree of a morphism $g:\mathcal{C}\to \mathcal{R}$
is $k$ if $g$ takes $\mathcal{C}_{l}$ to $\mathcal{R}_{l+k}$ for all
$l$.

Let $ \mathcal{C}'=\hom_{\Z_2}(G,\Z_2)\otimes \mathcal{R}$ be graded
such that if $x$ is a basis element of $G$, then its dual $x^{\ast}\in
\mathcal{C}'$ has degree $|x^{\ast}|=-|x|$. There is an obvious degree
preserving isomorphism $\psi:\mathcal{C}^{\odot}\to \mathcal{C}'$
defined by $\psi(f)=\sum_{i}f(g_{i})g_{i}^{\ast}$ where $(g_{i})$ is a
basis of $G$ and $(g_{i}^{\ast})$ is the dual basis. We define the
differential of $\mathcal{C}^{\odot}$, $\partial ^{\ast}$, as the
adjoint of $\partial$:
$$
\langle \partial^{\ast} f,x \rangle = \langle f,\partial x \rangle \
,\ \forall x\in \mathcal{C},f\in \mathcal{C}^{\odot}~.~$$ Clearly,
$\mathcal{C}^{\odot}$ continues to be a chain complex (and not a
co-chain complex).

An additional algebraic notion will be useful: the co-chain complex
$\mathcal{C}^{\ast}$ associated to $\mathcal{C}$. To define it, for a
graded $\Z_{2}$-vector space $V$ let $V^{inv}$ be the graded vector
space obtained by reversing the degree of the elements in $V$: if
$v\in V^{inv}$, then its degree is $|v|=-deg_{V}(v)$. Clearly,
$(V\otimes W)^{inv}= V^{inv}\otimes W^{inv}$.

For the complex $\mathcal{C}$ as above we let
$\mathcal{C}^{\ast}=(\mathcal{C}^{\odot})^{inv} =
\hom_{\mathbb{Z}_2}(G,\mathbb{Z}_2)^{inv}\otimes \mathcal{R}^{inv}$.
The complex $\mathcal{C}^{\ast}$ is obviously a co-chain complex and
its differential is a $\mathcal{R}^{inv}$-module map.  The cohomology
of $\mathcal{C}$ is then defined as
$H^{k}(\mathcal{C})=H^{k}(\mathcal{C}^{\ast})$.  Obviously, there is a
canonical isomorphism: $H_{-k}(\mathcal{C}^{\odot}) \cong
H^{k}(\mathcal{C}^{\ast})$.

A particular case of interest here is when
$\mathcal{C}=\mathcal{C}(L;\mathcal{R}; f,\rho,J)$.  In this case we
denote:
$$QH^{k}(L;\mathcal{R})=
H^{k}(\mathcal{C}(L;\mathcal{R};f,\rho,J)^{\ast})~.~$$

Notice that the chain morphisms $\eta: \mathcal{C}\to
\mathcal{C}^{\odot}$ of degree $-n$ are in $1-1$ correspondence with
the chain morphisms of degree $-n$:
$$\tilde\eta:\mathcal{C}\otimes_{\mathcal{R}} \mathcal{C}\to
\mathcal{R}~.~$$ via the formula $\tilde{\eta}(x\otimes
y)=\eta(x)(y)$.  Here the ring $\mathcal{R}$ on the right hand-side is
considered as a chain complex with trivial differential.

For $n\in \Z$ and any chain complex $\mathcal{C}$ as before we let
$s^{n}\mathcal{C}$ be its $n$-fold suspension. This is a chain complex
which coincides with $\mathcal{C}$ but its graded so that the degree
of $x$ in $s^{n}\mathcal{C}$ is $n+$ the degree of $x$ in
$\mathcal{C}$.  A particular useful case where both duality and
suspension appear is in the following sequence of obvious
isomorphisms: $H_{k}(s^{n}\mathcal{C}^{\odot})\cong
H_{k-n}(\mathcal{C}^{\odot})\cong H^{n-k}(\mathcal{C}^{\ast})$.

\begin{prop}\label{p:duality}
   Let $n=\dim(L)$.  There exists a degree preserving morphism of
   chain complexes:
   $$\eta: \mathcal{C}(L;\mathcal{R};f,\rho,J)\to
   s^{n}(\mathcal{C}(L;\mathcal{R};f,\rho,J)^{\odot})$$ which is a
   morphism of $\mathcal{R}$-modules and induces an isomorphism in
   homology. In particular, we have an isomorphism: $\eta:
   QH_{k}(L;\mathcal{R})\to QH^{n-k}(L;\mathcal{R})$. The
   corresponding (degree $-n$) bilinear map
   $$H(\tilde\eta):QH(L;\mathcal{R})\otimes_{\mathcal{R}} 
   QH(L;\mathcal{R})\to \mathcal{R}$$ coincides with the product
   described at point~\ref{I:qmod-qprod} of Theorem~\ref{thm:alg_main}
   composed with the augmentation $\epsilon_{L}$. When
   $\mathcal{R}=\La$ the pairing $H(\tilde\eta)$ is non-degenerate.
   Moreover, for any $k\in \Z$ the induced pairing:
   $$H(\tilde\eta)^{0}:QH_{k}(L)\otimes_{\Z_{2}} QH_{n-k}(L)
   \to \La_{0}=\Z_{2}$$ is non-degenerate.
\end{prop}

{\em Proof of Proposition \ref{p:duality}}.  For any two pearl
complexes $\mathcal{C}(L;\mathcal{R};f,\rho_{L},J)$ and
$\mathcal{C}(L;\mathcal{R};f',\rho'_{L},J')$ the construction at point
\S \ref{subsubsec:defin_alg_str} e. provides a comparison chain
morphism relating them.  There is an alternative way to construct a
comparison map
$$\phi^{f,f'}: \mathcal{C}(L;\mathcal{R};f,\rho_{L},J)\to
\mathcal{C}(L;\mathcal{R};f',\rho_{L},J')$$ in case $f$ and $f'$ are
in general position. In homology, this induces the same morphism as
the one provided by the map $\phi^{F,\tilde{\rho_{L}},\tilde{J}}$
constructed at point e in \S\ref{subsubsec:defin_alg_str}. This
alternative comparison map is useful in the understanding of duality.
The definition of this map is:
$$\phi^{f,f'}(x)=\sum_{\mathcal{T}'}
\# (\mathcal{P}_{\mathcal{T}'})y t^{\mu[\mathcal{T}']/N_{L}}$$ where
the sum is taken over all the trees $\mathcal{T}'$ of symbol $(x:y)$,
$x\in\Crit(f)$, $y\in \Crit(f')$ and $|x|-|y|+\mu([\mathcal{T}'])=0$.
We put in this case $f_{1}=f$ and $f_{2}=f'$.  The exit rule - point v
in \S\ref{subsec:moduli_def} B - needs to be slightly modified for
these trees: in the tree $\mathcal{T}'$ there is one special vertex
$v_{0}$ so that for all vertices above it the exit rule is
$\Theta(f_{1})=f_{1}$, for all the vertices below it the exit rule is
$\Theta(f_{2})=f_{2}$ and at $v_{0}$ the exit rule is
$\Theta(f_{1})=f_{2}$. Condition iv in \S\ref{subsec:moduli_def} B is
also slightly modified in the sense that the vertex $v_{0}$ is allowed
to verify $\omega([v_{0}])=0$.  The marked point selector is as at
point \S\ref{subsubsec:defin_alg_str} a. The duality map
$$\eta : \mathcal{C}(L;\mathcal{R};f,\rho, J)\to
s^{n}(\mathcal{C}(L;\mathcal{R};f,\rho,J)^{\odot})$$ is defined as the
composition of two maps, the first is the canonical identification of
chain complexes obtained by ``reversing'' the flow
$\eta'_{f}:\mathcal{C}(L;\mathcal{R}; -f,\rho, J)\to
s^{n}(\mathcal{C}(L;\mathcal{R};f,\rho, J)^{\odot})$ (which sends each
critical point $x\in \Crit_{k}(-f)$ to $x\in \Crit_{n-k}(f)$) and the
second is $\phi^{F,\rho_{L}, J}$ where $F$ is a Morse homotopy between
$f$ and $-f$ so that $\eta=\eta'_{f}\circ \phi^{F,\rho_{L},J}$.  To
prove the identity $H(\tilde{\eta})=\epsilon_{L}(-\ast-)$, let $f'$ be
another Morse function in generic position with $f$. In homology
$\eta_{\ast}=\phi^{G,\rho_{L},J}_{\ast}\circ (\eta'_{f'})_{\ast}\circ
\phi^{F',\rho_{L}, J}_{\ast}$ where $F'$ is a Morse homotopy from $f$
to $-f'$ and $G$ is a Morse homotopy from $f'$ to $f$. Thus we also
have $\eta_{\ast}=\phi^{G,\rho_{L},J}_{\ast}\circ
(\eta'_{f'})_{\ast}\circ \phi^{f,-f'}_{\ast}$.  The relation we want
to justify follows by comparing the moduli spaces associated to the
trees $\mathcal{T}'$ of symbol $(x:y)$ with $x\in \Crit(f)$,
$y\in\Crit(-f')$ used in the definition of $\phi^{f,-f'}$ and the
moduli spaces associated to trees $\mathcal{T}$ of symbol $(x,y:m)$
(with $f=f_{1}$, $f'=f_{2}$) where $x\in \Crit(f)$, $y\in \Crit(f')$,
$m\in \Crit_{0}(f_{3})$ used in the definition of the product $-\ast
-$ at the point b in \S\ref{subsubsec:defin_alg_str}.  Here $m$ is the
unique minimum of the function $f_{3}$.  Indeed it is immediate to see
that the $0$-dimensional such moduli spaces are in bijection and this
implies the claimed identity.

It remains to prove that the pairing $H(\tilde\eta)^0$ (and thus $H(\tilde{\eta})$) 
is
non-degenerate when $\mathcal{R} = \Lambda$. From now on we put
$\mathcal{R} = \Lambda$ and omit it from the notation.

Let $\mathcal{C}$ be a finite rank free $\Lambda$-chain complex (e.g.
$\mathcal{C} = \mathcal{C}(L; f, \rho, J)$). Consider the following
pairing:
\begin{equation} \label{eq:pairing} \Theta : H_k(\mathcal{C}) \otimes
   H_{-k}(\mathcal{C}^{\odot}) \to \Lambda_0 = \mathbb{Z}_2,
\end{equation}
which is defined as follows. Given two classes $a \in
H_k(\mathcal{C})$, $g \in H_{-k}(\mathcal{C}^{\odot})$ choose cycles
representing them, $a = [\alpha]$, $g = [\varphi]$, and define
$\Theta(a \otimes g) = \varphi(\alpha)$. It is easy to see that
$\Theta$ is well defined. We will prove below the following.
\begin{lem} \label{l:Theta-non-degenerate} The pairing $\Theta$ is
   non-degenerate.
\end{lem}
Note that in view of the canonical isomorphisms $QH_*(L) \cong
H_{*-n}(\mathcal{C}(L;f, \rho, J)^{\odot})$ the non-degeneracy of
$\Theta$ (for $\mathcal{C} = \mathcal{C}(L;f, \rho, J)$) implies that
$H(\tilde{\eta})^0$ is non-degenerate.

We now proceed to prove Lemma~\ref{l:Theta-non-degenerate}.  Given $l
\in \mathbb{Z}$ denote by $(\hom_{\Lambda}(H(\mathcal{C}),
\Lambda))_l$ the space of $\Lambda$-linear morphisms $h:
H(\mathcal{C}) \to \Lambda$ that have degree $l$.  Consider now the
following canonical map:
$$\rho: H_l(\mathcal{C}^{\odot}) \to 
(\hom_{\Lambda}(H(\mathcal{C}), \Lambda))_l,$$ defined as follows.
Given $g \in H_l(\mathcal{C}^{\odot})$, choose a cycle $\varphi \in
\mathcal{C}^{\odot}_{l} = (\hom_{\Lambda}(\mathcal{C}, \Lambda))_l$
that represents $g$. Clearly, $\varphi$ descends to a map
$H_*(\mathcal{C}) \to \Lambda_{*+l}$ which we define to be $\rho(g)$.
It is easy to see that the map $\rho$ is well defined.  Note also that
we have $$\Theta(a \otimes g) = \rho(g)(a), \quad \forall a \in
H_k(\mathcal{C}), g \in H_{-k}(\mathcal{C}^{\odot}).$$
\begin{lem} \label{l:non-deg} Let $\mathcal{C}$ be as above. Then for
   every $l \in \mathbb{Z}$ the map $\rho$ is an isomorphism.
\end{lem}

Before proving this lemma that let us see how it implies the
non-degeneracy of $\Theta$ (hence that of $H(\tilde{\eta})^0$).
\begin{proof}[Proof that $\Theta$ is non-degenerate]
   Let $0 \neq a \in H_k(\mathcal{C})$. Choose a homomorphism $\phi_k:
   H_k(\mathcal{C}) \to \Lambda_0 = \mathbb{Z}_2$ with $\phi_k(a) \neq
   0$.  Extend $\phi_k$ to a an $\Lambda$-linear homomorphism $\phi:
   H_*(\mathcal{C}) \to \Lambda_{*-k}$ (this extension can be done by
   linearity over $\Lambda$ in degrees $* = k + q N_L$ and by $0$ in
   all other degrees).  Clearly $\Theta(a \otimes \rho^{-1}(\phi)) =
   \phi(a) \neq 0$.

   Assume now that $0 \neq g \in H_{-k}(\mathcal{C}^{\odot})$. Then
   $\rho(g) : H_*(\mathcal{C}) \to \Lambda_{*-k}$ is a non-trivial
   homomorphism. This means that there exists $j \in \mathbb{Z}$ and
   $b \in H_j(\mathcal{C})$ such that $\rho(g)(b) \neq 0$. As
   $\rho(g)(b) \in \Lambda_{j-k}$ it follows that $N_L \mid (j-k)$.
   Put $a = t^{(j-k)/N_L} b \in H_k(\mathcal{C})$. Clearly $\rho(g)(a)
   \neq 0$, which implies that $\Theta(a \otimes g) \neq 0$. This
   concludes the proof of the non-degeneracy of $\Theta$, modulo the
   proof of Lemma~\ref{l:non-deg}.
\end{proof}

To prove Lemma~\ref{l:non-deg} we need some more preparation. Let
$\mathcal{R}$ be a commutative graded ring and $M$ a graded
$\mathcal{R}$-module. Denote by $\pi_i: M \to M_i$ the projection on
the $i$'th component of $M$. Let $N \subset M$ be a submodule. We say
that $N$ is a graded submodule if for every $x \in N$ we have
$\pi_i(x) \in N$ for every $i \in \mathbb{Z}$. In that case the
grading of $M$ induces a grading on $N$ and $N$ becomes a graded
$\mathcal{R}$-module by itself. Note that not every submodule of a
graded module is graded. However:
\begin{lem} \label{l:gr-mod-homogeneous}
   \begin{enumerate}[i.]
     \item A submodule $N \subset M$ is a graded submodule iff it is
      generated (over $\mathcal{R}$) by a collection $\{x_s\}_{s \in
        \mathcal{S}}$ of homogeneous elements.  In particular, if
      $N_1, N_2 \subset M$ are graded submodules then so is $N_1 +
      N_2$. \label{I:gr-mod-1}
     \item Let $\mathcal{R} = \Lambda$. Let $M$ be a free finite rank
      graded $\Lambda$-module and $N \subset M$ a graded submodule.
      Then there exists a graded submodule $Q \subset M$ which is a
      complement of $N$, i.e. $N \oplus Q = M$.
      \label{I:gr-mod-2}
   \end{enumerate}
\end{lem}
\begin{proof}
   The proof of statement~\ref{I:gr-mod-1} is straightforward, so we
   omit it.

   We prove~\ref{I:gr-mod-2}. Choose a {\em homogeneous} element $x_1
   \in M \setminus N$ (if there are no such elements clearly $N = M$).
   Put $Q^{(1)} = \Lambda x_1$. We claim that $N \cap Q^{(1)} = 0$.
   Indeed, assume that $0 \neq \lambda x_1 \in N$ for some $\lambda
   \in \Lambda$. As $x_1$ is homogeneous and $N$ is a graded
   submodule, all the homogeneous components of $\lambda x_1$ must lie
   in $N$. In particular there exists $r \in \mathbb{Z}$ such that
   $t^r x_1 \in N$.  As $t^r$ is invertible it follows that $x_1 \in
   N$. A contradiction.

   We now continue the same construction inductively, namely we choose
   a homogeneous element $x_2 \in M \setminus (N + Q^{(1)})$. We claim
   that $\Lambda x_2 \cap (N + Q^{(1)}) = 0$. The argument is similar
   to the preceding one (for $N \cap \Lambda x_1 = 0$). The point is
   that $N + Q^{(1)}$ is a graded submodule. Put $Q^{(2)} = Q^{(1)} +
   \Lambda x_2$. Clearly we have $N \cap E^{(2)} = 0$. Note also that
   $Q^{(2)}$ and $N + Q^{(2)}$ are both graded submodules of $M$.

   Continuing this inductive construction we obtain, after a finite
   number of steps $\nu$, the desired complement $Q = Q^{(\nu)}$ which
   satisfies $N \oplus Q = M$. It is important here that $M$ is free
   of finite rank and that $\Lambda$ is a PID. These two conditions
   assure that every submodule of $M$ is also free with rank $\leq $
   the rank of $M$. In particular the process of defining $Q$
   concludes in a finite number of steps.
\end{proof}
\begin{rem}
   We remark that the statement at point~\ref{I:gr-mod-2} does not
   seem to hold if we replace $\Lambda$ by more general graded rings
   $\mathcal{R}$. In order for the proof above to work we need that
   every non-trivial element in {\em each} $\mathcal{R}_j$ ($\forall
   \, j \in \mathbb{Z}$) is invertible. This obviously holds for
   $\mathcal{R} = \Lambda$, but not for $\mathcal{R} = \Lambda^+$ for
   example.
\end{rem}

Coming back to a finite rank free $\Lambda$-chain complex
$(\mathcal{C},d)$, denote by $Z = \ker d \subset \mathcal{C}$ the
cycles and by $B = d(\mathcal{C}) \subset \mathcal{C}$ the boundaries.
Note that both $Z$ and $B$ are graded $\Lambda$-submodules of
$\mathcal{C}$. The following Lemma is an immediate consequence of
Lemma~\ref{l:gr-mod-homogeneous}-~\ref{I:gr-mod-2}.
\begin{lem} \label{l:gr-complement} There exist graded
   $\Lambda$-submodules $E \subset \mathcal{C}$ and $Z' \subset Z$
   such that $Z$ and $\mathcal{C}$ split as direct sums of graded
   $\Lambda$-modules: $$Z = Z' \oplus d(E), \quad \mathcal{C} = Z'
   \oplus d(E) \oplus E.$$ In particular, the restriction of $d$ to
   $E$, $d_{E} = d|_{E} : E \to d(E)$ is an isomorphism and $d(E) =
   B$.  Moreover, $E \oplus d(E)$ is an acyclic complex and
   $H_*(\mathcal{C}) \cong Z'_*$.
\end{lem}
This decomposition is of course not canonical. 
\begin{proof}[Proof of Lemma~\ref{l:non-deg}] 
   We first show that $\rho$ is injective. Suppose that $\rho(g) = 0$.
   Choose a cycle $\varphi:\mathcal{C}_* \to \Lambda_{*+l}$
   representing $g$. As $\rho(g)=0$ we have $\varphi|_{Z'} = 0$ and
   since $\varphi$ is a cycle we also have $\varphi|_{d(E)}=0$. Define
   $\psi: \mathcal{C}_* \to \Lambda_{*+l+1}$ by $\psi|_{Z'} =
   \psi|_{E} = 0$ and $\psi|_{d(E)} = \varphi \circ d_{E}^{-1}$.
   Clearly we have $\psi \circ d = \varphi$ which means that $\varphi$
   is a boundary, hence $g = [\varphi]=0$.  This shows that $\rho$ is
   injective.

   It remains to show it is surjective. Let $\varphi:H_*(\mathcal{C})
   \to \Lambda_{*+l}$ be an element in
   $(\hom_{\Lambda}(H(\mathcal{C}), \Lambda))_l$. View $\varphi$ as
   $\varphi: Z'_{*} \to \Lambda_{*+l}$. Extend $\varphi$ by $0$ to $Z'
   \oplus d(E) \oplus E$. Call this extension $\varphi'$.  Clearly
   $\varphi'$ is a cycle in $\mathcal{C}_{l}^{\odot}$ and
   $\rho[\varphi'] = \varphi$. This concludes the proof of
   Lemma~\ref{l:non-deg} as well as that of Proposition
   \ref{p:duality}.
\end{proof}

\begin{rem}\label{rem:Pduality_more}
   a.  The relation between the duality above and Poincar\'e duality
   is as follows: in case $\mathcal{C}(-)$ in the statement is
   replaced with the Morse complex $C(f)$ of some Morse function
   $f:L\to \R$ (and we take $\mathcal{R}=\Z_{2}$) we may define the
   morphism $\eta:C(f)\to s^{n}(C(f)^{\odot})$ as a composition of two
   morphisms with the first being the usual comparison morphism
   $C(f)\to C(-f)$ and the second $C(-f)\to s^{n}(C(f)^{\odot})$ given
   by $\Crit(f)\ni x \to x^{\ast}\in\hom_{\Z_2}(C(f),\Z_2)^{inv}$. We
   have the identifications
   $H_{k}(s^{n}(C(f)^{\odot}))=H_{k-n}(C(f)^{\odot})=H^{n-k}(C(f))$
   and the morphism $\eta$ described above induces in homology the
   Poincar\'e duality map: $H_{k}(L)\to H^{n-k}(L)$.

   b. Proposition \ref{p:duality} also shows that $QH(L)$ together
   with the bilinear map $\epsilon_{L}(-\circ-)$ is a Frobenius
   algebra, though not necessarily commutative.

   c.  The quantum inclusion, $i_{L}$, the duality map, $\eta$, and
   the Lagrangian quantum product determine the module structure by
   the following formula (which extends (\ref{eq:inclusion_mod})):
   \begin{equation}\label{eq:mod_inclusion}
      \langle h, i_{L}(x\circ y) \rangle = \eta(y) (PD(h)\circledast x)
   \end{equation}
   where $h\in H^{\ast}(M;\Z_{2})$, $x,y\in QH_{\ast}(L;\mathcal{R})$.
   Of course, here $\eta(y)\in H_{\ast}(s^{n}(\hom_{\mathcal{R}}
   (\mathcal{C}(L;\mathcal{R};f),\mathcal{R}))$ so that it can be
   evaluated on $QH_{\ast}(L;\mathcal{R})$. As in formula
   (\ref{eq:inclusion_mod}), the pairing on the left side is the
   $\mathcal{R}$-linear extension of the standard Kronecker pairing.
\end{rem}

\subsection{Wide Lagrangians and identifications with singular
  homology} \label{sb:wide-canonical}

Let $L \subset (M, \omega)$ be a monotone wide Lagrangian. This means
that there exists an isomorphism $QH_*(L) \cong (H(L;\mathbb{Z}_2)
\otimes \Lambda)_*$.  However, in general {\em there is no such
  canonical isomorphism!}

To explain this better, denote by $\mathcal{F} = (f, \rho)$ pairs of
Morse data. For any two pairs $\mathcal{F}=(f, \rho)$ and
$\mathcal{F}' = (f', \rho')$ and any two choices of almost complex
structures $J$ and $J'$ denote by $\Psi_{\scriptscriptstyle
  (\mathcal{F}',J'), (\mathcal{F}, J)}:
H_*(\mathcal{C}(L;\mathcal{F},J)) \to H_*(\mathcal{C}(L;\mathcal{F}',
J'))$ the canonical isomorphism between the pearl homologies (as
described at point e in~\S\ref{subsubsec:defin_alg_str}). Denote by
$\Psi^{\textnormal{{\tiny Morse}}}_{{\scriptscriptstyle \mathcal{F}',
    \mathcal{F}}}: H_*(\mathcal{F}) \to H_*(\mathcal{F}')$ the
canonical isomorphism between the Morse homologies associated to
$\mathcal{F}$ and $\mathcal{F}'$. From this point of view,
$H_{\ast}(L;\Z_{2})\otimes \La$ is identified with the family of
homologies $H_{\ast}(\mathcal{F})\otimes\La$ related by the canonical
isomorphisms mentioned above. Similarly, the quantum homology
$QH_{\ast}(L)$ is identified with the family of homologies
$H_{\ast}(\mathcal{C}(L;\mathcal{F},J))$ together with the canonical
isomorphisms $\Psi_{\scriptscriptstyle (\mathcal{F}',J'),
  (\mathcal{F}, J)}$.  Therefore, specifying a map $I:
H(L;\mathbb{Z}_2) \otimes \La\to QH(L)$ is equivalent to having a
family of maps $I_{(\mathcal{F}, J)} : H(\mathcal{F})\otimes\La \to
H(\mathcal{C}(L; \mathcal{F}, J))$ indexed by regular pairs
$(\mathcal{F}, J)$ such that the following diagram commutes for every
two such pairs $(\mathcal{F},J)$, $(\mathcal{F}', J')$:
\begin{equation} \label{eq:HF-H}
   \begin{CD}
      (H(\mathcal{F}) \otimes \Lambda)_* @>{\Psi^{\textnormal{{\tiny
              Morse}}}_{{\scriptscriptstyle \mathcal{F}',
            \mathcal{F}}}}>> (H(\mathcal{F}') \otimes \Lambda)_* \\
      @V{I_{(\mathcal{F},J)}}VV @VV{I_{(\mathcal{F}',J')}}V \\
      H_*(\mathcal{C}(L;\mathcal{F},J)) @>{\Psi_{\scriptscriptstyle
          (\mathcal{F}',J'), (\mathcal{F}, J)}}>>
      H_*(\mathcal{C}(L;\mathcal{F}',J'))
   \end{CD}
\end{equation}
Of course, in order to define such a family of maps it is enough to
choose a reference pair $(\mathcal{F}_0, J_0)$, define
$I_{(\mathcal{F}_0, J_0)}$ and then all the rest of the
$I_{(\mathcal{F},J)}$ are uniquely determined.

The point is that, in general, these choices do not lead to a
canonical map $I$.  To illustrate this, consider for simplicity the
case when $L$ admits a perfect Morse function and consider only Morse
data $\mathcal{F} = (f, \rho)$ where $f:L \to \mathbb{R}$ is a perfect
Morse function. Write the pearl differential $d$ as $d = d_0 + d'$,
where $d_0$ is the Morse differential. As $f$ is perfect we have $d_0
= 0$, so that $d = d'$. Moreover, since we assume that $L$ is wide, a
dimension comparison shows that $d'$ must vanish too (for otherwise
the rank of $QH(L)$ would be smaller than that of $H(L) \otimes
\Lambda$). Thus $d=0$ for every pair $(\mathcal{F}, J)$ as above. It
follows that: $$H_*(\mathcal{F}) = \mathbb{Z}_2 \langle
\textnormal{Crit}_*(f) \rangle, \quad
H_*(\mathcal{C}(L;\mathcal{F},J)) = (\mathbb{Z}_2 \langle
\textnormal{Crit}(f) \rangle \otimes \Lambda)_*.$$ At first glance it
seems that a natural isomorphism between the singular and quantum
homologies can be defined by $I_{(\mathcal{F},J)}(x) = x$ for every $x
\in \textnormal{Crit}(f)$ for every $(\mathcal{F}, J)$ (with the Morse
function in $\mathcal{F}$ being perfect). A more careful inspection
shows that if we define the isomorphisms $I_{(\mathcal{F}, J)}$ in
this way the diagram~\eqref{eq:HF-H} might not commute. A close look
at the definition of the comparison morphism $\Psi_{\scriptscriptstyle
  (\mathcal{F}',J'), (\mathcal{F}, J)}$ from point e
in~\S\ref{subsubsec:defin_alg_str} (see also an alternative
description in the proof of Proposition~\ref{p:duality}) shows that
$\Psi_{\scriptscriptstyle (\mathcal{F}',J'), (\mathcal{F}, J)}$ might
differ from $\Psi^{\textnormal{{\tiny Morse}}}_{{\scriptscriptstyle
    \mathcal{F}', \mathcal{F}}}$ by some quantum terms. In fact we
have:
\begin{equation} \label{eq:psi-morse-qh} \Psi_{\scriptscriptstyle
     (\mathcal{F}',J'), (\mathcal{F}, J)} = \Psi^{\textnormal{{\tiny
         Morse}}}_{{\scriptscriptstyle \mathcal{F}', \mathcal{F}}} +
   \sum_{i\geq 1} \Phi^i_{\scriptscriptstyle (\mathcal{F}',J'),
     (\mathcal{F}, J)} t^i,
\end{equation}
where the term $\Phi^i$ maps $\mathbb{Z}_2 \langle
\textnormal{Crit}_*(f) \rangle$ to $\mathbb{Z}_2 \langle
\textnormal{Crit}_{*+i N_L}(f') \rangle$ and is defined by counting
elements in some moduli spaces involving $J$ and $J'$-holomorphic
disks with total Maslov index $iN_L$. (See the precise description in
the proof of Proposition~\ref{p:duality} in~\S\ref{Sbs:duality}.)  It
is not hard to write down examples where some of the quantum terms
$\Phi^i$ do not vanish (see~\cite{Bi-Co:Yasha-fest, Bi-Co:qrel-long}).
In fact, this turns out to be the case for the Clifford torus
$\mathbb{T}_{\textnormal{clif}} \subset {\mathbb{C}}P^n$
(see~\S\ref{subsec:Cliff_calc}).

Despite the above there are situations in which a {\em canonical}
isomorphism $QH(L) \cong H(L;\mathbb{Z}_2) \otimes \Lambda$ exists, at
least in some degrees.

\begin{prop} \label{p:wide-can-iso} Let $L^n \subset (M^{2n}, \omega)$
   be a monotone Lagrangian (not necessarily wide or narrow).
   \begin{enumerate}[i.]
     \item For every $q \geq n-N_L+2$ there exists a canonical
      isomorphism $I: H_q(L;\mathbb{Z}_2) \longrightarrow
      Q^{+}H_{q}(L)$.  Moreover, this isomorphism maps the fundamental
      class $[L]$ to the unity in $Q^{+}H(L)$. \label{I:can-iso-1}
     \item If $L$ is not narrow then the isomorphism $I$
      from~\ref{I:can-iso-1} exists also for $q = n-N_L+1$.
      \label{I:can-iso-2}
     \item If $L$ is wide, the isomorphism $I$ induces a canonical
      embedding $H_{q}(L;\Z_{2})\otimes \La_{\ast} \hooklongrightarrow
      QH_{q+\ast}(L)$ for every $q \geq n-N_L + 1$.  In particular
      (for wide Lagrangians), if $N_L \geq n+1$ we have a canonical
      isomorphism $(H(L;\mathbb{Z}_2) \otimes \Lambda)_* \cong
      QH_*(L)$. \label{I:can-iso-3}
   \end{enumerate}
\end{prop}

\begin{proof} 
   Let $\mathcal{F}=(f,\rho)$ be a a pair formed by a Morse function
   and a Riemannian metric on $L$ and let $J$ be an almost complex
   structure on $M$ such that the pearl complex
   $\mathcal{C}^{+}(L;\mathcal{F}, J)$ as well as the Morse complex
   $C(\mathcal{F})$ are defined. Throughout the proof we will assume
   without loss of generality that $f$ has a unique maximum which we
   denote by $m$.

   Write the pearl differential $d$ on $\mathcal{C}^+(L;\mathcal{F},
   J) = C(\mathcal{F}) \otimes \Lambda^+$ as $$d =
   \partial_0 + \partial_1 t + \cdots + \partial_{\nu} t^{\nu},$$
   where $\partial_0$ is the Morse differential and $\partial_i$ is an
   operator acting as $\partial_i: C_{k}(\mathcal{F}) \to C_{k-1+i
     N_L}(\mathcal{F})$. For degree reasons we have:
   \begin{equation}\label{eq:ident_degree}
      C_{\geq n-N_{L}+1}(\mathcal{F}) = 
      \mathcal{C}^{+}_{\geq n-N_{L}+1}(L;f,\rho,J).
   \end{equation}
   Moreover, $d = \partial_0$ on $C_{\geq n-N_L+2}(\mathcal{F})$ and
   $d = \partial_0 + \partial_1 t$ on $C_{n-N_L + 1}(\mathcal{F})$,
   where $\partial_1: C_{n-N_L+1}(\mathcal{F}) \to C_n(\mathcal{F})$.

   Point~\ref{I:can-iso-1} now easily follows since $x \in C_{\geq
     n-N_L+2}(\mathcal{F})$ is a $\partial_0$-cycle iff it is a
   $d$-cycle and $x$ is a $\partial_0$-boundary iff it is a
   $d$-boundary. Therefore, the map $$\tilde{I}:C_q(\mathcal{F})
   \longrightarrow \mathcal{C}^{+}_{q}(L;\mathcal{F},J), \quad
   \tilde{I}(x)= x,$$ descends to an isomorphism $I$ in homology.  As
   $m$ represents the fundamental class, $I$ clearly sends $[L]$ to
   the unity of $QH^+(L)$. This completes the proof
   of~\ref{I:can-iso-1} except of the canonicity of $I$ which will be
   proved soon.

   We turn to point~\ref{I:can-iso-2}.  We claim again that $x \in
   C_{n-N_{L}+1}(\mathcal{F})$ is a $\partial_0$-cycle iff it is a
   $d$-cycle and $x$ is a $\partial_0$-boundary iff it is a
   $d$-boundary. Indeed, the claim is obvious for boundaries since $d
   = \partial_0$ on $C_{n-N_L + 2}(\mathcal{F})$. It remains to show
   that $\partial_0$ and $d$-cycles coincide on
   $C_{n-N_L+1}(\mathcal{F})$. Let $x \in C_{n-N_L+1}(\mathcal{F})$.
   We have $d(x) = \partial_0(x) +
   \partial_1(x)t$. This implies that if $x$ is a $d$-cycle then it is
   also a $\partial_0$-cycle. Suppose now that $x$ is a
   $\partial_0$-cycle. We then have $d(x) =
   \partial_0(x) + \partial_1(x) t = \partial_1(x)t$. If $d(x) \neq 0$
   then $d(x) = mt$ which implies that $m$ is a boundary hence $QH(L)
   = 0$ and $L$ is narrow, contrary to our assumption. Thus $d(x)=0$
   and $x$ is a $d$-cycle.
   
   We can now extend the definition of $\tilde{I}$ to
   $\tilde{I}:C_{n-N_L+1}(\mathcal{F}) \longrightarrow
   \mathcal{C}^{+}_{n-N_L+1}(L;\mathcal{F},J)$ by $\tilde{I}(x) = x$,
   and as before $\tilde{I}$ descends to an isomorphism $I$ in
   homology.

   To conclude the proofs of
   points~\ref{I:can-iso-1},~\ref{I:can-iso-2} it remains to show that
   $I$ is canonical in the sense discussed before the statement of the
   proposition. To see this, write the map $\tilde{I}$ as
   $\tilde{I}_{\scriptscriptstyle(\mathcal{F}, J)}$ to denote the
   relation to the data $(\mathcal{F}, J)$. For degree reasons it
   follows that the maps $\Phi^i_{\scriptscriptstyle
     (\mathcal{F}',J'), (\mathcal{F}, J)}$ in~\eqref{eq:psi-morse-qh}
   vanish on $\mathcal{C}^+_q$ for $q \geq n-N_L+1$, hence the squares
   in~\eqref{eq:HF-H} commute. This completes the proof of the first
   point of the proposition.

   We now turn to the proof of~\ref{I:can-iso-3}.  Consider the
   canonical map $p: Q^+H(L) \to QH(L)$ induced by the extension of
   coefficients $\Lambda^+ \to \Lambda$. The embedding
   $H_q(L;\mathbb{Z}_2) \otimes \Lambda_* \hooklongrightarrow
   QH_{q+*}(L)$ is induced by the map $p \circ I: H_q(L;\mathbb{Z}_2)
   \to QH_q(L)$. So, the proof is reduced to showing that $p \circ I$
   is an injection. To see this, let $x \in C_q(\mathcal{F})$ be a
   $\partial_0$-cycle with non-trivial Morse homology class
   $[x]_{{\tiny Morse}}$, where $q \geq n-N_L+1$. By what we have just
   proved, $x$ is also a $d$-cycle. We have to prove that $x$, when
   viewed as an element in $\mathcal{C}_q(L; \mathcal{F},J)$, is not a
   $d$-boundary. Consider the minimal model $\mathcal{C}_{min}(L)$
   together with the structural map
   $\phi:\mathcal{C}(L;\mathcal{F},J)\to \mathcal{C}_{min}(L)$ as
   constructed in~\S\ref{subsec:minimal}. Recall that by that
   construction $\phi_0(x) = [x]_{{\tiny Morse}} \neq 0$, hence
   $\phi(x) \neq 0$. On the other hand, by~\ref{rem:product_min} c,
   the differential of $\mathcal{C}_{min}(L)$ vanishes because $L$ is
   wide. As $\phi$ is a chain map it follows that $x$ cannot be a
   $d$-boundary.
\end{proof}

\section{Proofs of the main theorems}\label{sec:proof_main}
This section is focused on proving the three main theorems of the
introduction.

\

Before we go on with the proof we would like to make a small but
useful algebraic observation which will be used many times in the
sequel. Consider the graded vector space $H(L;\mathbb{Z}_2) \otimes
\Lambda^{+}$ endowed with the  grading coming from both factors.  Let $a
\in (H(L;\mathbb{Z}_2) \otimes \Lambda^{+})_l$ be a homogeneous
element (of degree $l$). Then we can decompose $a$ in a unique way as:
$$a = \sum_{r \geq 0} a_{l+rN_L} t^r, \quad 
a_{l+rN_L} \in H_{l_+rN_L}(L;\mathbb{Z}_2).$$ Suppose now that
$a_{l+r_0N_l} = [L] \in H_n(L;\mathbb{Z}_2)$ for some $r_0$.  In that
case we will say that $a$ contains $[L]t^{r_0}$.  (Note that this can
happen only if $l+r_0N_L = n$.) Then, as $H_{>n}(L;\mathbb{Z}_2) = 0$,
$|t|<0$, the decomposition of $a$ cannot contain terms with $t$ of
higher order than $r_0$, i.e. $$a = [L]t^{r_0} + a_{n-N_L}t^{r_0-1} +
a_{n-2N_L}t^{r_0-2} + \cdots.$$ We will abbreviate this by writing $a
= [L]t^{r_0} + l.o.(t)$, where $l.o.(t)$ stands for terms of lower
order in $t$. Similarly, if a homogeneous element $a$ contains
$[pt]t^{l_0}$ for some $l_0 \geq 0$, then must have $a = [pt]t^{l_0} +
h.o.(t)$, where $h.o.(t)$ stands for terms of higher order in $t$.

A similar discussion applies to homogeneous elements in the positive
quantum homology $QH_*(M;\Lambda^+) = (H(M;\mathbb{Z}_2) \otimes
\Lambda^+)_*$, as well as in the positive pearl complex
$\mathcal{C}^+(L;f, \rho, J)$ in case the function $f$ has a unique
maximum and a unique minimum.

\subsection{Proof of Theorem~\ref{thm:floer=0-or-all}} \label{sb:prf-hf=0-or-all}

The argument is based on the minimal model machinery
from~\S\ref{subsec:minimal}. Consider the pearl complex
$\mathcal{C}^{+}(f,J)$ and recall from \S\ref{subsec:minimal} that
there exists a chain complex $(\mathcal{C}^{+}_{min}(L) =
H(L;\mathbb{Z}_2) \otimes \Lambda^+,\delta)$, unique up to
isomorphism, and chain morphisms $\phi:\mathcal{C}^{+}(f,J)\to
\mathcal{C}^{+}_{min}(L)$, $\psi:\mathcal{C}^{+}_{min}(L)\to
\mathcal{C}^{+}(f,J)$ so that $\phi\circ\psi=id$, $\delta_{0}=0$
(where $\delta_{0}$ is obtained from $\delta$ by putting $t=0$) and
$\phi$, $\psi$, $\phi_{0}$, $\psi_{0}$ induce isomorphisms in quantum
and Morse homologies.  By Remark \ref{rem:product_min} the quantum
product in $\mathcal{C}^{+}(f,J)$ can be transported by the morphisms
$\phi$ and $\psi$ to a product $\ast:
\mathcal{C}^{+}_{min}(L)\otimes\mathcal{C}^{+}_{min}(L)\to
\mathcal{C}^{+}_{min}(L)$ which is a chain map and a quantum
deformation of the singular intersection product and so that $[L]\in
H_{n}(L;\Z_{2})$ is the unity at the chain level (notice though that,
as the maps $\phi$ and $\psi$ are not canonical, this product is not
canonical either). As discussed before we put
$\mathcal{C}_{min}(L)=\mathcal{C}^{+}_{min}(L)\otimes_{\La^{+}}\La =
H(L;\mathbb{Z}_2) \otimes \Lambda$.  As in the statement of the
theorem we assume that $H_{\ast}(L;\Z_{2})$ is generated by $H_{\geq
  n-l}(L;\Z_{2})$.  In view of Remark \ref{rem:product_min} b. the
first point of the theorem reduces to the next lemma.

\begin{lem} \label{lem:min_diff} Suppose that $N_{L}\geq l+1$.  If
   $\delta$, the differential of the minimal pearl complex, does not
   vanish, then $[L]$ is a boundary in $\mathcal{C}_{min}(L)$,
   $QH(L)=0$, and $N_{L}=l+1$.
\end{lem}
\begin{proof}
   There are two possibilities: either $\delta = 0$ on
   $H_{n-l}(L;\mathbb{Z}_2)$, or $\delta \neq 0$ on that homology.

   Assume first that $\delta = 0$ on $H_{n-l}(L;\mathbb{Z}_2)$.  We
   claim that $\delta = 0$ everywhere. To show this we will prove by
   induction that $\delta = 0$ on $H_{\geq n-l-s}(L;\mathbb{Z}_2)$ for
   every $s \geq 0$.
   
   Indeed, for $s=0$ this is true since $N_L \geq l+1$ implies that
   $\delta = 0$ on $H_{\geq n-l+1}(L;\mathbb{Z}_2)$, and moreover we
   have assumed that $\delta = 0$ on $H_{n-l}(L;\mathbb{Z}_2)$.
   Assume now that the assertion is true for some $s \geq 0$ and let
   $x \in H_{\geq n-l-s-1}(L;\mathbb{Z}_2)$. By the assumptions of
   Theorem~\ref{thm:floer=0-or-all} we can write $x = \sum_j a_j$
   where each $a_j$ is expressed as (classical) intersection products
   of elements from $H_{\geq n-l}(L;\mathbb{Z}_2)$. We now claim that
   $\delta(a_j) = 0$ for every $j$. To see this write $a_1 = x_1 \cdot
   \ldots \cdot x_r$ with $x_i \in H_{\geq n-l}(L;\mathbb{Z}_2)$,
   where $- \cdot -$ is the classical intersection product. We then
   have $\delta (x_{i})=0$ and we write $\delta(x_{1}\ast
   x_{2}\ast\ldots\ast x_{r})= \sum_{i} x_{1}\ast \ldots
   \delta(x_{i})\ast\ldots\ast x_{r}=0$. At the same time
   \begin{equation}\label{eq:prod_quantprod}
      x_{1}\ast x_{2}\ast\ldots\ast x_{r}= a_1 + \sum_{q>0}z_{q}t^{q},
   \end{equation} 
   with $z_{j}\in H_{\geq n-l-s}(L;\Z_{2})$. (Recall that $|t| = -N_L
   \leq -2$). By the induction hypothesis we have $\delta (z_{j})=0$,
   hence $\delta(a_1) = 0$.  The same argument shows that $\delta(a_j)
   = 0$ for every $j$.  It follows that $\delta(x) = 0$. This proves
   that $\delta = 0$ on $H_{\geq n-l-s-1}(L;\mathbb{Z}_2)$ and
   completes the induction.

   We now turn to the second case: $\delta \neq 0$ on
   $H_{n-l}(L;\mathbb{Z}_2)$. First note that we must have $N_L=l+1$.
   Indeed, if $N_L \geq l+2$ then by degree reasons $\delta = 0$ on
   $H_{n-l}(L;\mathbb{Z}_2)$ and by what we have just proved we obtain
   $\delta = 0$ everywhere, a contradiction. Thus $N_L = l+1$. By
   degree reasons again it follows that $\delta$ sends
   $H_{n-l}(L;\mathbb{Z}_2)$ non-trivially to $H_n(L;\mathbb{Z}_2)t$.
   Thus there exists $x \in H_{n-l}(L;\mathbb{Z}_2)$ such that
   $\delta(x) = [L]t$. This implies that $[L]$ is a boundary. As $[L]$
   is the unity of $QH(L)$ we also obtain $QH(L) = 0$.
\end{proof}

We now pursue with the proof of the second point of
Theorem~\ref{thm:floer=0-or-all}.  Thus we assume that $L$ is narrow
and so $[L]$ is a boundary in $\mathcal{C}_{min}(L)$ and $N_{L}\leq
l+1$.  Let $K$ be the constant in the statement of the theorem,
$K=\max\{l+1,n+1-N_{L}\}$ when $N_{L}< l+1$ and $K=l+1$ when
$N_{L}=l+1$. Notice that
the degree $n$ component of $\mathcal{C}_{min}^{+}(L)$ is
one-dimensional.  This implies that, despite the fact that the minimal
pearl model is determined only up to a non-canonical isomorphism, the
generator in degree $n$ is canonical. It will be denoted (as before) by $[L]$.  

In the following lemma we denote the differential of the complex
$\mathcal{C}^{+}_{min}(L)$ by $\delta^+$ to distinguish it from its
extension $\delta = \delta^{+} \otimes 1$ defined on
$\mathcal{C}_{min}(L) = \mathcal{C}^{+}_{min}(L) \otimes_{\Lambda^+}
\Lambda$. The main step is:
\begin{lem} \label{lem:bdry_order} Either there exists some $x\in
   H_{\ast}(L;\Z_{2})$ so that $\delta^+(x)=[L]t^{q}+ l.o.(t)$ or
   there are $y,z\in H_{\ast}(L;\Z_{2})$ so that $y\ast z =
   [L]t^{q}+l.o.(t)$, where in both cases $0 < qN_{L} \leq K$.
\end{lem}

\begin{proof}
   As $L$ is narrow, the first point of
   Theorem~\ref{thm:floer=0-or-all} implies that $N_L \leq l+1$.
   Assume first that $N_L = l+1$. Then by definition $K = l+1$. In
   this case, as the proof of Lemma~\ref{lem:min_diff} shows, there
   exists $x \in H_{n-l}(L;\mathbb{Z}_2)$ such that $\delta^+ x =
   [L]t$. Thus $x$ satisfies the statement of our lemma with $q=1$.

   We will assume from now on that $N_L < l+1$, and so $K \geq l+1,
   n+1-N_L$.  Let $w\in H_{\ast}(L;\Z_{2})$ be an element of maximal
   degree so that $\delta^+ w$ contains $[L]t^s$ for some $s \geq 0$.
   More precisely, denote by $\rho \in H^{n}(L;\mathbb{Z}_2)$ the
   generator (so that $\langle \rho, [L] \rangle = 1$). We require
   that $w$ is of minimal degree so that $\langle \rho,\delta w
   \rangle \not=0 \in \Lambda^+$. (Here and in what follows we extend
   the Kronecker pairing $\langle \cdot, \cdot \rangle$ to
   $H_*(L;\mathbb{Z}_2) \otimes \Lambda^+$ by linearity over
   $\Lambda^+$.)  Note that such a $w$ must exist, since $L$ is narrow
   hence $[L]t^r$ must be a $\delta^+$-boundary for some $r\geq 1$.

   If $|w|\geq n-l$, the statement of our lemma is verified with $x =
   w$ and $q=s$ because $q N_L = n - | \delta^+ w | = n+1 - |w|\leq
   l+1 \leq K$. Therefore we assume from now on that $|w|< n-l$. We
   know that $H_{\geq n-l}(L;\Z_{2})$ generates $H_{\ast}(L;\Z_{2})$
   as an algebra.  In particular, $|w|< n-l$ implies that $w$ is
   decomposable with respect to the intersection product.  We now
   write $$w=w_{1}\cdot w_{2}=w_{1}\ast w_{2}+\sum_{i>0}z_{i}t^{i},
   \quad \textnormal{with } \; |w| < |w_{1}|< n, \;\; |w| < |w_{2}| <
   n, \;\; |w| < |z_{i}|.$$ (Of course, $w$ can be a sum of such
   products but this does not make any difference in the argument and,
   in terms of notation, it is simpler to assume that just one such
   monomial appears.) Now
   $$\langle \rho, \delta^+ w\rangle=\langle \rho,
   (\delta^+ w_{1}) \ast w_{2} +w_{1}\ast (\delta^+ w_{2})\rangle +
   \sum_{i>0} \langle \rho, z_{i}\rangle t^{i}.$$ By the maximality of
   $|w|$, and the fact that $|w|<|z_i|$, we see that the second term
   on the left vanishes and we also get that for either $w_{1}$ or
   $w_{2}$, say $w_{1}$ (the other case is similar) we have $\langle
   \rho, (\delta^+ w_{1})\ast w_{2}\rangle = t^{q'}$ for some $q'>0$.
   We now write $\delta^+ w_{1}=\sum_{i>0} u_{i}t^{i}$ and we deduce
   that for some $i>0$ we have $\langle \rho, u_{i}\ast w_{2} \rangle
   = t^{q'-i}$.  Notice that $|u_{i}|=|w_{1}|+iN_{L}-1$. We put
   $q=q'-i$ (clearly $q\geq 0$ and we will show below that $q>0$). We
   now get: $$n-qN_{L}=|u_{i}\ast
   w_{2}|=|u_{i}|+|w_{2}|-n=|w_{1}|+iN_{L}-1 +|w_{2}|-n =|w|+i
   N_{L}-1\geq N_{L}-1.$$ Thus, $qN_{L}\leq n-N_{L}+1\leq K$ and the
   statement of our lemma will be verified with $y = u_i$ and $z =
   w_2$.

   It remains only to check that $q>0$. Assume by contradiction that
   $q=0$, or equivalently that $q'=i$. This implies that $u_i * w_2 =
   [L]$. But for degree reasons this cannot happen since $|w_2| < n$.
   A contradiction.
\end{proof}

To prove the second point of Theorem~\ref{thm:floer=0-or-all} we will
use Lemma~\ref{lem:bdry_order} to show that $L$ is uniruled of order
$K$. For this, we fix a generic almost complex structure $J$ as well
as a point $P\in L$. Fix a Morse function $f$ and a Riemannian metric
$\rho_L$ on $L$ so that the pair $(f, \rho_L)$ is Morse-Smale.
Moreover, we choose $f$ so that $P$ is its unique maximum. We also
pick a second Morse function $f_{1}$ so that the pair $(f_{1},
\rho_{L})$ is also Morse-Smale, and $f$ and $f_{1}$ are in general
position. We assume that $J$ is generically chosen so that
$\mathcal{C}^{+}(L;f,\rho_{L},J)$ and $\mathcal{C}^{+}(L;
f_{1},\rho_{L},J)$ are both defined as well as the relevant product.
As above, we let $\mathcal{C}^{+}_{min}(L)$ be the minimal pearl
complex and we fix $\phi,\psi,\phi_{1},\psi_{1}$, the structure maps
associated to $(f,\rho_{L},J)$ and, respectively, to
$(f_{1},\rho_{L},J)$ as constructed in the proof of
Proposition~\ref{prop:min_model} in~\S\ref{subsec:minimal}.

The following technical result is an easy consequence of the proof of
Proposition \ref{prop:min_model} and is valid independently of whether
$L$ is narrow or not.
\begin{lem} \label{lem:ex_disks2}
   \begin{enumerate}[i.]
     \item If there exists $z\in \Crit(f)$ so that
      $\phi(z)=[L]t^{s}+l.o.(t)$, $s>0$, then there exists $w\in
      \Crit(f)$ so that $dw=Pt^{s'}+l.o.(t)$ with $0<s'\leq s$.
      \label{I:ex_disks2-1}
     \item Let $a \in H_*(L;\mathbb{Z}_2)$ be a homogeneous element
      such that $\psi(a) = P t^s + l.o.(t)$, $s>0$. Then there exists
      $w \in \textnormal{Crit}(f)$ such that $dw = Pt^{s'} + l.o.(t)$
      with $0< s' < s$. \label{I:ex_disks2-2}
   \end{enumerate}
\end{lem}

\begin{proof} 
   We begin with point~\ref{I:ex_disks2-1}. As in the proof of
   Proposition~\ref{prop:min_model}, change the basis in
   $\Z_{2}\langle\Crit(f)\rangle$ so that the generators forming the
   new basis are of three types $B_{X}, B_{Y}\subset \ker (d_{0})$ and
   $B_{Y'}$ so that $B_{Y}$ and $B_{Y'}$ are in bijection and
   $d_{0}(B_{Y'})=B_{Y}$ (where $d_{0}$ is the Morse differential).
   For $y\in B_{Y}$ we denote by $y'\in B_{Y'}$ the element so that
   $d_{0}(y')=y$. As $\mathcal{C}^+_n(L;f, \rho, J) = \mathbb{Z}_2 P$
   we have $P\in B_{X}$.  The map $\phi:\mathcal{C}^{+}(L;f,\rho,J)\to
   \mathcal{C}^{+}_{min}(L)$ is defined so that for $x\in B_{X}$,
   $\phi(x)=[x]$ ($[x]$ is the Morse homology class of $x$), for
   $y'\in B_{Y'}$, $\phi(y')=0$ and for $y\in B_{Y}$, $\phi(y)=
   \phi(y-dy')$.  Let $u \in B_{Y}$ be a generator of the highest
   degree among the the elements of $B_{Y}$ with the property that
   there exists $0<s'\leq s$ with $\phi(u)=[L]t^{s'}+l.o.(t)$ (since
   $\phi(x)=[x]$ for $x\in B_{X}$, $\phi(y')=0$ for $y'\in B_{Y'}$ and
   $\phi(z) = [L]t^{s}+l.o.(t)$ with $s>0$, there must be such a $u$).
   Write $$u - d u' = \sum_{i>0} x_i t^i + \sum_{j>0} y_j t^j +
   \sum_{k>0} y'_k t^k, \quad \textnormal{with } x_i \in B_X, y_j \in
   B_Y, y'_k \in B_{Y'}.$$ We now have: $$[L]t^{s'} + l.o.(t) =
   \phi(u) = \phi(u-d u') = \sum_{i>0} \phi(x_i) t^i + \sum_{j>0}
   \phi(y_j) t^j.$$ Note that $|y_j|>|u|$ and therefore by the
   maximality of $u$ none of the terms $\phi(y_j)$ can contribute an
   $[L]t^{s''}$, $s''>0$ to that sum. Moreover, none of the terms
   $\phi(y_j)$ can contribute $[L]$ to that sum since $\phi$ of an
   element in $B_{Y}$ is divisible by $t$. It follows that there
   exists $i_0$ such that the term $\phi(x_{i_0})t^{i_0}$ contributes
   the element $[L]t^{s'}$. As $\phi(x_{i_0}) =
   [x_{i_0}]_{\textnormal{{\tiny Morse}}}$ it follows that $x_{i_0} =
   P$ and $i_0 = s'$. As the degree $n$ part of $B_X$ is $P$, and
   $B_{Y}$, $B_{Y'}$ do not contain elements of degree $n$, it follows
   that $u-du' = Pt^{s'} + l.o.(t)$. As $u$ is a linear combination of
   pure critical points (it doesn't involve $t$'s) we now obtain that
   $du' = Pt^{s'} + l.o.(t)$ (we work here over $\mathbb{Z}_2$ so $P =
   -P$). Finally, there must be a critical point $w$ participating in
   $u'$ (which is a linear combination of critical points) so that $dw
   = Pt^{s'} + l.o.(t)$. This completes the proof of
   point~\ref{I:ex_disks2-1}.

   We turn to the proof of~\ref{I:ex_disks2-2}. Write
   \begin{equation} \label{eq:psi-a} \psi(a) = Pt^{s} + z_{s-1}t^{s-1}
      + \cdots + z_1 t + z_0,
   \end{equation}
   with $z_i \in \mathbb{Z}_2 \langle \textnormal{Crit}(f) \rangle$.
   Note that $z_0 = \psi_0(a)$ and that by the construction of $\phi$
   and $\psi$ in the proof of Proposition~\ref{prop:min_model}
   in~\S\ref{subsec:minimal} we also have $\phi(z_0) = a$. Recall also
   that $\phi \circ \psi = id$. Using this, and applying $\phi$ to
   both sides of~\eqref{eq:psi-a} we obtain:
   $$0 = [L]t^s + \phi(z_{s-1})t^{s-1} + \cdots + \phi(z_1) t.$$
   Clearly not all of $\langle \rho,\phi(z_1)\rangle, \ldots, \langle \rho,\phi(z_{s-1})\rangle$ can vanish
   (where, as before, $\rho\in H^{n}(L;\Z_{2})$ is the generator).
   Let $1 \leq j \leq s-1$ be an index such that $\langle \rho,\phi(z_j)\rangle
   \neq 0$. We then have $\phi(z_j) = [L] t^{s-j} + l.o.(t)$.  By
   point~\ref{I:ex_disks2-1} just proved, there exists $w$ and $0< s'
   \leq s-j < s$ such that $dw = Pt^{s'} + l.o.(t)$.
\end{proof}

We continue with the proof of point~\ref{I:dichotomoy-narrow} of
Theorem~\ref{thm:floer=0-or-all}. We begin by analyzing the first
possibility resulting from Lemma \ref{lem:bdry_order}:
$\delta^{+}x=[L]t^{q}+l.o.(t)$ for some $x\in H_{\ast}(L;\Z_{2})$ with
$0< qN_{L}\leq K$. 

Consider the map $\phi:\mathcal{C}^{+}(L;f,\rho_{L}, J)\to
\mathcal{C}^{+}_{min}(L)$. As the degree $n$ part of $B_{X}$ consists of
$P$ only, we have $\phi(P)=[L]$.  By the definition of $\phi$ there
exists $u \in \mathbb{Z}_2 \langle B_X \rangle$ such that $\phi(u) =
x$.  Write $du = \sum_{i \geq 0} a_i t^i$. We have: $[L]t^{q}+l.o.(t)
= \delta^+ x = \delta^+ \phi(u) = \phi(du) = \sum_{i \geq 0}
\phi(a_i)t^i$. Thus there exists $0 \leq j \leq q$ such that
$\phi(a_j)t^j = [L]t^{q}+l.o.(t)$. There are two possibilities: either
$j=q$ or $j<q$. In case $j=q$ we must have $\phi(a_j) = [L]$ hence
$a_j = P$ and it follows that $du = Pt^q + l.o.(t)$. The element $u$
might not be a single critical point of $f$ but a linear combination
of such. However there must be a critical point $w$ participating in
the linear combination $u$ such that $dw = Pt^q + l.o.(t)$. In case
$j<q$ we obtain $\phi(a_j) = [L]t^{q-j} + l.o.(t)$ and as $q-j>0$ we
deduce from~\ref{lem:ex_disks2} that there exists $w \in
\textnormal{Crit}(f)$ so that $dw = Pt^{q'}+l.o.(t)$ with $0< q'\leq
q-j$. Summarizing, we see that in both cases ($j=q$ and $j<q$) that
there is $w \in \textnormal{Crit}(f)$ so that $dw=Pt^{s}+l.o.(t)$ with
$0< s\leq q$.  This implies that there exists a {\em non-constant}
$J$-disk through $P$ of Maslov index at most $qN_{L}$.

It remains to discuss the second case: $y\ast z= [L]t^{q}+l.o.(t)$.
The argument is similar. By definition $y*z = \phi(\psi_1(y) *
\psi(z))$.  Write $\psi_1(y)=\sum_{i\geq 0} y_i t^i$, $\psi(z) =
\sum_{j\geq 0} z_j t^j$, with $y_i \in \mathbb{Z}_2 \langle
\textnormal{Crit}(f_1)\rangle$, $z_j \in \langle
\textnormal{Crit}(f)\rangle$ being homogeneous elements. The equality
$[L]t^q + l.o.(t) = \sum_{i,j} \phi(y_i*z_j) t^{i+j})$ implies that
there exist $i,j \geq 0$ such that $\phi(y_i*z_j) t^{i+j} = [L]t^q +
l.o.(t)$. Write $y_i*z_i = \sum_{k \geq 0} p_{k}t^{k}$. We get that
there exists $k \geq 0$ such that $\phi(p_k)t^{k+i+j} = [L] t^q +
l.o.(t)$. Now there are two possibilities: either $k+i+j < q$ or
$k+i+j = q$.

In the first case ($k+i+j < q$) we get $\phi(p_k) = [L]t^{q-(k+i+j)} +
l.o.(t)$ and so by Lemma~\ref{lem:ex_disks2}-\ref{I:ex_disks2-1} there
exists $w \in \textnormal{Crit}(f)$ such that $dw = P t^{s'} +
l.o.(t)$ with $0<s' \leq q-(k+i+j)$. It follows that there exists a
non-constant $J$-disk through $P$ with Maslov index $ \leq s' N_L \leq
q N_L$.

In the second case ($k+i+j = q$) we have $\phi(p_k) = [L]$ hence $p_k
= P$ and $y_i*z_j = Pt^k + l.o.(t)$. If $k>0$ there exists a
non-constant $J$-disk through $P$ with Maslov index $\leq k N_L \leq
qN_L$. In case $k=0$ we have $y_i*z_j = P$ hence for degree reasons
$z_j=P$ (and $y_i=P_1$, where $P_1$ is the maximum of $f_1$). It
follows that $\psi(z) = P t^j + l.o.(t)$. We have $j > 0$, for
otherwise $\psi(z)=P$ so $z=[L]$ which is impossible in view of our
starting equality $y*z = [L]t^q + l.o.(t)$ with $y \in
H_*(L;\mathbb{Z}_2)$ and $q>0$. Thus $\psi(z) = P t^j + l.o.(t)$ with
$0<j\leq q$. By Lemma~\ref{lem:ex_disks2}-~\ref{I:ex_disks2-2} there
exists $w \in \textnormal{Crit}(f)$ with $dw = Pt^{j'} + l.o.(t)$ with
$0<j'<j \leq q$ and it follows that there exists a non-constant
$J$-disk through $P$ with Maslov index $\leq j'N_L < qN_L$. This
concludes the proof of Theorem~\ref{thm:floer=0-or-all}. \Qed

\subsection{Proof of Theorem~\ref{theo:geom_rig}}
\label{sb:prf-thm-geom-rig} 

Recall that we now suppose that $M$ is point invertible of order $k$.
This means that in the quantum homology of $M$ with coefficients in
$\Gamma^{+}=\Z_{2}[s]$ there exists $a\in QH_{\ast}(M; \Gamma^{+})$,
$a=a_{0}+a_{1}s$ with $0 \neq a_{0}\in H_{\ast}(M;\mathbb{Z}_2)$ and
$a_{1}\in QH_{\ast}(M; \Gamma^{+})$ so that $[pt]\ast
a=[M]s^{k/2C_M}$.  Recall that here $|s|=-2 C_M$.  Denote
$QH(M;\La^{+})=QH(M)\otimes_{\Z_{2}[s]}\La^{+}$.  Clearly, we also
have in $QH(M;\La^{+})$, $[pt]\ast a=[M]t^{k/N_{L}}$.

\

We start with the point i. of the theorem. We first notice that the
relation $[pt]\ast a= [M]t^{k/N_{L}}$ implies $|a|-2n=2n - k$ and as
$a=a_{0}+a_{1}s$ we have $0\leq |a|\leq 2n$ and so $k=4n-|a|\geq 2n$.
We now use the module structure
$$QH(M;\La^{+})\otimes Q^{+}H(L)\to Q^{+}H(L)$$
to write:
\begin{equation}\label{eq:pt_equation} 
   a*([pt]*[L]) = (a*[pt])*[L] = ([pt]*a)*[L] = 
   [M]*[L] t^{k/N_{L}}=[L]t^{k/N_{L}}~.~
\end{equation}
We need to analyze equation (\ref{eq:pt_equation}) at the chain level.
For this, we fix a Morse function $f:L\to \R$ with a single maximum
$P_f$ as well as a Morse function $g:M\to \mathbb{R}$ with a single
maximum $P_g$ and a single minimum $m_g$. We also fix Riemannian
metrics $\rho_L$ and $\rho_M$ on $L$ and $M$. The Morse complex of $g$
tensored with $\La^{+}$ will be denoted by $C^{+}(g)$.  We also fix a
minimal pearl complex for $L$, $\mathcal{C}^{+}_{min}(L)$, together
with the two associated structural maps $\phi$ and $\psi$ as
in~\S\ref{subsec:minimal}. We use the module operation (on the chain
level) in the form:
$$C^{+}(g)\otimes \mathcal{C}^{+}_{min}(L)\to \mathcal{C}^{+}_{min}(L),$$
by transporting the module operation $C^{+}(g)\otimes
\mathcal{C}^{+}(L;f, \rho_L,J) \to \mathcal{C}^{+}(L;f, \rho_L, J)$
via the structural maps $\phi$, $\psi$, i.e. for $h \in C^+(g)$,
$\alpha \in \mathcal{C}^{+}_{min}(L)$ we define $h * \alpha = \phi
(a*\psi(\alpha))$.

We write
\begin{equation} \label{eq:m*L}
   y=m_g\ast [L]=\sum_{i>0}z_{i}t^{i}~,~ \quad
   \textnormal{where } z_i \in H_*(L;\mathbb{Z}_2).
\end{equation}
Note that there are no classical terms here (i.e. $i=0$) for degree
reasons, since $|y|= -n$.

\begin{lem}\label{lem:complL_unir}
   There exists $0< i < \frac{k}{N_{L}}$ such that $z_i \neq 0$.
\end{lem}

\begin{proof}
   We write $y$ as a sum of three terms: $y=S_{1}+z't^{k/N_{L}} +
   S_{2}$ with $$S_{1}=\sum_{i=1}^{k/N_{L}-1}z_{i}t^{i}\ \ , \ \
   S_{2}=\sum_{i\geq k/N_{L}+1}z_{i}t^{i}$$ and $z_{i},z'\in
   H_{\ast}(L;\mathbb{Z}_2)$.  Notice that $S_{2}=0$ because $k\geq
   2n$, $|y|=-n$, $|z_{i}|\leq n$.

   Choose a cycle $a'\in C^{+}(g)$ which represents $a$. We have:
   $$a'\ast y= a'\ast S_{1}+ a'\ast z't^{k/N_{L}}$$
   and thus, $a'\ast S_{1}+ (a'\ast z' - [L])t^{k/N_{L}} \in
   Im(\delta^{+})$.

   We now claim that $a' * z' = 0$. To see this, first note that
   $|a'\ast z'|=|a|+|z'|-2n = (4n-k) + (-n+k) - 2n = n$.  Write $a'*z'
   = \sum_{q \geq 0} b_q t^q$, with $b_q \in H_*(L;\mathbb{Z}_2)$.  We
   have $|b_q| = |a'*z'| + qN_L = n + q N_L$, hence $b_q = 0$ for
   every $q\geq 1$. Thus $a'*z' = b_0$. Assume by contradiction that
   $b_0 \neq 0$. Then $|a'| = n+2n-|z'| \geq 2n$ and so $|a'|=2n$,
   hence $a'=P_f$ and $a=[M]$. This is impossible in view of our
   assumption that $[pt]*a = [M]t^{k/N_L}$. This proves that
   $a'*z'=0$.

   We now have: 
   \begin{equation} \label{eq:prod_bdry} a'\ast
      S_{1}-[L]t^{k/N_{L}}\in Im(\delta^{+}).
   \end{equation}
   From this equality we deduce $S_{1}\not=0$ and the statement of the
   Lemma. Indeed, if $S_{1}=0$, then $[L]$ is a boundary in
   $\mathcal{C}_{min}(L;\La)$ which implies that $L$ is narrow,
   contradicting our assumption.
\end{proof}

We continue with the proof of point i of Theorem~\ref{theo:geom_rig}.
In view of Lemma~\ref{lem:complL_unir} choose the {\em minimal} index
$0<i_0<\frac{k}{N_L}$ such that $z_{i_0} \neq 0$. We have $m_g*[L] =
z_{i_0} t^{i_0} + h.o.(t)$. ($h.o.(t)$ stands for higher order terms
in $t$.) We now have: $\phi(m_g*\psi([L])) = z_{i_0}t^{i_0} +
h.o.(t)$.  But $\psi([L]) = P_f$, hence $\phi(m_g*P_f) =
z_{i_0}t^{i_0} + h.o.(t)$.  Note that the classical term in $m_g*P_f$
vanishes and so $m_g*P_f = ut^l + h.o.(t)$ where $0 \neq u \in
\mathbb{Z}_2 \langle \textnormal{Crit}(f)\rangle$ and $l>0$. As
$\phi(m_g*P_f) = z_{i_0}t^{i_0} + h.o.(t)$ it follows that $l \leq
i_0$. By the definition of the moduli spaces giving the module action
(in~\S\ref{subsubsec:defin_alg_str}), this implies the claim at point i
of our theorem: for a generic $J$ there exists a non-constant $J$-disk
$v:(D, \partial D) \to (M,L)$ with $v(0) = m_g$ and such that
$\mu([v]) \leq l N_L \leq i_0 N_L \leq (\frac{k}{N_L}-1)N_L =
(k-N_L)$. In particular 
\begin{equation} \label{eq:width-i_0} 
   w(M \setminus L) \leq i_0 N_L \eta \leq (k-N_L) \eta.
\end{equation}
This completes the proof of point i of our theorem.

\ 

We now turn to the proof of the point ii of the theorem.  Recall that
$S_1 = \sum_{i=i_0}^{k/N_L -1} z_i t^i$ and that $1 \leq i_0 \leq
k/N_L - 1$. By assumption $L$ is wide so $\delta^{+}=0$, hence
by~\eqref{eq:prod_bdry} we get $a'*S_1 = [L]t^{k/N_L}$. Expending this
equality gives:
$$\sum_{i_0 \leq r+i \leq k/N_{L}}a'_{r}\ast z_{i}t^{r+i} = 
[L]t^{k/N_{L}},$$ where we have written $a'=\sum_{r \geq
  0}a'_{r}t^{r}$ (with $a'_{r}\in \mathbb{Z}_2 \langle \Crit(g)
\rangle$). The key remark is that:
\begin{equation}\label{eq:uni_L}
   \exists\ r\geq 0, \; i\geq i_0 \geq 1, \; \; \; 
   \textnormal{ such that } 
   (a'_{r}\ast 
   z_{i})t^{r+i}=[L]t^{k/N_L}+l.o.(t).
\end{equation}
This means that $\phi(a'_r * \psi(z_i))t^{r+i} = [L]t^{k/N_L} +
l.o.(t)$. Write $\psi(z_i) = \sum_{q \geq 0} x_{q}t^{q}$, where $x_{q}
\in \mathbb{Z}_2 \langle \textnormal{Crit}(f) \rangle$. It follows
that there exists $q$ such that $\phi(a'_r * x_{q})t^{q+r+i} =
[L]t^{k/N_L} + l.o.(t)$. Finally, writing $a'_r * x_{q} = \sum_{s \geq
  0} p_s t^s$ we deduce that there exists $s$ such that
$$\phi(p_s) t^{s+q+r+i} = [L] t^{k/N_L} + l.o.(t).$$
Put $\tau = \frac{k}{N_L} - (s+q+r+i)$. There are two main cases to be
considered: $s+q+r+i < k/N_L$ (i.e. $\tau>0$) and $s+q+r+i = k/N_L$
(i.e. $\tau = 0$). Before considering each case it is important to
note that as $i\geq i_0 \geq 1$ we always have $s, q, r <
\frac{k}{N_L}$.

\noindent \underline{{\textbf{Case 1:} $\mathbf{\tau>0}$.}} We have
$\phi(p_s) = t^{\tau}[L] + l.o.(t)$ with $\tau>0$ and we deduce from
Lemma~\ref{lem:ex_disks2} that there exists a critical point $w$ such
that $dw = P_f t^{\tau'} + l.o.(t)$ with $0<\tau' \leq \tau$.  It
follows that there exists a non-constant $J$-disk through $P_f$ with
Maslov index $\leq \tau' N_L \leq \tau N_L < k$, which proves the
desired uniruling property of $L$. In view of~\eqref{eq:width-i_0} we
also have $$w(L) + 2 w(M \setminus L) \leq 2\tau N_L \eta + 2i_0 N_L
\eta = 2(\frac{k}{N_L} - s - q - r - i +i_0)N_L \eta \leq 2k \eta.$$

\noindent \underline{{\textbf{Case 2:} $\mathbf{\tau=0}$.}} This means
that $\phi(p_s) = [L]$, hence $p_s=P_f$. Therefore 
\begin{equation} \label{eq:a'r*x_q}
   a'_r * x_q = P_f t^s + l.o.(t).
\end{equation}
There are again two cases: $s>0$ and $s=0$.

\noindent \underline{{\textbf{Case 2-i:} $\mathbf{\tau=0, s>0}$.}}  We
obtain from~\eqref{eq:a'r*x_q} that there exists a non-constant
$J$-disk through $P_f$ with Maslov index $\leq s N_L < k$.  As in
case~1 above we also have $$w(L)+2 w(M \setminus L) \leq 2s N_L \eta +
2i_0 N_L \leq 2(s+i)N_L \eta \leq 2k \eta.$$

\noindent \underline{{\textbf{Case 2-ii:} $\mathbf{\tau=0, s=0}$.}}
We will show now that this case is impossible.  To see this, first
note that by~\eqref{eq:a'r*x_q} we have that $a'_r*x_q = P_f$, hence
$a'_r = P_g$, $x_q = P_f$.  This implies that $a = [M]t^r + l.o.(t)$.
Write $a = [M]t^r + a_{r-1}t^{r-1} + \cdots + a_1 t + a_0$, where $a_j
\in H_*(M; \mathbb{Z}_2)$ are homogeneous elements. Recall that
$[pt]*a = [M]t^{k/N_L}$. Therefore $$[M]t^{k/N_L} = [pt]*a = [pt] t^r
+ [pt]*a_{r-1}t^{r-1} + \cdots [pt]*a_1 t + [pt]*a_0.$$ It follows
that there exists $0\leq j \leq r-1$ such that $([pt]*a_j)t^j =
[pt]t^r + h.o.(t)$, hence $([pt]*a_j) = [pt]t^{r-j} + h.o.(t)$.
Clearly this equality takes place in the image of the inclusion
$QH(M;\Gamma^+) \to QH(M; \Lambda^+)$ defined by $s \to t^{2C_M/N_L}$,
therefore we actually have in $QH(M; \Gamma^+)$:
\begin{equation} \label{eq:pt*a_j} 
   ([pt]*a_j) = [pt]s^{(r-j)N_L/2C_M} + h.o.(s).
\end{equation}
Note also that by the definition of $a_j$ we have $a_j \neq [M]$. We
will now show that such a relation is impossible in quantum homology.
To see this note that $r-j>0$ since $r-j=0$ would give $[pt]*a_j =
[pt]$ which is possible only if $a_j = [M]$ which is not the case. As
$r-j>0$, the relation~\eqref{eq:pt*a_j} implies that there exists a
homology class $A \in H_2^S(M)$ with $2c_1(A) = (r-j)N_L$ such that
$GW([pt], a_j, [M];A) \neq 0$. In particular, for generic $J$, the
moduli space of (simple) $J$-holomorphic rational curves
$u:\mathbb{C}P^1 \to M$ in the class $A$ which pass through a given
point in $M$ and intersect a cycle representing $a_j$ is not empty. To
estimate the dimension of this space denote by $\mathcal{M}(A,J)$ the
space of simple rational curves in the class $A$ and by $G =
Aut(\mathbb{C}P^1) \approx PSL(2,\mathbb{C})$ the group of
biholomorphisms of $\mathbb{C}P^1$. Consider the evaluation map $$ev:
\bigl(\mathcal{M}(A,J) \times \mathbb{C}P^1 \times \mathbb{C}P^1
\bigr)/G \to M \times M, \quad ev(u,z_1, z_2) = (u(z_1), u(z_2)).$$
The moduli space in question is $ev^{-1}(pt \times W^u_{a'_j})$,
where recall that $a'_j \in \mathbb{Z}_2 \langle \textnormal{Crit}(g)
\rangle$ is a Morse cycle representing $a_j$ and $W^u_{a'_j}$ stands
for the unstable submanifolds associated to the critical points in
$a'_j$.  By transversality we obtain the following dimension formula:
$$\dim ev^{-1}(pt \times W^u_{a'_j}) = 2n + 2c_1(A) + 2+2-6 + 
|a_j| - 4n = -2n+ (r-j)N_L - 2 + |a_j|.$$ On the other hand,
from~\eqref{eq:pt*a_j} we have $|a_j| = 2n-(r-j) N_L$. Putting this
into the dimension formula we get $\dim ev^{-1}(pt \times W^u_{a'j}) =
-2$, contradicting the fact that this space is not empty.  This rules
out Case 2-ii and concludes the proof of Theorem~\ref{theo:geom_rig}.

\Qed

\subsection{Proof of Theorem \ref{thm:action}} \label{subsec:proofT3}

We first recall the definition of the spectral invariants as well as
some other basic facts and we fix some conventions.

Consider a generic pair $(H,J)$ consisting of a $1$-periodic
Hamiltonian $H:M\times S^{1}\to \R$ and an almost complex structure
$J$ so that the Floer complex $CF_{\ast}(H,J)$ is well defined.
(Here, $CF(H,J)$ is the Floer complex for periodic orbits Floer
homology.) Let $I=\{\overline{\gamma}=(\gamma,\hat{\gamma})\}/\sim$
where $x$ is a contractible $1$-periodic orbit of the Hamiltonian flow
of $H$, $\hat{\gamma}:D \to M$ is a disk-capping of $\gamma$ (i.e.
$\hat{\gamma}|_{\partial D} = \gamma$) and the equivalence relation
$\sim$ is $\overline{\gamma}\sim\overline{\gamma}'$ if
$\gamma=\gamma'$ and $\omega(\hat{\gamma})=\omega(\hat{\gamma}')$.
Notice that $I$ is a $\Gamma$-module (we recall that
$\Gamma=\Z_{2}[s^{-1},s]$), the elements of $\Gamma$ acting by
changing the capping: $s \cdot (\gamma, \hat{\gamma}) = (\gamma,
\hat{\gamma_1})$, where $\omega(\hat{\gamma_1}) = \omega(\hat{\gamma})
- 2C_M \eta$. As $\Lambda$ is a $\Gamma$-module we will define the
Floer complex of interest here as: $CF(H,J;\Lambda)=\Z_{2}\langle
I\rangle \otimes_{\Gamma}\Lambda$ endowed with the usual Floer
differential.

Fix also a Morse function $f:L\to \R$ as well as a Riemannian metric
$\rho$ on $L$ so that the pearl complex $\mathcal{C}^{+}(L;f,\rho,J)$
is well defined.

We need to provide a Floer-theoretic description of our module
operation $\circledast$ which involves the two complexes above. This
is based on moduli spaces $\mathcal{P}'_{\mathcal{T}}$ similar to the
ones used in~\S\ref{subsubsec:defin_alg_str} c except that the vertex
of valence three in the string of pearls is now replaced by a
half-tube with boundary on $L$ and with the $-\infty$ end on an
element $\overline{\gamma} \in I$.  The symbol of the tree is
$(\overline{\gamma},x:y)$.  The total homotopy class $\lambda$ of the
configuration obtained in this way is computed by using the capping
associated to $\gamma$ to close the semi-tube to a disk and adding up
the homotopy class of this disk to the homotopy classes of the other
disks in the string of pearls.  More explicitly, a half tube as before
is a solution
$$u:(-\infty,0]\times S^{1}\to M$$
of Floer's equation
\begin{equation}\label{eq:Floer_semi_tube}\partial
u/\partial s +J\partial u/\partial t+\nabla H(u,t)=0
\end{equation}
with the boundary conditions $$u(\{0\}\times S^{1})\subset L \ \
\lim_{s\to-\infty}u(s,t)=\gamma(t)~.~$$ The marked points on the
``exceptional" vertex which corresponds to $u$ are so that the point
$u(0,1)$ is an exit point for a flow line and $u(0,-1)$ is the entry
point. Both compactification and bubbling analysis for these moduli
spaces are similar to what has been discussed before to which is added
the study of transversality and bubbling for the spaces of half-tubes
as described by Albers in~\cite{Alb:extrinisic}. As described
in~\cite{Alb:extrinisic}, an additional assumption is needed for this
part: $H$ is assumed to be such that no periodic orbit of $X^H$ is
completely included in $L$.

\begin{figure}[htbp]
   \psfig{file=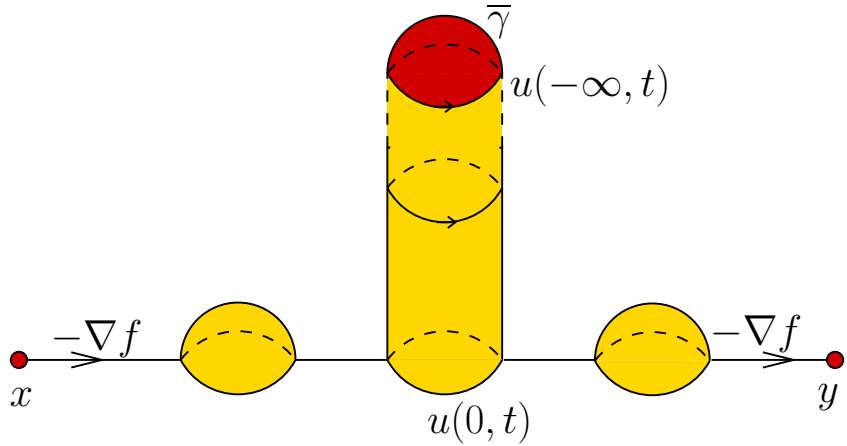, width=0.7 \linewidth}
   \caption{An element $v \in \mathcal{P}'_{\mathcal{T}}$.}
   \label{f:pearls-tube}
\end{figure}

Counting elements in these moduli spaces defines an operation:
$$\circledast_{F}: CF(H,J; \Lambda) \otimes_{\Lambda} 
\mathcal{C}(L;f,\rho,J)\to \mathcal{C}(L;f,\rho,J)~.~$$ Fix a
Morse-Smale pair $(g, \rho_M)$ on $M$ and let $C^{+}(g)$ be the
corresponding Morse complex tensored with $\La^{+}$. Recall the module
action defined in~\S\ref{subsubsec:defin_alg_str} c:
$$\circledast : C^{+}(g)\otimes_{\Lambda^+} 
\mathcal{C}^{+}(L;f,\rho,J)\to \mathcal{C}^{+}(L;f,\rho,J)~.~$$ There
are maps induced by the inclusion $\La^{+}\to \La$:
$$C^{+}(g)\to C^{+}(g)\otimes_{\La^{+}}\La =
C(g) \ \mathrm{and} \ \mathcal{C}^{+}(L;f,\rho,J)\to
\mathcal{C}(L;f,\rho,J)$$ which we will denote in both cases by $p$.

We will now use the Hamiltonian version of the
Piunikin-Salamon-Schwarz homomorphism~\cite{PSS}: $\widetilde{PSS}:
C(g) \to CF(H,J;\La)$. Standard arguments show that there is a chain
homotopy $\xi : C^{+}(g)\otimes_{\Lambda^+}
\mathcal{C}^{+}(L;f,\rho,J)\to \mathcal{C}(L;f,\rho,J)$ which
verifies:
\begin{equation}\label{eq:chtpy_PSS_mod}
   \widetilde{PSS}(p(x))\circledast_{F} p(y)-p(x\circledast y)=
   d\xi (x\otimes y) - \xi (d(x\otimes y)), 
\end{equation}
for every $x \in C^+(g)$, $y \in \mathcal{C}^+(L; f, \rho, J)$.

\

The Floer complex $CF_{\ast}(H,J)$ is filtered by the values of the
action functional
$$\mathcal{A}_{H}(\overline{\gamma})=\int_{S^{1}}H(\gamma(t),t)dt 
-\int_{D} \hat{\gamma}^{\ast}\omega$$ where
$\overline{\gamma}=(\gamma,\hat{\gamma})$, with $\gamma$ a
contractible $C^{\infty}$-loop in $M$ and $\hat{\gamma}$ a cap of this
loop.  This action is compatible with the action of $\Gamma$ and we
extend it on the generators of $CF(H,J;\Lambda) = \Z_{2}\langle
I\rangle \otimes_{\Gamma}\Lambda$ by:
$\mathcal{A}_{H}(\overline{\gamma}\otimes
t^{k})=\mathcal{A}_{H}(\overline{\gamma})-k\eta N_{L}$ (where $\eta$
is the monotonicity constant). The filtration of order $\nu\in\R$ of
the Floer complex, $CF^{\leq\nu}$, is the graded $\mathbb{Z}_2$-vector
space generated by all the elements $\overline{x} \otimes \lambda$ of
action at most $\nu$.

We emphasize that all the homology and co-homology classes to be
considered below are homogeneous. We now recall the definition of
spectral invariants following Schwarz~\cite{Sc:action-spectrum} and
Oh~\cite{Oh:spec-inv-1}. Fix $\alpha\in QH_*(M;\Lambda) =
(H(M;\mathbb{Z}_2)\otimes \Lambda)_*$ and define the spectral
invariant $\sigma(\alpha,H)$ of $\alpha$ by:
\begin{equation}\label{eq:spectral_hlgy}
   \sigma(\alpha,H) =
   \inf\{\nu : PSS (\alpha)\in
   \mathrm{Image}(\ H(CF^{\leq\nu})\to HF(H,J; \Lambda)\ )\},
\end{equation}
where $PSS:QH_*(M;\Lambda)\to HF_*(H,J;\Lambda)$ is the morphism
induced in homology by $\widetilde{PSS}$. Notice that by convention we
have $\sigma(0; H)=-\infty$.  Assuming that $H$ is normalized, it is
well known that $\sigma(\alpha,H)$ depends only on the class
$[\phi^H]\in \widetilde{Ham(M)}$ and on $\alpha$, and is therefore
denoted by $\sigma(\alpha, \phi^{H})$. We refer the reader
to~\cite{Oh:spec-inv-1, Oh:spec-inv-2, Oh:spec-inv-3, Oh:spec-inv-4,
  Sc:action-spectrum, McD-Sa:Jhol-2} for the foundations of the theory
of spectral invariants. See also~\cite{Vi:generating-functions-1} for
an earlier approach to the subject.

Let $L\subset M$ be a monotone Lagrangian submanifold.
Theorem~\ref{thm:action} is an immediate consequence of the first part
of Lemma~\ref{lem:action} below. To state it we fix some more
notation. As discussed before, the inclusion $\La^{+}\to \La$ induces
a map $p:\mathcal{C}^{+}(L;f,\rho,J)\to \mathcal{C}(L;f,\rho,J)$ which
is canonical in homology. We continue to denote the induced map in
homology by $p$ too. Denote by $IQ^{+}(L)$ the image of
$p:Q^{+}H(L)\to QH(L)$ and notice that $IQ^{+}(L)$ is a $\La^{+}$
module so that it makes sense to say whether a class $z\in IQ^{+}(L)$
is divisible by $t$ in $IQ^{+}(L)$: this means that there is some
$z'\in IQ^{+}(L)$ so that $z=tz'$.

\begin{lem} \label{lem:action}
   \begin{enumerate}[i.]
     \item Assume that $\alpha\in QH_*(M;\La^{+})$, $x,y\in
      Q^{+}H_*(L)$ are so that $p(y)$ is not divisible by $t$ in
      $IQ^{+}(L)$ and $\alpha*x=yt^{s}$.  Then we have the following
      inequality for every $\phi \in \widetilde{Ham}(M)$:
      $$\sigma(\alpha;\phi)\geq  
      \mathrm{depth}_{L}(\phi) - sN_{L}\eta ~.~$$
     \item Let $x\in Q^{+}H_*(L)$ and let $\phi\in\widetilde{Ham}(M)$.
      Then: $$\sigma(i_{L}(x);\phi)\leq \mathrm{height}_{L}(\phi)
      $$ where $i_{L}:QH(L)\to QH(M;\Lambda)$ is the quantum inclusion
      from Theorem~\ref{thm:alg_main} iii.
   \end{enumerate}
\end{lem}
The second point of the lemma is an extension of a result of Albers
\cite{Alb:extrinisic}.

Before proving Lemma~\ref{lem:action}, we show how it implies
Theorem~\ref{thm:action}. Indeed, if $L\subset M$ is not narrow, then
$[L]\in QH(L)$ is not trivial and we have $[M]\ast [L]=[L]$ which
implies the first point of Theorem~\ref{thm:action} because, for
degree reasons, $p([L])$ is not divisible by $t$ in $IQ^{+}(L)$.
Moreover, if $M$ is point invertible of order $k$, then there is $a\in
QH(M;\La^{+})$ so that $[pt]\ast a= [M]t^{k/N_{L}}$.  Therefore,
setting $a'=a\ast [L]\in Q^{+}H(L)$ we get $[pt]\ast
a'=[L]t^{k/N_{L}}$ and by applying the lemma for $\alpha=[pt]$,
$x=a'$, $y=[L]$ we deduce Theorem~\ref{thm:action} ii. \Qed

\begin{proof}[Proof of Lemma~\ref{lem:action}.]
   i. We fix $\phi\in\widetilde{Ham(M)}$.  By inspecting the
   definition of \textit{depth} in \S \ref{subsubsec:spec-rig} we see
   that the inequality we need to prove is reduced to showing that for
   every normalized Hamiltonian $H$ with $[H]=\phi$ there exists a
   loop $\gamma: S^1 \to L$ such that:
   \begin{equation}\label{eq:loops_tubes}
      \sigma(\alpha, \phi) - 
      \int_{S^{1}}H(\gamma(t),t)dt+s\eta N_{L} \geq 0.
   \end{equation}
   By a small perturbation of $H$ we may assume that no closed orbit
   of $H$ is contained in $L$.

   Given any $\epsilon>0$, in view of the definition of
   $\sigma(\alpha, H)$, we may find in
   $CF^{\leq\sigma(\alpha,H)+\epsilon}$ a cycle $\zeta$ with
   $[\zeta]=PSS(\alpha)\in HF(H,J;\Lambda)$.  Write $\zeta=\sum
   \overline{\gamma_{i}}\otimes t^{{k}_{i}}$ where
   $\overline{\gamma_{i}}$ are generators of $CF(H,J)$ and
   $t^{k_{i}}\in\Lambda$, $k_{i}\in \Z$.  Represent also $x$ as a
   cycle in $\mathcal{C}^{+}(L;f,\rho,J)$, $x=[x']$ with $x'=\sum_{i
     \geq 0} x_{i}t^{i}$, $x_{i}\in \Z_{2}\langle\Crit(f)\rangle$.
   Similarly, represent also $y$ by a cycle $y'$ in
   $\mathcal{C}^{+}(L;f,\rho,J)$.  From
   equation~\eqref{eq:chtpy_PSS_mod} we deduce that
   $\zeta\circledast_{F} x'-y't^{s}\in Im(d)$, where $d$ is the
   differential in $\mathcal{C}(L;f,\rho,J)$.  Write
   $\zeta\circledast_{F} x'=\sum_{i} z_{i}t^{i}$ with $i\in \Z$,
   $z_{i}\in\Z_{2}\langle\Crit(f)\rangle$ (note that here we cannot
   assume anymore that $i \geq 0$ only). The fact that $y'$ is not
   divisible by $t$ in $IQ^{+}(L)$ implies that there is some
   $z_{r}\not=0$ with $r\leq s$. But this means that there are
   $\gamma_{i}$ and $x_{j}$ so that $(\overline{\gamma_{i}} \otimes
   t^{k_{i}})\circledast_F x_{j}t^{j}=z_{r}t^{r}+\ldots$ (where $\ldots$
   stands for other terms).  This means that there are critical points
   $x_j^0, z_r^0 \in \Crit(f)$ (participating in $x_j$ and $z_r$) so
   that the moduli space $\mathcal{P}'_{\mathcal{T}}$ (described at
   the beginning of the section), of symbol
   $(\overline{\gamma_{i}},x_j^0:z_r^0)$ and with $\mu(\mathcal{T}) =
   (r-j-k_{i})N_{L}$, is not void.  We now consider an element
   $v\in\mathcal{P}'_{\mathcal{T}}$ and we focus on the corresponding
   half-tube $u$ (which is part of $v$). The usual energy estimate for
   this half-tube gives:

   $$
   0\leq\int_{-\infty}^{0}\int_{S^{1}}||\partial u/\partial
   s||^{2}dtds= \int_{(-\infty,0]\times S^{1}}u^{\ast}\omega +
   \int_{S^{1}}H(\gamma_{i}(t),t)dt -\int_{S^{1}}H(u(0,t),t)dt~,~
   $$
   hence: $$\int_{S^{1}}H(u(0,t),t)dt\leq \int_{(-\infty,0]\times
     S^{1}} u^{\ast} \omega +\int_{S^{1}} H(\gamma_{i},t)dt.$$ We now
   claim that:
   \begin{equation}\label{eq:energy_est}
      \mathcal{A}_{H}(\gamma_{i})+(r-j-k_{i})\eta N_{L}\geq 
      \int_{(-\infty,0]\times
        S^{1}}u^{\ast}\omega+\int_{S^{1}}H(\gamma_{i}(t),t)dt~.~
   \end{equation}
   Indeed, $\eta \mu(\mathcal{T})$ equals the symplectic area of all
   the disks in $v$ $+$ the area of the tube $u$ $+$ the area of the
   cap $\hat{\gamma_i}$ corresponding to $\overline{\gamma_{i}}$. The
   inequality~\eqref{eq:energy_est} now follows because the disks in
   $v$ are $J$-holomorphic hence their area is non-negative. But now
   $\sigma(\alpha,H)+\epsilon\geq
   \mathcal{A}_{H}(\overline{\gamma_{i}} \otimes
   t^{k_{i}})=\mathcal{A}_{H}(\overline{\gamma_{i}})-k_{i}\eta N_{L}$
   and as $j \geq 0$, $s\geq r$ we obtain:
   $$\sigma(\alpha,H)+\epsilon +s\eta N_{L}\geq \int_{(-\infty,0]\times
     S^{1}}u^{\ast}\omega+\int_{S^{1}}H(\gamma_{i}(t),t)dt$$ so that
   by taking $\gamma(t)=u(0,t)$ we deduce
   inequality~\eqref{eq:loops_tubes}.

   \

   ii. Given a Hamiltonian $H$ with $\phi=\phi^{H}$, a Morse function
   $f$, a generic metric $\rho$ and a generic almost complex structure
   $J$ we will define a chain map
   $$\tilde{i}_{L}:\mathcal{C}(L;f,\rho,J)\to CF(H,J; \Lambda)$$
   so that the maps induced in homology by $\widetilde{PSS} \circ i_L$
   and by $\tilde{i_L}$ are equal. To describe this map, fix a
   particular capping $\hat{\gamma}'$ for each contractible
   $1$-periodic, orbit $\gamma$ of the Hamiltonian vector field
   $X^{H}$ of $H$. We denote these pairs by $\tilde{\gamma} =
   (\gamma,\hat{\gamma}')$.

   For a critical point $p\in \Crit(f)$ we define
   \begin{equation} \label{eq:tilde-i_L}
      \tilde{i}_{L}(p)=\sum_{\mathcal{T},\gamma} \#_{2}
      (\mathcal{P}''_{\mathcal{T}})\ \tilde{\gamma} \otimes
      t^{\mu(\mathcal{T})/N_{L}},
   \end{equation}
   where the moduli spaces $\mathcal{P}''_{\mathcal{T}}$ are similar
   to the ones used in~\S\ref{subsubsec:defin_alg_str} d except that
   the last (exceptional) vertex there as well as its exiting edge are
   replaced here by a Floer semi-tube; the Maslov index
   $\mu(\mathcal{T})$ is the sum of the Maslov indices of the disks in
   the chain of pearls summed with the Maslov index of the tube glued
   to the disk $\hat{\gamma}'$ with reversed orientation.  More
   precisely, the moduli spaces $\mathcal{P}''_{\mathcal{T}}$ used
   here correspond to trees $\mathcal{T}$ of symbol $(p:
   \tilde{\gamma})$. An element $v\in \mathcal{P}''_{\mathcal{T}}$
   consists of a pair $(u',u'')$ where $u''$ is a Floer semi-tube
   $$u'':[0,\infty)\times S^{1}\to M$$
   verifying Floer's equation~\eqref{eq:Floer_semi_tube} with the
   boundary conditions $$u''(\{0\}\times S^{1})\subset L \ , \
   \lim_{s\to\infty}u''(s,t)=\gamma(t)$$ and $u'$ is a string of
   pearls $u'=(u_{1},\ldots, u_{k})$ in $M$ associated to $f$,
   starting at the critical point $p\in \Crit(f)$ and so that the last
   incidence condition is
   $$\exists t>0,\ \gamma_{t}^{f}(u_{k}(1))=u''(0, -1)~.~$$
   In other words, $u'$ is an element as of a moduli space as those
   considered in the construction of the pearl differential
   \S\ref{subsubsec:defin_alg_str} a except that the end point is not
   $\in \Crit(f)$ but $u''(0,-1)$.  The Maslov index is given by
   $\mu(\mathcal{T})=\mu(u')+\mu(u''\# (\hat{\gamma}')^{-1})$ where
   $(\hat{\gamma}')^{-1}$ is the disk with the opposed orientation
   compared to $\hat{\gamma}'$, and $u''\# (\hat{\gamma}')^{-1}$
   indicates the surface obtained by gluing the tube $u''$ and the
   capping disk $(\hat{\gamma}')^{-1}$ along $\gamma$. The sum
   in~\eqref{eq:tilde-i_L} is taken over all $(\mathcal{T}, \gamma)$
   such that $|p|-\mu(\tilde{\gamma})+\mu(\mathcal{T})=0$.  It is easy
   to see that the definition of $\tilde{i}_L$ does not depend on the
   specific choice of the cappings $\hat{\gamma}'$ associated to each
   $\gamma$.

   The regularity issues for the moduli spaces
   $\mathcal{P}''_{\mathcal{T}}$ are similar to those discussed
   before. Finally, standard arguments show that by extending this
   definition by linearity over $\Lambda$ we obtain a chain map and
   that, the map induced in homology by $\tilde{i}_L$ coincides with
   $PSS \circ i_{L}$.

   The next step is to establish an action estimate for the
   configurations $v=(u',u'') \in \mathcal{P}''_{\mathcal{T}}$
   considered above. We recall that if $\tilde{\gamma}$ is a capped
   orbit as above, the element $\tilde{\gamma}\otimes t^{k}$ is a
   generator of $CF(H,J;\Lambda)$ and its action is
   $\mathcal{A}_{H}(\tilde{\gamma})-k\eta N_{L}$.  The energy estimate
   associated to $u''$ gives:
   $$0\leq\int_{S^{1}}H(u''(0,t),t)dt - \int_{S_{1}} H(\gamma(t),t) dt+
   \int_{[0,\infty)\times S^{1}}(u'')^{\ast}\omega,$$ and so
   $$\mathcal{A}_{H}(\tilde{\gamma})= -
   \int_{D}(\hat{\gamma}')^{\ast}\omega+\int_{S^{1}}H(\gamma(t),t) dt
   \leq \int_{S^{1}}H(u''(0,t),t)dt +
   \omega([u''\#(\hat{\gamma}')^{-1}])~.~$$ Clearly, $\omega([u''\#
   (\hat{\gamma}')^{-1}])= \mu(\mathcal{T})\eta-\omega([u'])$ and as
   $\omega([u'])\geq 0$ we deduce
   \begin{equation} \label{eq:A_H-i_L}
      \mathcal{A}_{H}(\tilde{\gamma}\otimes t^{\mu(\mathcal{T})/N_L} \leq
      \mathrm{height}_{L}(\phi)~.~
   \end{equation}

   Let $x\in Q^{+}H(L)$, and let $x'=\sum_{i \geq 0}
   x_{i}t^{i}\in\mathcal{C}^{+}(L;f,\rho,J)$, $x_{i}\in \mathbb{Z}_2
   \langle \Crit(f)\rangle$ be a pearl cycle that represents $x$.
   Write $\tilde{i}_{L}(x')=\sum_{i}\tilde{\gamma}_{i} \otimes
   t^{k_{i}}$. Consider {\em any} of the terms in this sum, say
   $\tilde{\gamma}_{j} \otimes t^{k_{j}}$.  There exists $r \geq 0$
   and a critical point $x_r^0$ participating in $x_r$, so that
   $\tilde{i}_{L}(x^0_{r}t^{r})$ contains $\tilde{\gamma}_{j} \otimes
   t^{k_{j}}$.  As $\tilde{i}_{L}$ is $\Lambda$-linear this means that
   $\tilde{i}_{L}(x^0_{r})$ contains $\tilde{\gamma}_{j} \otimes
   t^{k_{j}-r}$. From~\eqref{eq:A_H-i_L} we now obtain 
   $$\mathcal{A}(\tilde{\gamma}_{j}t^{k_{j}}) =
   \mathcal{A}(\tilde{\gamma}_{j}t^{(k_{j}-r)}) - rN_{L}\eta\leq
   \mathrm{height}_{L}(\phi)-rN_{L}\eta \leq
   \mathrm{height}_{L}(\phi).$$ Finally, since $\tilde{i}_L$ and $PSS
   \circ i_L$ coincide in homology we can represent $PSS (i_L(x))$ as
   a linear combination of of generators of $CF(H,J; \Lambda)$ each of
   action at most $\mathrm{height}_{L}(\phi)$ which implies our claim.
\end{proof}

\begin{rem} \label{r:lem-action}
   \begin{enumerate}[a.]
     \item Sometimes the point ii of Lemma~\ref{lem:action} can be
      used to estimate from above spectral invariants of homology
      classes $\alpha\in H_{\ast}(M)$. For example, it is easy to see
      that $i_{L}([L])=\textnormal{inc}([L])$, where
      $\textnormal{inc}_*: H_*(L;\mathbb{Z}_2) \to
      H_*(M;\mathbb{Z}_2)$ is the map induced by the inclusion $L
      \subset M$. Therefore whenever $\textnormal{inc}_*([L]) \neq 0$
      we obtain $\sigma(\textnormal{inc}_*([L]),\phi)\leq
      \mathrm{height}_{L}(\phi)$ for any $\phi\in\widetilde{Ham}(M)$.
      \item In a point invertible manifold the first part of
      Lemma~\ref{lem:action} provides an estimate from below of
      $\sigma([pt],\phi)$ and so, in view of the proofs of the
      intersection results discussed in
      Corollaries~\ref{cor:inters_i_j} and~\ref{cor:inters}, it is
      particularly important to get also an estimate from the above.
      The natural idea is to write $[pt]=i_{L}(x)$ for some class $x$.
      However, there are cases when $[pt]$ is not in the image of this
      map $i_{L}$ - see for example the case of the quadric $Q^{2n}$
      described in \S\ref{Sb:quad-lag}.
     \item In case $\Lambda=\Gamma$ we have $CF(H,J; \Lambda)=CF(H,J)
      = \Z_{2}\langle I\rangle$, where $I$ is the set of contractible
      $1$-periodic orbits of $X^H$ together with all possible cappings
      (modulo the usual identifications) $I = \{\overline{\gamma} =
      (\gamma, \hat{\gamma}) \} / \sim$. In this case the map
      $\tilde{i}_{L}$ can be written as $\tilde{i}_{L}(p) =
      \sum_{\mathcal{T}, \overline{\gamma}}
      \#_{2}(\mathcal{P}'''_{\mathcal{T}})\overline{\gamma}$ where the
      moduli space $\mathcal{P}'''_{\mathcal{T}}$ contains
      configurations as those in $\mathcal{P}''_{\mathcal{T}}$ but
      with the additional condition that $\mu(\mathcal{T})=0$. Indeed,
      as $\Lambda=\Gamma$ any element $\tilde{\gamma}\otimes t^{k}$
      can be written uniquely as some $\overline{\gamma}$.

      If additionally, we have $N_{L}>n+1$, then a dimension count
      shows that for the configurations $v=(u',u'')$ used to define
      $\tilde{i}_{L}$ we have $\mu(u')=0$ and so there are no
      $J$-disks present in the definition of $\tilde{i}_{L}$. Under
      these assumptions $\tilde{i}_L$ coincides with a map introduced
      by Albers in~\cite{Alb:extrinisic}.
     \item It is possible to define a pseudo-valuation $\nu: QH(L)\to
      \Z\cup \{\infty\}$ as follows. Notice first that for any $a\in
      QH(L)$ there exists $k\in \mathbb{Z}$ so that $t^{k}a\in
      IQ^{+}(L)$.  Define
      $$\nu(a) = \max \{s \in \mathbb{Z} \mid 
      t^{-s} a \in IQ^+(L) \} \in \mathbb{Z} \cup \{\infty\}.$$ It is
      easy to see that $\nu$ is well defined, and that it verifies
      $\nu(a) \geq 0$ iff $a \in IQ^+(L)$, $\nu(a)=\infty$ iff $a=0$,
      $\nu(a+b)\geq \min\{\nu(a),\nu(b)\}$, $\nu(a\ast b)\geq
      \nu(a)+\nu(b)$, and $\nu(ta)=\nu(a)+1$.  A similar function to
      $\nu$ has already been considered by Entov-Polterovich
      \cite{En-Po:rigid-subsets} in the context of ambient quantum
      homology.  The inequality at point~i of Lemma~\ref{lem:action}
      can now be reformulated as:
      $$\sigma(\alpha,\phi)-\mathrm{depth}_{L}(\phi)\geq 
      (\nu(x)-\nu(\alpha\ast x))N_{L}\eta, \quad \forall \, \alpha\in
      QH(M;\La), \; x\in QH(L).$$
   \end{enumerate}
\end{rem}

\subsection{Proof of Theorem \ref{cor:inters_i_j}}
\label{sb:jcirci}

Recall the setting of this theorem.  Given $L_0, L_1 \subset M$,
monotone Lagrangian submanifolds we have the two associated rings
$\Lambda_0 = \mathbb{Z}_2[t_0^{-1}, t_0]$, $\Lambda_{1} =
\mathbb{Z}_2[t_1^{-1}, t_1]$ be the associated rings, graded by $\deg
t_0 = -N_{L_0}$ and $\deg t_1 = N_{L_1}$ as well as the ring
$\Lambda_{0,1}=\Lambda_0 \otimes_{\Gamma} \Lambda_1$ where
$\Gamma=\Z_{2}[s^{-1},s]$, $|s|=-2C_{min}$.  Recall also the two
canonical maps: the quantum inclusion $i_{L_0}:
QH_*(L_0;\Lambda_{0,1}) \to QH_*(M;\Lambda_{0,1})$ and $j_{L_1}:
QH_*(M;\Lambda_{0,1}) \to QH_{*-n}(L_1;\Lambda_{0,1})$, defined by
$j_{L_1}(a) = a * [L_1]$. The claim of the theorem is that if the
composition:
$$j_{L_1} \circ i_{L_0} :
QH_*(L_0;\Lambda_{0,1}) \longrightarrow QH_{*-n}(L_1;\Lambda_{0,1}).$$
does not vanish, then $L_{0}$ and $L_{1}$ intersect.

We start the proof with a little more preparation.  First note that
since $\Lambda_{0,1}$ is a $\Gamma$-module we can naturally extend the
definition of periodic orbit Floer homology to coefficients in
$\Lambda_{0,1}$ as the homology of the complex $CF(H,J;\Lambda_{0,1})
= CF(H,J) \otimes_{\Gamma} \Lambda_{0,1}$. We denote this homology by
$HF(H,J;\Lambda_{0,1})$. Moreover, the PSS isomorphism naturally
extends to this case and we get an isomorphism $PSS:
HF_*(H,J;\Lambda_{0,1}) \to QH_*(M; \Lambda_{0,1})$. Similarly, we can
extend the action functional to the generators of
$CF(H,J;\Lambda_{0,1})$ by defining: $\mathcal{A}_{H}(\overline{x}
\otimes t_0^{k_0} \otimes t_1^{k_1}) = \mathcal{A}_{H}(\overline{x}) -
k_0 \eta_0 N_{L_0} - k_1 \eta_1 N_{L_1}$. Here $\eta_i =
(\omega/\mu)|_{H_2^D(M,L_i)}$, $i=0,1$, are the monotonicity constants
of the Lagrangians. (Clearly, $\eta_0 = \eta_1$, unless
$\omega|_{\pi_2(M)} = 0$ in which case we anyway have $C_M = \infty$,
$\Gamma = \mathbb{Z}_2$ hence $CF(H,J; \Lambda_{0,1}) = CF(H,J)
\otimes \Lambda_0 \otimes \Lambda_1$.) It is easy to see that this
extension of the action is well defined.  With these conventions we
have as before a filtration on $HF(H,J;\Lambda_{0,1})$ by action and
we can define spectral numbers $\sigma_{\Lambda_{0,1}}(\alpha, \phi)$
for every $\alpha \in QH(M; \Lambda_{0,1})$,
$\phi\in\widetilde{Ham}(M)$, in a standard way.  A straightforward
algebraic argument shows that for classes $\alpha \in QH(M) \subset
QH(M; \Lambda_{0,1})$ (as well as $\alpha \in QH(M; \Lambda_i)$,
$i=0,1$) these ``new'' spectral numbers coincide with the usual ones,
i.e.  $\sigma_{\Lambda_{0,1}}(\alpha, \phi) = \sigma(\alpha,\phi)$.
(The point is that $\Lambda_{0,1}$ is a free module over $\Gamma$.)
We will also need the ring $\Lambda^{+}_{0,1} = \Lambda^+_0
\otimes_{\Gamma^+} \Lambda^+_1$. As before we have $$\Lambda^{+}_{0,1}
\cong \mathbb{Z}_2[t_0, t_1] / \{ t_0^{2C_M/N_{L_0}} =
t_1^{2C_M/N_{L_1}}\}~.~$$

Next we remark that Lemma~\ref{lem:action} continues to hold if we
replace $L$ by one of the $L_i$'s, say $L_0$, replace $\Lambda$ by
$\Lambda_{0,1}$, $\Lambda^+$ by $\Lambda^+_{0,1}$ and the condition
that ``$p(y)$ is not divisible by $t$ in $IQ^{+}(L)$'' by ``$p(y)$ is
not divisible by $t_{0}$ in the image of the map $p$'' with $p$ the
canonical ``change of coefficients'' map $p:QH(L_{0};\La^{+}_{0,1})\to
QH(L_{0};\La_{0,1})$. The proof of the lemma carries out to this case
without any essential modifications.

Since $j_{L_1} \circ i_{L_0} \neq 0$ there exists $x\in
QH(L;\Lambda^+_{0,1})$ so that $j_{L_1}\circ i_{L_0}(x)\not=0$.  From
the modified version of Lemma~\ref{lem:action} discussed above, we
deduce that for some constant $K$ depending only on $j_{L_1}\circ
i_{L_0}(x)$ and for any $\phi\in\widetilde{Ham}(M)$ we have
$$\mathrm{depth}_{L_1}(\phi)-K\leq\sigma(i_{L_0}(x)) \leq
\mathrm{height}_{L_0}(\phi)~.~$$

Now assume by contradiction that $L_0\cap L_1=\emptyset$. Pick a
normalized Hamiltonian $H$ which is constant equal to $C_0$ on $L_0$
and constant equal to $C_1$ on $L_1$ with $C_1>C_0+K$. This
immediately leads to a contradiction and concludes the proof of
Theorem \ref{cor:inters_i_j} \qed

\

We now pass to the proof of Corollary~\ref{cor:inter_Mas_Chern}.  Put
$L_1 = L$ and let $L_0 \subset M$ be a non-narrow monotone Lagrangian.
The claim follows if we show that if $[pt]\ast [L_{1}]$ is not
divisible by $t^{2C_{M}/N_{L_{1}}}$ in $IQ^{+}(L_{1})$, then
$j_{L_{1}}\circ i_{L_{0}}\not=0$. We first fix a Morse function
$f_{0}:L_{0}\to \R$ and a metric $\rho_{0}$ on $L_{0}$ as well as an
almost complex structure $J$ on $M$ so that the pearl complex
$\mathcal{C}(L_0;\La_{0}; f_{0},\rho_{0},J)$ is defined.  We assume
that $f_{0}$ has a unique minimum $m_{0}$.  To simplify the notation,
we put $c_{i}=2C_{M}/N_{L_{i}}$.

By the non-degeneracy part in Proposition \ref{p:duality} there exists
a class $\alpha \in QH_{0}(L_{0};\La_{0})$ which is non-zero and is
represented by a pearl cycle of the form $m_{0}+\sum_{i>0} x_{i}
t_{0}^{i}$ with $x_{i}\in\Crit(f_{0})$. A priori this cycle belongs to
$\mathcal{C}(L_0; f_0, \rho_0, J)$, but as $|m_0|=0$ and $|t_0|<0$ all
the the powers of $t_0$ in this cycle must be non-negative. Thus, in
fact this cycle is in $\mathcal{C}^{+}(L_0; f_0, \rho_0, J)$ and
$\alpha \in IQ^+(L_0)$. In view of the coefficients extension
morphisms $QH(L_0; \Lambda^+_0) \to QH(L; \Lambda^+_{0,1}) \to QH(L;
\Lambda_{0,1})$ we will view from now on $\alpha$ as an element of the
image of these maps i.e. $ \alpha \in IQ^{+}(L_0; \Lambda_{0,1})
\subset QH(L_0; \Lambda_{0,1})$. Here we have used again the ring
$\Lambda^+_{0,1} = \Lambda^+_0 \otimes_{\Gamma^+} \Lambda^+_1 \cong
\mathbb{Z}_2[t_0, t_1] / \{ t_0^{c_0} = t_1^{c_1} \}$ and the
coefficients extension morphisms induced by the obvious inclusions
$\Lambda^+_0 \to \Lambda^{+}_{0,1} \to \Lambda_{0,1}$.

As $i_{L_{0}}$ extends (at the chain level) the inclusion in singular
homology we can write $i_{L_{0}}(\alpha)=[pt]+\sum_{j>0}
a_{j}t_{0}^{j}$ with $a_{j}\in H_{\ast}(M;\Z_{2})$. Notice that
$QH(L_{1};\La_{0,1})=QH(L_{1};\La_{1})\otimes_{\Gamma} \La_{0}$
because $\mathcal{C}(L_{1};\La_{0,1};f_{1},\rho_{1},J)=
(\mathcal{C}(L_{1};\La_{1};f_{1},\rho_{1},J)\otimes _{\Gamma}\La_{0},
d_{\La_{1}}\otimes id)$ and $\La_{0,1}$ is a free $\Gamma$-module.
Taking this into account, we now apply $j_{L_{1}}$ to
$i_{L_{0}}(\alpha)$ and we obtain:
\begin{equation} \label{eq:identity_la01}
   (j_{L_{1}}\circ i_{L_{0}})(\alpha)= y+ \sum_{j>0} y_{j}t_{0}^{j},
\end{equation}
where we have denoted $y=[pt]\ast [L_{1}]\in QH(L_{1};\La_{1})\otimes
1$ and $y_j = a_{j}\ast [L_{1}] \in QH(L_{1};\La_{1})\otimes 1$. It is
important to notice that in fact $y, y_{j}\in IQ^{+}(L_{1}) \otimes 1
\subset IQ^+(L_1; \Lambda_{0,1}) \subset QH(L_1; \Lambda_{0,1})$. Now
suppose by contradiction that $j_{L_{1}}\circ i_{L_{0}}(\alpha)=0$.
As $y \in IQ^{+}(L_{1}) \otimes 1$ identity~\eqref{eq:identity_la01}
implies that the second term on its right-hand side belongs to
$IQ^{+}(L_{1}) \otimes 1$. This can only happen if for every $j$ with
$y_j \neq 0$ we have $c_0 | j$, so that $t_0^j = (t_1^{c_1})^{j/c_0}$.
It now follows that $y$ is divisible by $t_1^{c_1}$, and obviously
this divisibility property continues to hold also in $IQ^+(L_1)$. A
contradiction. \Qed

\section{Various examples and computations.}\label{sec:exa_comp}
The first three subsections below contain the proofs of the
computational theorems in~\S\ref{Sbs:examples-comput} and of their
corollaries from~\S\ref{subsubsec:methods}.  The last subsection
contains the justification of Example~\ref{ex:narrow}.

\subsection{Lagrangians in $\C P^{n}$ with $2H_{1}(L;\Z)=0$.}
Here we prove Theorem \ref{T:qstruct-2H_1=0} and its Corollary
\ref{cor:RP}.

We recall our notation: we denote by $h = [{\mathbb{C}}P^{n-1}] \in
H_{2n-2}({\mathbb{C}}P^n;\mathbb{Z}_2)$ the class of a hyperplane so
that in the quantum homology $QH({\mathbb{C}}P^n)$ we have:
$$h^{*j} =
\begin{cases}
   h^{\cap j}, & 0 \leq j \leq n \\ [{\mathbb{C}}P^n] s, & j=n+1
\end{cases}
~.~$$

We will use quantum homology with coefficients in
$\Lambda=\Z_{2}[t^{-1},t]$ and so we recall that $QH(\C
P^{n};\La)=QH(\C P^{n})\otimes_{\Gamma} \Lambda$, where $\Gamma =
\mathbb{Z}_2[s^{-1}, s]$, $\deg s = -(2n+2)$, and $\Lambda$ becomes a
$\Gamma$-module by $s\to t^{(2n+2)/N_{L}}$.  Obviously, $h$ is
invertible in $QH(\C P^n)$ so that the existence of the module action
claimed in Theorem \ref{thm:alg_main} directly implies the first part
of:

\begin{lem} \label{T:Lag-CPn} Let $L \subset {\mathbb{C}}P^n$ be a
   monotone Lagrangian with $N_L \geq 2$. Then $QH_*(L)$ is
   $2$-periodic, i.e. $QH_i(L) \cong QH_{i-2}(L)$ for every $i \in
   \mathbb{Z}$ and the homomorphism $QH_i(L) \to QH_{i-2}(L)$ given by
   $\alpha \mapsto h*\alpha$ is an isomorphism for every $i \in
   \mathbb{Z}$.  Moreover, $H_{1}(L;\Z)\not=0$.
\end{lem}

\begin{proof} The only part that still needs to be justified is that
   $H_{1}(L;\Z)\not=0$.  But if $H_{1}(L;\Z)=0$, then $N_{L} = 2
   C_{{\mathbb{C}}P^n} = 2n+2$ and by
   Theorem~\ref{thm:floer=0-or-all}~i we deduce that $L$ is wide (take
   $l = n$ in that theorem). The first part of the lemma implies in
   this case that $QH_*(L) \cong (H(L;\mathbb{Z}_2) \otimes
   \Lambda)_*$ is $2$-periodic which is impossible by degree reasons.
   Indeed, $(H(L;\mathbb{Z}_2) \otimes \Lambda)_n \neq 0$ but as $|t|
   = -2n-2$ we have $(H(L;\mathbb{Z}_2) \otimes \Lambda)_{n+2} \cong
   H_{n+2}(L;\mathbb{Z}_2) = 0$.
\end{proof}

\begin{rem}
   The first part of Lemma~\ref{T:Lag-CPn} was proved before by Seidel
   using the theory of graded Lagrangian
   submanifolds~\cite{Se:graded}.  The $2$-periodicity
   in~\cite{Se:graded} follows from the fact that ${\mathbb{C}}P^n$
   admits a Hamiltonian circle action which induces a shift by $2$ on
   graded Lagrangian submanifolds. Note that this is compatible with
   our perspective since that $S^1$-action gives rise to an invertible
   element in $QH({\mathbb{C}}P^n)$ (the Seidel element~\cite{Se:pi1,
     McD-Sa:Jhol-2}) whose degree is exactly $2n$ minus the shift
   induced by the $S^1$-action. In our case the Seidel element turns
   out to be $h$.
\end{rem}

We now focus on our main object of interest in the subsection.

\begin{lem}\label{lem:Maslov_RP}
   Let $L$ be a Lagrangian submanifold in $\C P^{n}$. If
   $2H_{1}(L;\Z)=0$ then $L$ is monotone, $N_{L}=n+1$, $L$ is wide and
   as a graded vector space we have $H_{\ast}(L;\Z_{2})\cong
   H_{\ast}(\R P^{n}; \Z_{2})$. Moreover, $QH_i(L) \cong \mathbb{Z}_2$
   for every $i \in \mathbb{Z}$.
\end{lem}

\begin{proof}
   Since $2H_1(L;\mathbb{Z})=0$ it is easy to see that $L$ is
   monotone.  Moreover, a simple computation shows that the minimal
   Maslov number of $L$ is $N_L = k(n+1)$ with $k \in\{ 1,2\}$. We
   already know from Lemma~\ref{T:Lag-CPn} that
   $H_{1}(L;\Z_{2})\not=0$ so that $H_{\ast}(L;\Z_{2})$ is generated
   as an algebra by $H_{\geq 1}(L;\Z_{2})$.  Thus, by
   Theorem~\ref{thm:floer=0-or-all}~i, $L$ is wide so that, again by
   Lemma~\ref{T:Lag-CPn}, we deduce that $(H(L;\Z_{2})\otimes \La)_*$
   is $2$-periodic. This $2$-periodicity implies (for degree reasons)
   that $N_L$ cannot be $2(n+1)$, hence $k=1$ and $N_L = n+1$.
   Moreover the $2$-periodicity implies that $H_{2i}(L;\mathbb{Z}_2)
   \cong H_0(L;\mathbb{Z}_2) = \mathbb{Z}_2$ for every $0 \leq 2i \leq
   n$.  Similarly we have: $H_1(L;\mathbb{Z}_2) \cong QH_1(L) \cong
   QH_1(L) t^{-1} = QH_{n+2}(L) \cong QH_n(L) \cong
   H_n(L;\mathbb{Z}_2) = \mathbb{Z}_2$. Applying the $2$-periodicity
   again we obtain $H_{2i+1}(L;\mathbb{Z}_2)\cong \mathbb{Z}_2$ for
   every $1 \leq 2i+1 \leq n$. Summing up we see that
   $H_j(L;\mathbb{Z}_2) \cong \mathbb{Z}_2 \cong
   H_j({\mathbb{R}}P^n;\mathbb{Z}_2)$ for every $0 \leq j \leq n$.

   As for the last statement regarding $QH_i(L)$, we have:
   $$QH_{2j}(L) \cong QH_0(L) \cong H_0(L;\mathbb{Z}_2)=\mathbb{Z}_2,
   \quad QH_{2j+1}(L) \cong QH_1(L) \cong H_1(L;\mathbb{Z}_2) \cong
   \mathbb{Z}_2.$$
\end{proof}

\begin{lem} \label{l:L-RPn} There is a map $\phi:L\to \R P^{n}$
   inducing an isomorphism in $\Z_{2}$-singular homology. In
   particular $H_{\ast}(L;\Z_{2})$ is isomorphic to $H_{\ast}(\R
   P^{n};\Z_{2})$ as an algebra. Moreover, the isomorphism
   $\phi_{\ast}$ identifies the classical external product
   $H_{\ast}(\C P^{n};\Z_{2})\otimes H_{\ast}(L;\Z_{2})\to
   H_{\ast}(L;\Z_{2})$ with the corresponding action for $\R P^{n}
   \subset {\mathbb{C}}P^n$.
\end{lem}
\begin{proof}
   Let $\alpha_{i}\in QH_{i}(L) \cong \mathbb{Z}_2$ be the generator.
   In view of the canonical isomorphism $QH_*(L) \cong
   (H(L;\mathbb{Z}_2) \otimes \Lambda)_*$ we have $H_j(L;\mathbb{Z}_2)
   \cong QH_j(L)$ for every $0 \leq j \leq n$. Therefore we will view
   $\alpha_j$, $0 \leq j \leq n$, also as elements of
   $H_j(L;\mathbb{Z}_2)$.

   We first claim that $\alpha_{n-1}\ast
   \alpha_{n-1}=\alpha_{n-1}\cdot\alpha_{n-1}=\alpha_{n-2}$ (where
   $-\cdot -$ is the classical intersection product). For degree
   reasons this is equivalent to
   $\alpha_{n-1}\cdot\alpha_{n-1}\not=0$. In turn, this is equivalent
   to showing that $\alpha^{1}\cup \alpha^{1}\not=0$ in
   $H^{2}(L;\Z_{2})$ where $\alpha^{1}\in H^{1}(L;\Z_{2})$ is the
   generator (and so is Poincar\'e dual to $\alpha_{n-1}$). From the
   fact that $H^{1}(L;\Z_{2})=\Z_{2}$ and $H_{1}(L;\Z)$ is $2$-torsion
   we obtain that the Bockstein homomorphism, $\beta :
   H^{1}(L;\Z_{2})\to H^{2}(L;\Z_{2})$, associated to the exact
   sequence $0\to \Z_{2}\to \Z_{4}\to \Z_{2}\to 0$ is not trivial. But
   $\beta=Sq^{1}$, the first Steenrod square, which in this degree
   coincides with the square cup-product, so that $\alpha^{1}\cup
   \alpha^{1}\not=0$. This proves that $\alpha_{n-1} * \alpha_{n-1} =
   \alpha_{n-1} \cdot \alpha_{n-1} = \alpha_{n-2}$.

   In view of the first part of Lemma \ref{T:Lag-CPn} we know that
   $h\ast \alpha_{i}=\alpha_{i-2}$ for all $i$.  As
   $\alpha_{n-1}\cdot\alpha_{n-1}=\alpha_{n-2}$ it follows that the
   $\Z_{2}$-singular homology of $L$ coincides as an algebra with that
   of $\R P^{n}$.  Let $\bar{\phi}: L\to \R P^{\infty}$ be the
   classifying map associated to $\alpha^{1}$. As $\dim(L)=n$ we
   deduce that $\bar{\phi}$ factors via a map $\phi:L\to \R P^{n}$ and
   as the induced map in cohomology $H^{1}(\phi):
   H^1({\mathbb{R}}P^n;\mathbb{Z}_2) \to H^1(L;\mathbb{Z}_2)$ is an
   isomorphism it follows that $\phi$ induces an isomorphism in
   homology in all degrees. Moreover, using the relation $h\ast
   \alpha_{i}=\alpha_{i-2}$ again, we deduce that the classical
   external product coincides with that for $\R P^{n}$.
\end{proof}

We now turn to the proof of Theorem~\ref{T:qstruct-2H_1=0}. Point ii
has already been proved (in the proof of Lemma~\ref{l:L-RPn}). Before
we go on, recall that we have denoted by $\alpha_i \in QH_i(L) \cong
\mathbb{Z}_2$ is the generator. Clearly we have $\alpha_{i-r(n+1)} =
\alpha_i t^r$ for every $i,r \in \mathbb{Z}$.

Another important fact we will need below is the following. By Theorem
\ref{thm:alg_main} the quantum inclusion $i_{L}:QH(L)\to QH(M;\La)$ is
determined by the module action and the augmentation $\epsilon_{L}$
via the formula
\begin{equation} \label{eq:PD-i_L}
   \langle PD(y),i_{L}(x) \rangle = \epsilon_{L}(y\ast x).
\end{equation}

We are now ready to prove points iii and iv of
Theorem~\ref{T:qstruct-2H_1=0}. Assume first that $n$ is even, $n =
2l$. Denote by $h_{2r} \in H_{2r}({\mathbb{C}}P^n;\mathbb{Z}_2)$ the
generator, so that $h_{2n-2} = h$ and $h_{2r} = h^{*(n-r)}$ for every
$0 \leq r \leq n$. Fix $0 \leq 2k \leq n$. For degree reasons we have
$i_L(\alpha_{2k}) = e h_{2k}$ for some $e \in \mathbb{Z}_2$.
Applying~\eqref{eq:PD-i_L} with $x = \alpha_{2k}$ and $y = h_{2n-2k}$
we obtain: $$ e = \epsilon_L(h_{2n-2k} * \alpha_{2k}) = \epsilon_L
(h^{* k} * \alpha_{2k}) = \epsilon_L(\alpha_0) = 1.$$ Now fix $1 \leq
2k+1 \leq n-1$. For degree reasons, $i_L(\alpha_{2k+1}) = f
h_{2k+n+2}t$ for some $f \in \mathbb{Z}_2$.
Applying~\eqref{eq:PD-i_L} with $x = \alpha_{2k+1}$, $y = h_{n-2k-2}$
we obtain $$ft = \epsilon_L(h_{n-2k-2}*\alpha_{2k+1}) =
\epsilon_L(h^{*(k+1+l)}* \alpha_{2k+1}) = \epsilon_L(\alpha_{-2l-1}) =
\epsilon_L(\alpha_0 t)=t,$$ hence $f = 1$.  This concludes the proof
for even $n$. The case $n=$ odd is very similar, so we omit the
details.

It remains to prove point i of Theorem~\ref{T:qstruct-2H_1=0}. For
this end, first notice that since $[L] = \alpha_n$ we have
$\alpha_{n-2} = h*[L]$. As both $[L] \in QH(L)$ and $h \in
QH({\mathbb{C}}P^n;\Lambda)$ are invertible (each in its respective
ring) it follows that $\alpha_{n-2}$ is invertible too. By the proof
of Lemma~\ref{l:L-RPn} we have $\alpha_{n-2} = \alpha_{n-1} *
\alpha_{n-1}$, hence $\alpha_{n-1}$ is invertible too. It follows that
$(\alpha_{n-1})^{*(n-i)} \neq 0 \in QH_i(L)$, hence $\alpha_i =
(\alpha_{n-1})^{*(n-i)}$. As this is true for every $i \in \mathbb{Z}$
the claim at point~i of Theorem~\ref{T:qstruct-2H_1=0} readily
follows. This concludes the proof of all the statements of
Theorem~\ref{T:qstruct-2H_1=0} \Qed

We now turn to proving Corollary~\ref{cor:RP}. We begin with point iv.
This follows easily from points iii and iv of
Theorem~\ref{T:qstruct-2H_1=0} by looking at the classical part of the
quantum inclusion $QH_*(L) \to QH_*({\mathbb{C}}P^n;\Lambda)$.  Point
iii follows in a similar way from the fact that $h*[L] =
\alpha_{n-2}$.

As point i and ii of Corollary~\ref{cor:RP} has already been proved it
now remains to prove points v, vi and vii of that corollary. We group
these in the next lemma.
\begin{lem} \label{l:uniruled-L-CPn} For a Lagrangian $L$ in $\C
   P^{n}$ with $2 H_{1}(L;\Z)=0$ we have:
   \begin{itemize}
     \item[i.] $(\C P^{n},L)$ is $(1,0)$-uniruled of order $n+1$.
     \item[ii.]  $L$ is $2$-uniruled of order $n+1$. Moreover, given
      two distinct points $x,y\in L$, for a generic $J$ there is an
      even but non-vanishing number of disks of Maslov index $n+1$
      whose boundary passes through these two points.
     \item[iii.] For $n=2$, $(\C P^{2},L)$ is $(1,2)$-uniruled of
      order $6$.
   \end{itemize}
\end{lem}

\begin{proof}
   Fix a Morse function $f:L\to \R$ with a single minimum and a single
   maximum and fix also a perfect Morse function $g:\C P^{n}\to \R$.
   Fix also Riemannian metrics $\rho_{L}$ on $L$ and $\rho_{M}$ on $M
   = {\mathbb{C}}P^n$ as well as an almost complex structure $J$ so
   that the pearl complex
   $\mathcal{C}(f)=\mathcal{C}(L;\Lambda;f,\rho_{L},J)$ and the Morse
   complex (tensored with $\Lambda$) $C(g)$ are defined as well as the
   module product:
   $$C(g)\otimes \mathcal{C}(f)\to \mathcal{C}(f)~.~$$
   Let $f':L\to \R$ be a second Morse function (again with a single
   minimum and maximum) and assume that the pearl complex
   $\mathcal{C}(f')=\mathcal{C}(L;\Lambda;f',\rho_{L},J)$ is defined
   as well as the quantum product:
   $$\mathcal{C}(f')\otimes\mathcal{C}(f)\to \mathcal{C}(f)~.~$$

   We now prove point i.  We have the relation
   \begin{equation}\label{eq:pt_act}
      [pt]\ast \alpha_{n}=h^{\ast_{n}}\ast \alpha_{n}= \alpha_{-n} =
      \alpha_{1}t \in QH(L)
   \end{equation}
   where, as before, $h\in H_{2n-2}(\C P^{n}; \Z_{2})$ is the
   generator. Denote by $w$ the maximum of $f$ and by $p$ the minimum
   of $g$. The critical point $w$ is a cycle in $\mathcal{C}(f)$ and
   $[w]=\alpha_{n}$. Thus, in view of relation (\ref{eq:pt_act}) we
   have $p\ast w \neq 0 \in \mathcal{C}_{-n}(f)$. As
   $\mathcal{C}_{-n}(f) = \mathcal{C}_1(f)t = \mathbb{Z}_2 \langle
   \textnormal{Crit}(f) \rangle t$ (the last equality being true for
   degree reasons) we obtain that $p\ast w$ has a summand which is of
   the type $yt$, where $y \in \textnormal{Crit}_1(f)$.  Given the
   definition of the module action
   in~\S\ref{subsubsec:defin_alg_str}~c this means that there is a
   $J$-disk of Maslov index $n+1$ through the point $p$. As we may
   choose $g$ so that the point $p$ is anywhere desired in $\C
   P^{n}\backslash L$ this implies point i.

   For the point ii. we will use the relation
   \begin{equation}
      \alpha_{n-1}\ast \alpha_{0}=\alpha_{n}t~.~
   \end{equation}
   To exploit this we denote by $m$ the minimum of $f$ and we let $c$
   be a cycle in $\mathcal{C}(f')$ which represents $\alpha_{n-1}$.
   Because $L$ is wide, $m$ is a Morse cycle and $N_L = n+1$, we
   deduce that $m$ is also a cycle in $\mathcal{C}(f)$ so that
   $[m]=\alpha_{0}$. Thus we have, at the chain level, $c\ast m = wt$.
   In view of the definition of the quantum product
   in~\S\ref{subsubsec:defin_alg_str}~b, we deduce that for generic
   $J$ there exists a $J$-disk of Maslov index $n+1$ through both $w$
   and $m$. To finish with this point we need now to remark that the
   number $n(m,w)$ of such disks is even.  Indeed, if $d$ is the
   differential of the pearl complex $\mathcal{C}(f)$, notice that for
   degree reasons the differential of $m$ has the form $dm=\epsilon
   wt$ where $\epsilon\in\Z_{2}$ is the parity of $n(m,w)$. But, as
   mentioned above, $L$ is wide and so $\epsilon=0$.

   For the third point we use the relation:
   \begin{equation}\label{eq:qnt_pr_pt}
      [pt]\ast \alpha_{0}=\alpha_{2}t^{2},
   \end{equation}
   and the fact that, when $n=2$, $\alpha_{2}=[w]$. At the chain
   level~\eqref{eq:qnt_pr_pt} becomes $p\ast m=w t^{2}$.  By
   interpreting this relation in terms of the moduli spaces used
   in~\S\ref{subsubsec:defin_alg_str}~c to define the module product
   we deduce that there is a ``chain of pearls'' of one of the
   following types:
   \begin{itemize}
     \item[a.]  two disks $u_{1}$, $u_{2}$ joined by a flow line of
      $-\nabla f$ so that $m\in u_{1}(\partial D)$, $w\in
      u_{2}(\partial D)$, $\mu(u_{1})=\mu(u_{2})=3$ and $p$ belongs to
      the image of one of the $u_{i}|_{\textnormal{Int\,} D}$'s.
     \item[b.] a single disk $u$ of Maslov index $2n+2=6$ whose
      interior goes through $p$ and with $m,w\in u(\partial D)$
   \end{itemize}
   Notice that given two points $k\in\C P^{n}\backslash L$, and $k'\in
   L$, for a generic $J$, there is {\em no} disk of Maslov index $n+1$
   passing through both $k$ and $k'$ because the virtual dimension of
   the moduli spaces of such disks equals $-1$.  Thus generically,
   case~a is not possible and so we are left with case~b which proves
   claim~iii of the lemma.
\end{proof}

\subsection{The Clifford torus.} \label{subsec:Cliff_calc} This
subsection consists a sequence of results in which we prove all the
properties claimed in Theorem~\ref{T:clif2-qstruct} and
Corollary~\ref{cor:clifford}.

\begin{lem} \label{l:T_clif-wide}
   The Clifford torus $\clift^{n}\in\C P^{n}$ is wide and
   $N_{\clift^{n}}=2$.
\end{lem}
This Lemma was first proved by Cho~\cite{Cho:Clifford} by a direct
computation. Below we give a somewhat different proof.

\begin{proof}
   We first notice that by Theorem \ref{thm:floer=0-or-all} any
   Lagrangian torus $L$ is narrow or wide and if $N_{L}\geq 3$, then
   it is wide.  In the case of the Clifford torus, $\clift^{n}=
   \{[z_0: \cdots: z_n] \in {\mathbb{C}}P^n \mid |z_0| = \cdots |z_n|
   \}\subset \C P^{n}$, a simple computation shows that it is monotone
   and that $N_{\mathbb{T}_{\textnormal{clif}}} = 2$. Moreover
   (see~\cite{Cho:Clifford}), with the standard complex structure on
   $\C P^{n}$ there are exactly $n+1$ families of disks of Maslov
   index $2$ with boundary on $\mathbb{T}^n_{\textnormal{clif}}$,
   $\gamma_{0},\gamma_{1},\ldots, \gamma_{n}$ so that for any point
   $x\in \mathbb{T}^n_{\textnormal{clif}}$ there is precisely one disk
   $\Delta_{i}(x)$ from the family $\gamma_{i}$ passing through $x$.
   In fact we can describe these disks explicitly as follows.  Write
   $x = [x_0: \cdots: x_n] \in \mathbb{T}^n_{\textnormal{clif}}$ with
   $|x_i| = 1$ for every $i$. Then the disk $\Delta_i(x)$ is given by
   $D \ni z \mapsto [x_0: \cdots: x_{i-1}:z:x_{i+1}: \cdots : x_n] \in
   {\mathbb{C}}P^n$.

   It is proved in~\cite{Cho:Clifford} that these disks are regular
   and we can choose a basis of $H_{1}(\clift^{n};\Z)$ represented by
   the curves $c_{i}=\partial (\Delta_{i}(x))$, $1\leq i\leq n$.  In
   this basis, $c_{0}=\partial (\Delta_{0}(x))\simeq -c_{1}-c_{2} -
   \cdots -c_{n}$. Using the criterion for the vanishing of Floer
   homology in Lemma \ref{P:criterion-QH=0-1} we see that the cycle
   $D_{1}$ defined there is null-homologous and so $\clift^{n}$ is
   wide.
\end{proof}

For the $2$-dimensional Clifford torus we now pass to verifying the
properties of the quantum product as stated in
Theorem~\ref{T:clif2-qstruct}. Before we go into these computations
recall from~\S\ref{sb:wide-canonical} that although
$\mathbb{T}^2_{\textnormal{clif}}$ is wide there might not be a
canonical isomorphism
$H(\mathbb{T}^2_{\textnormal{clif}};\mathbb{Z}_2) \otimes \Lambda
\cong QH(\mathbb{T}^2_{\textnormal{clif}})$. This turns out to be
indeed the case (see~\cite{Bi-Co:Yasha-fest, Bi-Co:qrel-long}).
However, by Proposition~\ref{p:wide-can-iso} we have canonical
embeddings $H_1(\mathbb{T}^2_{\textnormal{clif}};\mathbb{Z}_2) \otimes
\Lambda_* \hooklongrightarrow
QH_{1+*}(\mathbb{T}^2_{\textnormal{clif}})$ and
$H_2(\mathbb{T}^2_{\textnormal{clif}};\mathbb{Z}_2) \otimes \Lambda_*
\hooklongrightarrow QH_{2+*}(L)$. This implies, for degree reasons,
that:
\begin{equation} \label{eq:H(T)=QH(T)}
   QH_1(\mathbb{T}^2_{\textnormal{clif}}) \cong H_1(L;\mathbb{Z}_2),
   \quad QH_0(\mathbb{T}^2_{\textnormal{clif}}) \cong
   H_0(\mathbb{T}_{\textnormal{clif}};\mathbb{Z}_2) \oplus
   [\mathbb{T}^2_{\textnormal{clif}}]t,
\end{equation}
where the first isomorphism is canonical and the second isomorphism is
not canonical but the second summand on its right-hand side (involving
the fundamental class $[\mathbb{T}^2_{\textnormal{clif}}]t$) is
canonical.

In view of~\eqref{eq:H(T)=QH(T)}, let $w =
[\mathbb{T}^2_{\textnormal{clif}}] \in
H_2(\mathbb{T}^2_{\textnormal{clif}};\mathbb{Z}_2)$ be the fundamental
class and let $a=[c_{1}], \ b=[c_{2}] \in H_1(L;\mathbb{Z}_2) \cong
QH_1(\mathbb{T}^2_{\textnormal{clif}})$. By the preceding discussion
$w$, $a$, $b$ can be viewed as well defined elements of
$QH(\mathbb{T}^2_{\textnormal{clif}})$.

\begin{lem} There is an element $m\in QH_{0}(\clift^{2})$ which
   together with $wt$ generates $QH_{0}(\clift^{2})$ so that we have
   $a\ast b=m+ w t$, $b\ast a=m$, $a\ast a=b\ast b= w t$, $m\ast
   m=mt+wt^{2}$.
\end{lem}

\begin{proof}
   We consider a perfect Morse function
   $f:\mathbb{T}^2_{\textnormal{clif}}\to \R$ and, by a slight abuse
   in notation, we denote its minimum by $m$. Similarly, we denote its
   maximum by $w$ and we let the two critical points of index $1$ be
   denoted by $a'$ and $b'$. We pick $f$ so that the closure of the
   unstable manifold of $a'$ represents $a\in
   H_{1}(\mathbb{T}^2_{\textnormal{clif}};\mathbb{Z}_2)$ and the
   unstable manifold of the critical point $b'$ represents $b$.

   To simplify notation we denote the disk $\Delta_{i}(w)$ by $d_{i}$.
   See figure~\ref{f:clif-traj-1}.
   \begin{figure}[htbp]
      \begin{center}
         \epsfig{file=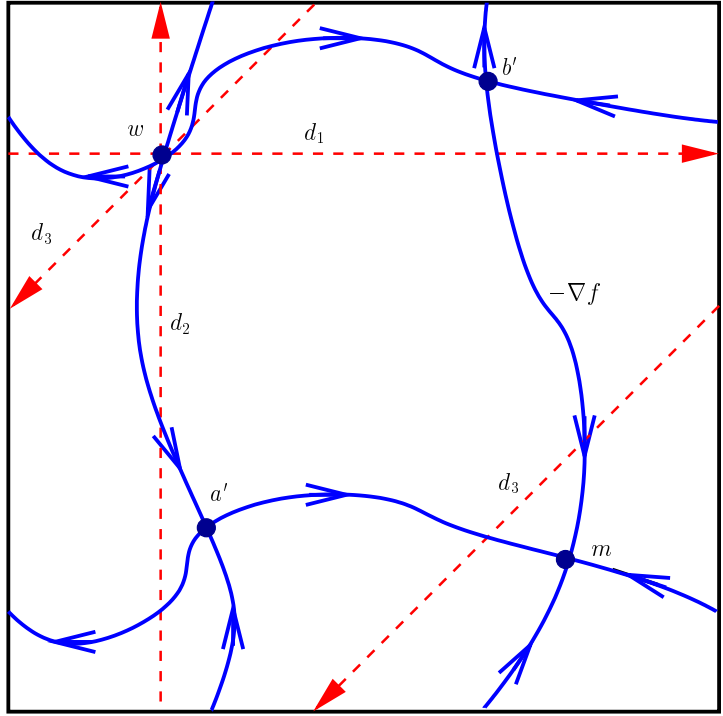, width=0.6\linewidth}
      \end{center}
      \caption{Trajectories of $-\nabla f$ and holomorphic disks on
        $\mathbb{T}^2_{\textnormal{clif}}$.}
      \label{f:clif-traj-1}
   \end{figure}
   By possibly perturbing the function $f$ slightly we may assume that
   the unstable manifold of $a'$ intersects $d_{2}$ and $d_{3}$ in a
   single point and is disjoint from $d_{1}$.  Similarly, we may
   assume that the unstable manifold of $b'$ intersects $d_{1}$ and
   $d_{3}$ at a single point and that this unstable manifold is
   disjoint from $d_{2}$. With these choices the pearl complex
   $(\mathcal{C}(f, J, \rho), d)$ is well defined. Here we take $J$ to
   be the standard complex structure of ${\mathbb{C}}P^2$, or a
   generic small perturbation of it and $\rho$ a generic small
   perturbation of the flat metric on
   $\mathbb{T}^2_{\textnormal{clif}}$.  As $f$ is perfect and
   $\clift^{2}$ is wide, the differential in $\mathcal{C}(f,J,\rho)$
   vanishes.  From now on we will view $m, a', b', w$ as generators
   (over $\Lambda$) of $QH_*(\mathbb{T}^2_{\textnormal{clif}})$.
   Recall that $m$ depends on the choice of $f$ in the sense that if
   we take another perfect Morse function $\widetilde{f}$ with minimum
   $\tilde{m}$, then $\tilde{m}$ might give an element of
   $QH_0(\mathbb{T}^2_{\textnormal{clif}})$ which is different than
   $m$. On the other hand $a', b', w \in QH$ are canonical.

   In order to compute the various products of $a'$ and $b'$ we use
   another perfect Morse function $g:\mathbb{T}^2_{\textnormal{clif}}
   \to \mathbb{R}$ with critical points $a''$, $b''$, $m''$, $w''$. We
   may choose $g$ to be a small perturbation of $f$ so that the
   unstable and stable manifolds of $a''$, $b''$ become ``parallel''
   copies of those of the corresponding points of $f$ (see
   figure~\ref{f:clif-traj-2}). Moreover, by taking $g$ to be close
   enough to $f $ (and keeping $J$ and $\rho$ fixed) we may assume
   that the comparison chain map $\Psi^{\textnormal{\tiny prl}} =
   \Psi_{\scriptscriptstyle (f, \rho, J), (g, \rho, J)} :
   \mathcal{C}(L; f, \rho, J) \to \mathcal{C}(L; g, \rho, J)$
   coincides with the Morse comparison chain map
   $\Psi^{\textnormal{{\tiny Morse}}}_{\scriptscriptstyle (f,\rho),
     (g, \rho)}$, namely:
   $$\Psi^{\textnormal{\tiny prl}}(a') = a'', \quad
   \Psi^{\textnormal{\tiny prl}}(b') = b'', \quad
   \Psi^{\textnormal{\tiny prl}}(m) = m'', \quad
   \Psi^{\textnormal{\tiny prl}}(w) = w''.$$ See point e
   in~\S\ref{subsubsec:defin_alg_str} as well as the proof of
   Proposition~\ref{p:duality} for various descriptions of the
   comparison map $\Psi^{\textnormal{\tiny prl}}$ (this map was
   denoted in the proof of Proposition~\ref{p:duality} by
   $\phi^{f,f'}$).

   \begin{figure}[htbp]
      \begin{center}
         \epsfig{file=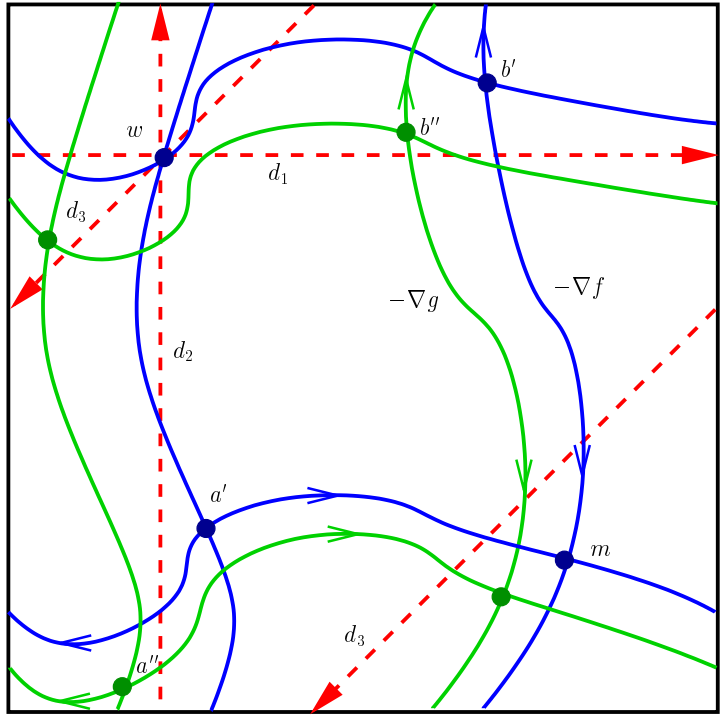, width=0.6\linewidth}
      \end{center}
      \caption{Trajectories of $-\nabla f$, $-\nabla g$ and
        holomorphic disks on $\mathbb{T}^2_{\textnormal{clif}}$.}
      \label{f:clif-traj-2}
   \end{figure}

   We now compute the product (on the chain level):
   $$*: \mathcal{C}(L; f, \rho, J) \otimes \mathcal{C}(L; g, \rho, J)
   \to \mathcal{C}(L; f, \rho, J).$$ For degree reasons we have:
   $$a'*b'' = m + \epsilon w t, \quad b'*a'' = m + \epsilon' wt, 
   \quad \textnormal{for some } \epsilon, \epsilon' \in
   \mathbb{Z}_2.$$ By the definition of the quantum product,
   $\epsilon$ is the number modulo $2$ of $J$-disks with $\mu = 2$
   going - in clockwise order !  - through the following points: one
   point in the unstable manifold of $a'$ then $w$ and, finally one
   point in the unstable manifold of $b''$.  Similarly, $\epsilon'$ is
   the number modulo $2$ of disks with $\mu=2$ going in order through
   a point in the unstable manifold of $b'$, $w$ and then a point in
   the unstable manifold of $a''$. There is a single disk through $w$
   which also intersects both the unstable manifolds of $a'$ and $b'$
   - the disk $d_{3}$.  However, the order in which the three types of
   points appear on the boundary of this disk implies that precisely
   one of $\epsilon$ and $\epsilon'$ is non-zero. Looking at
   figure~\ref{f:clif-traj-2} we see that for our choices of Morse
   data and $J$ we actually have $\epsilon=1$, $\epsilon' = 0$. Thus
   $a'*b'' = m+wt$, $b'*a'' = m$, hence in
   $QH(\mathbb{T}^2_{\textnormal{clif}})$ we have $a*b = m + wt$ and
   $b*a = m$.

   Next we compute $a*a$ and $b*b$ via $a'*a''$ and $b'*b''$.  To this
   end first note that $a'* a'' = \delta wt$ with $\delta\in\{0,1\}$
   (the classical term vanishes here since in singular homology we
   have $a \cdot a = 0$). There are precisely two pseudo-holomorphic
   disks that go through $w$ as well as through both unstable
   manifolds of $a'$ and of $a''$: the disks $d_{2}$ and $d_{3}$.  It
   is at this point that we use the fact that $[d_{2}]=b$,
   $[d_{3}]=-a-b$.  Indeed, this means that the order in which these
   three points lie on the boundary of each of these two disks is
   opposite. Thus, exactly one of these disks will contribute to
   $\delta$ and so $\delta=1$. (In fact, looking at
   figure~\ref{f:clif-traj-2} we see that the relevant disk is $d_2$.)
   A similar argument shows $b'\ast b''= wt$.  The formula for $m*m$
   follows now from the associativity of the product.  Indeed:
   $$m*m = (a*b+wt)*(b*a) = a*(b*b)*a + b*a t = mt+wt^2.$$
   (Recall that we are working over $\mathbb{Z}_2$.)
  \end{proof}

  \begin{rem}\label{rem:higher_cliff_prod} For the $n$-dimensional
     Clifford torus, $\clift^{n}\subset \C P^{n}$, let $t_1, \ldots,
     t_n$ be a basis of $H_{n-1}(\clift^{n};\Z_{2})$ dual to the basis
     $[c_1], \ldots, [c_n] \in H_{1}(\clift^{n};\Z_{2})$, with respect
     to the (classical) intersection product. The same argument as
     that giving the product $a\ast b$, $b\ast a$ in the proof of the
     lemma above shows that for $i\not=j$, $t_{i}\ast t_{j}+t_{j}\ast
     t_{i}=wt$ where $w$ represents the fundamental class.
  \end{rem}

  We now turn to determining the quantum module structure (points ii
  and iii in Theorem \ref{T:clif2-qstruct}). We recall that $h\in
  H_{2}(\C P^{2};\Z_{2})$ is the class of a hyperplane, hence in this
  case of a projective line $\mathbb{C}P^1 \subset {\mathbb{C}}P^2$.

  \begin{lem}\label{lem:clift_module}
     With the notations above we have:
     \begin{itemize}
       \item[i.] $h\ast a=at$, $h\ast b=bt$, $h\ast w=wt$, $h\ast m =
        mt$.
       \item[ii.] $i_{L}(m)=[pt]+ht+[\C P^{2}]t^{2}$,
        $i_{L}(a)=i_{L}(b)=i_{L}(w)=0$.
     \end{itemize}
  \end{lem}

  \begin{proof} We will make use of a second geometric fact concerning
     the Clifford torus: there is a symplectomorphism homotopic to the
     identity, $\bar{\phi}: \C P^2\to \C P^2$, whose restriction to
     $\mathbb{T}^2_{\textnormal{clif}}$ is the permutation of the two
     factors in $\mathbb{T}^2_{\textnormal{clif}} \approx S^{1}\times
     S^{1}$.  We now determine what is the map
     $$\tilde{\phi}:QH_{\ast}(\mathbb{T}^2_{\textnormal{clif}})\to
     QH_{\ast}(\mathbb{T}^2_{\textnormal{clif}})$$ which is induced by
     $\bar{\phi}$. For degree reasons we have $\tilde{\phi}(w)=w$,
     $\tilde{\phi}(a)=b$, $\tilde{\phi}(b)=a$ and, by
     Proposition~\ref{cor:symm}, we know that $\tilde{\phi}$ is a
     morphism of algebras (from this it also follows immediately that
     $\tilde{\phi}(m)=m+wt$).

     We now compute $h\ast a$ and $h\ast b$. We have, $h\ast a=h\ast
     \tilde{\phi}(b)=\tilde{\phi}(h\ast b)$. Now $h\ast a= (u_{1}a +
     u_{2}b)t$ with $u_{1},u_{2}\in \mathbb{Z}_2$ which implies that
     $h\ast b= (u_{1}b + u_{2} a)t$. As in Corollary \ref{T:Lag-CPn}
     we also have that $h \ast (-):
     H_{1}(\mathbb{T}^2_{\textnormal{clif}};\mathbb{Z}_2)\to
     H_{1}(\mathbb{T}^2_{\textnormal{clif}};\mathbb{Z}_2)t$ is an
     isomorphism.  This implies that precisely one of $u_{1},u_{2}$ is
     non zero.  Assume first that $u_{1}=0$ and $u_{2}=1$.  Then
     $h\ast a= bt$, $h\ast(h\ast a)=at^{2}$ and $h\ast(h\ast(h\ast
     a))=bt^{3}$ which is not possible because $h^{\ast 3}=[\C P^{2}]
     t^{3}$ (where, $[\C P^{2}]$ denotes the fundamental class of $\C
     P^{2}$) and $[\C P^{2}]\ast a=a$. Thus we are left with
     $u_{1}=1$, $u_2=0$ as claimed.

     To compute $h*w$ write $h\ast w t=h\ast( a\ast a)= (h\ast a)\ast
     a=(a\ast a)t=wt^{2}$.  Similarly $h\ast m=h\ast (b\ast a)=(h\ast
     b)\ast a= mt$.

     Finally, point ii is an immediate consequence of the first point
     and of formula (\ref{eq:inclusion_mod}) in
     Theorem~\ref{thm:alg_main} iii.
  \end{proof}

  Finally, we need to justify the uniruling properties of the Clifford
  torus as described in Corollary \ref{cor:clifford}.

\begin{lem}
   For $n\geq 2$, $(\C P^{n},\clift^{n})$ is $(1,0)$-uniruled of order
   $2n$ and $\clift^{n}$ is uniruled of order $2$. For $n=2$, $(\C
   P^{2},\clift^{2})$ is $(1,1)$-uniruled of order $4$.
\end{lem}
\begin{proof}
   As $\clift^{n}$ is wide of minimal Maslov number $2$ and $\C P^{n}$
   is point invertible of order $2n+2$ we deduce from Theorem
   \ref{theo:geom_rig} that $(\C P^{n},\clift^{n})$ is uniruled of
   order (at most) $2n$. The fact that $\clift^{n}$ is uniruled of
   order $2$ follows immediately from the relation $t_{i}\ast
   t_{j}+t_{j}\ast t_{i}=wt$ from Remark \ref{rem:higher_cliff_prod}.
   Indeed, this relation implies the existence of a disk of Maslov
   index $2$ through $w$ (for generic $J$). There is also a direct
   proof of this, based on the fact that the families of $J$-disks
   $\gamma_{i}$ are regular and thus, being of minimal possible area,
   they persist under generic deformations of $J$. Finally, for $n=2$,
   with the notation in Lemma \ref{lem:clift_module} we have the
   relation $[pt]\ast m=mt^{2}$ where $[pt]=h^{* 2}$. We consider a
   Morse function $g:\C P^{2}\to \R$ which is perfect and we denote
   its minimum by $p$.  The previous relation gives (at the chain
   level): $p\ast m=mt^{2}$ where $m$ is the minimum of a perfect
   Morse function $f:\clift^{2}\to \R$ (so that the respective pearl
   complex and all the relevant operations are defined).  This means
   that there is a configuration consisting of one of the following:
   \begin{enumerate}[a.]
     \item one $J$-disk with $\mu=4$ through $p$, whose boundary is on
      $\mathbb{T}^2_{\textnormal{clif}}$ and contains $m$.
     \item two $J$-disks, each with $\mu=2$, related by a negative
      gradient flow line of $f$ so that one of these two disks goes
      through $p$ and the boundary of the other contains $m$.
   \end{enumerate}
   To prove our claim we only have to notice that possibility b can
   not arise for a generic $J$.  Indeed, generically, the set of
   points in $\C P^{2}$ which lie in the image of some $J$-disk of
   Maslov index $2$ is only $3$-dimensional and so, generically, these
   disks avoid $p$.
\end{proof}

\subsection{Lagrangians in the quadric} Here we prove
Theorem~\ref{T:quadric-qstruct} and Corollary~\ref{cor:quadric}.

Let $Q \subset {\mathbb{C}}P^{n+1}$ be a smooth complex
$n$-dimensional quadric, where $n \geq 2$. More specifically we can
write $Q$ as the zero locus $Q= \{ z \in {\mathbb{C}}P^{n+1} \mid
q(z)=0\}$ of a homogeneous quadratic polynomial $q$ in the variables
$[z_0: \cdots: z_{n+1}] \in {\mathbb{C}}P^{n+1}$, where $q$ defines a
quadratic form of maximal rank. We endow $Q$ with the symplectic
structure induced from ${\mathbb{C}}P^{n+1}$. (Recall that we use the
normalization that the symplectic structure $\omega_{\textnormal{FS}}$
of ${\mathbb{C}}P^{n+1}$ satisfies $\int_{\mathbb{C}P^1}
\omega_{\textnormal{FS}} = 1$.) When $n \geq 3$ we have by Lefschetz
theorem $H^2(Q;\mathbb{R}) \cong \mathbb{R}$, therefore by Moser
argument all K\"{a}hler forms on $Q$ are symplectically equivalent up
to a constant factor. When $n=2$, $Q \subset {\mathbb{C}}P^3$ is
symplectomorphic to $(\mathbb{C}P^1 \times \mathbb{C}P^1, \omega_{FS}
\oplus \omega_{FS})$. Also note that the symplectic structure on $Q$
(in any dimension) does not depend (up to symplectomorphism) on the
specific choice of the defining polynomial $q$ (this follows from
Moser argument too since the space of smooth quadrics is connected).

\subsubsection{Topology of the quadric} \label{Sb:top-quadric} The
quadric has the following homology:
\begin{equation*}
   H_i(Q;\mathbb{Z}) \cong
   \begin{cases}
      0 & \textnormal{if }i=\textnormal{odd} \\
      \mathbb{Z} & \textnormal{if } i=\textnormal{even} \neq n
   \end{cases}
\end{equation*}
Moreover, when $n=$\,even, $H_n(Q;\mathbb{Z})\cong \mathbb{Z} \oplus
\mathbb{Z}$. To see the generators of $H_n(Q;\mathbb{Z})$, write
$n=2k$. There exist two families $\mathcal{F}, \mathcal{F}'$ of
complex $k$-dimensional planes lying on $Q$
(see~\cite{Gr-Ha:alg-geom}). Let $P \in \mathcal{F}$, $P' \in
\mathcal{F}'$ be two such planes belonging to different families. Put
$a=[P]$, $b=[P']$. Then $H_n(Q;\mathbb{Z}) = \mathbb{Z}a \oplus
\mathbb{Z} b$ and $h^{\bullet k}=a+b$. Moreover, we have:
\begin{equation} \label{Eq:cap-quad}
   \begin{aligned}
      & \textnormal{for } k=\textnormal{odd}: & a \cdot b = [pt],\; a
      \cdot a =
      b \cdot b =0, \\
      & \textnormal{for } k=\textnormal{even}: & a \cdot b = 0, \; a
      \cdot a = b \cdot b = [pt].
   \end{aligned}
\end{equation}
Here and in what follows we have denoted by $\cdot$ the intersection
product in singular homology.

\subsubsection{Quantum homology of the quadric}
\label{Sb:qhom-quadric} Let $h \in H_{2n-2}(Q;\mathbb{Z})$ be the
class of a hyperplane section (coming from the embedding $Q \subset
{\mathbb{C}}P^{n+1}$), $p \in H_0(Q;\mathbb{Z})$ the class of a point
and $u \in H_{2n}(Q;\mathbb{Z})$ the fundamental class. We will first
describe the quantum cohomology over $\mathbb{Z}$. Define
$\Lambda^{\mathbb{Z}} = \mathbb{Z}[t, t^{-1}]$ where $\deg t = -N_L$.
Here $N_L$ is the minimal Maslov number of a Lagrangian submanifold
that will appear later on. Note that that $c_1(Q) = nPD(h)$, hence
$N_L | 2n$. Let $QH(Q;\Lambda^{\mathbb{Z}})=H(Q;\mathbb{Z}) \otimes
\Lambda^{\mathbb{Z}}$ be the quantum homology endowed with the quantum
product $*$.

\begin{prop}[See~\cite{Beau:quant}] \label{P:qhom-quadric}
   The quantum product satisfies the following identities:
   \begin{align*}
      & h^{*j}=h^{\bullet j} \; \forall \, 0 \leq j \leq n-1, \quad
      h^{*n}=2p+2u  t^{2n/N_L}, \\
      & h^{*(n+1)}=4h t^{2n/N_L}, \quad p*p=u t^{4n/N_L}.
   \end{align*}
   When $n=$\,even we have the following additional identities:
   \begin{enumerate}[i.]
     \item $h*a=h*b$.
     \item If $n/2=$\,odd then $a*b=p$, $a*a=b*b=u t^{2n/N_L}$.
     \item If $n/2=$\,even then $a*a=b*b=p$, $a*b=u t^{2n/N_L}$.
   \end{enumerate}
\end{prop}
\begin{proof}
   The first three identities and the fact that $h*a=h*b$ are proved
   in~\cite{Beau:quant}. To prove the remaining two identities write
   $n=2k$. Recall from~\cite{Beau:quant} that
   $$(a-b)*(a-b) = \bigl(\#(a-b) \cdot (a-b)\bigr) 
   \frac{1}{2}(h^{*n}-4u t^{2n/NL}) =
   \bigl(\#(a-b)\cdot(a-b)\bigr)(p-u t^{2n/N_L}).$$
   Substituting~\eqref{Eq:cap-quad} in this we obtain:
   \begin{equation} \label{Eq:qprod-quad-1}
      (a-b)*(a-b) = (-1)^k 2 (p-u t^{2n/N_L}).
   \end{equation}
   On the other hand we have $h^{*k} = h^{\bullet k} = a+b$, hence
   \begin{equation} \label{Eq:qprod-quad-2}
      (a+b)*(a+b) = h^{*n} = 2p + 2u  t^{2n/N_L}.
   \end{equation}
   Next we claim that $a*a = b*b$. Indeed $a*a-b*b = (a+b)*(a-b) =
   h^{*k}*(a-b)=0$. The desired identities follow from this together
   with~\eqref{Eq:qprod-quad-1},~\eqref{Eq:qprod-quad-2}.
\end{proof}

\subsubsection{Quantum structures for Lagrangian submanifolds of the quadric.}
\label{Sb:quad-lag}
The quadric $Q$ has Lagrangian spheres. To see this write $Q$ as $Q =
\{z_0^2 + \cdots + z_n^2 = z_{n+1}^2\} \subset {\mathbb{C}}P^{n+1}$.
Then $L = \{ [z_0: \cdots : z_{n+1}] \in Q \mid z_i \in \mathbb{R},
\forall\, i \}$ is a Lagrangian sphere. We assume from now on that
$n\geq 2$.

\begin{lem} \label{l:quad-quant-L} Let $L \subset Q$ be a Lagrangian
   submanifold with $H_{1}(L;\Z)=0$.  Then, $N_{L}=2n$, $L$ is wide
   and there is a canonical isomorphism $QH(L) \cong H(L;\mathbb{Z}_2)
   \otimes \Lambda$. Moreover, if we denote by $\alpha_0 \in QH_0(L)$
   the class of a point, by $\alpha_n \in QH_n(L)$ the fundamental
   class and similarly by $p \in QH_0(Q)$ the class of the point and
   by $u \in QH_{2n}(Q)$ the fundamental class, then we have:
   \begin{enumerate}[i.]
     \item $p*\alpha_0 = \alpha_0 t$, $p*\alpha_n = \alpha_n t$.
     \item $i_L(\alpha_0) = p + u t$.
     \item If $n$ is even then $\alpha_0 * \alpha_0 = \alpha_n t$.
   \end{enumerate}
\end{lem}
\begin{rem} \label{r:quad-z2-z} Suppose that $L$ is a monotone
   Lagrangian which is orientable and relative spin (see~\cite{FO3}
   for the definition). In that case, it is possible to coherently
   orient the moduli spaces of pseudo-holomorphic disks with boundary
   on $L$ using the theory of~\cite{FO3}. It seems very likely that
   these orientations are compatible with the quantum operations based
   on our pearly moduli spaces, hence we expect our theory to work
   over $\mathbb{Z}$. Assuming this, let $L$ be a Lagrangian as in
   Lemma~\ref{l:quad-quant-L} and suppose in addition that $L$ is
   relative spin ($H_1(L;\mathbb{Z})=0$ automatically implies
   orientability). Then we expect the formulae in~i and~ii to become:
   \begin{enumerate}[i'.]
      \item $p*\alpha_0 = -\alpha_0 t$, $p*\alpha_n = -\alpha_n t$.
      \item $i_L(\alpha_0) = p - u t$.
   \end{enumerate}
\end{rem}

\begin{proof}[Proofs of Lemma~\ref{l:quad-quant-L} and
   Remark~\ref{r:quad-z2-z}]

   Following Remark~\ref{r:quad-z2-z} we will carry out the proof over
   the ring $\mathbb{K}$ which is either $\mathbb{Z}_2$ or
   $\mathbb{Z}$. In the latter case the proof is not 100\% rigurous in
   the sense that it depends on the verification that our theory
   indeed works over $\mathbb{Z}$. We remark that for $\mathbb{K} =
   \mathbb{Z}_2$ the proof below is completely rigorous (and in this
   case we may also drop the assumptions that $L$ is orientable and
   relative spin). We will use the ring $\Lambda = \mathbb{K}[t^{-1},
   t]$ with the same grading as before, i.e.  $\deg t = -N_L$.

   Due to $H_{1}(L;\Z)=0$ and $C_Q=n$ we see that $N_{L}=2n$. By
   Theorem \ref{thm:floer=0-or-all} we deduce that $L$ is wide.
   Moreover, by Proposition~\ref{p:wide-can-iso} there is a canonical
   isomorphism $QH_*(L) \cong (H(L;\mathbb{K}) \otimes \Lambda)_*$.

   We first prove the lemma and the remark under the additional
   assumption that $n = \dim L \geq 3$. The case $n=2$ will be treated
   separately at the end of the proof.

   We start with the statement at point ii'. It easily follows from
   the definition of the quantum inclusion map that $i_L(\alpha_0) = p
   + e u t$, for some $e \in \mathbb{K}$.  Clearly $h * \alpha_0 = 0$
   since $h * \alpha_0$ belongs to $QH_{-2}(L) \cong QH_{2n-2}(L)=0$
   (since $2n-2 > n$). Therefore we have
   $$0=i_L(h*\alpha_0) = h*(p+e u t) = h*p + eh t.$$
   On the other hand a simple computation based on the identities of
   Proposition~\ref{P:qhom-quadric} gives $h*p=h t$. It follows that
   $e=-1$. This proves point ii'.

   We turn to proving point i'. By Proposition~\ref{P:qhom-quadric} $p
   \in QH_0(Q;\Lambda)$ is an invertible element, hence $p *
   (-):QH_i(L) \to QH_{i-2n}(L)$ is an isomorphism for every $i$. But
   $QH_0(L) \cong \mathbb{K} \alpha_0$ and $QH_{-2n}(L) \cong
   \mathbb{K} \alpha_0 t$.  Therefore $p*\alpha_0 = \epsilon \alpha_0
   t$, where $\epsilon = \pm 1$. It remains to determine the precise
   sign of $\epsilon$. Using the formula in ii' we obtain
   \begin{equation} \label{eq:i_L(p*alpha_0)}
      i_L(p*\alpha_0) = i_L(\epsilon \alpha_0 t) = \epsilon (pt -ut^2).
   \end{equation}
   On the other hand we have $$i_L(p * \alpha_0) = p*i_L(\alpha_0) =
   p*(p-ut) = ut^2 - pt.$$ Comparing this to~\eqref{eq:i_L(p*alpha_0)}
   immediately shows that $\epsilon = -1$. The proof of the identity
   $p* \alpha_n = -\alpha_n t$ is similar. This concludes the proof of
   point i'.

   We now turn to the proof in case $n=2$.  In this case $Q \approx
   S^2 \times S^2$ endowed with the split symplectic form $\omega
   \oplus \omega$ with both $S^2$ factors having the same area. Put $a
   = [S^2 \times \textnormal{pt}]$, $b=[\textnormal{pt} \times S^2]$
   $\in H_2(Q; \mathbb{Z})$ and denote by $\textnormal{inc}_*:
   H_*(L;\mathbb{Z}) \to H_*(Q; \mathbb{Z})$ the (classical) map
   induced by the inclusion $L \subset Q$. Note that $L$ must be a
   Lagrangian sphere, hence $\int_L \omega = 0$ and
   $\textnormal{inc}_*([L]) \cdot \textnormal{inc}_*([L]) = -2$. It
   follows that $\textnormal{inc}_*([L]) = \pm (a-b)$. Finally, in
   this dimension the hyperplane class $h$ satisfies $h = a+b$.

   As $n=2$ we have $N_L = 4$ and so $\deg t = -4$.  As before, since
   $p$ is invertible we can write $p*\alpha_0 = \epsilon \alpha_0 t$,
   where $\epsilon = \pm 1$, and $i_L(\alpha_0) = p + e ut$ with $e
   \in \mathbb{Z}$. It follows that
   $$i_L(p * \alpha_0) = p*(p+e u t) = ut^2 + e p t.$$ On the other
   hand we also have:
   $$i_L(p*\alpha_0) = i_L(\epsilon \alpha_0 t) = \epsilon t(p+eut) =
   \epsilon e ut^2 + \epsilon p t.$$ It follows that $\epsilon e = 1$,
   hence $e = \epsilon = \pm 1$. This proves formulas i and ii over
   $\mathbb{Z}_2$ (that $p*\alpha_2 = \pm \alpha_2 t$ follows
   immediately from the fact that $p$ is invertible).

   It remains to determine the sign of $e$ and $\epsilon$, so we now
   work over $\mathbb{Z}$. For this end write $h*\alpha_0 = r
   \alpha_2$ with $r \in \mathbb{Z}$.  Note that $\alpha_2 = [L]$ so
   $$i_L(h*\alpha_0) = i_L(r \alpha_2 t) = r \textnormal{inc}_*([L]) t
   = \pm r (a-b) t.$$ (Here we have used the fact that for the
   fundamental class $[L]$ we have $i_L([L]) =
   \textnormal{inc}_*([L])$.) On the other hand $$i_L(h*\alpha_0) = h*
   i_L(\alpha_0) = h*(p+e u t) = ht + e ht = (1+e)ht = (1+e)(a+b)t.$$
   It follows that $(1+e)(a+b)t = \pm r (a-b)t$. This implies $r = 1+e
   = 0$, hence $e = -1$. The proof of formulae i, i', ii, ii' is now
   complete for every $n \geq 2$.

   Finally, we prove iii (only over $\mathbb{Z}_2$). By
   Proposition~\ref{P:qhom-quadric} when $n=$\,even the element $a \in
   QH_n(Q;\Lambda)$ is invertible (even if we work with coefficients
   in $\mathbb{Z}_2$). Therefore $a*\alpha_n =\alpha_0$ and
   $a*\alpha_0 = \alpha_n t$. It follows that
   $$\alpha_0* \alpha_0 = (a*\alpha_n)*\alpha_0 = a*
   (\alpha_n*\alpha_0) = a*\alpha_0 = \alpha_n t.$$
\end{proof}

The following result shows that for $n=$\,even, at least
homologically, spheres are the only type of Lagrangian in $Q$ with
$H_1(L;\mathbb{Z})=0$.
\begin{thm} \label{T:quad-L} Assume $n = \dim_{\mathbb{C}}Q=$\,even.
   Let $L \subset Q$ be a Lagrangian submanifold with
   $H_1(L;\mathbb{Z})=0$. Then $H_*(L;\mathbb{Z}_2) \cong
   H_*(S^n;\mathbb{Z}_2)$.
\end{thm}
\begin{proof}
   In view of the isomorphism $QH_*(L)\cong (H(L;\mathbb{Z}_2)\otimes
   \Lambda)_*$, for every $q \in \mathbb{Z}$, $0 \leq r <2n$ we have:
   \begin{equation} \label{Eq:QH-quad} QH_{2nq+r}(L) \cong
      \begin{cases}
         H_r(L;\mathbb{Z}_2) & \textnormal{if } 0 \leq r \leq n\\
         0 & \textnormal{if } n+1 \leq r \leq 2n-1
      \end{cases}
   \end{equation}
   Reducing modulo $2$ the identities from
   Proposition~\ref{P:qhom-quadric} it follows that $a\in QH_n(Q;\La)$
   is an invertible element. Therefore $a * (-):QH_i(L) \to
   QH_{i-n}(L)$ is an isomorphism for every $i \in \mathbb{Z}$.  It
   now easily follows from~\eqref{Eq:QH-quad} that
   $H_i(L;\mathbb{Z}_2)=0$ for every $0<i<n$.
\end{proof}

We are not aware of the existence of any Lagrangian submanifolds in $Q$
with $H_1(L;\mathbb{Z})=0$ which are not diffeomorphic to a sphere,
and it is tempting to conjecture that spheres are indeed the only
examples.

\begin{rem} Theorem~\ref{T:quad-L} can be also proved by Seidel's
   method of graded Lagrangian submanifolds~\cite{Se:graded}. Indeed
   for $n=$\,even the quadric has a Hamiltonian $S^1$-action which
   induces a shift by $n$ on $QH_*(L)$. To see this write $n=2k$ and
   write $Q$ as $Q=\{ \sum_{j=0}^k z_jz_{j+1+k}=0\}$. Then $S^1$ acts
   by $t\cdot [z_0: \cdots : z_{2k+1}] = [tz_0: \cdots: tz_k:z_{k+1}:
   \cdots: z_{2k+1}]$.  A simple computation of the weights of the
   action at a fixed point gives a shift of $n$ on graded Lagrangian
   submanifolds in the sense of~\cite{Se:graded}.

   When $n=$\,odd our methods (as well as those of~\cite{Se:graded})
   do not seem to yield a  result similar to Theorem~\ref{T:quad-L}.
   However the works of Buhovsky~\cite{Bu:TSn} and of
   Seidel~\cite{Se:TSn} may provide evidence that such a result should
   hold.
\end{rem}

Denote $\mathcal{J}$ the space of almost complex structures compatible
with the symplectic structure of $Q$. The next result is a
straightforward consequence of Lemma~\ref{l:quad-quant-L} and it
concludes the proofs of the properties claimed in Theorem
\ref{T:quadric-qstruct} and Corollary \ref{cor:quadric}.
\begin{lem} \label{C:quad-disks} Let $L \subset Q$ be a Lagrangian
   submanifold with $H_1(L;\mathbb{Z})=0$. Assume
   $n=\dim_{\mathbb{C}}Q \geq 2$. Then the following holds:
   \begin{enumerate}
     \item[i.] Let $x \in L$ and $z \in Q \setminus L$. Then for every
      $J \in \mathcal{J}$ there exists a $J$-holomorphic disk $u:(D,
      \partial D) \to (Q,L)$ with $u(-1)=x$, $u(0)=z$ and
      $\mu([u])=2n$.
     \item[ii.] Assume that $n=$\,even. Let $x', x'', x'' \in L$. Then
      for every $J \in \mathcal{J}$ there exists a $J$-holomorphic
      disk $u:(D, \partial D) \to (Q, L)$ with $u(e^{2\pi i/3})=x'$,
      $u(1)=x''$, $u(e^{4 \pi i/3})=x'''$ and $\mu([u])=2n$.
   \end{enumerate}
\end{lem}
\begin{proof}
   The first point follows as usual by considering a Morse function
   $f:L\to \R$ with a single maximum and a single minimum as well as a
   perfect Morse function $h:Q\to \R$. We let the minimum of $h$ be
   denoted by $p$ (by a slight abuse in notation we identify the
   critical points of $h$ and the corresponding singular homology
   classes) and we denote the minimum of $f$ by $m$ and its maximum by
   $w$. As $L$ is wide both $m$ and $w$ are cycles in the associated
   pearl complex.

   Point~i in Lemma~\ref{l:quad-quant-L} gives, at the chain level,
   $p\ast m=m t$ which implies the first point of our lemma.  The
   second point is proved by considering a second Morse function $f':L
   \to \mathbb{R}$ with a unique minimum $m'$. Relation~iii in
   Lemma~\ref{l:quad-quant-L} now gives (on the chain level) $m*m' =
   wt$, which proves the needed statement.
\end{proof}

\subsection{Narrow Lagrangians in $\C P^{n}$}
\label{sb:narrow}

The purpose of this section is to construct the monotone narrow
Lagrangians mentioned in Example \ref{ex:narrow}. The construction is
based on the decomposition technique developed in~\cite{Bi:Barriers}
and on the Lagrangian circle bundle construction
from~\cite{Bi:Nonintersections}.

Let $(M^{2n}, \omega)$ be a symplectic manifold for which $[\omega]
\in H^2(M; \mathbb{R})$ admits an integral lift in $H^2(M;
\mathbb{Z})$. Fix such a lift $a_{\omega}$. Let $\Sigma^{2n-2} \subset
M^{2n}$ be a symplectic hyperplane section in the sense that $\Sigma$
is a symplectic submanifold whose homology class is dual to a positive
multiple of $a_{\omega}$, i.e.  $PD [\Sigma] = k a_{\omega} \in
H^2(M;\mathbb{Z})$ for some integer $k > 0$. By rescaling $\omega$ we
will assume from now on, without loss of generality, that $k=1$.

Assume further that $M$ is a complex manifold, that $\omega$ is a
K\"{a}hler form and that $\Sigma \subset M$ is a complex submanifold
(so that $\Sigma \subset M$ is a smooth ample divisor).  Put
$\omega_{_{\Sigma}} = \omega|_{\Sigma}$ and $a_{\Sigma} =
a_{\omega}|_{\Sigma} \in H^2(\Sigma;\mathbb{Z})$. Let $\pi: P \to
\Sigma$ be a circle bundle with Euler class $a_{\Sigma}$ and $\alpha$
a connection $1$-form on $P$ normalized so that $d\alpha = -
\pi^*\alpha$. Denote by $E_{\Sigma} \to \Sigma$ the associated unit
disk bundle, $E_{\Sigma} = \bigl(P \times [0,1)\bigr) / \sim$, where
$(p', 0) \sim (p'', 0)$ iff $\pi(p') = \pi(p'')$. We endow
$E_{\Sigma}$ with the following symplectic structure:
$\omega_{\textnormal{can}} = \pi^*\omega_{\Sigma} + d(r^2 \alpha)$,
where $r$ is the second coordinate on $P \times [0,1)$. Note that with
our normalization $\omega_{\textnormal{can}}|_{\Sigma} =
\omega_{\Sigma}$ and the area of each fibre of $E_{\Sigma}$ with
respect to $\omega_{\textnormal{can}}$ is $1$.

By the results of~\cite{Bi:Barriers} there exists a compact isotropic
CW-complex $\Delta \subset M \setminus \Sigma$ and a symplectomorphism
$F: (E_{\Sigma}, \omega_{\textnormal{can}}) \longrightarrow (M
\setminus \Delta, \omega)$. Moreover, for every $x \in \Sigma \subset
E_{\Sigma}$ we have $F(x) = x$. In most cases $\Delta$ is a Lagrangian
CW-complex, i.e. $\dim \Delta = \frac{1}{2} \dim M$ -- this is called
the {\em critical} case. In special situations it may happen that
$\dim \Delta < \frac{1}{2} \dim M$, which we call the {\em
  subcritical} case. The dimension of $\Delta$ is in fact determined
by the critical points of a plurisubharmonic function $\varphi: M
\setminus \Sigma \to \mathbb{R}$ canonically determined by $\Sigma$
and the complex structure of $M$. The CW-complex $\Delta$ is called
the isotropic (or sometimes Lagrangian) skeleton. We refer the reader
to~\cite{Bi:Barriers} for more details on this type of decompositions.
See also~\cite{El-Gr:convex, El:Stein} for the foundations of
symplectic geometry of Stein manifolds, as well as~\cite{Bi-Ci:Stein,
  Bi-Ci:closed, Bi:Nonintersections} for applications of these
concepts to questions on Lagrangian submanifolds.  We will identify
from now on $(M \setminus \Delta, \omega)$ with $(E_{\Sigma},
\omega_{\textnormal{can}})$ via the map $F$.

Let $L\subset (\Sigma, \omega_{\Sigma})$ be a Lagrangian submanifold.
Fix $0< r_0 < 1$. Put $$\Gamma_L = \pi^{-1}(L) \times \{ r_0 \}
\subset E_{\Sigma} \approx M \setminus \Delta.$$ Note that $\pi:
\Gamma_L \to L$ is a circle bundle isomorphic to the restriction of $P
\to \Sigma$ to $L$. A simple computation shows that $\Gamma_L$ is
Lagrangian with respect to $\omega$. We will view $\Gamma_L$ as a
Lagrangian submanifold of $M$, but it is important to note that {\em
  $\Gamma_L$ is disjoint from $\Delta$}. We remark also that
$\Gamma_L$ depends on the value of $r_0$.  In fact, different values
of $r_0$ give rise to Lagrangians $\Gamma_L$ with different area
classes. Below we will make a specific choice of $r_0$ and call
$\Gamma_L$ the {\em Lagrangian circle bundle over $L$}. We refer the
reader to~\cite{Bi:Nonintersections} for more details on the subject.

Suppose now that $L \subset \Sigma$ is monotone with proportionality
constant $\eta = \omega / \mu$.
\begin{prop} \label{p:gamma_L} Assume that $\dim M \geq 6$, or that
   $\dim M = 4$ and $\Delta$ is subcritical. Let $r_0^2 =
   \frac{2\eta}{2\eta+1}$.  Then the Lagrangian $\Gamma_L \subset M$
   is monotone. It has minimal Maslov number $N_{\Gamma_L} = 2$ and
   proportionality constant $\widehat{\eta} = \frac{\eta}{2\eta + 1}$.
\end{prop}
\begin{proof}
   Fix $A \in \pi_2(M, \Gamma_L)$ and let $u: (D, \partial D) \to (M,
   \Gamma_L)$ be a representative of $A$. As $\dim \Delta + 2 < \dim
   M$ we may assume by transversality that the image of $u$ is
   disjoint from $\Delta$, hence lies in $E_{\Sigma}$. Denote by $x_1,
   \ldots, x_k$ the intersection points of $u$ with $\Sigma$ and
   assume that they are all transverse. Moreover, we may assume that
   each $x_i$ corresponds to a single interior point $z_i \in D$, so
   that $u^{-1}(x_i) = \{z_i\}$. After a suitable homotopy of $u$ (rel
   $\partial D$) we may assume that the points $x_i$ all lie in $L$.
   Denote by $D_{x_i} \subset E_{\Sigma}$ the disk of radius $r_0$
   lying in the fibre over $x_i$ (i.e.  $D_{x_i} = (\pi^{-1}(x_i)
   \times [0, r_0]) / \sim$). Note that the boundary of $D_{x_i}$ lies
   in $\Gamma_L$. After a further homotopy of $u$ we may assume that
   there exists small disks $B_i \subset D$ around each $z_i$ such
   that $u$ maps each $B_i$ to $\pm D_{x_i}$. Here, $\pm$ stands for
   the two possible orientations on $D_{x_i}$, according to whether
   $u|_{B_i}: B_i \to D_{x_i}$ preserves or reverses orientation. Put
   $S = D \setminus (\cup_{i=1}^r \textnormal{Int\,} B_i)$. Put $v =
   u|_{S}$. Clearly the image of $v$ is disjoint from $\Sigma$ and
   moreover $v$ maps the boundary of $S$ to $\Gamma_L$. After another
   homotopy of $v$, rel $\partial S$ we may also assume that the image
   of $v$ lies in $P \times \{r_0\} \subset E_{\Sigma}$. Note that
   $$\omega_{\textnormal{can}}|_{P
     \times \{r_0\}} = (\pi^* \omega_{\Sigma} + 2r dr \wedge \alpha +
   r^2 d\alpha)|_{P \times \{r_0\} } = (1-r_0^2)\pi^*
   \omega_{\Sigma},$$ hence we have:
   $$\int_{S} v^* \omega =
   (1-r_0^2) \int_{S} (\pi\circ v)^*\omega_{\Sigma}.$$ Denote by
   $\epsilon_i \in \{ -1, 1\}$ the intersection index of $u|_{B_i}$
   with $\Sigma$. We have:
   \begin{equation} \label{eq:area-mu}
      \begin{aligned}
         & \omega(A) = \int_{D} u^* \omega = \sum_{i=1}^k
         \int_{B_i}u^* \omega + \int_{S}v^*\omega = \bigl
         (\sum_{i=1}^k \epsilon_i \bigr) r_0^2 + (1-r_0^2) \int_{S}
         (\pi\circ v)^*\omega_{\Sigma}, \\
         & \mu(A) = \sum_{i=1}^k \mu([u|_{B_i}]) + \mu([v]) =
         2(\sum_{i=1}^k \epsilon_i) + \mu([v]).
      \end{aligned}
   \end{equation}
   Denote by $\mu_L: H_2(\Sigma, L) \to \mathbb{Z}$ the Maslov index
   of $L \subset \Sigma$. A simple computation shows that $\mu([v]) =
   \mu_{L}([\pi \circ v])$ (see Proposition~4.1.A
   in~\cite{Bi:Nonintersections} and its proof.) Next, note that $[\pi
   \circ v]$ in fact lies in the image of $\pi_2(\Sigma, L) \to
   H_2(\Sigma, L)$. By the monotonicity of $L$ we now get: $\int_{S}
   (\pi\circ v)^*\omega_{\Sigma} = \eta \mu_{L}([\pi \circ v])$. Using
   this and~\eqref{eq:area-mu} we deduce that $\Gamma_L \subset M$
   will be monotone if $\frac{r_0^2}{2} = (1-r_0^2) \eta$. Solving
   this equation gives $r_0^2 = \frac{2\eta}{2\eta+1}$.
\end{proof}

\begin{remnonum}
   The Lagrangian $\Gamma_L$, when viewed as a submanifold of $M
   \setminus \Sigma$, is obviously monotone too (in fact, for every
   value of $r_0$). Its minimal Maslov number (as a Lagrangian in $M
   \setminus \Sigma$), $N'_{\Gamma_L}$, satisfies $N'_{\Gamma_L} =
   N_{L}$. See~\cite{Bi:Nonintersections} for more details.
\end{remnonum}
Based on the above we can construct examples of narrow Lagrangians in
${\mathbb{C}}P^n$.

\subsubsection{Narrow Lagrangians in ${\mathbb{C}}P^n$}
\label{sbsb:narrow-cpn}
Consider $M = {\mathbb{C}}P^n$, $n \geq 3$, endowed with the following
normalization of the standard symplectic structure
$\omega'_{\textnormal{FS}} = 2 \omega_{\textnormal{FS}}$.  (The
normalization here is made so that $[\omega'_{\textnormal{FS}}] \in
H^2({\mathbb{C}}P^n;\mathbb{Z})$ is $2$ times the generator.)  Let
$\Sigma = Q^{2n-2} \subset {\mathbb{C}}P^n$ be the smooth complex
quadric hypersurface, given for example by $Q = \{ z_0^2 + \cdots +
z_n^2 = 0\}$. The Lagrangian skeleton in this case is $\Delta =
{\mathbb{R}}P^n = \{[z_0: \cdots: z_n] \mid z_i \in \mathbb{R}, \,
\forall\, i\}$.  See~\cite{Bi:Barriers} for the computation.

Let $L \subset Q^{2n-2}$ be any monotone Lagrangian (e.g. a Lagrangian
sphere), and consider $\Gamma_L \subset {\mathbb{C}}P^n$ constructed
as above. By construction, $\Gamma_L \cap {\mathbb{R}}P^n =
\emptyset$. By Corollary~\ref{cor:RP} ${\mathbb{R}}P^n$ is wide. It
follows from Corollary~\ref{cor:inters} that $\Gamma_L$ is narrow.

The same construction actually works also for $M = {\mathbb{C}}P^2$,
although $\Delta$ is not subcritical. In this case $Q \approx S^2$ and
we can take $L \subset S^2$ to be a circle which divides $S^2$ into
two disks of equal areas. The corresponding Lagrangian circle bundle
$\Gamma_L$ is a $2$-dimensional torus in ${\mathbb{C}}P^2$. The fact
that $\Gamma_L$ is monotone follows from a direct computation of
Maslov indices and areas for each of the three generators of
$\pi_2({\mathbb{C}}P^2, L) \cong \mathbb{Z}^{\oplus 3}$.  Thus we
obtain a narrow Lagrangian torus $\Gamma_L \subset {\mathbb{C}}P^2$.
We remark that $\Gamma_L$ is not symplectically equivalent to the
Clifford torus $\mathbb{T}^2_{\textnormal{clif}} \subset
{\mathbb{C}}P^2$ since the latter is wide. On the other hand, these
two tori, $\mathbb{T}_{\textnormal{clif}}$ and $\Gamma_L$ turn out to
be Lagrangian isotopic one to the other. It would be interesting to
understand the relation of this example with Chekanov's exotic
torus~\cite{Chekanov:exotic-torus} as well as with the
works~\cite{El-Po:lag-knots, Ble-Po:tori}.

\subsubsection{More examples} \label{sbsb:more-narrow} One can iterate
the Lagrangian circle bundle construction by looking at hyperplane
sections of hyperplane sections $\Sigma' \subset \Sigma \subset M$
etc. (with different choices of $\Sigma$'s as well as different
choices of $L$'s) and obtain many examples of narrow monotone tori in
${\mathbb{C}}P^n$.  It would be interesting to figure out how many of
them are symplectically non-equivalent. It would also be interesting
to understand the relation of these tori to the recent series of
pairwise non-equivalent Lagrangian tori constructed
in~\cite{Che-Schl:tori}.

\section{Open questions.}

Traditionally, the class of monotone Lagrangians has been of interest
because it provides a context in which Floer homology remains
reasonably simple to define and, simultaneously, is sufficiently rich
so as to provide a wide variety of examples.  However, the structural
rigidity properties discussed in this paper indicate that this class
is also interesting in itself.  We remark that wide monotone
Lagrangians also satisfy a form of numerical (or arithmetic) rigidity
(some results on this can be found in \cite{Bi-Co:qrel-long}).

Of course, many questions remain open at this time. An obvious issue
is whether higher order operations - beyond the module and product
structures, in particular - can be used to produce further extensions
of the results proved here.  A considerable amount of additional
technical complications are involved in setting up the machinery
needed to deal with that degree of generality so we have not pursued
this avenue here. In a different direction, it is clearly possible to
further pursue relative packing computations as well as various Gromov
radius estimates.
 
Another obvious problem is to establish the theory described here with
coefficients in $\Z$. As already mentioned in~\S\ref{Sb:quad-lag}
Remark~\ref{r:quad-z2-z}, we expect our theory to work over
$\mathbb{Z}$ however we have not rigorously checked the needed
compatibility with orientations.  Still, it is instructive to see an
example showing that this issue is important for certain applications.

Let $Q \subset {\mathbb{C}}P^{n+1}$ be a smooth complex quadric
hypersurface endowed with the symplectic structure induced from
${\mathbb{C}}P^{n+1}$. The following corollary shows that the
composition $j_{L_1} \circ i_{L_0}$ introduced in
\S\ref{sec:Lagr_int_str} does not vanish for a class of Lagrangians in
the quadric, provided that we work with $\mathbb{Z}$ (rather than
$\mathbb{Z}_2$) as the ground ring of coefficients and so - by
Theorem~\ref{cor:inters_i_j} (again with $\Z$-coefficients) - any two
Lagrangians in this class intersect.  We mark the Corollary with a
$^*$ to indicate that its proof is not 100\% rigorous.
\begin{corspec} \label{c:jcirci-quad} Let $L_0, L_1 \subset Q$ be two
   Lagrangians with $H_1(L_i; \mathbb{Z}) = 0$, $i=0,1$ and assume in
   addition that $L_0, L_1$ are relative spin (see~\cite{FO3} for the
   definition). (e.g. $L_0$ and $L_1$ are two Lagrangian spheres).
   Then, over $\mathbb{Z}$, the composition $j_{L_1} \circ i_{L_0}$
   does not vanish. In particular $L_0 \cap L_1 \neq \emptyset$.
\end{corspec}
\begin{proof}[Proof $^*$]
   As $H_1(L_i;\mathbb{Z}) = 0$, the Lagrangians $L_0, L_1$ are
   orientable, hence in view of the relative spin condition we can
   orient all the moduli spaces of disks following~\cite{FO3}.

   The condition $H_1(L_i;\mathbb{Z}) = 0$ implies that $N_{L_0} =
   N_{L_1} = 2C_Q = 2n$. Therefore in the ring $\Lambda_{0,1}$ (from
   \S\ref{sec:Lagr_int_str} ) we have $t_0=t_1$ or in other words
   $\Lambda_{0,1} \cong \Lambda_{0} \cong \Lambda_1 \cong
   \mathbb{Z}[t^{-1}, t]$, with $\deg t = -2n$. (Note again, we are
   using $\mathbb{Z}$ as the ground ring.)

   We will now use the notation from~\S\ref{Sb:quad-lag},
   Lemma~\ref{l:quad-quant-L} and Remark~\ref{r:quad-z2-z}. Recall
   that by this Lemma and this Remark we have $i_{L_0}(\alpha_0) = p - ut$,
   where $\alpha_0 \in QH_0(L_0)$ is the generator, $p \in QH_0(Q)$ is
   the class of a point and $u \in QH_{2n}(Q)$ is the fundamental
   class. Denoting by $\alpha'_n \in QH_n(L_1)$ the fundamental class
   we now have by the same lemma and remark (now applied to $L_1$):
   $$j_{L_1} \circ i_{L_0} (\alpha_0) = (p-ut)*\alpha'_n =
   -\alpha'_nt - \alpha'_n t = -2\alpha'_n t \neq 0.$$ By
  Theorem~\ref{cor:inters_i_j}, $L_0 \cap L_1 \neq \emptyset$.
\end{proof}

\

We conclude with two conjectures which, we believe, have a significant
structural significance for the understanding of the subject so that
we want to make them explicit here. We recall that here, as all along
the paper, we include in the definition of a monotone Lagrangian
submanifold the condition $N_{L}\geq 2$.

\begin{cnj}\label{cnj:dichotomy}
   Any monotone Lagrangian submanifold is either narrow or wide.
\end{cnj}

\begin{cnj}\label{cnj:intersection}
   In a point invertible manifold, if two monotone Lagrangian
   submanifolds do not intersect, then at least one of them is narrow.
\end{cnj}

\begin{remnonum}\label{rem:conj}
   a. As shown in Theorem \ref{thm:floer=0-or-all} the dichotomy {\em
     narrow - wide} can be established in many relevant cases and we
   can prove it in a few more. It is true, for example, for $n=\dim L
   \leq 3$ (at least when $L$ admits a perfect Morse function).

   There is an equivalent statement of the conjecture which is worth
   indicating here.  Recall the map $p_{\ast}: Q^{+}H(L)\to QH(L)$
   induced by the change of coefficients $\La^{+}\to \La$ and that we
   denote by $IQ^{+}(L)$ its image. It is easy to see that the kernel
   of $p_{\ast}$ consists precisely of the torsion ideal $T^{+}(L)$ of
   $Q^{+}H(L)$,
   $$T^{+}(L)=\{z\in Q^{+}H(L) : \ \exists\ m\in\N ,\ t^{m}z=0\}~.~$$
   It is a simple exercise to see that $L$ is wide iff $T^{+}(L)=0$
   and $L$ is narrow iff $T^{+}(L)=Q^{+}H(L)$. Thus the {\em wide -
     narrow} conjecture is equivalent to showing that the torsion
   ideal of any monotone Lagrangian can only be $0$ or coincide with
   the entire ring.

   b. The difficulty in proving the second conjecture is caused by the
   following phenomenon (see also Theorem \ref{cor:inters_i_j}).
   First, notice that the result immediately follows if one can show
   that there is a constant $C$ and a class $\alpha\in QH(M;\Lambda)$
   (with $M$ the ambient symplectic manifold) so that for any
   monotone, non-narrow Lagrangian $L\subset M$ and any
   $\phi\in\widetilde{Ham}(M)$ one has:
   \begin{equation}\label{eq:cj_first}
      {\rm depth}_{L}(\phi)-C \leq \sigma(\alpha,\phi) 
      \leq {\rm height}_{L}(\phi)+C.
   \end{equation}
   By Lemma \ref{lem:action} i, if $\alpha$ is invertible (for
   example, $\alpha=[pt]$ for a point invertible manifold) the left
   inequality~\eqref{eq:cj_first} follows because $\alpha$ acts
   non-trivially on $QH(L)$. The second inequality is implied by the
   second point of the same Lemma if one can show $\alpha\in
   Im(i_{L})$.  Finding a class $\alpha$ which satisfies both
   properties is however quite non-trivial.  Notice that in a
   point-invertible manifold not only is the left inequality
   in~\eqref{eq:cj_first} satisfied for $\alpha=pt$ but we can also
   deduce the estimate
   $\sigma^{\ast}([\omega^{n}],\phi):=\inf\{\sigma([M]+s^{-1}x,\phi)\
   | \ x\in QH(M;\Lambda)\}\leq {\rm height}_{L}(\phi)+k$ where $x\in
   Q^{+}H(M)$, $s$ is the Novikov variable in
   $\Gamma=\Z_{2}[s^{-1},s]$, and $\sigma^{\ast}(\omega^{n},\phi)$ is
   by definition the infimum given above (this notation is justified
   because it coincides with the {\em co-homological} spectral
   invariant of the class $[\omega^{n}]$).  It is clear from the
   ``triangle inequality'' that $\sigma([M],\phi)\geq
   \sigma([pt],\phi)$ but it is in general not easy to show that
   $\sigma^{\ast}([\omega^{n}],\phi)\geq \sigma([pt],\phi)$.

\end{remnonum}


\begin{thebibliography}{FOOO}

\bibitem[Alb1]{Alb:PSS}
P.~Albers.
\newblock A {L}agrangian {P}iunikhin-{S}alamon-{S}chwarz morphism and two
  comparison homomorphisms in {F}loer homology.
\newblock Preprint (2005), can be found at
  \verb+http://lanl.arxiv.org/pdf/math/0512037+.

\bibitem[Alb2]{Alb:extrinisic}
P.~Albers.
\newblock On the extrinsic topology of {L}agrangian submanifolds.
\newblock {\em Int. Math. Res. Not.}, 2005(38):2341--2371, 2005.

\bibitem[ALP]{ALP}
M.~Audin, F.~Lalonde, and L.~Polterovich.
\newblock Symplectic rigidity: {L}agrangian submanifolds.
\newblock In M.~Audin and J.~Lafontaine, editors, {\em Holomorphic curves in
  symplectic geometry}, volume 117 of {\em Progr. Math.}, pages 271--321,
  Basel, 1994. Birkh\"{a}user Verlag.

\bibitem[BC1]{Bar-Cor:NATO}
J.-F. Barraud and O.~Cornea.
\newblock Homotopic dynamics in symplectic topology.
\newblock In P.~Biran, O.~Cornea, and F.~Lalonde, editors, {\em Morse theoretic
  methods in nonlinear analysis and in symplectic topology}, volume 217 of {\em
  NATO Sci. Ser. II Math. Phys. Chem.}, pages 109--148, Dordrecht, 2006.
  Springer.

\bibitem[BC2]{Bar-Cor:Serre}
J.-F. Barraud and O.~Cornea.
\newblock Lagrangian intersections and the {S}erre spectral sequence.
\newblock {\em Annals of Mathematics}, 166:657--722, 2007.

\bibitem[BC3]{Bi-Ci:Stein}
P.~Biran and K.~Cieliebak.
\newblock Lagrangian embeddings into subcritical {S}tein manifolds.
\newblock {\em Israel J. Math.}, 127:221--244, 2002.

\bibitem[BC4]{Bi-Ci:closed}
P.~Biran and K.~Cieliebak.
\newblock Symplectic topology on subcritical manifolds.
\newblock {\em Commentarii Mathematici Helvetici}, 76:712--753, 2002.

\bibitem[BC5]{Bi-Co:jcirci}
P.~Biran and O.~Cornea.
\newblock In preparation.

\bibitem[BC6]{Bi-Co:Yasha-fest}
P.~Biran and O.~Cornea.
\newblock Lagrangian quantum homology.
\newblock Preprint (2008).

\bibitem[BC7]{Bi-Co:qrel-long}
P.~Biran and O.~Cornea.
\newblock Quantum structures for {L}agrangian submanifolds.
\newblock Preprint (2007). Can be found at
  \verb+http://arxiv.org/pdf/0708.4221+.

\bibitem[Bea]{Beau:quant}
A.~Beauville.
\newblock Quantum cohomology of complete intersections.
\newblock {\em Mat. Fiz. Anal. Geom.}, 2(3-4):384--398, 1995.

\bibitem[Bir1]{Bi:Barriers}
P.~Biran.
\newblock Lagrangian barriers and symplectic embeddings.
\newblock {\em Geom. Funct. Anal.}, 11(3):407--464, 2001.

\bibitem[Bir2]{Bi:Nonintersections}
P.~Biran.
\newblock Lagrangian non-intersections.
\newblock {\em Geom. Funct. Anal.}, 16(2):279--326, 2006.

\bibitem[BP]{Ble-Po:tori}
L.~Blechman and L.~Polterovich.
\newblock In preparation.

\bibitem[Buh1]{Bu:products}
L.~Buhovsky.
\newblock Multiplicative structures in lagrangian floer homology.
\newblock Preprint, can be found at math.SG/0608063.

\bibitem[Buh2]{Bu:packing}
L.~Buhovsky.
\newblock One explicit construction of a relative packing.
\newblock Preprint, can be found at arXiv:0803.2774v3 [math.SG].

\bibitem[Buh3]{Bu:TSn}
L.~Buhovsky.
\newblock Homology of {L}agrangian submanifolds in cotangent bundles.
\newblock {\em Israel J. Math.}, 143:181--187, 2004.

\bibitem[Che]{Chekanov:exotic-torus}
Y.~Chekanov.
\newblock Lagrangian tori in a symplectic vector space and global
  symplectomorphisms.
\newblock {\em Math. Z.}, 223(4):547--559, 1996.

\bibitem[Chi]{Chiang:RP3}
R.~Chiang.
\newblock New {L}agrangian submanifolds of {${\mathbb{C}}P^n$}.
\newblock {\em Int. Math. Res. Not.}, 2004(45):2437--2441, 2004.

\bibitem[Cho1]{Cho:Clifford}
C.-H Cho.
\newblock Holomorphic discs, spin structures, and floer cohomology of the
  {C}lifford torus.
\newblock {\em Int. Math. Res. Not.}, 2004(35):1803--1843, 2004.

\bibitem[Cho2]{Cho:products}
C.-H Cho.
\newblock Products of {F}loer cohomology of torus fibers in toric {F}ano
  manifolds.
\newblock {\em Comm. Math. Phys.}, 260(3):613--640, 2005.

\bibitem[CL1]{Cor-La:Cluster-1}
O.~Cornea and F.~Lalonde.
\newblock Cluster homology.
\newblock Preprint (2005), can be found at
  \verb+http://xxx.lanl.gov/pdf/math/0508345+.

\bibitem[CL2]{Cor-La:Cluster-2}
O.~Cornea and F.~Lalonde.
\newblock Cluster homology: an overview of the construction and results.
\newblock {\em Electron. Res. Announc. Amer. Math. Soc.}, 12:1--12, 2006.

\bibitem[CO]{Cho-Oh:Floer-toric}
C.-H Cho and Y.-G. Oh.
\newblock Floer cohomology and disc instantons of {L}agrangian torus fibers in
  {F}ano toric manifolds.
\newblock {\em Asian J. Math.}, 10(4):773--814, 2006.

\bibitem[CR]{Co-Ra:Morse-Novikov}
O.~Cornea and A.~Ranicki.
\newblock Rigidity and gluing for {M}orse and {N}ovikov complexes.
\newblock {\em J. Eur. Math. Soc.}, 5(4):343--394, 2003.

\bibitem[CS]{Che-Schl:tori}
Y.~Chekanov and F.~Schlenk.
\newblock Monotone lagrangian tori in {$\mathbb{R}^{2n}$}, {$\mathbb{C}P^n$}
  and products of spheres.
\newblock In preparation.

\bibitem[EG]{El-Gr:convex}
Y.~Eliashberg and M.~Gromov.
\newblock Convex symplectic manifolds.
\newblock In {\em Several complex variables and complex geometry Part 2 (Santa
  Cruz, CA, 1989)}, volume~52 of {\em Proc. Sympos. Pure Math.}, pages
  135--162, Providence, RI, 1991. Amer. Math. Soc.

\bibitem[Eli]{El:Stein}
Y.~Eliashberg.
\newblock Topological characterization of {S}tein manifolds of dimension
  {$>2$}.
\newblock {\em Internat. J. Math.}, 1(1):29--46, 1990.

\bibitem[EP1]{El-Po:lag-knots}
Y.~Eliashberg and L.~Polterovich.
\newblock The problem of {L}agrangian knots in four-manifolds.
\newblock In {\em Geometric topology (Athens, GA, 1993)}, volume~2 of {\em
  AMS/IP Stud. Adv. Math.}, pages 313--327. Amer. Math. Soc., Providence, RI,
  1997.

\bibitem[EP2]{En-Po:rigid-subsets}
M.~Entov and L.~Polterovich.
\newblock Rigid subsets of symplectic manifolds.
\newblock Preprint (2007), can be found at
  \verb+http://arxiv.org/pdf/0704.0105+.

\bibitem[FOOO]{FO3}
K.~Fukaya, Y.-G. Oh, H.~Ohta, and K.~Ono.
\newblock Lagrangian intersection {F}loer theory - anomaly and obstruction.
\newblock Preprint.

\bibitem[Fra]{Fr:msc}
U.~Frauenfelder.
\newblock Gromov convergence of pseudo-holomorphic disks.
\newblock {\em Journal of Fixed Point Theory and its Applications}, (3):??--??,
  2008.

\bibitem[Fuk]{Fu:Morse-homotopy}
K.~Fukaya.
\newblock Morse homotopy and its quantization.
\newblock In {\em Geometric topology (Athens, GA, 1993)}, volume 2.1 of {\em
  AMS/IP Stud. Adv. Math.}, pages 409--440, Providence, RI, 1997. Amer. Math.
  Soc., Amer. Math. Soc.

\bibitem[GH]{Gr-Ha:alg-geom}
P.~Griffiths and J.~Harris.
\newblock {\em Principles of algebraic geometry}.
\newblock Pure and Applied Mathematics. John Wiley \& Sons, Inc., New York,
  1978.

\bibitem[Gro]{Gr:phol}
M.~Gromov.
\newblock Pseudoholomorphic curves in symplectic manifolds.
\newblock {\em Invent. Math.}, 82(2):307--347, 1985.

\bibitem[KO]{Kw-Oh:discs}
D.~Kwon and Y.-G. Oh.
\newblock Structure of the image of (pseudo)-holomorphic discs with totally
  real boundary condition. {W}ith an appendix by {J}ean-{P}ierre {R}osay.
\newblock {\em Comm. Anal. Geom.}, 8(1):31--82, 2000.

\bibitem[Laz1]{Laz:decomp}
L.~Lazzarini.
\newblock Decompostion of a {$J$}-holomorphic curve.
\newblock Preprint. Can be downloaded at \\
  \verb+http://www.math.jussieu.fr/~lazzarin/articles.html+.

\bibitem[Laz2]{Laz:discs}
L.~Lazzarini.
\newblock Existence of a somewhere injective pseudo-holomorphic disc.
\newblock {\em Geom. Funct. Anal.}, 10(4):829--862, 2000.

\bibitem[McD]{McD:Ham_U}
D.~McDuff.
\newblock Hamiltonian {$S^{1}$}-manifolds are uniruled.
\newblock Preprint, can be found at arXiv:0706.0675v2 [math.SG].

\bibitem[MS]{McD-Sa:Jhol-2}
D.~McDuff and D.~Salamon.
\newblock {\em {$J$}-holomorphic curves and symplectic topology}, volume~52 of
  {\em American Mathematical Society Colloquium Publications}.
\newblock American Mathematical Society, Providence, RI, 2004.

\bibitem[Oh1]{Oh:spec-inv-4}
Y.-G. Oh.
\newblock Mini-max theory, spectral invariants and geometry of the
  {H}amiltonian diffeomorphism group.
\newblock Preprint (2002), can be found at
  \verb+http://xxx.lanl.gov/pdf/math/0206092+.

\bibitem[Oh2]{Oh:HF1}
Y.-G. Oh.
\newblock Floer cohomology of {L}agrangian intersections and pseudo-holomorphic
  disks. {I}.
\newblock {\em Comm. Pure Appl. Math.}, 46(7):949--993, 1993.

\bibitem[Oh3]{Oh:spectral}
Y.-G. Oh.
\newblock Floer cohomology, spectral sequences, and the {M}aslov class of
  {L}agrangian embeddings.
\newblock {\em Internat. Math. Res. Notices}, 1996(7):305--346, 1996.

\bibitem[Oh4]{Oh:relative}
Y.-G. Oh.
\newblock Relative floer and quantum cohomology and the symplectic topology of
  lagrangian submanifolds.
\newblock In C.~B. Thomas, editor, {\em Contact and symplectic geometry},
  volume~8 of {\em Publications of the Newton Institute}, pages 201--267.
  Cambridge Univ. Press, Cambridge, 1996.

\bibitem[Oh5]{Oh:spec-inv-2}
Y.-G. Oh.
\newblock Construction of spectral invariants of {H}amiltonian paths on closed
  symplectic manifolds.
\newblock In J.~Marsden and T.~Ratiu, editors, {\em The breadth of symplectic
  and {P}oisson geometry. {F}estschrift in honor of {A}lan {W}einstein}, volume
  232 of {\em Progr. Math.}, pages 525--570, Boston, MA, 2005. Birkh\"{a}user
  Boston.

\bibitem[Oh6]{Oh:spec-inv-3}
Y.-G. Oh.
\newblock Spectral invariants, analysis of the {F}loer moduli space, and
  geometry of the {H}amiltonian diffeomorphism group.
\newblock {\em Duke Math. J.}, 130(2):199--295, 2005.

\bibitem[Oh7]{Oh:spec-inv-1}
Y.-G. Oh.
\newblock Lectures on {F}loer theory and spectral invariants of {H}amiltonian
  flows.
\newblock In P.~Biran, O.~Cornea, and F.~Lalonde, editors, {\em Morse theoretic
  methods in nonlinear analysis and in symplectic topology}, volume 217 of {\em
  NATO Sci. Ser. II Math. Phys. Chem.}, pages 321--416, Dordrecht, 2006.
  Springer.

\bibitem[PSS]{PSS}
S.~Piunikhin, D.~Salamon, and M.~Schwarz.
\newblock Symplectic {F}loer-{D}onaldson theory and quantum cohomology.
\newblock In C.~B. Thomas, editor, {\em Contact and symplectic geometry},
  volume~8 of {\em Publications of the Newton Institute}, pages 171--200.
  Cambridge Univ. Press, Cambridge, 1996.

\bibitem[Sch]{Sc:action-spectrum}
M.~Schwarz.
\newblock On the action spectrum for closed symplectically aspherical
  manifolds.
\newblock {\em Pacific J. Math.}, 193(2):419--461, 2000.

\bibitem[Sei1]{Se:pi1}
P.~Seidel.
\newblock {$\pi_1$} of symplectic automorphism groups and invertibles in
  quantum homology rings.
\newblock {\em Geom. Funct. Anal.}, 7(6):1046--1095, 1997.

\bibitem[Sei2]{Se:graded}
P.~Seidel.
\newblock Graded {L}agrangian submanifolds.
\newblock {\em Bull. Soc. Math. France}, 128(1):103--149, 2000.

\bibitem[Sei3]{Se:TSn}
P.~Seidel.
\newblock Exact {L}agrangian submanifolds of {$T^*S^n$} and the graded
  {K}ronecker quiver.
\newblock In Y.~Eliashberg S.~Donaldson and M.~Gromov, editors, {\em Different
  faces of geometry}, Int. Math. Ser. (N. Y.), pages 349--364. Kluwer/Plenum,
  New York, 2004.

\bibitem[Vit]{Vi:generating-functions-1}
C.~Viterbo.
\newblock Symplectic topology as the geometry of generating functions.
\newblock {\em Math. Ann.}, 292(4):685--710, 1992.

\end{thebibliography}
\end{document}